\documentclass[11pt]{article}
\usepackage{lmodern}
\usepackage[T1]{fontenc}
\usepackage{comment}
\usepackage{amsthm}

\usepackage[margin=3cm]{geometry}
\usepackage{amsmath,amssymb}
\usepackage[mathscr]{euscript}
\usepackage[usenames,dvipsnames]{color}
\usepackage[colorlinks]{hyperref}
\usepackage{pstricks}
\usepackage{graphicx,caption}
\usepackage{dsfont}
\usepackage[all]{xy}
\usepackage{tabularx}
\usepackage{bbm}
\usepackage{todonotes}

\usepackage{tikz}
\usepackage{tikz-3dplot}
\usepackage{amsmath}
\usetikzlibrary{arrows.meta,decorations.markings}
\usetikzlibrary{3d,decorations.markings}

\tikzset{
  ->-/.style={decoration={
    markings,
    mark=at position #1 with {\arrow{>}}},postaction={decorate}},
  ->-/.default=0.5
}

\newcommand{\dd}{\mathrm{d}}

\newcommand{\C}{ \mathbb{C} } 
\newcommand{\E}{ \mathbb{E} } 
\newcommand{\M}{ \mathbb{M} } 
 
\newcommand{\R}{ \mathbb{R} } 
\renewcommand{\P}{ \mathbb{P} } 
 
\newcommand{\Z}{ \mathbb{Z} } 

\newcommand{\Lc}{ \mathcal{L} }

\newcommand{\id}{ \mathrm{id} }

\newcommand{\Ad}{ \mathrm{Ad} }
\newcommand{\gfrak}{ \mathfrak{g} }

\newcommand{\eqlaw}{ \stackrel{\Lc}{=} }

\newcommand{\ind}{ \mathrm{ind} } %index of critical point in Morse (not used)
\newcommand{\argmax}{ \mathrm{argmax} }
\newcommand{\argmin}{ \mathrm{argmin} }

\numberwithin{equation}{section}

\newtheorem{thm}{Theorem}[section]

\newtheorem{corollary}[thm]{Corollary}

\newtheorem{definition}[thm]{Definition}
\newtheorem{example}[thm]{Example}
\newtheorem{lemma}[thm]{Lemma}

\newtheorem{rmk}[thm]{Remark}

\newcommand{\Addresses}{
  \bigskip

 Reda Chhaibi\par\nopagebreak \textsc{Laboratoire Jean Alexandre Dieudonn\'e and Institut Universitaire de France, Universit\'e C\^ote d'Azur, Campus Sciences, Parc Valrose, 28 avenue Valrose,
06108 Nice Cedex 02, France}\par\nopagebreak
  \par\nopagebreak
  \textit{E-mail address:} \texttt{reda.chhaibi@univ-cotedazur.fr} 

  \bigskip

  Nguyen Viet Dang
   \par\nopagebreak\textsc{IRMA and Institut Universitaire de France, Universit\'e de Strasbourg, 7 rue Ren\'e
Descartes, 67084 Strasbourg Cedex, France}\par\nopagebreak
  \textit{E-mail address:} \texttt{nvdang@unistra.fr}
  
   \bigskip
   
 Yannick Guedes Bonthonneau
 \par\nopagebreak\textsc{D\'epartement de math\'ematiques et applications, \'Ecole normale sup\'erieure, CNRS, 45 rue d'Ulm, 75230 Paris Cedex 05, France}\par\nopagebreak
  \textit{E-mail address:} \texttt{yguedesbonthonne@dma.ens.fr}
  
   \bigskip

 Gabriel Rivi\`ere\par\nopagebreak \textsc{Laboratoire de Math\'ematiques Jean Leray (UMR CNRS 6629), Nantes Universit\'e, 2 rue de la Houssini\`ere, 44322 Nantes Cedex 03, France}\par\nopagebreak
  \par\nopagebreak
  \textit{E-mail address:} \texttt{gabriel.riviere@univ-nantes.fr}
  
   \bigskip

  Tat Dat T\^o 
  \par\nopagebreak\textsc{Institut de Math\'ematiques de Jussieu-Paris Rive Gauche, Sorbonne Universit\'e, 4 place Jussieu,
75252 Paris Cedex 05, France}\par\nopagebreak
  \textit{E-mail address:} \texttt{tat-dat.to@imj-prg.fr}

}

\title{The Yang-Mills measure on surfaces via Morse theory}
\author{R. Chhaibi, N.V. Dang, Y. Guedes Bonthonneau, G. Rivi\`ere and T.D. T\^o}
\date{\today}
\begin{document}
\maketitle

\begin{abstract}
We introduce a Morse theoretical approach to the construction of the Yang--Mills measure on the space of connections of a compact Riemannian surface. This provides a direct continuous version of this measure which was previously obtained through lattice approximations by Chevyrev in the case of the flat torus and by one of the authors and Nohra for general compact Riemannian surfaces. 

The starting point is the new notion of a Morse gauge together with the resolution of random cohomological equations associated to Morse--Smale vector fields. This is achieved by improving exponential convergence to equilibrium results for Morse--Smale gradient flows that were obtained by two of the authors in the context of the study of Ruelle spectra and by Jia, Stewart and Sverak in the context of simplified models from fluid mechanics. Combining these random solutions with the data given by the Morse complex, we introduce a free Yang-Mills measure on space of connections and, using classical tools from stochastic differential equations, we show how to make sense of holonomies for random connections along a large class of curves. 

Finally, by setting a proper conditioning of this free measure through these random holonomies, we define the Yang--Mills measure and we compute its partition function together with the law of random holonomies with respect to this measure, recovering the formulas from the works of Migdal, Witten and L\'evy. 
\end{abstract}

\newpage

% --------------------------------------------------------------------
% Table of contents
\setcounter{tocdepth}{2}
%\medskip
\hrule
\tableofcontents
\medskip
\medskip
\hrule

\newpage

% --------------------------------------------------------------------

\section{Introduction}

Let $(\Sigma,h)$ be a smooth ($\mathcal{C}^\infty$), compact, connected and oriented Riemannian surface and let $G$ be a connected and compact linear Lie group with Lie algebra $\mathfrak{g}$. A smooth connection $\nabla$ on the \emph{trivial bundle} $P:=\Sigma\times G\rightarrow \Sigma$ can be identified with a smooth one-form $A$ that is $\mathfrak{g}$-valued, i.e. $\nabla=d+A$ with $A\in\Omega^1(\Sigma,\mathfrak{g})$. The curvature of the connection is then defined as
$$
F(A)= dA+A\wedge A\ \in\ \Omega^2(\Sigma,\mathfrak{g}),
$$
and it gives rise to the so-called Yang--Mills functional on the space of connections:
$$
\mathcal{S}_{\text{YM}}(A):=\frac12\int_{\Sigma}\left\|\star_h F(A)\right\|^2_{\mathfrak{g}}d\upsilon,
$$
where $\upsilon$ is the Riemannian volume associated with $h$, $\star_h:\Omega^2(\Sigma,\mathfrak{g})\rightarrow \mathcal{C}^\infty(\Sigma,\mathfrak{g})$ is the Hodge star map and where $\|.\|_{\mathfrak{g}}$ is a norm\footnote{The reader may keep in mind the example where $G=\operatorname{SU}(N)$ and the scalar product we shall use on the Lie algebra $\mathfrak{su}(N)=\mathfrak{g}=T_{\text{Id}}SU(N)$ is given by the trace $ \left\|B \right\|^2:= \operatorname{Tr}\left(B^*B \right) =-\operatorname{Tr}(B^2),$
since $B^*=-B$ for anti-Hermitian matrices.} on $\mathfrak{g}$ that is invariant under the adjoint representation, $\text{Ad}_\mathrm{g}(\mathfrak{a})=\mathrm{g}\mathfrak{a}\mathrm{g}^{-1}$, $(\mathrm{g},\mathfrak{a})\in G\times\mathfrak{g}$. The classical Yang--Mills theory is concerned among other things with the study of the critical points of this functional which is a classical and difficult topic in nonlinear partial differential equations (especially in dimension $4$). In dimension $2$, it is a celebrated result of Atiyah and Bott~\cite{AtiyahBott83} that this functional is a perfect $G$-equivariant Morse function with infinitely many critical values that can be computed explicitly. In fact, one of the key features of this problem is its invariance by gauge transformations. The action of $\mathrm{g}\in C^\infty(\Sigma,G)$ on connections by gauge transformation is indeed defined by the map 
$$
\nabla\mapsto \nabla_{\mathrm{g}}:=\mathrm{g}\nabla\mathrm{g}^{-1}\quad \Leftrightarrow\quad A\mapsto A_{\mathrm{g}}:=\mathrm{g}A\mathrm{g}^{-1}-(d\mathrm{g})
\mathrm{g}^{-1},
$$
and one has $F(A_\mathrm{g})=\mathrm{g}F(A)\mathrm{g}^{-1}.$ In particular, for all $\mathrm{g}\in\mathcal{C}^\infty(\Sigma,G)$, one has
$\mathcal{S}_{\text{YM}}(A)=\mathcal{S}_{\text{YM}}(A_\mathrm{g})$ and critical points come into families.

Besides studying this functional, other quantities playing a central role to describe spaces of connections are the so-called holonomies. Indeed these geometric objects allow us to identify these spaces (modulo gauge choice) as subsets of $G$--valued morphisms on the space of loops of $\Sigma$ (modulo conjugation by $G$) -- see~\cite{Levyphd, LevyMaster} for details. These holonomies are defined as follows. Given a $\mathcal{C}^1$ curve $\gamma:[0,1]\rightarrow\Sigma$, one can then define the parallel transport of $\nabla$ along $\gamma$ by solving the following ordinary differential equation:
\begin{equation}\label{e:parallel-transport}
 \frac{d\mathrm{g}_\gamma(t)}{dt}+\mathrm{g}_\gamma(t)A_{\gamma(t)}(\gamma'(t))=0,\quad\mathrm{g}_\gamma(0)=\text{Id}_G.
\end{equation}
The holonomy along $\gamma$ is then defined as $\operatorname{Hol}_\gamma(A):=\mathrm{g}_\gamma(1)\in G$. These geometric quantities turn out to be central when considering the quantum Yang--Mills theory where one aims at defining a proper path integral. Recall that Yang--Mills theory appeared in the 1950s as a gauge-theoretical foundation for the strong and weak nuclear forces in physics and that it has since become a standard topic in mathematical physics. As for the classical theory, the case of quantum Yang--Mills in dimension $4$ where one aims at constructing a non-perturbative, continuum Yang–Mills measure consistent with the Osterwalder--Schrader axioms is a notorious difficult problem. Even the case of dimension $3$ is still not fully understood despite significant progress that was obtained using stochastic quantization. Yet, as in the classical theory, the $2$-dimensional case is much better understood and the probabilistic constructions of this Quantum Field Theory (QFT) has a rich history at the crossroads of mathematics and physics starting from the seminal works of Migdal~\cite{Migdal}.

On the quantum side and from the mathematical perspective, one is in fact interested in giving proper mathematical sense to the following formal measure on the space of connections and for a large class of test functions $\Phi$:
\begin{equation}\label{e:def-formal-YM}
\int_{\mathcal{A}/\mathcal{G}}\Phi(A) e^{-\mathcal{S}_{\text{YM}}(A)}\mathrm{d} A,
\end{equation}
where $\mathcal{A}/\mathcal{G}$ is a subspace of connections in $\mathcal{D}^\prime(\Sigma,T^*\Sigma\otimes\mathfrak{g})$ modulo the action of the gauge group $\mathcal{G}$. In particular, the case $\Phi=1$ corresponds to the so-called \textbf{partition function} in Quantum Field Theory and one expects more generally to encompass geometric relevant test functions of the form $\Phi(A)=\Psi(\text{Hol}_\gamma(A))$ in view of computing the law of holonomies along a given path $\gamma$. This means that one should be able to define holonomies along curves even if $A$ may have low regularity almost surely. To circumvent these issues, this question was tackled by defining a measure on the space of $G$-valued morphisms on spaces of loops thanks to the above identification through holonomies. In fact, early works by Gross, King, and Sengupta~\cite{GKS}, by Driver on the plane~\cite{Dri89} and later by Sengupta on general surfaces~\cite{Sengupta97} constructed this measure as a gauge-invariant random holonomy law. A complementary viewpoint, developed in particular by L\'evy~ \cite{Levyphd} on two-dimensional Markovian holonomy fields, emphasizes the structure of the Yang–Mills measure as a Markov process indexed by loops, with transition kernels given by heat kernels on the structure group $G$. L\'evy's work provides a deep understanding of the Wilson loop observables and the Makeenko–Migdal equations from a probabilistic and geometric perspective.

However, a more analytic question remained open for some time: can one construct the Yang–Mills measure \emph{directly} as a random \emph{distributional connection}, i.e., as a probability measure on a space of $\mathfrak{g}$-valued distributional $1$-forms? This is subtle because one has to fix a proper gauge otherwise the measure does not look absolutely continuous with respect to any free field measure. Moreover, one knows from standard properties of the white noise that such connections have their curvature lying in $H^{\kappa}$ for every $\kappa<-1$ where $H^{\kappa}$ denotes the standard Sobolev scale. In particular, the action functional is infinite on this rough set of fields and holonomies are also hard to define. For the flat torus $\mathbb{T}^2$, Chevyrev recently solved these problems and he achieved this construction by defining a Yang--Mills measure as scaling limit of some lattice gauge model from discrete Yang--Mills theory~\cite{Chevyrev}.
Even more recently and relying on the Morse gauge which is introduced in the present work, Nohra and one of the authors have pushed this program further~\cite{DangNohra}. In this reference, they construct the Yang–Mills measure as a random distributional $1$-form on compact surfaces of arbitrary genus equipped with an arbitrary smooth area form as scaling limit of discrete connections coming from lattice gauge theory on $\Sigma$. Actually, they prove the following universality theorem: their continuum Yang--Mills measure arises as the scaling limit of a wide class of lattice gauge theories (including Wilson, Manton, and Villain actions) on any compact surface. 

The present work generalizes Chevyrev's result to arbitrary closed surfaces $\Sigma$ of genus $g \ge 0$, complementing and extending in other directions the results of \cite{DangNohra}. Indeed, while \cite{DangNohra} focuses on universality from lattice approximations, we use dynamical systems methods to provide a direct and self-contained continuum construction of the Yang–Mills measure as a random distributional connection without passing through a lattice limit. Our approach requires overcoming several obstacles that will be described precisely later on. Among others, we construct \emph{a novel global and singular Morse gauge} on a general surface and prove it addresses the issue of Gribov copies. Using the Morse complex, we also give a careful treatment of the cohomological zero-modes which contains information on the topology of the moduli space of flat connections appearing in Witten's seminal work~\cite{Witten1}. 
%Together with \cite{DangNohra}, this is the first construction of the Yang–Mills measure on an arbitrary closed surface as a random distributional connection, providing a continuum counterpart to the lattice universality results of \cite{DangNohra}.
An informal summary of our main results is as follows.

\medskip

\textbf{Informal main Theorem.} \emph{Given $a\in\Sigma$, there exists a measure $\mu_{\operatorname{YM}}$ supported in $H^{-1-}_{\operatorname{loc}}(\Sigma\setminus\{a\},T^*\Sigma\otimes\mathfrak{g})$ whose partition function verifies Witten's formula~\cite{Witten1} and for which one can define and compute laws of random holonomies for a large class of piecewise $\mathcal{C}^1$ curves.}

\medskip

For the sake of consistency, we will in fact compute explicitly the laws of these random holonomies for a large class of small loops and we will verify that they satisfy Migdal's formulas~\cite{Migdal, Levyphd}. The computation along more general curves would require heavier combinatorial work that we do not pursue here. We emphasize that the proof given in the present work is self-contained and in particular completely independent of the results and of the lattice gauge approach from~\cite{DangNohra} even if both works rely on the same choice of a Morse gauge. To the best of our knowledge, \cite{DangNohra} and the present work provide the first construction of the Yang–Mills measure on an arbitrary closed surface as a random distributional connection. The fact that we need to fix a base point $a$ on the surface comes from our construction through Morse theory and from the existence of an attractor for the induced gradient dynamics. In view of constructing the Yang--Mills measure, we will in fact make use of tools and ideas from Morse theory and the theory of hyperbolic dynamical systems that we will combine with methods from probability. Morse theory will be used to fix a dynamical gauge, a key ingredient for defining the Yang--Mills measure. More precisely, we will proceed in three main steps of different nature and of independent interest:
\begin{enumerate}
 \item \textbf{Classical Yang--Mills theory in $2d$ through the lens of hyperbolic dynamical systems}. In this first step, we will revisit some aspects of the classical Yang--Mills theory through the lens of dynamical systems theory. We will explain how to put the Yang--Mills functional into a normal form using a Morse--Smale gradient flow on $\Sigma$ associated with a Morse function whose maximum is reached at the point $a$. We refer to this step as defining a \emph{Morse gauge} on the space of connections. Among other things, we will show how to express a connection in terms of its curvature and of the resolvent of the corresponding gradient vector field. This is achieved by building on ideas arising in the study of Ruelle spectra in hyperbolic dynamical systems and by adapting there some ideas that were initially used in the context of certain models from fluid mechanics by Jia, Stewart and Sverak. The use of these tools from dynamical systems and PDE is one of the main novelties of the present work both in the classical and in the quantum parts.
 
 \item \textbf{Solving random cohomological equations in $2d$}. Given a $\mathfrak{g}$-valued white noise on $\Sigma$, we show how to solve random cohomological equations associated with these Morse--Smale vector fields. We then prove that the random solutions that we obtain can be integrated along a large class of curves $\gamma$ and that they give rise to a $\mathfrak{g}$-valued reparametrized Brownian motion. This allows us to define stochastic analogues of~\eqref{e:parallel-transport} that can be solved using the standard theory of stochastic differential equations.
 
 \item \textbf{Quantum Yang--Mills theory in $2d$}. With these tools at hand, we define first a free boundary Yang--Mills measure by using the solutions to these random cohomological equations and a correction through the associated Morse complex to take into account the topology of $\Sigma$. We explain how to make sense of random holonomies along a large class of curves and we use holonomies around small loops surrounding $a$ to couple the Gaussian part of our measure with the correction given by the Morse complex. Since the coupling is nonlinear, this allows us to define a Yang--Mills measure which is \textbf{non Gaussian} and whose partition function can be computed explicitly and verifies Witten's formula. Finally, we illustrate our construction by computing the law of random holonomies along small loops of the surface and we show that they verify the formulas appearing in L\'evy's works.
\end{enumerate}

\section{Main results}

We now describe in more detail the main results pertaining the three main issues discussed at the end of the introduction and leading to the continuous construction of the Yang--Mills measure via Morse theory. Recall that our goal in the present paper is twofold. On the one hand, we introduce a novel gauge named \textbf{Morse gauge} and we explain in which sense this gauge slices the space $\mathcal{A}$ of connections and provides 
a nice model for the orbit space $\mathcal{A}/\mathcal{G}$ of connections modulo gauge choice. This concerns classical Yang--Mills theory. On the other hand, we give an application of this new gauge to quantum Yang--Mills theory by constructing a Gibbs measure on the \emph{infinite dimensional orbit space} $\mathcal{A}/\mathcal{G}$.

\subsection{Classical Yang--Mills theory in \texorpdfstring{$2d$}{2d} through the lens of hyperbolic dynamical systems} 
We let $f:\Sigma\rightarrow\mathbb{R}$ be a smooth ($\mathcal{C}^\infty$) Morse function which has exactly $2g+2$ critical points where $g$ is the genus of $\Sigma$. The Morse assumption means that the critical points are nondegenerate. Denoting by $\text{Crit}(f):=\{a_j:1\leqslant j\leqslant 2g+2\}$ the set of critical points, we also suppose that
$$
f(a_1)<f(a_2)<\ldots<f(a_{2g+1})<f(a_{2g+2}).
$$
The points $(a_j)_{2\leq j\leq 2g+1}$ are the saddle points of $f$, equivalently its critical points of index $1$. We can define an adapted metric on $\Sigma$ (which may be different from $h$) such that the metric is Euclidean in Morse charts near critical points and such that the corresponding gradient vector field $V$ has the Morse--Smale property. See~\S\ref{s:morse} for details. We denote by $\varphi_f^t:\Sigma\rightarrow\Sigma$ the corresponding gradient flow which is the simplest example of a hyperbolic (or Axiom A) dynamical system in the sense of Smale~\cite{Smale67}. One can then introduce the unstable manifolds of each saddle point of $f$:
$$
\forall\ 2\leqslant j\leqslant 2g+1,\quad W^u(a_j):=\left\{x\in\Sigma:\ \lim_{t\rightarrow+\infty}\varphi_f^{-t}(x)=a_j\right\}.
$$
These are embedded submanifolds that are diffeomorphic to $\mathbb{R}$ whose closure is a circle containing $a_{2g+2}$ and one can define the corresponding current of integration $[W^u(a_j)]$. It is a classical result due to Laudenbach~\cite{Laudenbach} that these currents of integration are generators of the De Rham cohomology in degree $1$.

Consider now the following transport equation:
\begin{equation}\label{e:def-gT}
\partial_t\mathrm{g}+\left(\mathcal{L}_V+A(V)\right) \mathrm{g}=0,\quad \mathrm{g}(t=0)= \text{Id}_G,
\end{equation}
where $\mathcal{L}_V=d\iota_V+\iota_Vd$ is the Lie derivative along the gradient vector field $V$. Equivalently, if we let $\tilde{\mathrm{g}}_t:=\varphi_f^{t*}(\mathrm{g}_t)$, it solves the parallel transport equation
\begin{equation}\label{e:holonomy-equation}
\partial_t \tilde{\mathrm{g}}+\varphi_f^{t*}(A(V))\tilde{\mathrm{g}}=0\ \Leftrightarrow\ \frac{d\tilde{\mathrm{g}}_t^{-1}}{dt}-\tilde{\mathrm{g}}_t^{-1}\varphi_f^{t*}(A(V))=0,\quad \tilde{\mathrm{g}}(t=0)=\text{Id}_G,
\end{equation}
which can be compared with~\eqref{e:parallel-transport}. In other words, $\mathrm{g}_t$ describes the parallel transport induced by $\nabla:=d+A$ along the flowlines of the gradient vector field. For every $T>0$, we set 
$$
A_T:=A_{\mathrm{g}_T^{-1}}=\mathrm{g}_T^{-1}A\mathrm{g}_T+\mathrm{g}_T^{-1}d\mathrm{g}_T,
$$
and our first main theorem reads:
\begin{thm}[Classical Morse gauge]\label{t:normalform}
%Let $G$ be a compact Lie group and $(\Sigma,h)$ be a $C^\infty$ closed compact surface with a Morse function $f$ s.t. $V=\nabla f$ satisfies the Smale transversality condition and such that $V$ is linear in Morse charts near critical points. For every saddle point $a$ of $f$, we denote by $W^u(a)$ the corresponding unstable curve and by $[W^u(a)]$ the corresponding current of integration. 
%Then for any $C^\infty$ connection $\nabla=d+A$, there exists a family $(g_T)_{T\geqslant 0}$ in $C^\infty(\Sigma,G)$ with $g_T(\min(f))=\id_G$ 
For any $A\in\Omega^1(\Sigma,\mathfrak{g})$, there exist $\mathrm{g}_\infty\in L^{\infty}(\Sigma,G)$ such that the following holds: 
\begin{enumerate}
\item $F(A_T)$ converges (in the sense of currents) to $F_\infty:=\mathrm{g}_\infty^{-1}F(A)\mathrm{g}_\infty$  as $T\rightarrow\infty$ and 
$$
\mathcal{S}_{\operatorname{YM}}(A)=\int_{\Sigma}\|\star_hF_\infty\|_{\mathfrak{g}}^2\upsilon,
$$
\item there exists $\beta_\infty\in\mathcal{D}^\prime(\Sigma,T^*\Sigma\otimes\mathfrak{g})\cap\ker(\iota_V)$ such that $\int_0^T\varphi_f^{-t*}(\iota_V(F_\infty)dt$ converges (in the sense of currents) to $\beta_\infty$ and $A_T$ converges (in the sense of currents) to
\begin{equation}\label{eq:Morsegauge1}
A_\infty:=\sum_{j=2}^{2g+1}\mathfrak{b}_j[W^u(a_j)]+\beta_\infty,
\end{equation}
where $\mathfrak{b}_j:=\int_{W^s(a_j)}A\in\mathfrak{g}$;
\item $\beta_\infty$ and $\mathrm{g}_\infty$ are smooth on $W^u(a_1):=\Sigma\setminus\cup_{j=2}^{2g+1}\overline{W^u(a_j)}$ and, for every $1\leq q< 2$, $\beta_\infty$ belongs to $L^q(\Sigma,T^*\Sigma\otimes\mathfrak{g})$;
\item one has\footnote{Here, $[a]$ denotes the current of integration on the point $a$.} $dA_\infty=F_\infty-\left(\int_\Sigma F_\infty\right)[a_{2g+2}]$;
%\item $ \lim_{T\rightarrow +\infty} g_T^{-1}d g_T+g_T^{-1}A g_T = A_\infty$ in the sense of currents where $A_\infty$ is smooth on $\Sigma\setminus \cup_{a\in \text{saddle}} W^u(a)$,
%\item the limiting connection $A_\infty$ has the form:
%\begin{equation}\label{eq:Morsegauge1}
%A_\infty=\iota_V\beta+    \sum_{a\in \text{Crit}(f)_1}m_a [W^u(a)]
%\end{equation}
%\item where $\beta\in \mathcal{D}^{1}_\Gamma(\Sigma,\mathfrak{g}) $ has \textbf{wave front set} contained in the union $\overline{\cup_{a\in \text{saddle}} N^*W^u(a) }$ of conormal bundles of unstable curves, $m_a\in \mathfrak{g}$, in particular $\iota_VA_{\infty}=0$,
\item for any $\gamma:[0,1]\rightarrow\Sigma$ of class $\mathcal{C}^1$ such that $\gamma(0),\gamma(1)\in W^u(a_1)$, one has 
$$
\lim_{T\rightarrow +\infty}\operatorname{Hol}_{\gamma}(A_T)=\mathrm{g}_\infty(\gamma(1))^{-1}\operatorname{Hol}_{\gamma}(A)\mathrm{g}_\infty(\gamma(0));
$$
\item for any gauge equivalent $A_1,A_2 \in \Omega^1(\Sigma,\mathfrak{g})$, i.e. $A_2=\mathrm{g}^{-1}d\mathrm{g}+\mathrm{g}^{-1}A_1\mathrm{g}$ for some $\mathrm{g}\in C^\infty(\Sigma,G)$, the limiting 
connections $A_{1,\infty},A_{2,\infty}$ satisfy
$$A_{2,\infty}=\mathrm{g}(a_1)^{-1}A_{1,\infty} \mathrm{g}(a_1)\quad\text{on}\quad W^u(a_1),$$
and, under the assumption of the previous item, 
$$
\lim_{T\rightarrow +\infty}\operatorname{Hol}_{\gamma}(A_{2,T})=\lim_{T\rightarrow +\infty}\mathrm{g}(a_1)^{-1}\operatorname{Hol}_{\gamma}(A_{1,T})\mathrm{g}(a_1).
$$
\end{enumerate}
\end{thm}
Recall that $d[W^u(a)]=0$ and also from~\cite[Prop.~7.7]{DR19} that $\iota_V([W^u(a)]=0$ for every critical point $a$ of index $1$. In particular, one has 
$$
\boxed{\iota_V(A_\infty)=0,}
$$ 
and $A_\infty$ is said to be in Morse gauge.
We would like to point out that some closely related idea of dynamical gauge 
already appeared in the literature under the name of Anosov gauge~\cite{Schiavina1, Schiavina2}, we wonder if a similar result as Theorem~\ref{t:normalform} would make sense or could be established for the Anosov gauge.
We also emphasize that $A_\infty$ has low regularity so that its curvature does not a priori make sense due to the (ill-defined) quadratic terms $A_\infty\wedge A_\infty$. Despite that, the fourth item shows that this quadratic term is formally concentrated at the maximum of $f$. This illustrates how singular the Morse gauge is and this is somewhat reminiscent of the concentration of curvature phenomena discovered by Uhlenbeck in $4$--dimensional gauge theories~\cite{DonaldsonKronheimer}. Observe also that the regularity statement in the third item is somehow sharp in terms of $L^p$ regularity as $dA_\infty\notin H^{-1}$. The fifth item ensures in some sense the existence of the holonomy for the limit connection despite the low regularity of $A_\infty$. Finally, the last item tells us in which (weak) sense we can think of $A_\infty$ as a choice of gauge for $A$ in its gauge orbit. In summary, this Theorem provides a global gauge for our problem at the expense of losing the regularity of $A$ to take into account the topology of $\Sigma$ and its proof is given in \S\ref{s:Morse-gauge} building on the tools from~\S\ref{s:morse} and~\S\ref{s:spectralgap}.

More generally, any connection having the form given by equation~\eqref{eq:Morsegauge1} with $\beta_\infty\in\ker(\iota_V)$ is said to be in the \textbf{Morse gauge}. The main conceptual input of our work is in fact the introduction of this gauge which is a far reaching generalization of the classical axial gauge to general surfaces. It slices the space $\mathcal{A}$ of connections and it provides a nice model of the orbit space $\mathcal{A}/\mathcal{G}$. Moreover, this gauge gives in some sense a normal form for the Yang--Mills functional and any connection in the Morse gauge decomposes as a sum of some part that contributes to the functional which contains all the curvature contribution and some \emph{zero modes part} of the connection. The zero modes part is supported on unstable curves and already appeared in the representations of the Morse complex in terms of currents
~\cite{Laudenbach, HarveyLawson, DR21}. An important remark about why our gauge fixing does not suffer from the problem of Gribov ambiguities is that \emph{our limit gauge transformations are discontinuous} along the wedge of circles $\cup_{2\leqslant j\leqslant 2g+1}\overline{W^u(a_j)}$. Therefore the classical Gribov argument yielding a topological obstruction to global gauge fixing no longer applies since it relies on calculating homotopy groups of the group $\mathcal{C}^0(\Sigma,G)$ of continuous gauge transformations. An important feature of the Morse gauge is as follows: if the connection $\nabla=d+A$ we started with is flat, then $F_\infty=0$ and the connection $A_\infty$ is \emph{concentrated} on the union of unstable curves. This is a non abelian analogue of the quasi--isomorphism between the de Rham complex of smooth forms and the Morse complex of currents as shown in the works of Laudenbach~\cite{Laudenbach}.  

As we shall see, the existence of the limit curvature $F_\infty$ follows from elementary geometric and analytic considerations but it only provides limited regularity. Hence, the convergence of the integral term (as well as that of $A_T$) in the second item of our theorem requires a very careful analysis. If $F_\infty$ was smooth (or at least $\mathcal{C}^k$ with $k$ large enough), this would follow from the spectral analysis of gradient flows as it was developed by two of the authors in~\cite{DR19}. In order to overcome this regularity issue, we will refine the results from this reference by adapting some ideas that were used by Jia, Stewart and Sverak~\cite{JiaStewartSverak19} in the context of certain simplified models from fluid mechanics (the so-called De Gregorio equation). This is the content of Theorem~\ref{t:contraction} below which is the main analytical statement behind this first theorem and its upcoming probabilistic version, namely Theorem~\ref{t:random-cohomological}. Roughly speaking, Theorem~\ref{t:contraction} is an exponential decay of correlations result for the free transport equation associated with~\eqref{e:def-gT} and stated in a sharp families of weighted Lebesgue spaces. Gradient flows are in fact part of the larger family of Axiom A dynamical systems as it was primarily defined by Smale in the 1960s~\cite{Smale67}. The spectral study of Axiom A systems at work here has a long and well-established tradition going back to the works of Bowen and Ruelle in the 1970s~\cite{Bowen08, Baladi00} and revisited during the last thirty years using sophisticated tools from functional analysis. We refer the reader to~\cite{Baladi18, Demers_Kiamari_Liverani, Lefeuvre25} for books describing these recent progress. 
Finally, we note that we will say very little about the convergence of $\mathrm{g}_T$. Following the methods from~\cite{HarveyLawson01, DR19}, we could probably show the convergence to some limit $\mathrm{g}_\infty$ in a much stronger sense than what we will claim. 

\subsection{Solving random cohomological equations in \texorpdfstring{$2d$}{2d}}

Regarding the expression~\eqref{eq:Morsegauge1} of the Morse gauge and the formal definition~\eqref{e:def-formal-YM} of the Yang--Mills measure, it is natural to pick $\beta_\infty$ by choosing $F_\infty$ being a white noise. Hence, we let $\xi$ be a $\mathfrak{g}$-valued white noise on $\Sigma$ and we denote by $(\Omega,\mathcal{B},\mathbb{P})$ the corresponding probability space. Recall that $\xi$ is an element of $L^2(\Omega, H^{-1-\kappa}(\Sigma,\mathfrak{g}))$ for every $\kappa>0$ and that, for every $\psi$ in $\mathcal{C}^\infty(\Sigma,\mathfrak{g})$, one has
$$
\mathbb{E}\left(\left|\left\langle\xi, \psi\right\rangle\right|^2\right)=\left\|\psi\right\|_{L^2(\Sigma,\mathfrak{g})}^2.
$$
See Section~\ref{s:solve_random_connection_equation} for more details. In view of defining a random connection using the Morse gauge, we are thus left with solving the following random cohomological equation
\begin{equation}\label{e:random-cohomological-equation-intro}
 \mathcal{L}_V\left(A\right)=\xi\iota_V(\upsilon).
\end{equation}
This is the content of our second main theorem:
\begin{thm}[Random cohomological equations]\label{t:random-cohomological} With the above conventions and for every $\kappa>0$, the sequence
$$
\int_0^T\varphi_f^{-t*}\left(\xi\iota_V(\upsilon)\right)
$$
converges as $T\rightarrow +\infty$ in $L^2(\Omega, H^{-1-\kappa}(\Sigma,T^*\Sigma\otimes\mathfrak{g}))$ to some limit $A$ which satisfies almost surely
$$
\iota_V(A)=0\quad\text{and}\quad dA=\xi\upsilon-\left(\int_{\Sigma}\xi\upsilon\right)[a_{2g+2}].
$$
In particular, $A$ solves~\eqref{e:random-cohomological-equation-intro} almost surely.
\end{thm}
We denote the solution to this equation by $A=\mathcal{L}_V^{-1}\left(\xi\iota_V(\upsilon)\right)$ even if the inverse is only formal and has to be understood in a probabilistic sense. Again the difficulty in this result is that the regularity of $\xi$ is too rough to apply the spectral theory from~\cite{DR19}. Indeed the functional spaces adapted to the spectral analysis of $\mathcal{L}_V$ (and more generally for Axiom A systems) have to enjoy some anisotropic Sobolev regularity, meaning that the right hand side of~\eqref{e:random-cohomological-equation-intro} should have positive Sobolev regularity in certain direction even if it may have negative Sobolev regularity along others. Here, $\xi$ is not better than $H^{-1-}$ almost surely and we need to proceed in a different way to deal with this isotropic low regularity. One more time, the key ingredient is the decay of correlation property given in Theorem~\ref{t:contraction} which can be combined with probabilistic arguments. The resolution of cohomological equations is a standard topic in dynamical systems but, to the best of our knowledge, nothing seems to be known on the resolution of random cohomological equations. Hence, on top of its applications to Yang-Mills theory, Theorem~\ref{t:random-cohomological} also seems to provide the first example of resolution of such random equations.  The proof of this result is given in~\S\ref{s:solve_random_connection_equation} and let us now explain how it will allow us to pick connections at random and to build the Yang--Mills measure. Let us also mention that a resolution of~\eqref{e:random-cohomological-equation-intro} in sharper anisotropic spaces of distributions is given in~\cite{DNsemiclassical} by Nohra and one of us.

\subsection{Quantum Yang--Mills theory in \texorpdfstring{$2d$}{2d}}

Regarding the above results, we define a {\bf random connection} as
\begin{equation}\label{e:random-connection-intro}
A(\xi,\mathrm{b}):= \mathcal{L}_V^{-1}\left(\xi\iota_V(\upsilon)\right)+\sum_{j=2}^{2g+1}\log \left(\mathrm{b}_j \right) [W^u(a_j)],
\end{equation}
where $\xi:=\xi(\omega)$ is the white noise with probability space $(\Omega,\mathcal{B},\mathbb{P})$ and where $\mathrm{b}=(\mathrm{b}_j)_{j=2}^{2g+1}$ belongs to $G^{2g}$ and is distributed according to the normalized Haar measure $\mu_G^{\otimes 2g}$. Here $\log(\mathrm{b})$ is measurable and defined as an element in $\mathfrak{g}$ (with minimal norm) such that $\mathrm{b}=\exp(\log(\mathrm{b}))$. 

Before stating our result defining the Yang--Mills measure, let us gather a few extra properties of $\xi$. In fact, given a compact subset $K$ of $\Sigma$, one can find a sequence of smooth functions $(\psi_n)_{n\geq 1}$ such that $\psi_n$ converges to $\mathbf{1}_K$ in every $L^p$-space with $1\leq p<\infty$. One can verify that $(\xi\psi_n)_{n\geq 1}$ is a Cauchy sequence in $L^2(\Omega,H^{-1-\kappa}(\Sigma,\mathfrak{g}))$ for every $\kappa>0$ and the limit is independent of the smoothing sequence. It is denoted by $\xi_K$ and it is almost surely supported in $K$. If we denote by $\sigma(\xi_K)$ the $\sigma$-algebra generated by $\xi_K$ (plus the zero measure sets), one can verify that, for $K_1\subset K_2$, $\sigma(\xi_{K_1})$ is contained in $\sigma(\xi_{K_2})$. See Lemma~\ref{r:whitenoise-multiplication} for more details. One can then introduce 
$$
K_n:=\{x:d(x,a_{2g+2})\geqslant 1/n\}
$$ 
and denote the corresponding sequence of subalgebras by $(\mathcal{B}_{n})_{n\geqslant 1}$. This collection of $\sigma$-algebras is a filtration meaning that it is a nondecreasing sequence of sub $\sigma$-algebras of $\mathcal{B}$. One of the key observation in view of defining the Yang--Mills measure is that, for every $n\geqslant 1$ and for every $\psi\in\Omega^1(\Sigma,\mathfrak{g})$ supported in $K_n$,
\begin{equation}\label{e:measurability-test-function}
(\omega,\mathrm{b})\in\Omega\times G^{2g}\mapsto \langle A(\xi,\mathrm{b}),\psi\rangle\in \mathbb{R}\ \text{is}\ \mathcal{B}_n^G-\text{measurable},
\end{equation}
where the $\sigma$-algebra $\mathcal{B}_n^G$ is the lift to $\Omega\times G^{2g}$ of $\mathcal{B}_n$. Again, $(\mathcal{B}_n^G)_{n\geqslant 1}$ is a \emph{filtration} of the $\sigma$-algebra generated by $\mathcal{B}$ and the Borel subsets of $G^{2g}$. See Lemma~\ref{l:admissible-random-variable} for more precise statements. In the following, we will set $\mathcal{B}_{\infty}^G:=\cup_{n\geqslant 1}\mathcal{B}_n^G$  which is a priori only an algebra as it has no reason to be stable under countable union.

We are almost ready to state our two main Theorems and we just need to define a family of small curves for which we will be able to compute the law of random holonomies. More precisely, we say that a compact subset $D\subset \Sigma\setminus\text{Crit}(f)$ is an \emph{admissible disk} if it is homeomorphic to a disk and if $\partial D:=D\setminus\mathring{D}$ is the concatenation of two $\mathcal{C}^1$ curves that are both transverse to the gradient flowlines and that intersect at most one unstable manifold $(W^u(a_j))_{j=1}^{2g+1}$ or one stable manifold $(W^s(a_j))_{j=1}^{2g+1}$. See figure~\ref{fig:admiss_disc}
and
Section~\ref{ss:random-holonomies-YM} for more details.

\begin{figure}
    \centering
    \includegraphics[scale=0.35]{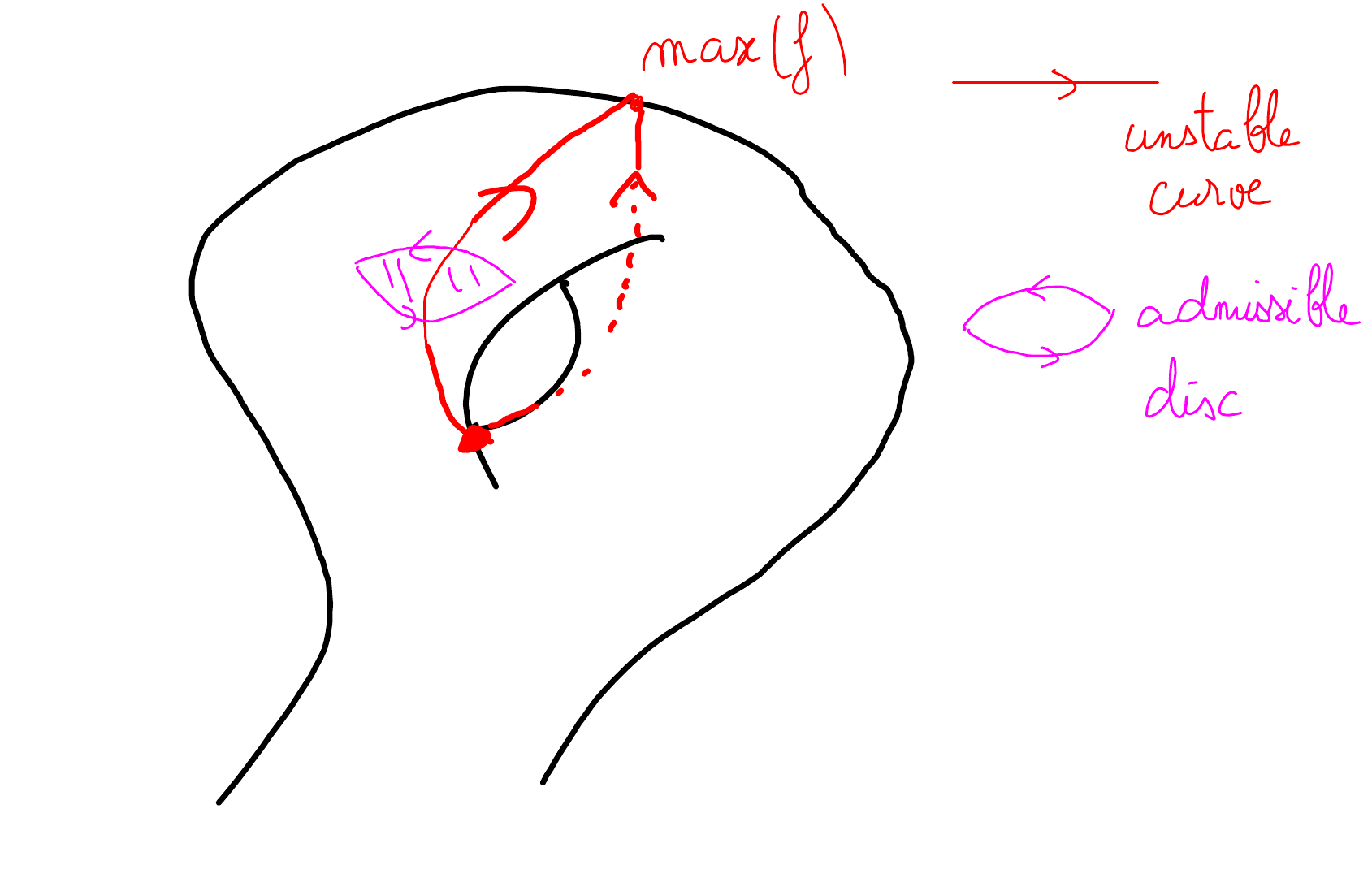} 
    \caption{An admissible disk.}
    \label{fig:admiss_disc}
\end{figure}

With these conventions, we can define a notion of random holonomy on the probability space $\Omega\times G^{2g}$ and a simplified statement reads as follows:
\begin{thm}[Random holonomies]\label{t:def-YM-general}
There exists a finitely additive functional 
$$
\mathbb{M}_{\operatorname{YM}}:\mathcal{B}_\infty^G\rightarrow \mathbb{R}_+
$$
such that the following holds
 \begin{enumerate}
  \item for all $n\geqslant 1$, $\mathbb{M}_{\operatorname{YM}}$ is a measure on $(\Omega\times G^{2g},\mathcal{B}_n^G)$;
  \item  for every continuous and piecewise $\mathcal{C}^1$ curve $\gamma:[0,1]\rightarrow \Sigma$ that is everywhere transverse to $V$, one can find a random variable $\mathbf{Hol}(\gamma)\in G$ such that, for every $n\geqslant 1$,  
  $$
  \gamma([0,1])\subset K_n\ \Longrightarrow\ \mathbf{Hol}(\gamma)\ \text{is}\ \mathcal{B}_{n}^G-\text{measurable}
  ,$$ 
  and, for every admissible disk $D$, the law of $\mathbf{Hol}(\partial D)$ is given by
  $$
  \boxed{\mathbb{M}_{\operatorname{YM}}\left(\mathbf{Hol}(\partial D)\in d\mathrm{g}\right)=\left(\sum_{\rho\in\widehat{G}}\frac{e^{-\frac{c_2(\rho)}{2}\upsilon(\Sigma\setminus D)}}{\left(\operatorname{dim} V_\rho\right)^{2g-1}}\chi_\rho(\mathrm{g})\right)\left(\sum_{\rho\in\widehat{G}}\operatorname{dim} V_\rho e^{-\frac{c_2(\rho)}{2}\upsilon(D)}\chi_\rho(\mathrm{g})\right)\mu_G(d\mathrm{g}),}
  $$
  where $\rho:G\rightarrow \operatorname{GL}(V_\rho)$ runs over the equivalent classes $\widehat{G}$ of unitary irreducible representations of $G$, $c_2(\rho)$ is the Casimir number of $\rho$ and $\chi_\rho$ is the character of $\rho$.
 \end{enumerate}
 In particular, the partition function is given by
  $$
  \boxed{Z_{\operatorname{YM}}:=\mathbb{M}_{\operatorname{YM}}(\Omega\times G^{2g})=\sum_{\rho\in\widehat{G}}e^{-\frac{c_2(\rho)}{2}\upsilon(\Sigma)}\left(\operatorname{dim} V_\rho\right)^{2-2g}.}
  $$
\end{thm}

We refer to~\S\ref{ss:def_lie_group} for a brief recollection on irreducible representations of $G$ and their relation to the heat kernel on $G$. In this statement, by finitely additive functional, we mean that, for every finite collection $(B_1,\ldots B_J)$ of elements of $\mathcal{B}_\infty^G$ that are pairwise disjoint, $\mathbb{M}_{\text{YM}}(\cup_jB_j)=\sum_j\mathbb{M}_{\text{YM}}(B_j)$. This restriction comes from the fact that $\mathcal{B}_\infty^G$ is not a priori a $\sigma$-algebra even if all elements of the filtration $(\mathcal{B}_n^G)_{n\geqslant 1}$ are and even if $\mathbb{P}_{\text{YM}}$ is a measure for each element of the filtration. This theorem allows us to define random holonomies as well as a measure on $(\Omega\times G^{2g},\mathcal{B}_n^G)$ in such a way that the law for random holonomies from~\cite{Levyphd} can be recovered at least for small loops\footnote{As already alluded, the computation for a general curve is more involved from the combinatorial point of view and we restrict ourselves to such loops for simplicity. See~\S\ref{ss:random-holonomies-YM} for details.}. We emphasize that, despite the fact that the measure depends on our choice of Morse function, the law of our random holonomies depends only on $\upsilon$. Hence, we obtain an holonomy process for a large class of piecewise $\mathcal{C}^1$ curves verifying the expected formulas for the Yang--Mills measure. Yet, as we shall see in \S\ref{sss:holonomy_convention} and \S\ref{ss:random-holonomies-YM}, the statement will be much more precise in the sense that \textbf{the map $\mathbf{Hol}(\gamma)$ will be precisely given by the solution to a stochastic version of~\eqref{e:parallel-transport} where $A$ has been replaced by the random connection $A(\xi,\mathrm{b})$}. The proof and a large part of our work consists in fact in making sense of these stochastic equations and their solutions for a large class of curves. 

Theorem~\ref{t:def-YM-general} defines a measure for measurable sets belonging to a filtration that do not see the maximum value of $f$ in a certain sense. This is due to the specific nature of this point which is an attractor for the gradient flow used to define our random connections. Despite that, this information is sufficient to define the Yang--Mills measure directly on a space of distributions and this requires us to introduce $\mathcal{S}(\Sigma\setminus\{a_{2g+2}\},T^*\Sigma\otimes\mathfrak{g})$ the space of smooth $\mathfrak{g}$-valued $1$-forms on $\Sigma$ all of whose derivatives vanish at any order on $a_{2g+2}$. This space contains $\Omega^1_c(\Sigma\setminus\{a_{2g+2}\},\mathfrak{g})$ as a dense subset. Its topological dual is denoted by $\mathcal{S}^\prime(\Sigma\setminus\{a_{2g+2}\},T^*\Sigma\otimes\mathfrak{g})$ and it consists of the space of $\mathfrak{g}$-valued currents on $\Sigma\setminus\{a_{2g+2}\}$ with moderate growth at $a_{2g+2}$. These are exactly the currents on $\Sigma\setminus\{a_{2g+2}\}$ that can be extended to the whole $\Sigma$. See~\S\ref{sss:proof-theo-YM} for more details. Our last main Theorem states the existence of the Yang--Mills measure on such spaces of distributions and it reads as follows:

\begin{thm}[Existence of the Yang--Mills measure]\label{t:def-YM} There exist a Hilbert space $\mathcal{H}$ satisfying
$$
\mathcal{S}(\Sigma\setminus\{a_{2g+2}\},T^*\Sigma\otimes\mathfrak{g})\subset \mathcal{H}\subset\mathcal{S}^\prime(\Sigma\setminus\{a_{2g+2}\},T^*\Sigma\otimes\mathfrak{g}),
$$
with continuous embeddings and a measure $\mu_{\operatorname{YM}}$ of total mass $Z_{\operatorname{YM}}$ on $\mathcal{H}$ endowed with the Borel $\sigma$-algebra such that the following holds
\begin{enumerate}
 \item for $\mu_{\operatorname{YM}}$ almost every $A$, $\iota_V(A)=0$;
 \item for every $p\geqslant 1$, for every $\psi_1,\ldots,\psi_p$ in $\Omega^1_c(\Sigma\setminus\{a_{2g+2}\},\mathfrak{g})$ and for every Borel set of $\R^p$, one has
 $$
 \mu_{\operatorname{YM}}\left(\left\{(\langle A,\psi_1\rangle,\ldots,\langle A,\psi_p\rangle)\in B\right\}\right)=\mathbb{M}_{\operatorname{YM}}\left(\left\{(\langle A(\xi,\mathrm{b}),\psi_1\rangle,\ldots,\langle A(\xi,\mathrm{b}),\psi_p\rangle)\in B\right\}\right);
 $$
 \item for every $\kappa>0$ and for $\mu_{\operatorname{YM}}$ almost every $A$, $A\in H_{\operatorname{loc}}^{-1-\kappa}\left(\Sigma\setminus\{a_{2g+2}\},T^*\Sigma\otimes\mathfrak{g}\right).$
\end{enumerate}
\end{thm}

Contrary to $\mathbb{M}_{\text{YM}}$, $\mu_{\text{YM}}$ is a $\sigma$-additive measure and not only finitely additive. The Hilbert space $\mathcal{H}$ is in fact a weighted Sobolev space of $\Sigma\setminus\{a_{2g+2}\}$ with negative Sobolev regularity and with bounded growth at $a_{2g+2}$. Yet, as this requires the introduction of more notations, we do not discuss the precise definition here. See~\S\ref{sss:proof-theo-YM} for more precise statements. The typical Sobolev regularity reached in our constructions is maybe not optimal in view of the results from~\cite{Chevyrev} which obtains $\mathcal{C}^{0-}$ regularity on the flat torus and of the results from~\cite{DangNohra} which show some anisotropic regularity of random connections on $W^u(a_1)$ on general compact surfaces. Despite that, it is worth noting that one cannot expect better than Sobolev $H^{-1/2-}$ regularity with this approach. This is due to the presence of the currents of integration $([W^u(a_j)])_{2\leqslant j\leqslant 2g+1}$ which have precisely this regularity. Yet, as it will be clear from our construction, one can expect that our random connections have more regularity along the flow lines and less in a transverse direction to the flow. This interesting and subtle question, which is addressed in~\cite{DNsemiclassical}, requires substantial extra work  that we will not discuss in this article.

In order to prove these last two theorems, we will first need to define the \emph{free boundary Yang--Mills measure}
$\mu_{\mathrm{YM}}^{\rm{free}}$
on the space of random connections or more generally on the Borel sets of $\Omega\times G^{2g}$. It is naturally associated with our definition~\eqref{e:random-connection-intro} of a random connection in Morse gauge:
\begin{definition}\label{def:free_YM} The free boundary Yang-Mills measure is defined as
$$
\int_{\ker(\iota_V)  } \Phi(A) \mu_{\operatorname{YM}}^{\operatorname{free}}(dA) :=\int_{\Omega\times G^{2g}}\Phi\left(A(\xi,\mathrm{b})\right)\P(d\omega)\mu_G^{\otimes 2g}(d\mathrm{b}),
$$
where $A(\xi,\mathrm{b})$ was defined in~\eqref{e:random-connection-intro} and where $\Phi:H^{-1-\kappa}(\Sigma,T^*\Sigma\otimes\mathfrak{g})\rightarrow\C$, $\kappa>0$ is a bounded and measurable function.
\end{definition}

\begin{comment}
The free boundary Yang-Mills measure $\mu_{\mathrm{YM}}^{\rm{free}}$ is the rigorous definition of the   Yang--Mills functional integral on a punctured surface $\Sigma\setminus \{a\}$ where our randomly chosen connections 
have no constraints at the puncture. This also means that the random connections charged by   $\mu_{\mathrm{YM}}^{\rm{free}}$ have no relation to global connections on the closed surface $\Sigma$ since any connection on $\Sigma$ restricted to the punctured surface $\Sigma\setminus \{a\}$ should have near identity holonomy along very small loops around $a$. 
\end{comment}
It follows from Theorem~\ref{t:random-cohomological} that a connection $A$ chosen randomly under this probability measure is almost surely in $H^{-1-\kappa}(\Sigma,T^*\Sigma\otimes\mathfrak{g})$ for every $\kappa>0$. The simplicity of $\mu_{\mathrm{YM}}^{\rm{free}}$ comes from the absence of any coupling between the components of the connection $A$ coming from the white noise $\xi$ and from the one coming from the unstable components $([W^u(a_j)])_{j=2}^{2g+1}$. Equivalently, the variables carrying the information on the curvature are independent from the ones carrying the topology of $\Sigma$. However, a connection randomly chosen under $\mu_{\rm{YM}}^{\rm{free}}$ cannot come from the closed surface $\Sigma$ since its random holonomy along small loops around $a_{2g+2}$ has no reason to be close to $\text{Id}_G$. In view of fixing this issue, we will perform a singular conditioning of $\mu_{\mathrm{YM}}^{\rm{free}}$ that will \emph{impose} 
the holonomy at $a_{2g+2}$ to be equal to $\text{Id}_G$. More concretely, we will blow--up our surface $\Sigma$ at $a_{2g+2}$ yielding a surface with boundary
$\mathcal{S}$ which is diffeomorphic to $\Sigma$ minus a disc (see e.g. Figure \ref{fig:Morse_flow}) and we will require that $\mathbf{Hol}(\partial\mathcal{S})$ is close to the identity. Indeed, one should think of random connections $A$ under the free boundary Yang--Mills measure $\mu_{\rm{YM}}^{\rm{free}}$ as connections living on the bordered surface $\partial \mathcal{S}$ chosen  
randomly under the Yang--Mills measure but without any constraint on the restriction of $A$ on the boundary $\partial\mathcal{S}$. This blow-up procedure explains somehow the choices of $\sigma$-algebras appearing earlier. 

In summary, we will construct the Yang--Mills measure on closed surfaces starting from some free boundary Yang--Mills measure which lives on the blow--up of the initial surface $\Sigma$ at the maximum of some Morse function $f$: this is the probability law of a Gaussian random connection plus some independent singular connection supported by the union of unstable curves of some Morse gradient flow. Then we recover the Yang--Mills measure on the closed surface $\Sigma$ by \emph{nonlinear conditioning} of the free boundary Yang--Mills measure in a way which is strongly reminiscent of the work of Sengupta in~\cite{Sengupta97}.
We would like to insist on two facts: first, the Yang--Mills measure we obtain on the closed surface after conditioning is \textbf{no longer Gaussian. And this non Gaussianity holds true on any surface, be it the sphere $\mathbb{S}^2$ or the torus $\mathbb{T}^2$}. In fact, it is not even Gaussian for the Maxwell theory, i.e. the case of $G=U(1)$.
Secondly, our random connections on the closed surface $\Sigma$ \textbf{can come from different topological types of bundles} over $\Sigma$. 
We refer to section~\ref{s:MaxwellMorse} for a detailed discussion of this fact in the $U(1)$ case where our measure defines random connections living on line bundles of different Chern classes.
Forgetting about the Morse gauge, our methods through random holonomy processes also bear some strong similarities with the early works of L\'evy on this topic~\cite[Chapter 2]{Levyphd}. Finally, we emphasize that, along our construction and besides defining random holonomy processes for large class of curves, we will also define a novel class of Yang--Mills observables
$$W_\gamma:\Omega\times G^{2g}\rightarrow\mathfrak{g}, $$
which are indexed by curves $\gamma$ belonging to some general class of curves seemingly optimal in terms of regularity and which are the random analogues of the classical variable $\int_\gamma A$.

\subsection{Literature review on mathematical 2D gauge theory}
In the physics literature, the study of the Yang--Mills measure in $2$ dimensions goes back to the work of Migdal~\cite{Migdal} in 1975 on the plane or planar domains. His approach was discrete and in the spirit of lattice gauge theories as described by Polyakov in~\cite{Polyakov}.
His construction was then extended to surfaces of any genus in the groundbreaking works of Rusakov~\cite{Rusakov2} and Witten~\cite{Witten1, Witten2} with relation to many topics in mathematics such as 
the Reidemeister torsion, volumes of character varieties, the Verlinde formula for conformal blocks and many other mathematical topics~\cite{Witten1}~\cite{Witten2}. 
From the mathematical perspective, the probabilistic construction of the Yang--Mills measure as a \textbf{holonomy process} was studied by Driver~\cite{Dri89}, 
 Gross--King--Sengupta~\cite{GKS} on $\mathbb{R}^2$ and later generalized by Fine~\cite{Fine}, 
 Sengupta~\cite{Sengupta97} and finally by L\'evy~\cite{Levyphd} on general surfaces.
 Except for the work of Lévy, these authors rely on the idea that the Yang--Mills measure becomes Gaussian once we fix the axial gauge (this is also called the temporal gauge in PDE theory) which is an old idea from physics which goes back to the seminal works of Kummer~\cite{Kummer} and Arnowitt--Fickler~\cite{ArnFick} in the 1960s. One of the nice properties of the axial gauge is that we do not need ghosts to fix the gauge and quantize the theory since 
 the Fadeev--Popov determinant is a constant for the axial gauge. This is reminiscent of our observation that, in our work, we formally make the following change of variables:
 $$
 \int_{\mathcal{A}/\mathcal{G}}\Phi(A) e^{-\mathcal{S}_{\text{YM}}(A)}dA=\int_{\Omega\times G^{2g}}\Phi\left(\mathcal{L}_V^{-1}\iota_V(F)+\sum_{a}\log\mathrm{b}_a[W^u(a)]\right)e^{-\frac{\|F\|^2}{2}}dF \mu_G(d\mathrm{b}).
 $$
 Hence, in some sense, we decide to normalize to $1$ the Jacobian determinant of our change of variables $(F,\mathrm{b})\mapsto A$.

 In~\cite{Levyphd, LevyMarkov}, L\'evy gave a unified construction of this measure as a very general type of stochastic process on the space of loops taking the celebrated Driver--Sengupta formula as starting point. In particular, the area Markov property plays a central and beautiful role in Lévy's approach in ~\cite{LevyMarkov}. Recently, in a breakthrough work, Chevyrev was able to construct for the first time a genuine Yang--Mills measure \textbf{directly} on distributional connections on flat $\mathbb{T}^2$ as scaling limit of discrete random connections
 whose holonomy also obey the Driver--Sengupta formula. This is the most relevant result from the viewpoint of our work. This advance stimulated a number of results from the stochastic partial differential equations community to explore the stochastic quantization of gauge theories, due to 
Chandra--Chevyrev--Hairer--Shen~\cite{CCHS1,CCHS2},  
Cao--Chatterjee~\cite{CC1,CC2}, Bringmann--Cao~\cite{BringCao1,BringCao2}, Chevyrev--Shen~\cite{ChevShen}, Shen--Smith--Zhu~\cite{SSZ1}, Shen--Zhu--Zhu~\cite{SZZ1} and many others.

 Finally, we mention that a parallel work of the third author with E. Nohra~\cite{DangNohra} constructed a \textbf{discrete version} of the 
Yang--Mills measure which is compatible with the lattice gauge theory on the surface $\Sigma$. They show its convergence in the scaling limit to the continuum measure of the present paper, proving that the measure constructed in the present paper can be obtained 
as scaling limit of statistical models defined on the lattice, in the same way as for the Ising model and the $\Phi^4_2$ and $\Phi^4_3$ models studied in constructive quantum field theory. In particular, the continuum measure constructed in the present paper is universal, independent of the details of the lattice discretization. As a consequence, the result of~\cite{DangNohra} together with the present work obtain a new intrinsic construction of the Yang--Mills measure, independent of previous constructions in the literature, and prove the convergence of partition functions on all compact surfaces~\footnote{This is done in~\cite{DangNohra}.}.
In fact in \cite{DangNohra}, the approach is to resolve globally the surface by some cylinder and using global system of fake polar coordinates on $\Sigma$ that is called pseudocoordinates.   Then in \cite[section 5.1]{DangNohra}, specially \cite[Equation (5.2) p.~36]{DangNohra}, one can find an alternative definition  
of the solution $A_{\mathcal{N}}=\mathcal{L}_V^{-1}\left(\xi \iota_V(\upsilon) \right)$ of the cohomological equation expressed in terms of a Brownian sheet written in the pseudocoordinates, in the notations of \cite{DangNohra}:
\begin{equation}
A_\mathcal{N}= \Psi^*\left(\partial_\theta\left\langle \xi, 1_{\square(r,\theta)} \right\rangle  \dd \theta\right).
\end{equation}

However, the solution $\mathcal{L}_V^{-1}\left(\xi \iota_V(\upsilon) \right)$ is controlled in anisotropic spaces in \cite[section 5.3]{DangNohra} 
instead of the Sobolev space $H^{-1-\kappa}$ which is used in the present paper. In 
\cite[p.~51--52]{DangNohra}, there is evidence but not a complete proof that both definitions define the same object~\footnote{A complete proof is part of one of the open questions that we listed}. Namely, in \cite{DangNohra}, it is proved that for any smooth curve $\gamma$ transverse to $V$ then
$$ \int_\gamma \Psi^*\left(\partial_\theta\left\langle \xi, 1_{\square(r,\theta)} \right\rangle  \dd \theta\right)= \int_\gamma \mathcal{L}_V^{-1}\left(\xi \iota_V(\upsilon) \right) $$
where the equality holds true in law as $\mathfrak{g}$--valued random variables. 

\subsubsection{Open problems}

Here we make a tentative list of some open questions which are left over by our work and that we believe would deserve further investigation:
\begin{enumerate}
\item Can we prove the Driver--Sengupta formula in full generality for our Yang--Mills measure for arbitrary embedded graphs in our surface?
\item Clarify the relation between our measure and the one constructed in \cite{DangNohra}, they should be the same and there is strong evidence for this but a simple direct proof is somehow missing.
\item What can we define in the bordered case or what should we do when there are multiple boundary components?
\item Related to the previous question: can we prove a spatial Markov property and some form of Segal gluing for our Yang--Mills measure? 
\item What happens if we change the Morse function? This is an extremely natural question. Then how do we compare the measures obtained under different choices of Morse--Smale flows?
\item Can we condition our Yang--Mills random connection to be some connection on some non trivial vector bundle of fixed topological type? The case of the abelian Yang--Mills measure conditioned on nontrivial $U(1)$-bundles of fixed Chern number is described in paragraph~\ref{sss:YMChern}. 
\item What happens for non compact groups? One should start with the $SL_2(\mathbb{R})$ case first.
\item  Can we randomly choose our gauge fixing? Could we compare our measure with the holonomy process constructed by T.Lévy~\cite{Levyphd}?
\item Can we put the connection in Morse gauge back in a different gauge? How does this compare to the connection constructed by Chevyrev~\cite{Chevyrev}, if we could go back to the Coulomb gauge could we recover the measures constructed by Chevyrev on the torus?
\item What happens on non orientable surfaces?
\item If $\nabla f$ is not Morse--Smale, can we still solve the cohomological equations? In higher dimension ($d=3,4$) could we use the same techniques to construct Maxwell and abelian Chern-Simons theories?
\item Could we study the critical values and critical points of the Yang--Mills functional in the Morse gauge?
\item How can this construction help for SPDE construction of $\mathsf{YM}_2$ on general surfaces? Could it be useful for constructing the Yang--Mills--Higgs theory on general surfaces?
\item Would our method be useful for the $3d$ Yang--Mills measure? For the moment, we strongly doubt that our method 
would have any relevance for the construction of $\mathsf{YM}_3$.
\end{enumerate}

\subsection{Organization of the article}

In Section~\ref{s:morse}, we review some material from classical Morse theory and we fix some notations that are used throughout the article. In Section~\ref{s:spectralgap}, we show the main analytical result on gradient flow. Namely, we prove a spectral gap estimate that will be used both in the proof of Theorem~\ref{t:normalform} and~\ref{t:random-cohomological}. In Section~\ref{s:Morse-gauge}, we give the proof of Theorem~\ref{t:normalform} using methods and ideas from the classical theory of hyperbolic dynamical systems. Along the way, we also describe families of classical observables. In Section~\ref{s:solve_random_connection_equation}, we define the white noise precisely together with some of its elementary properties and we then prove Theorem~\ref{t:random-cohomological}. The longer Section~\ref{s:integration-connection} describes large families along which we expect to define random holonomies and more general Yang--Mills observables. Then, we explain how to integrate the random solutions from Theorem~\ref{t:random-cohomological} along such curves and we show that the resulting process is a reparametrized $\mathfrak{g}$-valued Brownian process. After that, in Section~\ref{s:randomholonomy}, we explain how to solve the stochastic differential equations associated with these $\mathfrak{g}$-valued Brownian processes and we prove some basic properties of their solutions. Once all these tools are settled, we prove Theorems~\ref{t:def-YM-general} and~\ref{t:def-YM} in Section~\ref{s:definition} except for the definition and the law of random holonomies that is explained in Section~\ref{ss:random-holonomies-YM}.
Section~\ref{s:MaxwellMorse} is entirely devoted to a treatment of the abelian case where we give details on the Morse gauge fixing for connections on non trivial line bundles. Then we discuss how our Yang--Mills measure charges connections with different Chern numbers and how one can condition our Yang--Mills measure to line bundles of fixed Chern class.

 Finally, Appendix~\ref{a:reparametrization-sde} is mostly intended for readers less familiar with stochastic differential equations who will find here material and references on this material that is mostly used in Section~\ref{s:randomholonomy}.

\subsection{Acknowledgments} 

Two of the authors are partially supported by the Institut Universitaire de France and three others by the Agence Nationale de la Recherche through the ADYCT grant (ANR-20-CE40-0017) and the POAS grant (ANR-24-CE40-5511). T.D.T is partially supported by the project Emergence 2025-2026  (Sorbonne Université). The authors thank Philippe Carmona, Baptiste Chantraine, Ilya Chevyrev, Nicolas Depauw, Léonard Ferdinand, Colin Guillarmou, Stéphane Guillermou, Paul Laurain, Thibault Lefeuvre, Thibaut Lemoine, Thierry Lévy, Jiasheng Lin, Hao Shen, Michele Schiavina, Elias Nohra, Rongchan Zhu and Xiangchan Zhu for useful discussions related to various aspects of the article. 

N.V.D and T.D.T would like to thank Thierry Lévy for the amazing minicourses at Sorbonne University on his approach to the $\mathsf{YM}_2$ measure and for his encouragement on this long term project. N.V.D recognizes the influence of the works of Cekic--Lefeuvre \cite{CL19} that suggested that the analysis of hyperbolic dynamical systems might have some interesting applications in gauge theory. Then discussions with Schiavina~\cite{Schiavina1, Schiavina2} who introduced the notion of Anosov gauge for BF theories
convinced us that one could instead use Morse flows to define the Morse gauge considered in the present paper.  
  Finally, explanations by Rongchan and Xiangchan Zhu 
made us realize that we could generalize the work of Driver~\cite{Dri89} to the manifold case.

\section{Preliminaries on Morse functions}
\label{s:morse}
Let $\Sigma$ be a $\mathcal{C}^\infty$ compact surface which is oriented, connected and boundaryless. In this section, we review some material on gradient vector fields generated by a Morse function and we fix the conventions that are used throughout the article. The key ingredient in view of our construction of the Yang-Mills measure is Theorem~\ref{t:harveylawson} which describes the convergence to equilibrium for Morse-Smale gradient flows. This result will be further refined in Section~\ref{s:spectralgap} in order to fit the analytical problems we will encounter when constructing the so-called Morse gauge of a connection.

\subsection{Morse functions}

Let $f:\Sigma\rightarrow\R$ be a smooth function. We say that $f$ is a \emph{Morse function} if all its critical points are nondegenerate. It can be shown that the set of Morse functions is open and dense in the $\mathcal{C}^\infty$ topology. We denote the set of critical points of $f$ by $\text{Crit}(f)$ and a Morse function is said to be perfect if $|\text{Crit}(f)|=2g+2$ (where $g$ is the genus of the surface $\Sigma)$ and if all the critical values of $f$ are distinct. The set of perfect Morse functions is also open (and nonempty) in the $\mathcal{C}^\infty$-topology.

A fundamental property of Morse functions is the so-called Morse Lemma:
\begin{lemma}[Morse Lemma]
Let $f:\Sigma\rightarrow \mathbb{R}$ be a $\mathcal{C}^\infty$ Morse function. Then, for any $a$ in $\operatorname{Crit}(f)$, there exists a smooth chart centered at $a$ such that, in these local coordinates, the function $f$ reads
$$
f(x) = f(a)+\frac {1}{2}\left(\varepsilon_1x_1^2+\varepsilon_2 x_2^2\right),
$$
with $\varepsilon_1,\varepsilon_2\in\{-1,1\}$. If $\varepsilon_1=\varepsilon_2=1$, we say that $a$ has index $0$ (local minimum). If $\varepsilon_1=1,\varepsilon_2=-1$, we say that $a$ has index $1$ (saddle point). If $\varepsilon_1=\varepsilon_2=-1$, we say that $a$ has index $2$ (local maximum). 
\end{lemma}
With this Lemma at hand, we introduce the notion of locally flat   metric near critical points~\cite[\S2]{HarveyLawson}:
\begin{definition} Let $f:\Sigma\rightarrow \mathbb{R}$ be a $\mathcal{C}^\infty$ Morse function. We fix a Morse chart near every point in $\operatorname{Crit}(f)$. We say that the metric $\tilde{h}$ is locally flat near critical points if it reads in the local Morse coordinates
$$
\tilde{h}=dx_1^2+dx_2^2.
$$
\end{definition}
By a partition of unity argument, one can verify that such metrics exist. 

\subsection{Gradient flows}

We now fix $f:\Sigma\rightarrow \mathbb{R}$ a perfect $\mathcal{C}^\infty$ Morse function once and for all. Given a $\mathcal{C}^\infty$ metric $\tilde{h}$, one can define the corresponding gradient vector field $\nabla_{\tilde{h}}f$ through the relation
$$
\forall x\in\Sigma,\quad d_xf=\tilde{h}_x(\nabla_{\tilde{h}}f(x),.).
$$
This induces a complete and smooth flow on $M$ that we denote by $\varphi_f^t:\Sigma\rightarrow \Sigma$. One can verify that
\begin{equation}\label{e:lyapunov-property}
 \forall t_1,t_2\in\R,\ \forall x\in\Sigma,\quad
 f\circ\varphi_f^{t_2}(x)-f\circ\varphi_f^{t_1}(x)=\int_{t_1}^{t_2}\|d_{\varphi_f^t(x)}f\|_{\tilde{h}^*(x)}^2dt,
\end{equation}
where $\tilde{h}^*$ is the induced metric on $T^*\Sigma$. In particular, $f$ is nondecreasing along the flow lines of $\varphi_f^t$. A key property of gradient flows is that, for any $x\in M$, there exists $x_-$ and $x_+$ in $\text{Crit}(f)$ such that
\begin{equation}\label{e:limitset}
 \lim_{t\rightarrow\pm\infty}\varphi_f^t(x)=x_\pm.
\end{equation}
We say that the pair $(f,h)$ has the Morse-Smale property if there is no orbit connecting two distinct saddle points of $f$. It is shown in~\cite[Th.~14.4]{HarveyLawson} that there exists a Morse-Smale pair with $\tilde{h}:=h_f$ being a locally flat metric and $f$ a perfect Morse function. 

\textbf{From this point on, we will always assume that the pair $(f,h_f)$ is Morse-Smale, the function $f$ is perfect and the metric $h_f$ is locally flat near critical points.}  In particular, one has
\begin{enumerate}
 \item near the maximum of $f$,
 \begin{equation}\label{e:local-expression-max}
 \nabla_{h_f} f=-x_1\partial_{x_1}-x_2\partial_{x_2},\ \varphi^{t}_f(x_1,x_2)=(e^{-t}x_1,e^{-t}x_2),
 \end{equation}
 \item near the saddle points of $f$,
 \begin{equation}\label{e:local-expression-saddle}
 \nabla_{h_f} f=x_1\partial_{x_1}-x_2\partial_{x_2},\ \varphi_f^{t}(x_1,x_2)=(e^{t}x_1,e^{-t}x_2),
 \end{equation}
 \item near the minimum of $f$,
 \begin{equation}\label{e:local-expression-min}
 \nabla_{h_f} f=x_1\partial_{x_1}+x_2\partial_{x_2},\ \varphi_f^{t}(x_1,x_2)=(e^{t}x_1,e^{t}x_2).
 \end{equation}
\end{enumerate}

One important fact is that the dynamics is linear in Morse charts. Given a critical point $a$ of $f$, we define the stable and unstable manifolds of $a$:
$$
W^u(a):=\left\{x\in\Sigma:\ \lim_{t\rightarrow-\infty}\varphi_f^t(x)=a\right\},
$$
and
$$
W^s(a):=\left\{x\in\Sigma:\ \lim_{t\rightarrow+\infty}\varphi_f^t(x)=a\right\}.
$$
These are embedded submanifolds of $\Sigma$ which are diffeomorphic to $\R^{2-\text{ind}(a)}$ (for $W^u(a)$) and to $\R^{\text{ind}(a)}$ (for $W^s(a)$). As we are working with compact surfaces, one can verify that these submanifolds induce de Rham currents, i.e.
$$
\forall\psi\in\Omega^{2-\text{ind}(a)}(\Sigma),\quad\langle[W^u(a)],\psi\rangle:=\int_{W^u(a)}\psi,
$$
and 
$$
\forall\psi\in\Omega^{\text{ind}(a)}(\Sigma),\quad\langle[W^s(a)],\psi\rangle:=\int_{W^s(a)}\psi.
$$
Thanks to the local expression of the vector field, one has that (1) near the maximum $a_{2g+2}$, $[W^u(a_{2g+2})]=\delta_0(x_1,x_2)dx_1\wedge dx_2$; (2) near the saddle points $(a_j)_{2\leq j\leq 2g+1}$, $[W^u(a_j)]=\delta_0(x_2)dx_2$ and (3) near the minimum $a_1$, $[W^u(a_1)]=1$. In fact, as $f$ has a single minimum, one has globally $[W^u(a)]=1$ in the case where $\text{ind}(a)=0$.

\begin{rmk}
 All along the article, we adopt the following geometric convention. If $X$ denotes an oriented submanifold with corners in $\Sigma$, we will use the notation $[X]$ for the corresponding current of integration: $\left\langle [X] ,\varphi \right\rangle:=\int_X\varphi $ for $\varphi\in \Omega^\bullet(\Sigma)$. For quick recollections on currents in the formalism close to the present work, we refer the reader to~\cite[Appendix D p.~48]{DRNodal}, \cite[section 6.2 p.~1434]{DR19} and \cite[Appendix A p.~62]{DRDuke}.
\end{rmk}

\subsection{Convergence to equilibrium}
A key property of gradient flows for our analysis is the following theorem showing convergence to equilibrium for gradient flows:
\begin{thm}[Harvey-Lawson, Dang-Rivi\`ere]\label{t:harveylawson} Let $f:\Sigma\rightarrow\R$ be a $\mathcal{C}^\infty$ perfect Morse function and let $h_f$ be a $\mathcal{C}^\infty$ locally flat metric. Suppose that $(f,h_f)$ has the Morse-Smale property. Then, there exists $N_0\geq 1$ such that, for every $0\leq k\leq 2$, for every $0<\beta<1$ and for every $\psi\in\Omega^k(\Sigma)$, 
 $$
 \forall t\geq 0,\quad \varphi^{-t*}_f(\psi)=\sum_{a:\operatorname{ind}(a)=k}\left(\int_{W^s(a)}\psi\right)[W^u(a)]+\mathcal{O}_{\mathcal{D}^\prime}(e^{-t\beta}),
 $$
 where the constant in the remainder depends only on $\beta$ and $\|\psi\|_{\mathcal{C}^{N_0}}$.
\end{thm}
The fact that one has convergence was proved by Harvey and Lawson in~\cite{HarveyLawson01} while the rate of convergence and a full asymptotic expansion were obtained by two of the authors~\cite{DR19, DR21}. Note that this result holds true for much more general gradient flows but we restrict to the setting of interest for our analysis. In the upcoming section, we will try to refine this result in the case of forms of degree $0$ (meaning functions) for which convergence follows in fact directly from the dominated convergence Theorem. Our goal is to get more precise information on the norms involved in the remainder term. For degree $1$, we will not need information on the size of the remainder and we will directly use the original version of this convergence result as in~\cite{HarveyLawson01}.

\subsection{The case of Lie algebra valued forms} 

In view of our future applications, we briefly discuss how Theorem~\ref{t:harveylawson} extends to matrix valued forms. More precisely, we let $G$ be a compact Lie group. We denote by $\mathfrak{g}$ the corresponding Lie algebra. Recall that there is a natural mapping $\textbf{Ad}:\mathrm{g} \in G\mapsto \textbf{Ad}_{\mathrm{g}}\in\text{Aut}(G)$ defined as $\textbf{Ad}_\mathrm{g}(\mathrm{h})=\mathrm{g}\mathrm{h}\mathrm{g}^{-1}.$ The differential of this map at the identity element $\mathrm{e}\in G$ induces a Lie algebra automorphism denoted by $\text{Ad}_\mathrm{g}\in\text{Aut}(\mathfrak{g})$ (the so-called adjoint representation of $G$). All along the article, we will make the assumption that $G$ is a compact \emph{linear} Lie group, meaning that it is a compact subgroup of $\text{GL}_N(\mathbb{C})$. This implies that the adjoint representation reads $\text{Ad}_\mathrm{g}(\mathfrak{a})=\mathrm{g}\mathfrak{a}\mathrm{g}^{-1}$. From now on, we will fix a norm $\|.\|_{\mathfrak{g}}$ on $\mathfrak{g}$ that is invariant under the adjoint representation.

\begin{rmk}
According to~\cite[Cor.~4.22]{Knappbeyond}, any compact Lie group is isomorphic to a closed linear group. 
\end{rmk}

\begin{example}
 In the case where $G=\operatorname{SU}(N)$ with $N\geq 1$, one has $\operatorname{Ad}_\mathrm{g} X = \mathrm{g} X \mathrm{g}^{-1}$. A natural Hermitian norm on $\mathfrak{su}(N)$ is given by the Killing form
$$
\forall (X_1, X_2) \in \mathfrak{su}(N)^2, \quad \langle X_1, X_2\rangle_{\mathfrak{g}} = \operatorname{Tr}( X_1 X_2^*) = -\operatorname{Tr}( X_1 X_2).
$$
\end{example}

One can consider $\mathfrak{g}$-valued forms $\psi\in\Omega^k(\Sigma,\mathfrak{g})$ on which $\varphi_f^{t*}$ naturally acts by pull--back\footnote{One can for instance fix a basis of $\mathfrak{g}$ and the gradient flow acts coordinates by coordinates.} as we consider a trivial bundle $\Sigma\times\mathfrak{g}$. In that context, Theorem~\ref{t:harveylawson} reads, for all $t\geq 0$, for all $\psi_1\in\Omega^k(\Sigma,\mathfrak{g})$ $\psi_2\in\Omega^{2-k}(\Sigma,\mathfrak{g})$,
\begin{equation}\label{e:harveylawsonbundle}
\int_{\Sigma}\left\langle\varphi_{f}^{-t*}(\psi_1)\wedge \psi_2\right\rangle_{\mathfrak{g}}=\sum_{\text{ind}(a)=1}\left\langle\int_{W^s(a)}\psi_1,\int_{W^u(a)}\psi_2\right\rangle_{\mathfrak{g}}+\mathcal{O}_{\psi_1,\psi_2}(e^{-t\beta}).
\end{equation}
Equivalently, one has that $\varphi^{-t*}_f(\psi_1)$ converges weakly to 
$$
\sum_{\text{ind}(a)=1}\left(\int_{W^s(a)}\psi_1\right)[W^u(a)],
$$
in the sense of $\mathfrak{g}$-valued currents.

\section{Contraction on weighted \texorpdfstring{$L^p$}{Lp}-spaces}
\label{s:spectralgap}

In this section, we aim at proving a somehow refined version of Theorem 
~\ref{t:harveylawson} in the case of $0$-forms (meaning functions). More precisely, we will give a precise estimate on the remainder appearing in this Theorem. To do this, we will follow an idea due to Jia, Stewart and Sverak in the context of $1$-dimensional reductions of certain equations from fluid mechanics~\cite[\S3]{JiaStewartSverak19}. Namely, given $\gamma>1$ and $p\geq 1$ such that 
$\frac{2}{p}<\gamma,$
we introduce the following norm
$$
\forall\psi\in\mathcal{C}^{\infty}(\Sigma,\C),\quad\left\|\psi\right\|_{p,\gamma}:=\left(\int_{\Sigma}|\psi(x)|^p\frac{d\text{vol}_{h_f}(x)}{d(x,\operatorname{argmin} f)^{p\gamma}}\right)^{\frac1p},
$$
where $\text{vol}_{h_f}$ is the Riemannian volume induced by the Morse metric $h_f$. The point of introducing such a norm is that functions that are nonzero near the minimum of $f$ do not have finite norm (and thus do not belong to the corresponding Banach space). In other words, we force functions to vanish at $\argmin(f)$ at a certain rate which is controlled by the exponent $\gamma$.
If we also make the assumption that $\gamma<1+\frac{2}{p}$, then functions vanishing at first order near $\min f$ will have finite norm. We denote by $Y_{p,\gamma}$ the closure of $\mathcal{C}^{\infty}(\Sigma,\C)$ with respect to the $\Vert . \Vert_{p,\gamma}$-norm. The main result of this section is
\begin{thm}\label{t:contraction} Let $f:\Sigma\rightarrow\R$ be a $\mathcal{C}^\infty$ perfect Morse function and let $h_f$ be a $\mathcal{C}^\infty$ locally flat metric. Suppose that $(f,h_f)$ has the Morse-Smale property. 

Then, for every $p>2$, for every $\gamma>\frac{2}{p}$ and for every $0<\beta_0<\min\{\frac{1}{2}-\frac{1}{p},\gamma-\frac{2}{p}\}$ , one can find a constant $C>0$ such that
$$
\forall t\geq 0,\quad \left\|\varphi_f^{-t*}\right\|_{Y_{p,\gamma}\rightarrow Y_{2,\gamma}}\leq C e^{-\beta_0 t}.
$$ 
\end{thm} 

\begin{rmk}
We just deal with the case of complex valued functions but, as we work with trivial bundles, the analogues of these results for Lie algebra valued functions will follow immediately by acting coordinates by coordinates. As the manifold $\Sigma$ is compact, $\operatorname{vol}_{h_f}$ and $\upsilon$ are equivalent and the volume $\operatorname{vol}_{h_f}$ can be replaced by the Riemannian volume $\upsilon$ we used to define the Yang--Mills functional $\mathcal{S}_{\operatorname{YM}}$.
\end{rmk}

This theorem is the main analytical tool that allows us to define the Morse gauge in Theorem~\ref{t:normalform} and to solve the random cohomological equations in Theorem~\ref{t:random-cohomological}.
As we shall see in the proof, the fact that we require $p>2$ is rather important to deal with the behavior of the flow near saddle points. In~\cite{JiaStewartSverak19}, the case $p=2$ was allowed but this was due to the absence of such points in dimension $1$. In order to include the case $p=2$, one would need to make an assumption on the divergence of the vector field at the saddle points. Note that, for the flow in positive time, all the conditions are inverted and the conclusion remains unchanged (up to replacing $\argmin f$ by $\argmax f$ in the definition of the norm). The rest of this section is devoted to the proof of this result. Before proceeding to the proof, let us record some preliminary reductions. 

Given $u\in\mathcal{C}^{\infty}(\Sigma)$ and using the H\"older inequality, one has, for $p\geq 2$,
\begin{multline*}
\left\|\varphi_f^{-t*}(u)\right\|_{2,\gamma}=\left\|d(.,\operatorname{argmin} f)^{-\gamma}\varphi_f^{-t*}(d(., \operatorname{argmin} f)^{\gamma}d(.,\operatorname{argmin} f)^{-\gamma}u)\right\|_{L^2}\\
=\left\|e^{\frac{1}{2}\int_0^t(\text{\text{div}} V)\circ \varphi_f^sds}\varphi_f^{t*}d(.,\operatorname{argmin} f)^{-\gamma}\varphi_f^{-t*}(d(.,\operatorname{argmin} f)^{\gamma}d(.,\operatorname{argmin} f)^{-\gamma}u)\right\|_{L^2}\\
\leq \left\|e^{\frac{1}{2}\int_0^t(\text{\text{div}} V)\circ \varphi_f^sds}\varphi_f^{t*}(d(.,\operatorname{argmin} f)^{-\gamma})d(.,\operatorname{argmin} f)^{\gamma}\right\|_{L^{\frac{2p}{p-2}}}\|u\|_{p,\gamma},
\end{multline*}
where in the second equality we used the identity $\varphi_f^{t*} \text{vol}_{h_f} = e^{\int_0^t(\text{\text{div}} V)\circ \varphi_f^sds}   \text{vol}_{h_f}$ which follows from the next Lemma.
\begin{lemma}\label{l:divergence-pullback}
Let $\mathrm{v}$ be a volume form on a compact manifold $M$ of dimension $n$. Let $V$ be a smooth vector field with $(\varphi^t)_{t\in \mathbb{R}}:M\mapsto M$ the corresponding flow. Then we have the exact identity
\begin{equation}
\varphi^{t*}({\rm v})=    e^{\int_0^t(\operatorname{div}_{{\rm v}} (V))\circ \varphi^s ds}   {\rm v}.
\end{equation}
\end{lemma}
\begin{proof}
 The pull--back $u(t):=\varphi^{t*}{\rm v} \in \Omega^n(M)$ solves the transport equation (in the space of top degree forms)
 $\partial_tu-\mathcal{L}_Vu=0$ with initial condition $u(0)={\rm v}$. Now note that
 $ \mathcal{L}_V\varphi^{t*}{\rm v} =\varphi^{t*} \left(\mathcal{L}_V {\rm v} \right) = \varphi^{t*} \left(\text{div}_{{\rm v}}(V){\rm v}\right)=\left( \text{div}_{{\rm v}}(V)\circ \varphi^t\right)  \varphi^{t*}{\rm v}=\left( \text{div}_{{\rm v}}(V)\circ \varphi^t\right) u(t)  $ by definition of the divergence of $V$.
 Hence the transport equation rewrites
 $\partial_t u- \left( \text{div}_{{\rm v}}(V)\circ \varphi^t\right) u=0 $ whose solution reads
 $e^{\int_0^t(\text{div}_{{\rm v}} V)\circ \varphi^s ds}     {\rm v} $.
\end{proof}

Hence,
$$
\left\|\varphi_f^{-t*}\right\|_{Y_{p,\gamma}\rightarrow Y_{2,\gamma}}\leq
\left\|e^{\frac{1}{2}\int_0^t(\text{\text{div}} V)\circ \varphi_f^sds}\varphi_f^{t*}(d(.,\operatorname{argmin} f)^{-\gamma})d(.,\operatorname{argmin} f)^{\gamma}\right\|_{L^{\frac{2p}{p-2}}}
$$
In other words, we are left with proving the following Lemma from which Theorem~\ref{t:contraction} follows thanks to the last inequality. 
\begin{lemma}\label{l:keylemma}
 Let $f:\Sigma\rightarrow\R$ be a $\mathcal{C}^\infty$ perfect Morse function and let $h_f$ be a $\mathcal{C}^\infty$ locally flat metric. Suppose that $(f,h_f)$ has the Morse-Smale property. 
 
 Then, for every $p>2$, for every $\gamma>\frac{2}{p}$ and for every $0<\beta_0<\min\{\frac{1}{2}-\frac{1}{p},\gamma-\frac{2}{p}\}$ , one can find a constant $C>0$ such that
$$
\forall t\geq 0,\quad \left\|e^{\frac{1}{2}\int_0^t(\operatorname{div} V)\circ \varphi_f^sds}\varphi_f^{t*}(d(.,\operatorname{argmin} f)^{-\gamma})d(.,\operatorname{argmin} f)^{\gamma}\right\|_{L^{\frac{2p}{p-2}}}\leq C e^{-\beta_0 t}.
$$ 
\end{lemma}

\subsection{Partitions of \texorpdfstring{$\Sigma$}{Sigma} adapted to the dynamics} We fix (once and for all) some small enough $r_0>0$ such that, for every two distinct critical points $a$ and $b$, one has $B(a,2r_0)\cap B(b,2r_0)=\emptyset$, where $B(x,r)$ is the open ball of radius $r$ centered at $x$ (for the Riemannian distance induced by $h_f$). We also suppose that $r_0$ is small enough to ensure that $B(a,r_0)$ is contained in the Morse chart for any critical point.

We order the critical points $\{a_1,\ldots,a_{2g+2}\}$ of $f$ as follows:
$$
f(a_{2g+2})>f(a_{2g+1})>\ldots>f(a_1).
$$
With these conventions at hand, we introduce a partition of $\Sigma$ by letting
$$
\forall 1\leq j\leq 2g+2,\quad P_j:=B(a_j,r_0),
$$
and
$$
P_0:=\Sigma\setminus \left(\bigcup_{j=1}^{2g+2}P_j\right).
$$
As a direct consequence of the gradient dynamics, one has the following Lemma.

\begin{lemma}\label{l:crossingtime}[Escaping $P_0$] With the above assumptions, there exists $T_0>0$ such that, for every $x\in P_0$, one has
$$
\varphi_f^{[0,T_0]}(x)\cap \left(\bigcup_{j=1}^{2g+2}P_j\right)\neq\emptyset.
$$
\end{lemma}

It means that there is some uniform time $T_0$ such that for all points $x$ outside the union of critical balls, the flow trajectory $\varphi_f^{[0,T_0]}(x)$ from this point will necessarily visit the  union 
$\left(\bigcup_{j=1}^{2g+2}P_j\right)$ of critical balls. In other words, one cannot stay forever outside the critical region. This Lemma explains this simple fact in a quantitative way.
\begin{proof}
Suppose that, for every $N\geq 1$, one can find $x_N\in P_0$ such that $\varphi_f^{[0,N]}(x_N)\subset  P_0$ which is a closed subset. Hence, up to extraction, $x_N$ converges to some point $x_\infty$. For every $T>0$, one has that, for every $N\geq T$, $\varphi_f^{[0,T]}(x_N)\subset  P_0$. By letting $N$ go to $\infty$ along the good subsequence, one has $\varphi_f^{[0,T]}(x_\infty)\subset  P_0$ for every $T>0$. This contradicts the fact that $\varphi_f^t(x_\infty)$ converges to some point $(a_j)_{1\leq j\leq 2g+2}$.
\end{proof}
Given $x\in\Sigma$ and $N\geq 1$, we introduce the unique word $\alpha(x)=(\alpha_0(x),\ldots,\alpha_{N-1}(x))$ in $\{0,1\ldots, 2g+2\}^N$ such that 
$$
x\in P_{\alpha_0(x)},\  \varphi_f^{T_0}(x)\in P_{\alpha_1(x)},\ldots ,\varphi_f^{(N-1)T_0}(x)\in P_{\alpha_{N-1}(x)}.
$$

The reader has to think of this word as remembering what critical balls are being visited by the trajectory in the spirit of symbolic dynamics. Thanks to Lemma~\ref{l:crossingtime}, one has the following property.
\begin{lemma}\label{l:outsidebasicsets}[Few zeroes in a word] With the above assumptions and conventions, one has, for every $x\in \Sigma$ and for every $N\geq 1$,
$$
\left|\{0\leq j\leq N-1:\ \alpha_j(x)=0\}\right|\leq 2g+2.
$$
\end{lemma}

In other words, a trajectory is not allowed to visit the region $P_0$ too many times since once it has escaped a critical ball $P_j$ for some $j\in \{1,\dots,2g+2\}$ it has to transit through $P_0$ then visit another critical ball $P_k$ for $k\neq j$, $k\in \{1,\dots,2g+2\}$ and never come back to $P_j$ again and therefore it is immediate that if we wait long enough, in the worst case scenario, we can at most visit each $P_j, j\in \{1,\dots,2g+2\}$ exactly once.

\begin{proof} Suppose that, given a point $x\in \Sigma$, one can find integers $k_1<k_2<\ldots<k_{L}$ such that, for every $1\leq \ell\leq L$, $\varphi_f^{k_\ell T_0}(x)\in P_0$. Thanks to Lemma~\ref{l:crossingtime}, one knows that, for every $1\leq \ell\leq L$, there exists $T_\ell\in (k_\ell T_0, k_{\ell+1} T_0)$ such that $\varphi^{ T_\ell}(x)$ belongs to some ball $B(a_{j_\ell},r_0)$ with $1\leq j_\ell\leq L$. Thanks to~\eqref{e:lyapunov-property} and to the exact expressions~\eqref{e:local-expression-max},~\eqref{e:local-expression-saddle} and~\eqref{e:local-expression-min} of the flow in the Morse charts, one knows that once a point exits a ball $B(a_{j_\ell},r_0)$ it will never re-enter it in positive time. Hence all the $j_\ell$ have to be distinct from each other from which we infer $L\leq 2g+2$.
\end{proof}

In other words, the number of iterations where the trajectory exits the fixed neighborhoods of the critical points is uniformly bounded in terms of the topology of $\Sigma$ (precisely by the genus). By similar arguments, one has
\begin{lemma}\label{l:order}[Non vanishing numbers are increasing in a word] With the above assumptions and conventions, one has, for every $x\in \Sigma$, for every $N\geq 1$ and for every $0\leq i\leq j\leq N-1$, one has
$$
\alpha_i(x)\leq\alpha_j(x)\ \text{or}\ \alpha_j(x)=0.
$$
\end{lemma}
With these properties at hand, we introduce the refined partition $\bigvee_{\ell=0}^{N-1}\varphi_f^{-\ell T_0}(\mathcal{P})$, where $\mathcal{P}=(P_j)_{0\leq j\leq 2g+2}.$ Recall that it is composed of the sets
$$
P_\alpha:=P_{\alpha_0}\cap \varphi_f^{-T_0}(P_{\alpha_1})\cap\ldots \cap\varphi_{f}^{-(N-1)T_0}(P_{\alpha_{N-1}}),\quad \alpha=(\alpha_0,\alpha_1,\ldots,\alpha_{N-1})\in\{0,1,\ldots, 2g+2\}^{N}.
$$
The number of nonempty elements is in fact not too big in this partition.
\begin{lemma}\label{l:cardinal} With the above assumptions and conventions, there exists a constant $C_0>0$ such that, for every $N\geq 1$,
$$
\left|\left\{\alpha\in\{0,\ldots,2g+2\}^N:\ P_{\alpha}\neq\emptyset\right\}\right|\leq C_0 N^{6g+6}.
$$
\end{lemma}
\begin{proof} Thanks to Lemma~\ref{l:outsidebasicsets}, one knows that the number $\ell$ of occurrence of the letter $0$ is at most $2g+2$. Hence, we first have to choose the place for the letter $0$ in a sequence of length $N$. The number of possibilities is at most $  \left(\begin{array}{c}
     N  \\
     2g+2 
\end{array} \right)= \mathcal{O}(N^{2g+2})$. After that, we need to pick a strictly ordered sequence of length $\ell'\leqslant 2g+2$ among the letters $1,\ldots, 2g+2$. The number of possibilities is now at most $ \left(\begin{array}{c}
     N  \\
     2g+2 
\end{array} \right)=\mathcal{O}(N^{2g+2}).$ Once we have fixed an ordered sequence of length $\ell'$, we need to choose how many time each letter appears knowing that the length of the word is at most $N$. In other words, we need to decompose $N$ as a sum of $\ell'$ numbers. Again this is bounded by $\mathcal{O}(N^{2g+2})$. Gathering these three bounds, we get the expected result.
\end{proof}

\subsection{Proof of Lemma~\ref{l:keylemma}}
As $T_0$ is fixed once and for all, we can reduce the proof of the Lemma to times $t$ of the form $(N-1)T_0$ with $N\geq 1$. Hence, we are interested in the $L^q$ (with $q=\frac{2p}{p-2}$) norm of the function
$$
e^{\frac12\int_0^{(N-1)T_0}(\operatorname{div} V)\circ \varphi_f^s(x)ds}\frac{d(x,\operatorname{argmin} f)^{\gamma}}{d(\varphi_f^{(N-1)T_0}(x),\operatorname{argmin} f)^{\gamma}}.
$$
To do that, we will first bound this function for $x\in P_\alpha$ in terms of $\alpha$. By construction, one can find $C_0>0$ such that, for every $\alpha\in \{0,\ldots, 2g+2\}^{N}$ and for every $x\in P_\alpha$, one has
$$
e^{\frac12\int_0^{(N-1)T_0}(\operatorname{div} V)\circ \varphi_f^s(x)ds}\frac{d(x,\operatorname{argmin} f)^{\gamma}}{d(\varphi_f^{(N-1)T_0}(x),\operatorname{argmin} f)^{\gamma}}
\leq C_0 e^{-T_0|\{j:\alpha_j=2g+2\}|}e^{T_0(1-\gamma)|\{j:\alpha_j=1\}|}.
$$
Hence, one has
\begin{multline*}
 \int_{\Sigma}\left|e^{\frac12\int_0^{(N-1)T_0}(\operatorname{div} V)\circ \varphi_f^s(x)ds}\frac{d(x,\operatorname{argmin} f)^{\gamma}}{d(\varphi_f^{(N-1)T_0}(x),\operatorname{argmin} f)^{\gamma}}\right|^qd\text{vol}_{h_f}(x)\\
 \leq C_0^q\sum_{|\alpha|=N}\operatorname{vol}_{h_f}(P_\alpha) e^{-qT_0|\{j:\alpha_j=2g+2\}|}e^{qT_0(1-\gamma)|\{j:\alpha_j=1\}|}
\end{multline*}
Recalling from Lemma~\ref{l:cardinal} that the number of nonzero elements in the sum is $\mathcal{O}(N^{6g+6})$ and recalling from Lemma~\ref{l:outsidebasicsets} that $|\{j:\alpha_j=0\}|\leq 2g+2$, it is sufficient to conclude the proof with the following Lemma.
\begin{lemma}\label{l:volumelemma}
 With the above conventions and assumptions, there exists $C_1>0$ such that, for every $N\geq 1$ and for every $\alpha\in\{0,\ldots ,2g+2\}^N$,
 $$
 \operatorname{vol}_{h_f}(P_\alpha)\leq C_1 e^{-T_0|\{j:2\leq \alpha_j\leq 2g+1\}|} e^{-2T_0|\{j: \alpha_j=1\}|}.
 $$
\end{lemma}

This Lemma is again reminiscent of results in hyperbolic dynamical systems -- see e.g.~\cite[App.~A]{BowenRuelle} for Axiom A flows. The main difference with that reference is that the dynamics is much simpler but we have to deal with the global dynamics rather than the dynamics near a single basic set (here a critical point).

\begin{proof} Let us first observe that we only need to treat the case where there exists some $j$ such that $\alpha_j\neq 2g+2$. In order to prove this upper bound, we fix $J_1$ to be the largest integer such that $\alpha_{J_1}\neq 2g+2$ and $J_0$ to be the largest integer such that $\alpha_{J_0}=1$. In the case where $\alpha_{j}>1$ for every $j$, one sets $J_0=0$.  From the exact expressions for the vector field given in~\eqref{e:local-expression-min} and from Lemma~\ref{l:divergence-pullback}, we can write
$$
\operatorname{vol}_{h_f}(P_\alpha)= \int_{\varphi_f^{J_0T_0}(P_{\alpha})} e^{-\int_{-J_0T_0}^0(\text{div} V)\circ\varphi_f^{s}ds}d\text{vol}_{h_f}\leq C e^{-2J_0 T_0}\operatorname{vol}_{h_f}(\varphi_f^{J_0T_0}(P_{\alpha})),
$$
for some constant $C>0$ depending only on $f$. Hence, we are done if $J_1\leq j_0+1$ and we are left with dealing with the case $J_1> J_0+1$. In that situation, we can roughly bound the volume in the upper bound using only the element of the partition that are close to a saddle point. More precisely, one has
$$
\operatorname{vol}_{h_f}(\varphi_f^{J_0T_0}(P_{\alpha}))\leq C\operatorname{vol}_{h_f}\left(P_{\alpha_{J_0+2}}\cap\ldots\cap \varphi_f^{-(J_1-J_0-2)T_0}\left(P_{\alpha_{J_1-1}}\right)\right),
$$
for some constant $C>0$ depending only on $f$. Under this form, we can use the exact expressions of the vector field given in~\eqref{e:local-expression-saddle} to get an upper bound of the form $\mathcal{O}(e^{-(J_1-J_0)T_0})$. Indeed, this volume corresponds to the volume of a neighborhood of a fixed saddle point and all the indices $j$ are in fact equal. To see this, recall that, thanks to the Morse--Smale property, the neighborhoods $(P_j)_{2\leq j\leq 2g+1}$ of saddle points can be chosen from the start small enough to ensure that they do not intersect the stable manifolds of the other saddle points. In particular, all the points in one of these $P_j$ will not enter the neighborhoods of the other saddle points in forward time. Equivalently, all the $\alpha_j$ with $J_0+1< j<J_1-1$ are equal to some fixed integer between $2$ and $2g+1$ and we can work in the Morse chart to estimate the volume of this neighborhood.
\end{proof}

\textbf{For the rest of the article, we make the assumptions} that $f$ is a $\mathcal{C}^\infty$ perfect Morse function, that the pair $(f,h_f)$ has the Morse-Smale property and that the metric $h_f$ is $\mathcal{C}^\infty$ and locally flat near any critical point. In particular, Theorem~\ref{t:contraction} applies.

\section{Deterministic Morse gauge for the Yang-Mills functional}
\label{s:Morse-gauge}

In this section, we apply Theorem~\ref{t:contraction} in view of proving Theorem~\ref{t:normalform}. By reversing times, we will also use similar weighted $L^p$-spaces with $p>2$
and
the extra assumption that
$$
\frac{2}{p}<\gamma<1+\frac{2}{p}.
$$
Recall that this allows smooth functions vanishing at order $1$ near the minimum (resp. maximum) of $f$ to belong to the weighted space $Y_{p,\gamma}$ while constant functions are not. Recall also that Theorem~\ref{t:harveylawson} applies in the case of $\mathfrak{g}$-valued functions by considering the action by pullback on each coordinate. We emphasize that $h$ and $h_f$ are a priori different metrics. The metric $h_f$ is just used to define the gradient dynamics. We also introduce the following norms, for $0\leq k\leq 2$, $1\leq p<\infty$ and $\delta\geq 0$,
$$
\forall \psi\in\Omega^k(\Sigma,\mathfrak{g}),\quad \left\|\psi\right\|_{\mathcal{Z}^k_{\delta,p}}:=\left(\int_{\Sigma}\left\| \psi(x)\right\|_{\Lambda^k(T^*\Sigma)\times\mathfrak{g}}^{p}\frac{d\operatorname{vol}_{h_f}(x)}{d(x,\operatorname{argmax} f)^\delta}\right)^{\frac{1}{p}}.
$$
The completions of $\Omega^k(\Sigma,\mathfrak{g})$ with respect to these two norms are denoted by $\mathcal{Z}^k_{\delta,p}$. Thanks to Theorem~\ref{t:contraction} (adapted to the case of $\mathfrak{g}$-valued functions) and $F_\infty$ belongs to $L^\infty$, the integral 
$$\int_0^\infty\varphi_f^{-t*}(\iota_V(F_\infty))dt
$$
defines a ($\mathfrak{g}$-valued) de Rham current of degree $1$ by duality. The resulting current lies in the Banach space $\mathcal{Z}_{\delta,q}^1$ for every $q<2$ and for every $\delta>2$. In particular, it lies in every $L^q$ space with $q<2$. This is the content of the third item of Theorem~\ref{t:normalform}. For the sake of simplicity, we will write this term as
$$\mathcal{L}_{V}^{-1}(\iota_V(F_\infty)):=\int_0^\infty\varphi_f^{-t*}(\iota_V(F_\infty))dt.
$$
Hence, we are left with the proofs of the last four items of Theorem~\ref{t:normalform} and we will proceed in three steps. First, we will discuss the (weak) convergence of $\mathrm{g}_T$ and derive the convergence of $F(A_T)$. Then we will focus on the convergence of $A_T$: this is where Theorem~\ref{t:contraction} is crucially used. Finally, we will discuss the last item of Theorem~\ref{t:normalform} which is relevant to the slicing of the space of connections.

\subsection{Convergence of \texorpdfstring{$\mathrm{g}_T$}{gT}}
We begin with the following lemma
\begin{lemma}\label{l:convergence-gT} There exists $\mathrm{g}_\infty\in L^\infty(\Sigma,G)$ such that $\mathrm{g}_T$ converges to $\mathrm{g}_\infty$ for the weak-$\star$ topology on $L^\infty(\Sigma)$. Moreover, for every compact $K$ of $W^u(a_1)=\Sigma\setminus\overline{\bigcup_{\operatorname{ind}(a)=1}W^u(a)}$, $\mathrm{g}_T$ converges to $\mathrm{g}_\infty$ in the $L^\infty(K)$-topology.
\end{lemma}
Recall that $G$ is a compact linear group hence included in some linear space $\text{M}_N(\C)$. In particular, $\mathrm{g}_T$ can be identified with an element in $L^\infty(\Sigma, \text{M}_N(\C))$ and the convergence in the weak-$\star$ topology is understood in this sense in this Lemma (meaning against test functions in $L^1$). 

\begin{proof}
The element $\mathrm{g}_T$ is defined as the solution at time $T\geq 0$ to~\eqref{e:def-gT} and we have set $\tilde{\mathrm{g}}_T:=\varphi_f^{T*}(\mathrm{g}_T)$ which is the solution to~\eqref{e:holonomy-equation}.
%\begin{equation}\label{e:holonomy-equation}
%\partial_t \tilde{u}+\varphi_f^{t*}(A(V))\tilde{u}=0,\quad \tilde{u}(t=0)=\id.
%\end{equation}
Fix now a point $x\in \Sigma$. From this expression, one can verify that, up to taking its inverse, $\tilde{\mathrm{g}}_t(x)$ is the parallel transport (for the connection $d-A$) along the path joining $x$ to $\varphi_f^t(x)$ with initial condition $\text{Id}_G$ at $t=0$. Hence, $\mathrm{g}_T(x)$ is the element in $G$ corresponding to this holonomy. More precisely, as the holonomy is independent of the path parametrization, one can consider a smooth path $\gamma_x:[0,1]\rightarrow\Sigma$ joining $\lim_{t\rightarrow-\infty}\varphi_f^t(x)$ to $x$ following the flowline of $V$. Then, we denote by $\mathrm{h}_x(\tau)$, $\tau\in[0,1]$, the parallel transport associated with the connection $d-A$ along this path $\gamma_x$ with initial condition $\text{Id}_{G}$ at $\gamma_x(0)$. Now, for every $T>0$, there exists $\tau(T)\in[0,1]$ such that $\gamma_x(\tau(T))=\varphi_f^{-T}(x)$ and one has $\mathrm{g}_T(x)=\mathrm{h}_x(\tau(T))\mathrm{h}_x(1)^{-1}$. By letting $T\rightarrow+\infty$, one finds that, for every $x\in\Sigma$, $\mathrm{g}_T(x)$ converges to some limit element $\mathrm{g}_\infty(x)\in G$ corresponding to the inverse of the holonomy along the flow line joining $x$ to $\lim_{T\rightarrow +\infty}\varphi_f^{-T}(x)$ with initial condition $\text{Id}_G$ at $\lim_{T\rightarrow +\infty}\varphi_f^{-T}(x)$ with initial condition $\text{Id}_G$. With the notations from the introduction, $\mathrm{g}_{\infty}(x)=\text{Hol}_{\gamma_x}(-A)^{-1}$. Moreover, if we fix a compact set $K$ of $W^u(a_1)$, one has uniform convergence with respect to $x\in K$. 

As the group $G$ is a linear compact group, there exists a constant $C_0$ such that, for every $T\geq 0$, $\|\mathrm{g}_T\|_{L^\infty}\leq C_0$ and this remains true for $T=\infty$. Observe now that $(\mathrm{g}_T)_{T\geq 0}$ is bounded in $L^\infty$ which is the dual space to $L^1$ that we endow with its weak-$\star$ topology. In particular, we can extract a convergent subsequence in this topology. By uniqueness of the limit in $\mathcal{D}^\prime(\Sigma)$, one finds that $\mathrm{g}_T$ converges to $\mathrm{g}_\infty$ in the weak-$\star$ topology.
\end{proof}

As a direct corollary, one finds that
\begin{corollary}\label{c:convergence-curvature} The curvature $F(A_T):=F(A_{\mathrm{g}_T^{-1}})$ converges (for the weak-$\star$ topology on $L^{\infty}(\Sigma,\Lambda^2(T^*\Sigma)\otimes \mathfrak{g})$) to 
$$
F_\infty=\mathrm{g}_\infty^{-1} F(A)\mathrm{g}_\infty.
$$
\end{corollary}
\begin{proof} We fix a compact set $K$ of $\tilde{\Sigma}$. One has $F(A_T)$ converges to $F_\infty$ in the $L^\infty(K)$ topology on this set. Moreover, $\|F(A_T)\|_{\mathfrak{g}}=\|F(A)\|_{\mathfrak{g}}$ for every $T>0$. In particular, this is a bounded sequence in $L^\infty(\Sigma)$ and one can extract a convergent subsequence for the weak-$\star$ topology. By uniqueness of the limit in $\mathcal{D}^\prime$, one has convergence to $F_\infty$ in this topology. 
\end{proof}

We also record the following Lemma regarding the regularity of $\mathrm{g}_\infty$. 
\begin{lemma} The map $\mathrm{g}_\infty|_{W^u(a_1)}$ belongs to $\mathcal{C}^{\infty}(W^u(a_1),G)$. 
\end{lemma}
In particular, $F_\infty|_{W^u(a_1)}$ belongs to $\mathcal{C}^{\infty}(W^u(a_1),\mathfrak{g})$.
\begin{proof} We let $x\in W^u(a_1)$ and we fix a local chart $(-\varepsilon,\varepsilon)^2$ centered at $x$ and contained in $W^u(a_1)$. Up to shrinking the size of the local chart, we can pick a smooth map $\gamma:[0,1]\times(-\varepsilon,\varepsilon)^2\rightarrow\Sigma$ such that, for every $x$, $\gamma_x:t\in[0,1]\mapsto \gamma(t,x)$ is the flow line joining $a_1$ to $x$. From the proof of Lemma~\ref{l:convergence-gT}, one has $\mathrm{g}_\infty(x)=\text{Hol}_{\gamma_x}(-A)^{-1}$ hence the solution to some ordinary differential equation depending on the parameter $x$ through the curve $\gamma_x$. From the Cauchy--Lipschitz Theorem, one finds that $\mathrm{g}_\infty(x)$ depends smoothly on $x$.  
\end{proof}
Finally, we also deduce the following corollary on the resulting limit holonomy.
\begin{corollary}\label{c:limit-holonomy} Let $\gamma:[0,1]\rightarrow\Sigma$ be a $\mathcal{C}^1$ curve such that $\gamma(0),\gamma(1)\in W^u(a_1)$. Then, one has 
$$
\lim_{T\rightarrow+\infty} \operatorname{Hol}_\gamma(A_T)=\mathrm{g}_\infty(\gamma(1))^{-1}\operatorname{Hol}_\gamma(A)\mathrm{g}_\infty(\gamma(0)).
$$
\end{corollary}
We recall that $A_T:=A_{\mathrm{g}_T^{-1}}$ is the element of $\Omega^1(\Sigma,\mathfrak{g})$ associated with the connection $\nabla_{\mathrm{g}_T^{-1}}=\mathrm{g}_T^{-1}\nabla\mathrm{g}_T$, where $\nabla=d+A$. In particular, this result shows the $5$-th item of Theorem~\ref{t:normalform}.
\begin{proof} One knows that the solution to~\eqref{e:parallel-transport} at time $1$ gives the holonomy and that
$$
\operatorname{Hol}_\gamma(A_T)=\mathrm{g}_T(\gamma(1))^{-1}\operatorname{Hol}_\gamma(A)\mathrm{g}_T(\gamma(0)),
$$
from which the result follows.
\end{proof}

\subsection{Convergence of \texorpdfstring{$A_T$}{AT} and \texorpdfstring{$dA_T$}{dAT}}

In view of studying the convergence of $A_T$, we rewrite it as a sum of $\varphi_f^{-T*}(A)$ plus a perturbative term. To do this, we recall that $\mathrm{g}_T:=\varphi_f^{-T*}(\tilde{\mathrm{g}}_T).$
Hence, one has
\begin{equation}\label{e:duhamel-step1}
A_T=\mathrm{g}_T^{-1}A\mathrm{g}_T+\mathrm{g}_T^{-1}d\mathrm{g}_T=\varphi_f^{-T*}\left(\tilde{\mathrm{g}}_T^{-1}(d+\varphi_f^{T*}(A))\tilde{\mathrm{g}}_T\right).
\end{equation}
We now write
\begin{equation}\label{e:duhamel-step2}
\tilde{\mathrm{g}}_T^{-1}(d+\varphi_f^{T*}(A))\tilde{\mathrm{g}}_T=A+\int_0^T\frac{d}{dt}\left(\tilde{\mathrm{g}}_t^{-1}(d+\varphi_f^{t*}(A))\tilde{\mathrm{g}}_t\right)dt.
\end{equation}
Using that $\tilde{\mathrm{g}}_t$ solves~\eqref{e:holonomy-equation}, one finds that
$$
 \frac{d}{dt}\left(\tilde{\mathrm{g}}_t^{-1}(d+\varphi_f^{t*}(A))\tilde{\mathrm{g}}_t\right)=\varphi_f^{t*}\left(\mathrm{g}_t^{-1}\left(A(V)(d+A)-(d+A)A(V)+\mathcal{L}_V(A)\right)\mathrm{g}_t\right),
$$
which can be simplified thanks to the following observation
$$
A(V)(d+A)-(d+A)A(V)+\mathcal{L}_V(A)=\iota_V(dA+A\wedge A)=\iota_V(F(A)).
$$
Combining this with~\eqref{e:duhamel-step1} and~\eqref{e:duhamel-step2}, one finds
$$
A_T=\varphi_f^{-T*}(A)+\iota_V\int_0^T\varphi_f^{-(T-t)*}\left(\mathrm{g}_t^{-1}F(A)\mathrm{g}_t\right)dt.
$$
Equivalently, one has
\begin{equation}\label{e:duhamel-final}
 A_T=\varphi_f^{-T*}(A)+\iota_V\int_0^T\varphi_f^{-t*}\left(\mathrm{g}_{T-t}^{-1}F(A)\mathrm{g}_{T-t}\right)dt.
\end{equation}
We can derive the convergence of $A_T$ from this formula.  The first term converges
$$\varphi_f^{-T*}(A) \rightarrow\sum_{\operatorname{ind}(a)=1}\left(\int_{W^s(a)}A\right)[W^u(a)],$$  as $T\rightarrow \infty$ by Theorem \ref{t:harveylawson} applied in the matrix-valued case.  For the second term, we take  a smooth test form $\psi \in \Omega^1(\Sigma, \mathfrak{g})$, then 
\begin{align*}
\left\langle \iota_V\int_0^T\varphi_f^{-t*}\left(\mathrm{g}_{T-t}^{-1}F(A)\mathrm{g}_{T-t}\right)dt, \psi\right \rangle &= \int_0^T \left\langle \left(\mathrm{g}_{T-t}^{-1}F(A)\mathrm{g}_{T-t}\right),  \varphi_f^{t*}\psi(V) \right\rangle dt\\
&=  \int_0^{\infty} \left\langle \left(\mathbf{1}_{[0,T]} \mathrm{g}_{T-t}^{-1}F(A)\mathrm{g}_{T-t}\right),  \varphi_f^{t*}\psi(V) \right\rangle dt.
\end{align*}
Since 
$\mathbf{1}_{0,T}(t)\mathrm{g}_{T-t}^{-1}F_A\mathrm{g}_{T-t}$ converges to  $F_\infty= g_\infty^{-1}F(A)g_\infty $ as $T\rightarrow \infty$ and since the norm of $\mathrm{g}_{T-t}^{-1}F_A\mathrm{g}_{T-t}$ is uniformly bounded in $\mathfrak{g}$ and, by Theorem \ref{t:contraction},
$$ \|\varphi_f^{t*} \psi(V)\|_{L^2} \lesssim e^{-\beta_0 t} \| \psi(V)\|_{\mathcal{Z} ^0_{\delta, p} }, \, p>2,$$
we infer by the dominated convergence that, as $T\rightarrow \infty$
\begin{align*}
\left\langle \iota_V\int_0^T\varphi_f^{-t*}\left(\mathrm{g}_{T-t}^{-1}F(A)\mathrm{g}_{T-t}\right)dt, \psi \right\rangle \rightarrow   \int_0^{\infty} \left\langle \left( g_\infty^{-1}F(A)g_\infty  \right),  \varphi_f^{t*}\psi(V) \right\rangle dt. %=    \langle \iota_V\int_0^\infty\varphi_f^{-t*}(F_\infty)dt ,\psi \rangle. 
\end{align*}
Therefore as $T\rightarrow+\infty$, $A_T$ converges (in the sense of currents) to
 $$
 A_\infty=\sum_{\operatorname{ind}(a)=1}\left(\int_{W^s(a)}A\right)[W^u(a)]+\int_0^\infty\varphi_f^{-t*}\iota_V(F_\infty)dt,
 $$
 where the last term is understood through the above limit in $\mathcal{D}^\prime$. We now prove the convergence of $dA_T$. Using again \eqref{e:duhamel-final}, we have 
\begin{align*}
dA_T&=d\varphi_f^{-T*}(A)+ d \iota_V\int_0^T\varphi_f^{-t*}\left(\mathrm{g}_{T-t}^{-1}F(A)\mathrm{g}_{T-t}\right)dt\\
&=\varphi_f^{-T*}(dA)+ d \iota_V\int_0^T\varphi_f^{-t*}\left(\mathrm{g}_{T-t}^{-1}F(A)\mathrm{g}_{T-t}\right)dt.
\end{align*}
The first term converges to $\left(\int_\Sigma dA \right) [a_{2g+2}]=0$  as $T\rightarrow \infty$ by Theorem \ref{t:harveylawson} applied in the matrix-valued case.  For the second term, we take  $u \in \Omega^0(\Sigma, \mathfrak{g})$, then by Cartan's formula
\begin{align*}
\left\langle d\iota_V\int_0^T\varphi_f^{-t*}\left(\mathrm{g}_{T-t}^{-1}F(A)\mathrm{g}_{T-t}\right)dt, u \right\rangle&=\left\langle \mathcal{L}_V\int_0^T\varphi_f^{-t*}\left(\mathrm{g}_{T-t}^{-1}F(A)\mathrm{g}_{T-t}\right)dt, u \right\rangle\\
 &=-  \int_0^T \left\langle \left(\mathrm{g}_{T-t}^{-1}F(A)\mathrm{g}_{T-t}\right),  \varphi_f^{t*} \mathcal{L}_V(u)\right\rangle dt\\
&= - \int_0^{\infty} \left\langle \left(\mathbf{1}_{[0,T]} \mathrm{g}_{T-t}^{-1}F(A)\mathrm{g}_{T-t}\right),  \varphi_f^{t*}  \mathcal{L}_Vu \right\rangle dt.
\end{align*}
Note that $\mathcal{L}_Vu$ is a smooth function that vanishes at critical points hence it belongs to the weighted space we used. Therefore 
we have exponential decay in $L^2$ norm of $\varphi_f^{t*}  \mathcal{L}_Vu $.
Again, by dominated convergence as above, the second term converges to 
\begin{align*}
 \int_0^{\infty} \left\langle F_\infty,  \varphi_f^{t*}  \mathcal{L}_V u \right\rangle dt &= -  \left\langle F_\infty,\int_0^{\infty}  \varphi_f^{t*}  \mathcal{L}_V u dt \right\rangle \\
 &=   \left\langle F_\infty, u-u(a_{2g+2} ) \right\rangle,
\end{align*}
where $\int_0^{\infty}  \varphi_f^{t*}  \mathcal{L}_V u dt$ is the limit as $T\rightarrow\infty$ of $\int_0^T\varphi_f^{t*}\mathcal{L}_Vu dt=\varphi^{T*}(u)-u$ in the sense of distributions (in fact in some $L^p$ space). By Theorem~\ref{t:harveylawson}, it is equal to $u(a_{2g+2})-u $. 
This implies that $$
dA_T\rightharpoonup F_\infty-\left(\int_{\Sigma}F_\infty\right)[a_{2g+2}].
$$
This concludes the proof of Theorem~\ref{t:normalform} except for the regularity of $A_\infty$ on $W^u(a_1)$ and for the last item. Regarding the regularity of $A_\infty$ on $W^u(a_1)$, it amounts to prove the regularity of 
$$
\beta_\infty:=\int_0^{\infty}\varphi_f^{-t*}(\iota_V(F_\infty))dt=\int_0^{\infty}\varphi_f^{-t*}\iota_V\left(F_\infty-\left(\int_\Sigma F_\infty\right) [a_{2g+2}]\right)dt.
$$
We now fix some $x_0\in W^u(a_1)$ and some smooth cutoff function $\psi\in\mathcal{C}^{\infty}_c(W^u(a_1))$ that is identically equal to $1$ in a small open neighborhood of the gradient orbit joining $a_1$ to $x_0$. Hence, for $x$ close enough to $x_0$, one has $$
\beta_\infty(x)=\int_0^{\infty}\varphi_f^{-t*}(\iota_V(\psi F_\infty))dt=\int_0^{\infty}\varphi_f^{-t*}\iota_V\left(\psi F_\infty-\left(\int_\Sigma \psi F_\infty\right) [a_{2g+2}]\right)dt.
$$
Observe now that $\psi F_\infty$ belongs to $\mathcal{C}^{\infty}(\Sigma)$. In particular, it belongs to the anisotropic Sobolev spaces\footnote{The results in this reference are given for $\C$-valued currents but we can apply them coordinates by coordinates as we are dealing with a trivial bundle.} from~\cite[\S4.1]{DR19} and one has $\psi F_\infty-\left(\int_\Sigma \psi F_\infty\right) [a_{2g+2}]=\left(\text{Id}-\Pi_0\right)(\psi F_\infty)$ where $\Pi_0$ is the spectral projector for the eigenvalue $0$ of $\mathcal{L}_V$ acting on these anisotropic spaces~\cite[Prop.~6.9]{DR19}. Hence, one has that, for $x$ near $x_0$, one has $\beta_\infty=\mathcal{L}_V^{-1}\iota_V\left(\text{Id}-\Pi_0\right)(\psi F_\infty)$ where $\mathcal{L}_V^{-1}$ is the resolvent of the Lie derivative acting on the anisotropic spaces according to~\cite[Prop.~4.2]{DR19}. In particular, $\beta_\infty$ belongs to these anisotropic Sobolev spaces. By construction, they can be chosen to have arbitrarily large Sobolev regularity on $W^u(a_1)$~\cite[\S4.1]{DR19}. Hence, $\beta_\infty$ is smooth in a small neighborhood of $x_0$ which is the last part in the third item of Theorem~\ref{t:normalform}. We are left with proving the last item of this theorem.

\begin{figure}
    \centering
    \includegraphics[scale=0.6]{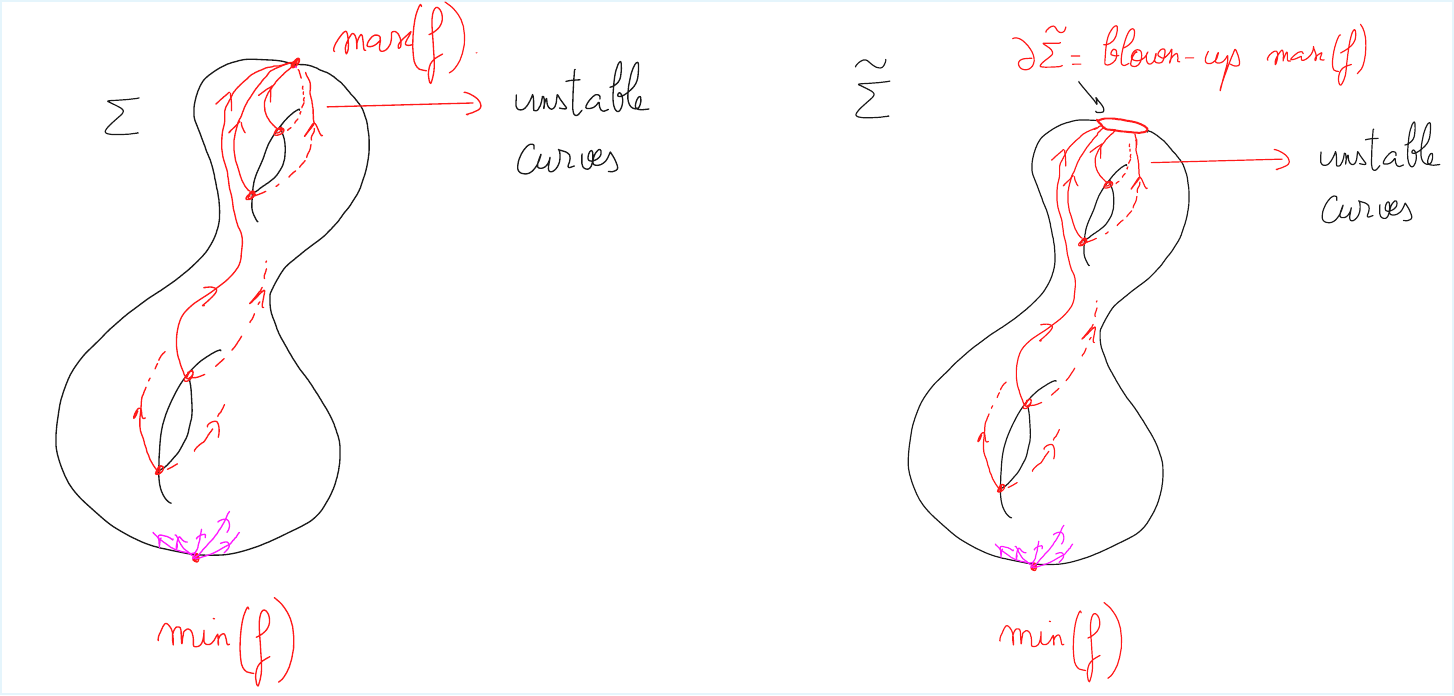}  
    \caption{Morse flow on $\Sigma$ and Blow-up at $\max (f)$}
    \label{fig:Morse_flow}
\end{figure}

\subsection{Slicing the space of connections by the Morse gauge}

In this paragraph, we examine in which precise sense the Morse gauge realizes a slicing of the space 
of connections, meaning that we prove the last item of Theorem~\ref{t:normalform}. Assume that $A_2={\rm g}^{-1}d{\rm g}+{\rm g}^{-1}A_1{\rm g}$ for some ${\rm g}\in \mathcal{C}^\infty(\Sigma,G)$ and observe that 
$$
\mathcal{L}_V^{A_2}={\rm g}^{-1}\mathcal{L}_V^{A_1}{\rm g},
$$ 
where $\mathcal{L}_V^A=\mathcal{L}_V+A(V)=\nabla^A\iota_V+\iota_V\nabla^A$ (with $\nabla^A=d+A$). Therefore,
the transport equation \eqref{e:def-gT} for $\mathrm{g}_{2}(t)$ reads
\begin{align*}
0=\partial_t u_2+  \mathcal{L}_V^{A_2}u_2= {\rm g}^{-1}\left( \partial_t +\mathcal{L}_V^{A_1} \right){\rm g}u_2, u_2(0,.)=\text{Id}_G.   
\end{align*}
Hence, on the one hand ${\rm g}u_2(t)$ solves the transport equation  $\left( \partial_t +\mathcal{L}_V^{A_1} \right)u=0$ with initial data $u(0)={\rm g}\in \mathcal{C}^\infty(\Sigma,G)$. On the other hand, if $u_1(t)$ solves 
 \begin{equation}
 \partial_tu+\mathcal{L}_V^{A_1} u=0,\quad u(t)=\text{Id}_G,
\end{equation}
then 
$ v(t)=u_1(t)\varphi_f^{-t*}({\rm g})$ solves
\begin{equation}
\partial_t v+ \mathcal{L}_V^{A_1} v=0,\quad v(0)={\rm g}\in \mathcal{C}^\infty(\Sigma, G).
\end{equation}
Hence, we get ${\rm g}u_2(t)= u_1 (t)  \varphi_f^{-t*} ({\rm g}) $, thus  $u_2(t)={\rm g}^{-1} u_1 (t)  \varphi_f^{-t*} ({\rm g})$ and  $ \mathrm{g}_{ 2,\infty}={\rm g}^{-1} g_{ 1,\infty} {\rm g}(a_1)$ on $W^u(a_1)$. Hence, by definition, one has
$$
F_{2,\infty}=\mathrm{g}_{2,\infty}^{-1}F(A_2)\mathrm{g}_{2,\infty}={\rm g}(a_1)^{-1}F_{1,\infty}{\rm g}(a_1).
$$
Recalling the second item of Theorem~\ref{t:normalform}, one finds from this last expression that
$$
A_{2,\infty}={\rm g}(a_1)^{-1}A_{1,\infty}{\rm g}(a_1)+\sum_{\text{ind}(a)=1}\left(\int_{W^s(a)}\left(A_2-{\rm g}(a_1)^{-1}A_1{\rm g}(a_1)\right)\right)[W^u(a)].
$$
In particular, $A_{2,\infty}={\rm g}(a_1)^{-1}A_{1,\infty}{\rm g}(a_1)$ on $W^u(a_1)$. Now, given a smooth path $\gamma:[0,1]\rightarrow\Sigma$ such that $\gamma(0),\gamma(1)\in W^u(a_1)$, one knows from Corollary~\ref{c:limit-holonomy} that
$$
\lim_{T\rightarrow+\infty}\text{Hol}_{\gamma}(A_{2,T})=\mathrm{g}_{2,\infty}(\gamma(1))^{-1}\text{Hol}_{\gamma}(A_{2})\mathrm{g}_{2,\infty}(\gamma(0))=\lim_{T\rightarrow+\infty}{\rm g}(a_1)^{-1}\text{Hol}_{\gamma}(A_{1,T}){\rm g}(a_1).
$$
This concludes the proof of Theorem~\ref{t:normalform}. In fact, as a by-product of this last item of Theorem~\ref{t:normalform}, one gets the following corollary.
\begin{corollary}[Classical admissible observables]\label{c:classical-observables} Let $\gamma_1,\gamma_2:[0,1]\rightarrow\Sigma$ be two $\mathcal{C}^1$ curves such that $\gamma_1(0),\gamma_1(1)$ belong to $W^u(a_1)$ and $\gamma_2([0,1])\subset W^u(a_1)$. Then, the maps
$$
A\in\Omega^1(\Sigma,\mathfrak{g})\mapsto \lim_{T\rightarrow +\infty}\operatorname{Hol}_{\gamma_1}(A_{T}) \in G,
$$
and
$$
A\in\Omega^1(\Sigma,\mathfrak{g})\mapsto\int_{\gamma_2} A_\infty\in\mathfrak{g}
$$
induce maps from the moduli space $\Omega^1(\Sigma,\mathrm{g})/\mathcal{C}^{\infty}(\Sigma,G)$ to $G/\mathbf{Ad}$ and $\mathfrak{g}/\operatorname{Ad}$ respectively, where $A_1\sim A_2\in \Omega^1(\Sigma,\mathfrak{g})$ means that $A_2=\mathrm{g}^{-1}d\mathrm{g}+\mathrm{g}^{-1}A_1\mathrm{g}$ for some $\mathrm{g}\in C^\infty(\Sigma,G).$
\end{corollary}
Here $G/\mathbf{Ad}$ (resp. $\mathfrak{g}/\operatorname{Ad}$) means that $\mathrm{g}_1\sim\mathrm{g}_2$ (resp. $\mathfrak{a}_1\sim\mathfrak{a}_2$) if there exists $\mathrm{h}\in G$ such that $\mathrm{g}_2=\mathrm{h}^{-1}\mathrm{g}_1\mathrm{h}$ (resp. $\mathfrak{a}_2=\mathrm{h}^{-1}\mathfrak{a}_1\mathrm{h}$). In other words, the first induced map takes values into adjoint orbits of $G$ while the second one takes values into adjoint orbits in $\mathfrak{g}$. Regarding the difference between the assumptions on $\gamma_1$ and $\gamma_2$, we notice that if we want to consider the first observables of the Corollary in terms of the Lie algebra $\mathfrak{g}$, there is an \emph{extra indetermination}
in the sense that we only know $\log\left(\mathrm{b}^{-1}\lim_{T\rightarrow +\infty}\operatorname{Hol}_{\gamma_1}(A_{T})  \mathrm{b}\right)$ which requires us to pick some determination of the logarithm. More precisely we only have access to the following subset of adjoint orbits in $\mathfrak{g}$:
$$
\left\{  \mathrm{b}^{-1}\mathfrak{a}  \mathrm{b}:\ \mathrm{b}\in G,\ \exp\left(\mathfrak{a}\right)=\lim_{T\rightarrow +\infty}\operatorname{Hol}_{\gamma_1}(A_{T})   \right\}.
$$
In the probabilistic set-up, we will define similar observables through the resolution of certain stochastic differential equations and up to appropriate assumptions on the curves $\gamma$.

\section{Random connections as solutions to random cohomological equations}\label{s:solve_random_connection_equation}

The goal of this section is to prove Theorem~\ref{t:random-cohomological} which will directly follow from the slightly more general statement from Theorem~\ref{thm:main1}. To that aim, we first collect a few definitions and properties of the \emph{white noise} on a compact Riemannian surface. Then, we explain how to solve the random cohomological equation~\eqref{e:random-cohomological-equation-intro}
with truncated white noise on the right hand side as forcing term. Finally, we describe the independence properties of these random solutions.

\subsection{White noise on compact Riemannian surfaces} Let us recall some basic facts on the white noise we shall use in the sequel.
Fix an orthonormal basis $(e_n)_{n\geq 1}$ of $L^2_\upsilon(\Sigma,\mathbb{R})$ endowed with the $L^2$-scalar product inherited from the Riemannian volume form $\upsilon$.
\begin{rmk}
The metric $h$ should not be confused with the metric $h_f$ we used to produce the gradient flow $V=\nabla f$. Here the volume $\upsilon$ is the one coming from $h$. Yet, observe that the corresponding $L^p$ spaces (as well as their weighted versions) are the same for the volume form coming from $h$ or from $h_f$.
\end{rmk}
 \begin{rmk} We do not require $(e_n)$ to be an orthonormal basis of Laplace eigenfunctions. However, in view of computing Sobolev norms, we will make use of the Laplace--Beltrami operator $\Delta_h$ acting on forms and one has
 $$
 \|\psi\|_{H^{\sigma}}^2:=\left\|(\operatorname{Id}+\Delta_h)^{\frac{\sigma}{2}}\psi\right\|_{L^2}^2.
 $$
 \end{rmk}

 We also let $(\mathfrak{b}_\ell)_{1\leq \ell\leq L}$ be an orthonormal basis of $\mathfrak{g}$ with $L=\dim\mathfrak{g}$ with respect to the inner product we have fixed on $\mathfrak{g}$. This naturally gives rise to an orthonormal basis of $L^2(\Sigma,\mathfrak{g})$ by letting $e_{n,\ell}(x):=e_n(x)\mathfrak{b}_\ell$. Once this basis is fixed, we define a random variable
\begin{equation}\label{eq_white_noise_omega}
    \xi:=\sum_{n\geqslant 1} \sum_{\ell=1}^LX_{n,\ell}e_{n,\ell}\in\mathfrak{g},
\end{equation}
where $(X_{n,\ell})_{n\geqslant 1,1\leqslant \ell\leqslant L}$ is a sequence of independent, identically distributed, real random variable with probability law $\mathcal{N}(0,1)$ on a probability space $\Omega:=\Omega_\xi$. This is a $\mathfrak{g}$-valued white noise and we emphasize that $\xi$ depends implicitly on the choice of the Riemannian metric $h$ on $\Sigma$ (through the choice of $\upsilon$). %For any $\psi\in L^\infty(\Sigma,\mathbb{R})$, this induces a random variable on the space of $2$-forms by letting
%$$
%F_\psi:= \xi \psi\upsilon.
%$$
%The Lemma~\ref{lem:whitenoiseregularity} shows that, for every $\sigma\in\mathbb{R}$,
%$$
%\mathbb{E}\left(\|F_\psi\|_{H^{\sigma}}^2\right)\leq L\|\psi\|_{L^\infty}^2\sum_{\lambda\in \sigma(\Delta_h)} (1+\lambda)^{\sigma},
%$$
%which is finite as soon as $\sigma<-1$ by Weyl's law.
More precisely, one has

\begin{lemma}\label{lem:whitenoiseregularity}
For every $\psi\in L^\infty(\Sigma,\R)$, the random series
\begin{equation}
F_\psi^N:=\sum_{n= 1}^N\sum_{\ell=1}^L  X_{n,\ell} \psi e_n \mathfrak{b}_\ell    
\end{equation}
converges almost surely to some limit $F_\psi$ in $L^2(\Omega,H^{-1-\kappa}(\Sigma,\mathfrak{g}))$ for all $\kappa>0$. 
\end{lemma}
In view of emphasizing its probabilistic nature, the limit will sometimes be denoted by $F_\psi:=\psi\xi\upsilon$ even if the product $\xi\psi$ is a priori ill-defined in the deterministic sense. See Lemma~\ref{r:whitenoise-multiplication} for more details on the choice of regularization.
\begin{proof}
By definition of the truncated series and as $X_{n,\ell}\sim \mathcal{N}(0,1)$, we get
\begin{align*}
\mathbb{E} \left( \Vert F_\psi^N \Vert^2_{H^{\sigma}} \right)=L\sum_{n=1}^N \left\langle e_n, \left(\mathbf{m}_\psi\left(1+\Delta_h \right)^{\sigma}  \mathbf{m}_\psi \right) e_n\right\rangle_{L^2(\Sigma)} 
\end{align*}
where $\mathbf{m}_\psi$ denotes the multiplication operator by $\psi\in L^\infty$.
Now note that for $\sigma<-1$, the operator $\left(1+\Delta_h \right)^{\sigma} $ is trace class by Weyl's law since $\mathbf{Tr}_{L^2}(1+\Delta_h)^\sigma=\sum_{\lambda\in \sigma(\Delta_h)}(1+\lambda)^\sigma <+\infty$. Therefore, the composite operator $\left(\mathbf{m}_\psi\left(1+\Delta_h \right)^{\sigma}  \mathbf{m}_\psi \right)$ is trace class as the composition of bounded operators $\mathbf{m}_\psi$ and  $\left(1+\Delta \right)^{\sigma} $ is trace class. Here we used the fact that trace class operators form a two sided ideal inside bounded operators on $L^2$ by \cite[Thm VI.19 p.~207]{ReedSimonVol1}. Therefore the series $\sum_{n=0}^\infty \left\langle e_n, \left(\mathbf{m}_\psi\left(1+\Delta_h \right)^{\sigma}  \mathbf{m}_\psi \right) e_n\right\rangle_{L^2(\Sigma)} $ is absolutely summable and converges to
$\mathbf{Tr}_{L^2}\left( \mathbf{m}_\psi\left(1+\Delta \right)^{\sigma}  \mathbf{m}_\psi \right)$ by~\cite[Theorem VI.18 p.~206]{ReedSimonVol1}.
\end{proof}
In the following, we shall also write 
$$
F_\psi=\xi_\psi\upsilon,\ \text{or}\ F_\psi=\xi_U\upsilon\ \text{when}\ \psi=\mathbf{1}_U\ \text{for some measurable set}\ U\subset \Sigma,
$$ 
when we want to emphasize the probabilistic nature of our random curvature. In the following, we will mostly take $U$ to be of the form $\{x\in \Sigma; f(x)\in I\}$ where $I$ is an interval of $\R$ with nonempty interior. One also has
\begin{lemma}\label{r:whitenoise-multiplication}
Let $\psi,\widetilde{\psi}\in L^\infty(\Sigma)$. For every sequence $(\widetilde{\psi}_n)_{n\geq 1}\in \mathcal{C}^\infty(\Sigma)$ that converges to $\widetilde{\psi}$ in $L^p$ (for every $p<\infty$), $(\xi_\psi\widetilde{\psi}_n)_{n\geq 1}$ also converges to $\xi_{\psi\widetilde{\psi}}$. In particular, the $\sigma$-algebra $\sigma(\xi_{\psi\widetilde{\psi}})$ generated by
$$
\left\{\langle\xi_{\psi\widetilde{\psi}},\phi\rangle^{-1}(B):\ \phi\in\mathcal{C}^\infty\left(\Sigma\right),\ B\ \text{Borel set of}\ \Sigma\right\}
$$
is contained in $\sigma(\xi_{\psi})$, i.e. the same $\sigma$-algebra with $\psi$ replacing $\psi\widetilde{\psi}$ in the definition.
\end{lemma}
\begin{proof}
  For such a sequence $(\widetilde{\psi}_n)_{n\geq 1}$, $\widetilde{\psi}_n\xi_\psi$ is well-defined in the sense of distributions and one has that $\xi_\psi\widetilde{\psi}_n$ is a Cauchy sequence that converges to $\xi_{\psi\widetilde{\psi}}$ in $L^{2}(\Omega,H^{-1-\kappa})$ as defined above. Indeed, for any $\phi\in \mathcal{C}^\infty(\Sigma)$ and as in the proof of Lemma~\ref{lem:whitenoiseregularity}, one has
$$
\mathbb{E}\left(\left\|\xi_\psi\phi\right\|^2_{H^{-1-\kappa}}\right)=L\mathbf{Tr}_{L^2}\left(\mathbf{m}_{\psi\phi}(1+\Delta_h)^{\sigma}\mathbf{m}_{\psi\phi}\right).$$
We write this trace using an orthonormal basis $(\mathbf{e}_\lambda)_{\lambda\in\sigma(\Delta_h)}$ and we find that
$$
\mathbb{E}\left(\left\|\xi_\psi\phi\right\|^2_{H^{-1-\kappa}}\right)=L\sum_{\lambda\in\sigma(\Delta_h)}(1+\lambda)^{\sigma}|\langle\psi\phi,\mathbf{e}_\lambda\rangle|^2\leq \|\psi\phi\|_{L^p}^2\sum_{\lambda\in\sigma(\Delta_h)}(1+\lambda)^{-1-\kappa}\|\mathbf{e}_\lambda\|_{L^{\frac{2p}{p-2}}}^2,
$$
for any $2\leq p\leq \infty$. By Sobolev injection, one has $\|\mathbf{e}_\lambda\|_{L^{\infty}}=\mathcal{O}((1+\lambda))$ and, thus by interpolation, $\|\mathbf{e}_\lambda\|_{L^{\frac{2p}{p-2}}}=\mathcal{O}((1+\lambda)^{\frac2p})$. Hence, taking $p<\infty$ large enough to ensure that $\frac{4}{p}<\kappa$ and thanks to the Weyl law, one finds
$$
\forall\phi\in \mathcal{C}^\infty(\Sigma),\quad\mathbb{E}\left(\left\|\phi\xi_\psi\right\|^2_{H^{-1-\kappa}}\right)\leq C_p\|\phi\psi\|_{L^p}^2,
$$
from which we can infer that $\widetilde{\psi}_n\xi_\psi$ converges to $\xi_{\widetilde{\psi}\psi}$ (as defined above). Finally, for the inclusion of the $\sigma$-algebra, it follows from the first part and from the fact that $\psi$ can be approximated by a sequence in $\mathcal{C}_c^\infty(U)$.
\end{proof}

\begin{rmk}\label{r:independence-white-noise} Given two measurable subsets $U_1$ and $U_2$ such that $\mathbf{1}_{U_1}\mathbf{1}_{U_2}=0$ $\upsilon$-almost everywhere, the two corresponding white noises $\xi_{U_1}$ and $\xi_{U_2}$ are independent using L\'evy's criterion. See Lemma~\ref{l:independent-connections} for the related case of random connections. Hence, if we are given two bounded continuous functionals $F_1$ and $F_2$ on $H^{-1-\kappa}\left(\Sigma\right)$, one has 
$$
 \E( F_1(\xi_{U_1})F_2(\xi_{U_2}) ) = 
 \E(F_1(\xi_{U_1}))\mathbb{E}(F_2(\xi_{U_2})).
$$
\end{rmk}

\subsection{Solving random cohomological equations}
We now prove the main result of this section from which Theorem~\ref{t:random-cohomological} follows by taking the case $\psi=1$.
\begin{thm}
\label{thm:main1}
Let $\xi$ be the $\gfrak$-valued white noise on $\Sigma$ defined in \eqref{eq_white_noise_omega} and let $\psi\in L^\infty(\Sigma)$. Set $F_\psi=\xi_\psi\upsilon.$ 
Then the $\mathfrak{g}$-valued random variable
$$
A_{\psi}=\mathcal{L}_V^{-1}\left( \iota_V F_\psi \right):=\lim_{T\rightarrow+\infty}\int_0^T\varphi_f^{-t*}\left(\iota_V F_\psi\right)dt
$$ 
converges in $L^2(\Omega, H^{-1-\kappa}({\Sigma}, \gfrak))$ for every $\kappa>0$. Moreover, one has
$$
A_\psi=\sum_{n\geqslant 1} \sum_{\ell=1}^LX_{\lambda,\ell}\mathcal{L}_V^{-1}\left(\psi e_{n,\ell}\upsilon(V)\right),
$$
where the sum also converges in $L^2(\Omega, H^{-1-\kappa})$.
Finally, $$
\iota_V(A_{\psi})=0,
$$
and
$$
d  A_\psi=F_\psi- \left(\int_\Sigma F_\psi\right)[a_{2g+2}].
$$
\end{thm}
In the proof of this Theorem, we will in fact verify that $\mathcal{L}_V^{-1}\left(\psi e_{n,\ell}\upsilon(V)\right)$ belongs to every $L^q$-space with $1<q<2$. One of the key points in this statement is that we are able to find a connection in Morse gauge (i.e. verifying $\iota_V(A_\psi)=0$) for a generic realization of the white noise. In particular, it solves the random cohomological equation $\mathcal{L}_V(A_\psi)=\iota_V(F_\psi)$. Moreover, regarding the statement in Theorem~\ref{t:normalform}, the curvature of this random connection can be considered to be the rescaled white noise $F_\psi=\xi _\psi\upsilon$. Later on, when constructing the Yang-Mills measure, we will for instance pick $\psi$ to be equal to the characteristic functions of certain open sets $U$, i.e. $\psi=\mathbf{1}_{U}$. In that case, it amounts to take a white noise $\xi_U:=\xi_{ \mathbf{1}_{U}}$ on $U\subset\Sigma $.

\begin{proof}
Let us show that $\mathcal{L}_V^{-1}\iota_{V}(F_\psi)$ is indeed well defined almost surely. To do that recall that $F_\psi$ belongs to $L^2(\Omega, H^\sigma(\Sigma,\Lambda^2(T^*\Sigma)\times\mathfrak{g})).$ Hence, one can define 
$$
\int_0^{T}\varphi_f^{-t*}(\iota_V(F_\psi))dt=\sum_{n\geqslant 1}\sum_{\ell=1}^L X_{n,\ell}\int_{0}^T\varphi_{f}^{-t*}(\psi e_{n,\ell}\iota_V(\upsilon)).
$$
Let us show that this defines a Cauchy sequence. To that aim, we compute 
\begin{align*}
\mathbb{E}\left(\left\|\sum_{n,\ell} X_{n,\ell} \int_{T_1}^{T_2}\varphi_{f}^{-t*}(\psi e_{n,\ell}\iota_V(\upsilon))dt\right\|_{H^{\sigma}}^2\right)=\sum_{n,\ell}\left\| \int_{T_1}^{T_2}\varphi_{f}^{-t*}(\psi e_{n,\ell}\iota_V(\upsilon))dt\right\|_{H^{\sigma}}^2
\end{align*}
where we used independence. Now by definition of Sobolev norms of currents of degree $1$~:
\begin{align*}
\left\| \int_{T_1}^{T_2}\varphi_{f}^{-t*}(\psi e_{n,\ell}\iota_V(\upsilon))dt\right\|_{H^{\sigma}}^2=\sum_{\lambda',\ell'} (1+ \lambda')^{\sigma} \left| \left\langle \widetilde{\mathbf{e}}_{\lambda',\ell'},\int_{T_1}^{T_2}\varphi_{f}^{-t*}(\psi e_{n,\ell}\iota_V(\upsilon))dt\right\rangle_{L^2(\Sigma, \Lambda^1(T^*\Sigma)\times\mathfrak{g})} \right| ^2\\
 =\sum_{\lambda',\ell'} (1+ \lambda')^{\sigma} \left| \left\langle \int_{T_1}^{T_2}\varphi_{f}^{t*} \iota_V\left( \star_h\widetilde{\mathbf{e}}_{\lambda',\ell'} \right) dt,\psi e_{n,\ell}\upsilon\right\rangle \right| ^2 
\end{align*}
where $(\widetilde{\mathbf{e}}_{\lambda',\ell})_{\lambda'\in \sigma(\Delta_{h,1}),\ell=1,\dots,L}$ denotes an orthonormal basis of the Hodge de Rham Laplacian $\Delta_{h,1}$ acting on $\mathfrak{g}$--valued $1$--forms and $\lambda'$ runs over the spectrum $\sigma(\Delta_{h,1})$ of the Hodge Laplacian on $1$--forms.
The first pairing is a scalar product on $1$--forms whereas the second pairing corresponds to the de Rham duality between ($\mathfrak{g}$-valued) differential forms of degree $0$ and $2$.
  
Therefore the above expectation rewrites~:
\begin{align*}
\mathbb{E}\left(\left\| \int_{T_1}^{T_2}\varphi_{f}^{-t*}(\iota_VF_\psi)dt\right\|_{H^{\sigma}}^2\right)=\sum_{n,\ell}\sum_{\lambda',\ell'}(1+\lambda')^{\sigma}\left|\left \langle \psi {e}_{n,\ell}\upsilon,\int_{T_1}^{T_2}\varphi_{f}^{t*}\left( \iota_V (\star \widetilde{\mathbf{e}}_{\lambda',\ell'})\right)dt\right\rangle\right|^2\\
\leq \|\psi\|_{L^\infty}\sum_{\lambda',\ell'}(1+\lambda')^{\sigma}\left\|\int_{T_1}^{T_2}\varphi_{f}^{t*}\left( \iota_V(\star \widetilde{\mathbf{e}}_{\lambda',\ell'})\right)dt\right\|^2_{L^2}.
\end{align*}

At this point, we will use the same weighted norm as in the proof of Theorem~\ref{t:normalform}, $\|\cdot\|_{\mathcal{Z}_{\delta,p}^0}$, where $p\geq 2$ and $\delta\geq 0$. Since $V$ vanishes at order $1$ at the minimum of $f$, one has that $\|\iota_V (\star \widetilde{\mathbf{e}}_{\lambda',\ell'})\|_{\mathcal{Z}_{\delta,p}^0}<\infty$ for every $p>2$ and every $\delta<2+p$. For $\delta\leq p$, one has in fact 
$$
\|\iota_V (\star \widetilde{\mathbf{e}}_{\lambda',\ell'})\|_{\mathcal{Z}_{\delta,p}^0}\leq C_{\delta,p}\| \widetilde{\mathbf{e}}_{\lambda',\ell'}\|_{L^p}.
$$
Now by Sobolev embeddings $\Vert \widetilde{\mathbf{e}}_{\lambda',\ell'} \Vert_{L^\infty(T^*\Sigma)}\lesssim \Vert \widetilde{\mathbf{e}}_{\lambda',\ell'}\Vert_{H^{1+\varepsilon}}= (1+\lambda')^{\frac{1+\varepsilon}{2}}$ for all $\varepsilon>0  $. Therefore by interpolation we get that, for all $p>2$,
$$
\|\iota_V (\star \widetilde{\mathbf{e}}_{\lambda',\ell'})\|_{\mathcal{Z}_{\delta,p}^0}\leq C_{\delta,p,\varepsilon} (1+\lambda')^{\frac{(1+\varepsilon)(p-2)}{2p}}. 
$$
Now, picking $2<\delta<p$, one can apply Theorem~\ref{t:contraction} together with this bound on the $\| . \|_{\mathcal{Z}_{\delta,p}^0} $ norms. It implies that, for every $p>2$,
$$
\left\|\int_{T_1}^{T_2}\varphi_{f}^{t*}\left(\star_h\widetilde{\mathbf{e}}_{\lambda',\ell'}(V)\right)dt\right\|^2_{L^2}\leq C_p e^{-\beta_p\min(T_1,T_2)}(1+\lambda')^{\gamma(p)},
$$
for some $\gamma(p)\rightarrow 0$ as $p\rightarrow 2$.
Hence 
$$
\mathbb{E}\left(\left\| \int_{T_1}^{T_2}\varphi_{f}^{-t*}(\iota_VF_\psi)dt\right\|_{H^{\sigma}}^2\right)\leq C\sum_{\lambda',\ell'} e^{-\beta_p\min(T_1,T_2)}(1+\lambda')^{\gamma(p)+\sigma}, 
$$ 
which converges absolutely for all $\sigma<-1$ (and up to choosing $p$ close enough to $2$). Hence, we have a Cauchy sequence in $L^2(\Omega,H^{s}(\Sigma,\mathfrak{g}))$ for every $s<-1$. In particular, $\mathcal{L}_V^{-1}(\iota_V(F_\psi))$ exists almost surely. Let us now consider the following sum
$$
\forall N\geqslant 1,\quad \tilde{A}_{\psi,N}=\sum_{n=1}^N\sum_{\ell=1}^LX_{n,\ell}\mathcal{L}_{V}^{-1}(\iota_V(\psi e_{n,\ell}\upsilon)),
$$
where we recall that, thanks to Theorem~\ref{t:contraction} (applied with the flow in positive time as in the first part of the proof), one has, for every $\Psi\in\Omega^1(\Sigma,\mathfrak{g})$, 
\begin{equation}\label{e:bound0}
\forall n\geqslant 1,\quad \left|\langle \Psi,\mathcal{L}_{V}^{-1}(\iota_V(\psi e_{n,\ell}\upsilon))\rangle\right|\leq C_\psi\|\Psi\|_{\mathcal{Z}_{p,\delta}^1},
\end{equation}
where $p>2$ and $2<\delta<p$ and where $C_\psi$ is independent of $n\geq 1$. In particular, $\mathcal{L}_{V}^{-1}(\iota_V(e_{n,\ell}\upsilon))$ belongs to every $L^q$ space with $1<q<2$. In other words, $\tilde{A}_{\psi,N}$ is an element in $L^q$ for every $q<2$. Again, we can compute
$$
\mathbb{E}\left(\left\|\tilde{A}_{\psi,N_1}-\tilde{A}_{\psi,N_2}\right\|_{H^\sigma}^2\right)=\sum_{n=N_1+1}^{N_2}\sum_{\ell}\sum_{\lambda',\ell'}(1+\lambda')^{\sigma}\left|\langle \widetilde{\mathbf{e}}_{\lambda',\ell'},
\mathcal{L}_{V}^{-1}(\iota_V(\psi e_{n,\ell}\upsilon)\rangle\right|^2.
$$
Using~\eqref{e:bound0} with $\delta\leq p$, one finds that
$$
\mathbb{E}\left(\left\|\tilde{A}_{\psi,N_1}-\tilde{A}_{\psi,N_2}\right\|_{H^\sigma}^2\right)\leq C \sum_{N_1<n\leq N_2,\ell}\sum_{\lambda',\ell'}(1+\lambda')^{\sigma}\left\| \widetilde{\mathbf{e}}_{\lambda',\ell'}\right\|_{L^p}^2.$$
Thanks to the Sobolev injection and as it is valid for every $p<2$, we have a Cauchy sequence so that $\tilde{A}_{\psi,N}$ converges in $L^2(\Omega, H^{-1-\kappa})$ to  
$$
\tilde{A}_\psi=\sum_{n,\ell}X_{n,\ell}\mathcal{L}_{V}^{-1}(\iota_V(\psi e_{n,\ell}\upsilon)).
$$
We are left with verifying that this is indeed equal to $A_\psi=\iota_V\mathcal{L}_V^{-1}(F_\infty)$. To see this, we write
$$
\tilde{A}_\psi-\int_0^{T}\varphi_f^{-t*}(\iota_V(F_\psi))=\sum_{n,\ell}X_{n,\ell}\int_{T}^{+\infty}\varphi_f^{-t*}(\iota_V(\psi e_{n,\ell}\upsilon))dt,
$$
and the same argument as for the convergence of $\int_0^{T}\varphi_f^{-t*}(\iota_V(F_\psi))$ shows the convergence to $0$ in $L^2(\Omega, H^{-1-\kappa})$ for every $\kappa>0$.

Finally, from the expression of $A_\psi$ as a converging sum, one has directly that $\iota_V(A_\psi)=0$ as expected. Similarly, one has
$$
dA_\psi= d \iota_V\mathcal{L}_V^{-1}(F_\psi)=\sum_{n,\ell}X_{n,\ell}d\mathcal{L}_{V}^{-1}(\iota_V(\psi e_{n,\ell}\upsilon)).
$$
For  $u \in \Omega^0(\Sigma, \mathfrak{g})$, we have
\begin{align*}
\left\langle d\mathcal{L}_{V}^{-1}(\iota_V(\psi e_{n,\ell}\upsilon)),   u \right\rangle &=  -\left\langle \mathcal{L}_{V}^{-1}(\iota_V(\psi e_{n,\ell}\upsilon)),   du \right\rangle \\
&=-  \lim_{T\rightarrow \infty} \left\langle \int_0^T \varphi_f^{-t*}(\iota_V(\psi e_{n,\ell}\upsilon)),   d u \right\rangle.\\
&=  \lim_{T\rightarrow \infty}  \left\langle d\iota_V\int_0^T\varphi_f^{-t*}\left(\psi e_{n,\ell}\upsilon \right)dt, u\right \rangle\\
& = \lim_{T\rightarrow \infty} \left\langle \mathcal{L}_V\int_0^T\varphi_f^{-t*}\left(\psi e_{n,\ell}\upsilon \right)dt, u \right\rangle, \, \text{ (by Cartan's formula)}\\
 &= -\lim_{T\rightarrow \infty}   \int_0^\infty \left\langle {\bf 1}_{[0,T]} \psi e_{n,\ell}\upsilon  ,  \varphi_f^{t*} \mathcal{L}_V(u)\right\rangle dt.
\end{align*}
By the same argument in the proof of Theorem~\ref{t:normalform} (when proving the convergence of $dA_T$), i.e. using the fact that $\mathcal{L}_Vu \in \mathcal{Z}^0_{\delta, p}$ is a smooth function that vanishes at critical points, and the dominated convergence theorem, the limit is 
\begin{equation}
\int_0^\infty \left\langle \psi e_{n,\ell}\upsilon ,  \varphi_f^{t*} \mathcal{L}_V(u)\right\rangle dt= \left\langle \psi e_{n,\ell}\upsilon , \int_0^\infty  \varphi_f^{t*} \mathcal{L}_V(u)\right\rangle=\langle  \psi e_{n,\ell}\upsilon, u(a_{2g+2}) -u  \rangle, 
\end{equation}
from which we infer that
$$
d\mathcal{L}_{V}^{-1}(\iota_V(\psi e_{n,\ell}\upsilon))=\psi e_{n,\ell}\upsilon-\left(\int_{\Sigma}\psi e_{n,\ell}\upsilon\right)[a_{2g+2}].
$$
\end{proof}

\subsection{Independent random connections}

For later applications, we conclude this section by discussing the independence properties of our random connections when we pick $\psi_1=\mathbf{1}_{U_1}$ and $\psi_2=\mathbf{1}_{U_2}$ with $U_1$ and $U_2$ disjoint measurable subsets of $\Sigma$. Indeed, one has
\begin{lemma}\label{l:independent-connections} Let $U_1$ and $U_2$ be two measurable subsets of $\Sigma$ such that $\mathbf{1}_{U_1}\mathbf{1}_{U_2}=0$ $\upsilon$-almost everywhere. Then, the random connections $A_{U_1}:=A_{\mathbf{1}_{U_1}}$ and $A_{U_2}:=A_{\mathbf{1}_{U_2}}$ from Theorem~\ref{thm:main1} are independent.
\end{lemma}

As in Remark~\ref{r:independence-white-noise}, the proof of this result follows from the facts that $\xi_{U_1}$ and $\xi_{U_2}$ are independent and that $\iota_V\mathcal{L}_V^{-1}$ has well-behaved probabilistic properties. We just provide the proof for the sake of completeness.
\begin{proof}
Let $\Psi_1,\Psi_2\in\Omega^1(\Sigma)$, let $1\leqslant k,\ell\leqslant L$ and set
$$
Z_{1,k}:=\left\langle \iota_V\mathcal{L}_{V}^{-1}(\mathbf{1}_{U_1}\langle\xi,\mathfrak{b}_k\rangle\upsilon),\Psi_1\right\rangle,\qquad
Z_{2,\ell}:=\left\langle \iota_V\mathcal{L}_{V}^{-1}(\mathbf{1}_{U_2}\langle\xi,\mathfrak{b}_\ell\rangle\upsilon),\Psi_2\right\rangle.
$$
Both are centered and depend linearly on $\xi$, hence belong to the first Wiener chaos of $\xi$. Any finite family of such variables is jointly Gaussian, so that independence is equivalent to decorrelation and it suffices to compute $\mathbb{E}(Z_{1,k}Z_{2,\ell})$. Using that the Gaussian coefficients of $\xi$ are independent and standard, and then duality, we find
\begin{align*}
\mathbb{E}\left(Z_{1,k}Z_{2,\ell}\right)
&=\delta_{k\ell}\sum_{n\geqslant 1}\left\langle \iota_V\mathcal{L}_{V}^{-1}(\mathbf{1}_{U_1}e_n\upsilon),\Psi_1\right\rangle\left\langle \iota_V\mathcal{L}_{V}^{-1}(\mathbf{1}_{U_2}e_n\upsilon),\Psi_2\right\rangle\\
&=\delta_{k\ell}\left\langle \mathbf{1}_{U_1}\iota_V\mathcal{L}_{-V}^{-1}(\Psi_1),\mathbf{1}_{U_2}\iota_V\mathcal{L}_{-V}^{-1}(\Psi_2)\right\rangle_{L^2_\upsilon(\Sigma)}\\
&=0,
\end{align*}
since $\mathbf{1}_{U_1}\mathbf{1}_{U_2}=0$ $\upsilon$-almost everywhere. The series converges because $\iota_V\mathcal{L}_{-V}^{-1}(\Psi_j)$ belongs to $L^2$, by Theorem~\ref{t:contraction} applied in positive time. Since such pairings generate the $\sigma$-algebras of $A_{U_1}$ and $A_{U_2}$, these two random connections are independent.
\end{proof}

\section{Integrating random connections along curves}
\label{s:integration-connection}

In order to define the Yang-Mills measure and to verify some of its main properties, we need to define a random holonomy process which is $G$-valued. This will be achieved in \S\ref{s:randomholonomy} but before that, we need to define its logarithmic version which is $\gfrak$-valued. 
Our goal in this section is to show that $A_{\psi}=\iota_V\mathcal{L}_V^{-1}\left(F_\psi\right)$ as defined in Theorem~\ref{thm:main1}
satisfies the area law. More precisely, we need to establish two facts. First, we want to trace $A_\psi$ on some large class of curves $\gamma$ despite its low regularity. In other words, we want to define a probabilistic analogue of the classical observables from Corollary~\ref{c:classical-observables}, i.e. give a probabilistic meaning of the formal integrals:
$$
W^\mathfrak{g}_\psi(\gamma):=\int_{\gamma} A_\psi.
$$
Recall that, in a deterministic way, $A_\psi$ is only $H^{-1-}(\Sigma)$ in Sobolev regularity and that the current of integration $[\gamma]$ has regularity $H^{-1/2-}(\Sigma)$. This makes this integral ill-defined from the point of view of the classical theory of distributions and there is no wavefront set property to help us as $A_\psi$ has isotropic $H^{-1-}$ regularity. Once this issue is settled, we need to compute the covariance
of the resulting Gaussian process in terms of the geometry of the Morse flow and the surface $\Sigma$. More precisely, we will express
$$
\Phi_\psi(\gamma):=\mathbb{E}\left(\left\|W^\mathfrak{g}_\psi(\gamma)\right\|^2_{\mathfrak{g}}\right)
$$
in terms of areas with respect to the area form $\upsilon$ used in our construction and of domains built from $\gamma$ and the gradient flow $\varphi_f^t$. This is in agreement with Yang--Mills theory where the law of random holonomies along loops are expressed in terms of areas surrounded by the loops~\cite{Witten1, Witten2}.
Along the way, this procedure defines a new class of quantum observables $W^{\mathfrak{g}}_\psi(\gamma)$ as in Corollary~\ref{c:classical-observables} which can be thought as logarithmic holonomies.

This long section is organized as follows. First, in \S\ref{ss:geometric-preliminary}, we introduce some families of curves along which we aim at defining $W^\mathfrak{g}(\gamma)$ and random holonomies. We gather some of their basic properties that will be extensively used in the following paragraphs and sections. Then, in~\S\ref{ss:arealawtransverse}, we explain how to integrate the random connections from Theorem~\ref{thm:main1} against these families of curves. After, we show in \S\ref{ss:g-valued-BM} that the resulting random variables define a $\mathfrak{g}$-valued Brownian motion. Finally, in~\S\ref{ss:blowup-Wiener-max}, we explain how this construction can also be carried out for small curves shrinking near $a_{2g+2}$.

\subsection{Geometric preliminaries}
\label{ss:geometric-preliminary}
%\textcolor{red}{Viet: insert some pictures to illustrate the three types of curves.}

\subsubsection{Elementary curves}

Recall that the critical points of $f$ were ordered as $f(a_1)<f(a_2)<\ldots <f(a_{2g+2})$. Given $x\in\Sigma\setminus\{a_{2g+2}\}$, we define the backward flowline $\mathscr{L}_x$ of $x$ as follows.
\begin{itemize}
 \item If $x\in W^u(a_1)$, we set
 $$
 \mathscr{L}_x:=\overline{\left\{\varphi_f^{-t}(x):t\geq 0\right\}},
 $$
 which is diffeomorphic to a compact segment.
 \item If $x\in W^u(a_j)$ for some $2\leq j\leq 2g+1$, we set
 $$
 \mathscr{L}_x:=\overline{\left\{\varphi_f^{-t}(x):t\geq 0\right\}}\cup\overline{W^{s}(a_j)}.
 $$
 In this second case, note that the corresponding curve is homeomorphic to a circle to which we have attached at the point $a_j$ a compact segment (eventually reduced to a point if $x=a_j$). We can decompose $\mathscr{L}_x$ into two curves homeomorphic to closed segments and whose intersection is given by $\overline{\left\{\varphi_f^{-t}(x):t\geq 0\right\}}$. Moreover, we write $$
 \mathscr{L}_x=\mathscr{L}_x^+\cup\mathscr{L}_x^{-},
 $$
 where $\mathscr{L}_x^\pm$ is the union of $\overline{\left\{\varphi_f^{-t}(x):t\geq 0\right\}}$ and $\overline{\left\{\varphi_f^{-t}(y):t\in\R\right\}}$ where $y$ is of the form $(0,x_2)$ in the Morse coordinates~\eqref{e:local-expression-saddle} with $\pm x_2> 0$.
\end{itemize}

The intuition here is that there are two ways to get a \textbf{broken Morse flowline} starting from $x\in W^u(a_j)$ going down to the saddle point $a_j$ and going down along some piece of stable curve to the critical point $a_1$. Each curve $\mathscr{L}_x^\pm$ corresponds to the two choices of broken curves.

In the following, we choose to orient $\mathscr{L}_x$ starting from the point $x$ and this gives rise to a current of integration $[\mathscr{L}_x]$. Similarly, we define $[\mathscr{L}_x^\pm]$.
In fact every smooth compact curve $\gamma$ which is transverse to the flow can be decomposed into simple building blocks 
that we call \textbf{elementary curves}.
We now define this notion of elementary curve and Figure \ref{fig:curve_types} illustrates the three types of curves we introduce. 
\begin{definition}[Elementary curves]\label{def_elementaty_curve}
Let $\gamma:[0,1]\rightarrow\Sigma$ be a $\mathcal{C}^1$ curve. We say that $\gamma$ is of elementary type if $\gamma([0,1])$ is transverse\footnote{It precisely means that, for every $t\in[0,1]$, $\text{Span}\{\gamma'(t),V(\gamma(t))\}=T_{\gamma(t)}\Sigma$.} to the flow lines of $\varphi_f^t$, $\gamma((0,1))\subset W^u(a_1)$ and if
$$
N(\gamma):=\sharp \gamma(\{0,1\})\cap (\Sigma\setminus W^u(a_1))\leq 1.
$$
Equivalently, $\gamma$ cuts at most once the union $\cup_{\ind(a)=1}W^u(a)$ of unstable curves and only at endpoints.

An elementary curve is said to be \textbf{positively (resp. negatively) $V$-oriented} if, for every $t\in[0,1]$, $\gamma'(t)\wedge V(\gamma(t))$ has the same (resp. opposite) orientation compared to $\Sigma$. A $V$-oriented elementary curve is said to be \textbf{primitive} if the map
$$
(s,t)\in(0,1)\times\mathbb{R}_+\mapsto\varphi_f^{-t}(\gamma(s))
$$
is injective.
\end{definition}

\begin{definition}[Type $\operatorname{I}$ and $\operatorname{II}$ curves] Let $\gamma$ be a $V$-oriented and primitive elementary curve. We say that $\gamma$ is of type $\operatorname{I}$ if $N(\gamma)=0$, i.e. $\gamma([0,1])\cap\cup_{\ind(a)=1}W^u(a)=\emptyset$. Otherwise, we say that $\gamma$ is of type $\operatorname{II}$.

For a curve $\gamma$ of type $\operatorname{II}$, we denote by $a(\gamma)$ the only saddle point of $f$ such that $\gamma(\{0,1\})\cap W^u(a(\gamma))\neq\emptyset$ and we say that it is of type $\operatorname{II}_\pm$ if there exists a sequence $(t_n)_{n\geq 1}$ in $(0,1)$ such that $\gamma(t_n)\in  W^u(a(\gamma))$ and such that, for $n$ large enough, the orbit of $\varphi_f^{-t}(\gamma(t_n))$ enters the region $\pm x_2\geq 0$ (in the Morse chart near $a(\gamma)$) for $t\geq 0$. 
\end{definition}
\begin{figure}
    \centering
    \includegraphics[scale=0.4]{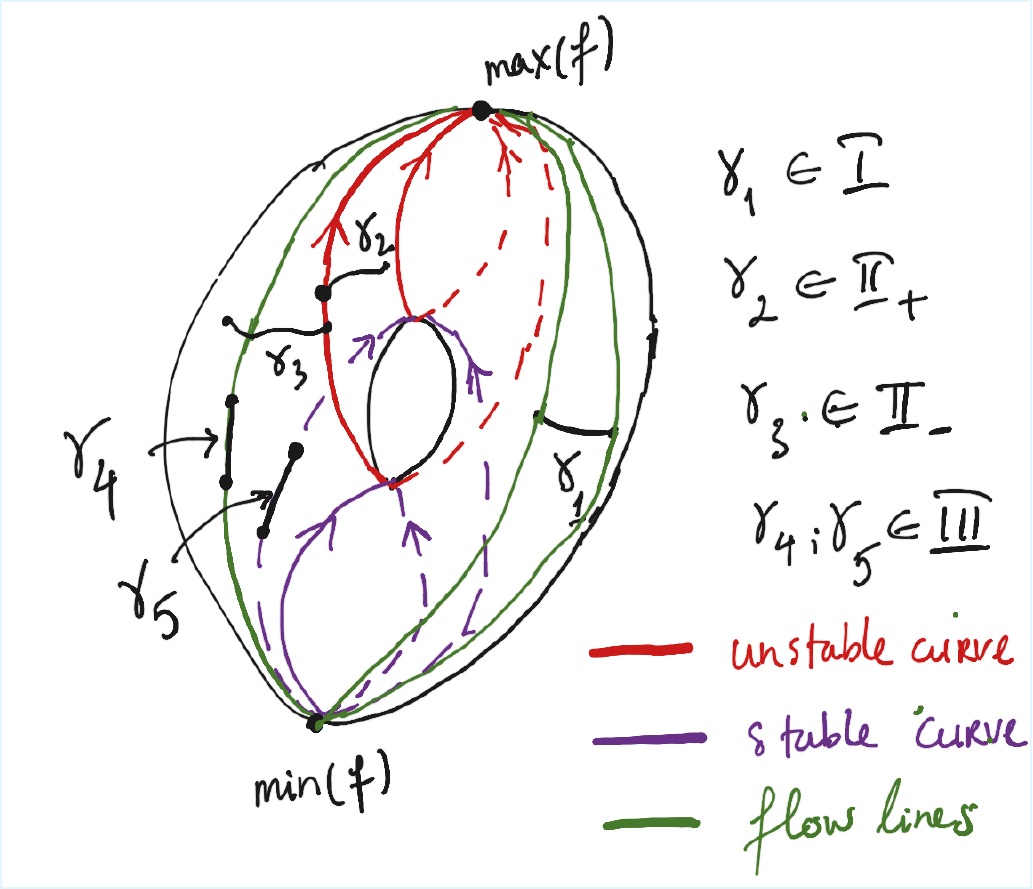} 
    \caption{An example of curve types.}
    \label{fig:curve_types}
\end{figure}
The general principles behind these definitions can be described as follows. First, \emph{elementary curves never meet critical points} and they are more generally transverse to the flowlines thanks to the transversality assumption. The orientation assumption is relative to the flowlines of the gradient flow while the primitive assumption ensures that there is no overlap inside the triangle $\{\varphi_f^{-t}(\gamma(s)):\ (s,t)\in(0,1)\times\mathbb{R}_+\}$ drawn by the gradient flowlines issued from $\gamma$. Finally, curves of type $\operatorname{I}$ never meet unstable curves or $\argmax(f)$ and,
if $\gamma$ is an elementary curve of type $\operatorname{II}$, only one of the endpoints of $\gamma$ meets some unstable curve and the rest of the curve minus this endpoint never intersects any other unstable curve.

With these conventions at hand, we can introduce the following currents of integration:
\begin{itemize}
 \item If $\gamma$ is of type $\operatorname{I}$, we set
 $$
 [\triangle(\gamma)]:=[\gamma]+[\mathcal{L}_{\gamma(1)}]-[\mathcal{L}_{\gamma(0)}],
 $$
which is a current of degree $1$. 
\item If $\gamma$ is of type $\operatorname{II}_\pm$ and if $\gamma(1)\in\Sigma\setminus W^u(a_1)$ (resp. $\gamma(0)$), we set
 $$
 [\Delta(\gamma)]:=[\gamma]+[\mathcal{L}_{\gamma(1)}^\pm]-[\mathcal{L}_{\gamma(0)}],\ \left(\text{resp.}\ :=[\gamma]+[\mathcal{L}_{\gamma(1)}]-[\mathcal{L}_{\gamma(0)}^\pm]\right)
 $$
which is a current of degree $1$. 
\end{itemize}
In both cases, if we denote by $\blacktriangle(\gamma)=\{\varphi_f^{-t}(\gamma(s)):t\geq 0,\ s\in(0,1)\}$, we orient this \emph{triangular domain} in such a way that the corresponding current of integration satisfies:
$$
\partial[\blacktriangle(\gamma)]:=-d[\blacktriangle(\gamma)]=[\triangle(\gamma)].
$$
\begin{rmk} Suppose that, in an oriented local chart $(x_1,x_2)\in(-\varepsilon,\varepsilon)^2$, the curve $\gamma$ is of the form $t\in I\mapsto (x_1(t), 0)$ and $V=\partial_{x_2}$. Then, if $x_1'(t)>0$ (resp. $x_1'(t)<0$), $\gamma$ is locally positively (resp. negatively) $V$-oriented. Moreover, if $\pm x_1'(t)>0$, $[\gamma]$ reads locally $\mp\delta_0(x_2) dx_2$ and $[\blacktriangle(\gamma)]=\mp\mathbf{1}_{\R_-}(x_2)$. 
\end{rmk}

These conventions are motivated by the following lemma.
\begin{lemma}[Fundamental geometric lemma on triangles]\label{lemm:triangles} 
Let $\gamma$ be a curve of type $\operatorname{I}$ or $\operatorname{II}$. Then, one has
$$
\forall T\geq 0,\quad [\blacksquare_T(\gamma)]:=\int_0^{T}\iota_V\varphi_f^{t*}[\gamma] dt\ \in\ L^\infty(\Sigma,\mathbb{R}),
$$
and, for every $1\leq p<\infty$,
$$
\lim_{T\rightarrow+\infty}\left\|[\blacksquare_T(\gamma)]-[\blacktriangle(\gamma)]\right\|_{L^p}=0.
$$
\end{lemma}
\begin{rmk}
As we shall see in the proof, 
$$
\forall T>0,\ \left\|[\blacksquare_T(\gamma)]-[\blacktriangle(\gamma)]\right\|_{L^\infty}=1,
$$
so that the restriction to $p<\infty$ is in fact sharp in this statement.
\end{rmk}

Recall that type $\operatorname{I}$ (or type $\operatorname{II}$) implies by definition that $\gamma$ is elementary, $V$-oriented and primitive. Figure \ref{fig:Rectangles_1} illustrates the definition of the dynamical rectangles $\blacksquare_T(\gamma)$.
\begin{figure}
    \centering
    \includegraphics[scale=0.35]{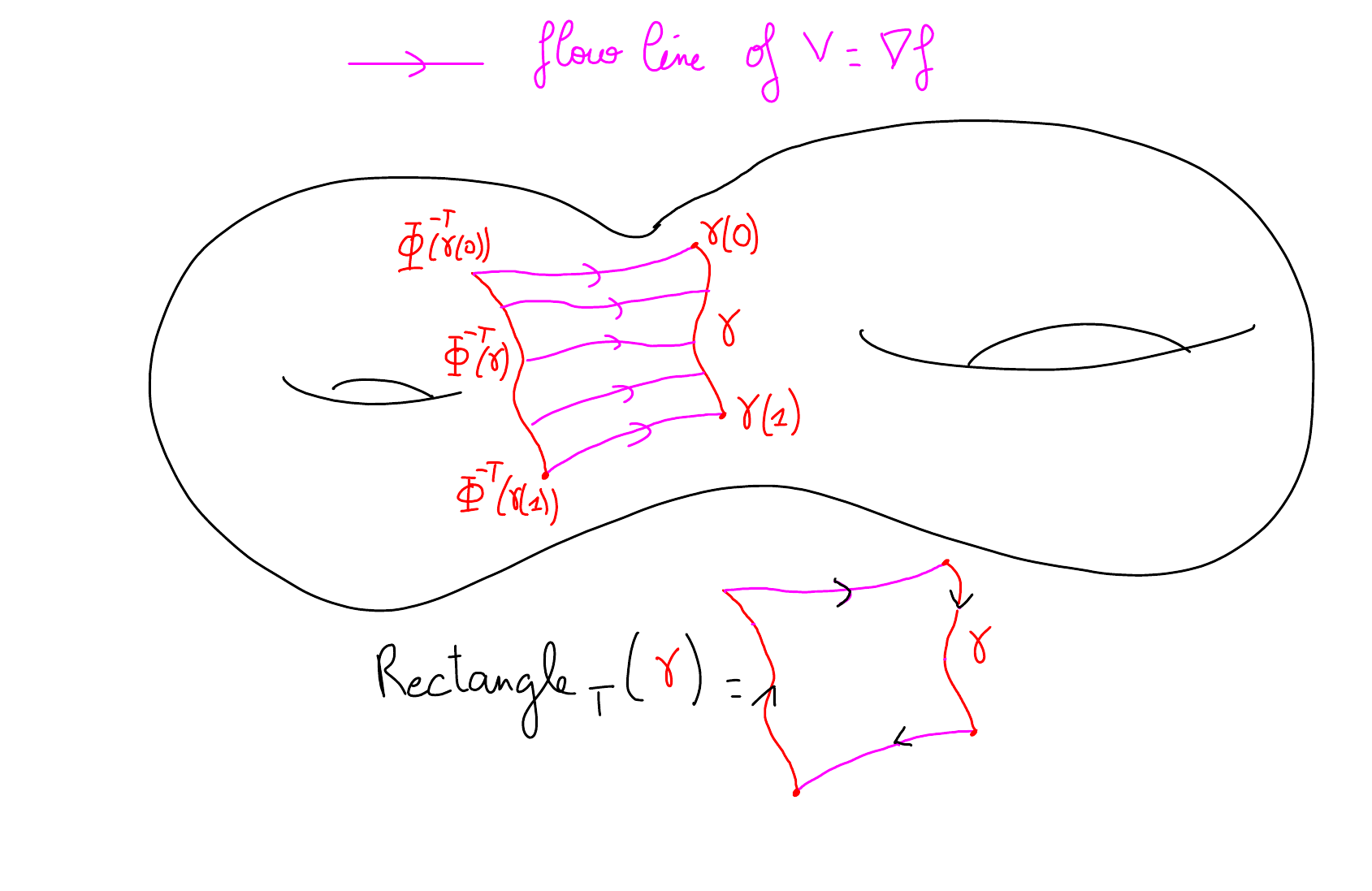} 
    \caption{An example of rectangle of $\gamma$.}
    \label{fig:Rectangles_1}
\end{figure}
Fig. \ref{fig:Triangle_1} illustrates some examples of triangles and the corresponding indicator of $\blacktriangle (\gamma)$.

\begin{figure}
    \centering
    \includegraphics[scale=0.35]{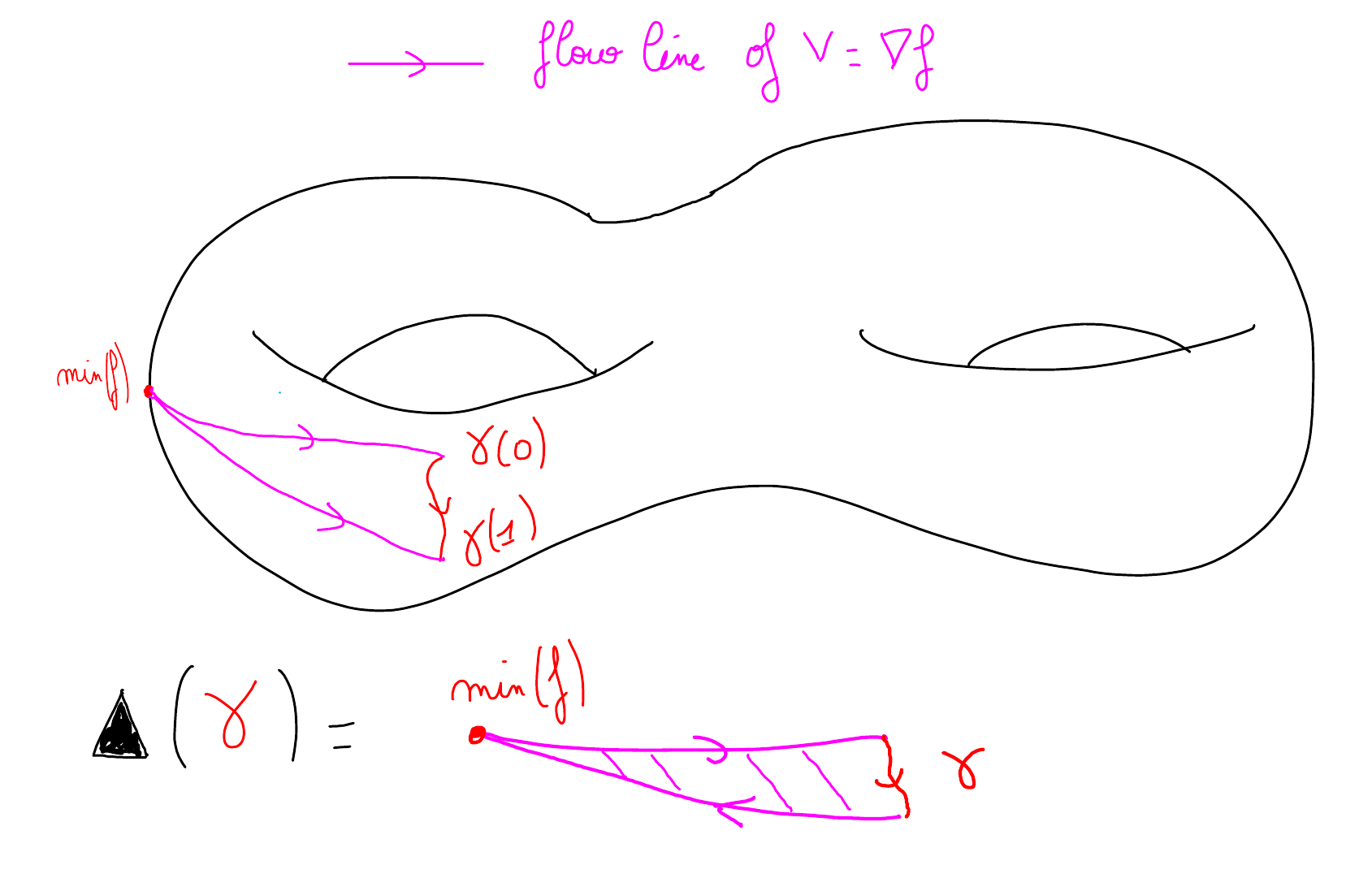} 
    \caption{A triangle $\triangle(\gamma)$.}
    \label{fig:Triangle_1}
\end{figure}

%THIS FIGURE IS MISLEADING. Figure \ref{fig:Triangle_1} illustrates some example of triangle and the corresponding indicator of $\blacktriangle (\gamma)$.

%\begin{figure}
%    \centering
%    \includegraphics[scale=0.35]{./figures/Triangle_1.pdf} 
%    \caption{A triangle $\triangle(\gamma)$.}
%    \label{fig:Triangle_1}
%\end{figure}

\begin{proof} The key ingredient of the proof is a geometric Lemma from~\cite[Lemma 2.1]{DRDuke}. In fact, it follows more specifically from the proof of this Lemma in this reference. Indeed, according to it (up to a verbatim adaptation to manifolds with corners), $[\blacksquare_T(\gamma)]$ is the current of integration on the submanifold (with corners) $\{\varphi_f^{-t}(\gamma(s)):\ 0\leq t\leq T,\ s\in[0,1]\}$ with the convention that, when restricted to some neighborhood of $\gamma([0,1])$, $-d[\blacksquare_T(\gamma)]=[\gamma]$, where $\gamma$ is oriented from $\gamma(0)$ to $\gamma(1)$. In particular, thanks to the primitive assumption, $[\blacksquare_T(\gamma)]$ can be identified with a measurable function with values in $\{-1,0,1\}$ (hence in $L^\infty$).

Suppose that $\gamma$ is of type $\operatorname{I}$ and fix some $\epsilon>0$. It implies that for $T>0$ large enough, $\varphi_f^T(\gamma([0,1]))$ lies in the ball $B(a_1,\epsilon)$ centered at the minimum of $f$. In particular, it means that the support of $[\blacksquare_T(\gamma)]-[\blacktriangle(\gamma)]$ lies in this ball. Recalling that it takes values either $-1$, $0$ or $1$, one finds that the $L^p$-norm is a $\mathcal{O}(\epsilon)$. Hence, it tends to $0$ for every $p<\infty$. Suppose now that $\gamma$ is of type $\operatorname{II}$ and say for instance that $\gamma(0)\notin W^u(a_1)$. If we fix $\epsilon>0$, one can find $s_0>0$ such that, for every $|s|\leq s_0$ and for every $t\geq 0$, $\varphi^{-t}(\gamma(s))$ lies in an $\epsilon$ neighborhood of $\mathcal{L}_{\gamma(0)}$. Hence, for every $1\leq p<\infty$, $\|[\blacksquare_T(\gamma([0,s_0]))]-[\blacktriangle(\gamma([0,s_0]))]\|_{L^p}=o(1)$ as $\epsilon\rightarrow 0^+$ (uniformly in $T\geq 0$). Applying the case of type $\operatorname{I}$ curves to $\gamma([s_0,1])$, we may conclude.
\end{proof}

Finally, we define a partial ordering relation among curves of type $\text{I}$ and $\text{II}$.

\begin{definition}\label{d:included-curve}
 Let $\gamma,\tilde{\gamma}$ be two curves which are of type $\operatorname{I}$ or $\operatorname{II}$. We say that $\tilde{\gamma}\preccurlyeq\gamma$ if there exists a bijective and increasing map $\tau:[0,1]\rightarrow[0,1]$, and such that, for every $t\in[0,1]$, $\tilde{\gamma}(t)\in\mathcal{L}_{\gamma\circ\tau(t)}$.
\end{definition}
Note that, if $\gamma$ is of type $\text{II}$, it may happen that $\tilde{\gamma} \preccurlyeq \gamma $ is of type $\text{I}$. Yet, if $\tilde{\gamma}$ is of type $\text{II}_\pm$, then every $\gamma$ such that $\tilde{\gamma}\preccurlyeq \gamma$ is also of type $\text{II}_\pm$. 

\begin{lemma}\label{l:reparametrization-C1} Let $\gamma,\tilde{\gamma}$ be two curves which are both of type $\operatorname{I}$ or $\operatorname{II}$ such that $\tilde{\gamma}\preccurlyeq\gamma$. Then, the map $\tau$ from Definition~\ref{d:included-curve} is of class $\mathcal{C}^1$. 
\end{lemma}

As a consequence of this Lemma, one finds that, up to a $\mathcal{C}^1$ reparametrization, the map $\tau$ in Definition~\ref{d:included-curve} can be chosen to be the identity.

\begin{proof} Given $\gamma$ as in the assumption of the lemma, one can introduce the following current $\langle[\gamma],\psi\rangle:=\int_0^1\psi_{\gamma(t)}(\gamma'(t))dt.$ Now, for $t\in[0,1]$ such that $\gamma\circ\tau (t)\in W^u(a_1)$, one defines the current $[L_t]$ as the current of integration on the flow line issued from $\tilde{\gamma}(t)$ and one has $\tau(t)=\pm\int_{\Sigma}[L_t]\wedge[\gamma]$ (with the sign depending on the orientation of the two curves). This map is of class $\mathcal{C}^1$ as both curves are of class $\mathcal{C}^1$ and as the gradient flow is also $\mathcal{C}^1$. Hence, the only difficulty is to check the regularity when $\gamma\circ\tau (t)\notin W^u(a_1)$. This can only happen for type $\text{II}$ curves and, by definition of these curves, at the endpoints of the interval, i.e. $t=0$ or $t=1$. Let us say $\gamma(0)\notin  W^u(a_1)$. If $\tilde{\gamma}(0)\notin W^u(a_1)$, then the same argument works as both $\gamma(0)$ and $\tilde{\gamma}(0)$ will lie on the same unstable manifold. Hence, one only needs to check the case where $\tilde{\gamma}(0)\in W^u(a_1)$ and $\gamma(0)\notin W^u(a_1)$. To deal with this case, it is sufficient to consider what happens in a Morse chart $(x_1,x_2)$ near a critical point of index $1$ with $\gamma(t)=r_0(2,t/2)$ and $\tilde{\gamma}(t)=r_0(t,1)$, for some small enough $r_0$. In that case, the map $\tau$ is just the identity map. 
\end{proof}

Given two curves as in Definition~\ref{d:included-curve}, one has the following:
\begin{lemma}\label{l:inclusion-curves}
 Let $\gamma,\tilde{\gamma}$ be two curves which are both of type $\operatorname{I}$ or $\operatorname{II}$ and such that $\tilde{\gamma}\preccurlyeq\gamma$. Then, one has
 $$
 \blacktriangle(\tilde{\gamma})\subset\blacktriangle(\gamma),
 $$
 where we denote by $\blacktriangle(\gamma)$ (resp. $\blacktriangle(\tilde{\gamma})$) the support of the current $[\blacktriangle(\gamma)]$ (resp. $[\blacktriangle(\tilde{\gamma})]$).
\end{lemma}

\begin{proof} This follows from the fact that, if $y\in\mathcal{L}_x$ (resp. $\mathcal{L}_x^\pm$), then $\mathcal{L}_y\subset\mathcal{L}_x$ (resp. $\mathcal{L}_y^\pm\subset\mathcal{L}_x^\pm$). Indeed, $\blacktriangle(\gamma)$ is the union of the curves $\mathcal{L}_{\gamma(t)}$ up to the case $t=0$ or $1$ where one may have to consider $\mathcal{L}_{\gamma(t)}^\pm$. Now taking $\tau=1$ for the reparametrization function, we obtain the expected conclusion.
\end{proof}

\subsubsection{Admissible curves}
Our purpose is now to define a set of admissible curves for computing random holonomies. In order to have the largest possible set, we introduce a last set of curves. 
\begin{definition}[Type $\operatorname{III}$ curves] Let $\gamma:[0,1]\rightarrow\Sigma$ be a $\mathcal{C}^1$ curve. We say that $\gamma$ is of type $\operatorname{III}$ if $\gamma([0,1])$ is included in a flow line of $\varphi_f^t$ and if it does not meet a critical point of $V$. For such curves, we set 
$$
[\triangle(\gamma)]=0,\ \text{and}\ [\blacktriangle(\gamma)]=0,
$$
where each respective current is understood as a degree $1$ (resp. $0$) current.
\end{definition}
Given two $\mathcal{C}^1$ curves $\gamma_j:[0,1]\rightarrow \Sigma$ such that $\gamma_1(1)=\gamma_2(0)$, we define their concatenation $\gamma_1\star \gamma_2:[0,1]\rightarrow \Sigma$ as follows
$$
\forall t\in\left[0,\frac12\right],\ \gamma_1\star \gamma_2(t)=\gamma_1(2t),\ \text{and}\ \forall t\in\left[\frac12,1\right],\ \gamma_1\star \gamma_2(t)=\gamma_2(2t-1). 
$$
Note that $\gamma_1\star \gamma_2$ is continuous but a priori only piecewise $\mathcal{C}^1$. This leads to the following definition.
\begin{definition}[Admissible curves]  Let $\gamma:[0,1]\rightarrow\Sigma$ be a continuous and piecewise $\mathcal{C}^1$ curve. We say that $\gamma$ is admissible if there exist $J$ curves $(\gamma_j)_{1\leq j\leq J}$ that are of type either $\operatorname{I}$, $\operatorname{II}$ or $\operatorname{III}$ (not necessarily all of the same type) such that
$$
\gamma=\gamma_1\star\left(\gamma_2\star\left(\ldots \star\gamma_J\right)\right).
$$
Then, we define
$$
[\triangle(\gamma)]=\sum_{j=1}^J[\triangle(\gamma_j)],\ \text{and}\ [\blacktriangle(\gamma)]=\sum_{j=1}^J[\blacktriangle(\gamma_j)].
$$
\end{definition}
Observe that, by construction, $[\triangle(\gamma)]$ and $[\blacktriangle(\gamma)]$ are independent of the decomposition into the curves $(\gamma_j)_{1\leq j\leq J}$ and that the following holds
\begin{equation}\label{e:closed-curve}
\partial[\blacktriangle(\gamma)]=[\triangle(\gamma)]-\sum_{\text{ind}(a)=1}n_\gamma(a)[W^s(a)],
\end{equation}
 where $n_\gamma(a)\in\mathbb{Z}$ counts the number of (algebraic) intersections between $\gamma$ and the unstable manifolds $(W^u(a))_{\text{ind}(a)=1}$. This contribution comes from the type $\operatorname{II}$ curves. The set of admissible curves is made of continuous curves which are piecewise $\mathcal{C}^1$ and the function $[\blacktriangle(\gamma)]:\Sigma\rightarrow\Z$ is referred to as the \emph{index function} of the closed (but not necessarily exact) curve $[\Delta(\gamma)]$.

 \begin{definition}[Index functions]
  The function $[\blacktriangle(\gamma)]$ is called the index of the oriented curve $\triangle(\gamma)$.
\end{definition}

This definition generalizes the usual notion of Hopf index which measures winding numbers. In the case of $\Sigma$ being the Riemann sphere and $f$ being the height map, removing $a_2:=\argmax f$ amounts to removing the north pole of the Riemann sphere and to consider the complex plane. In this case, $z \in \C \mapsto [\triangle(\gamma)](z) \in \Z$ associates to every curve the usual Hopf index of $\gamma$. See~\cite[Prop 4.16 p.~42]{Fulton} and we refer to Figure \ref{fig:index_lip} to illustrate examples of Hopf indices for various closed curves.

\begin{figure}
    \centering
    \includegraphics[scale=0.35]{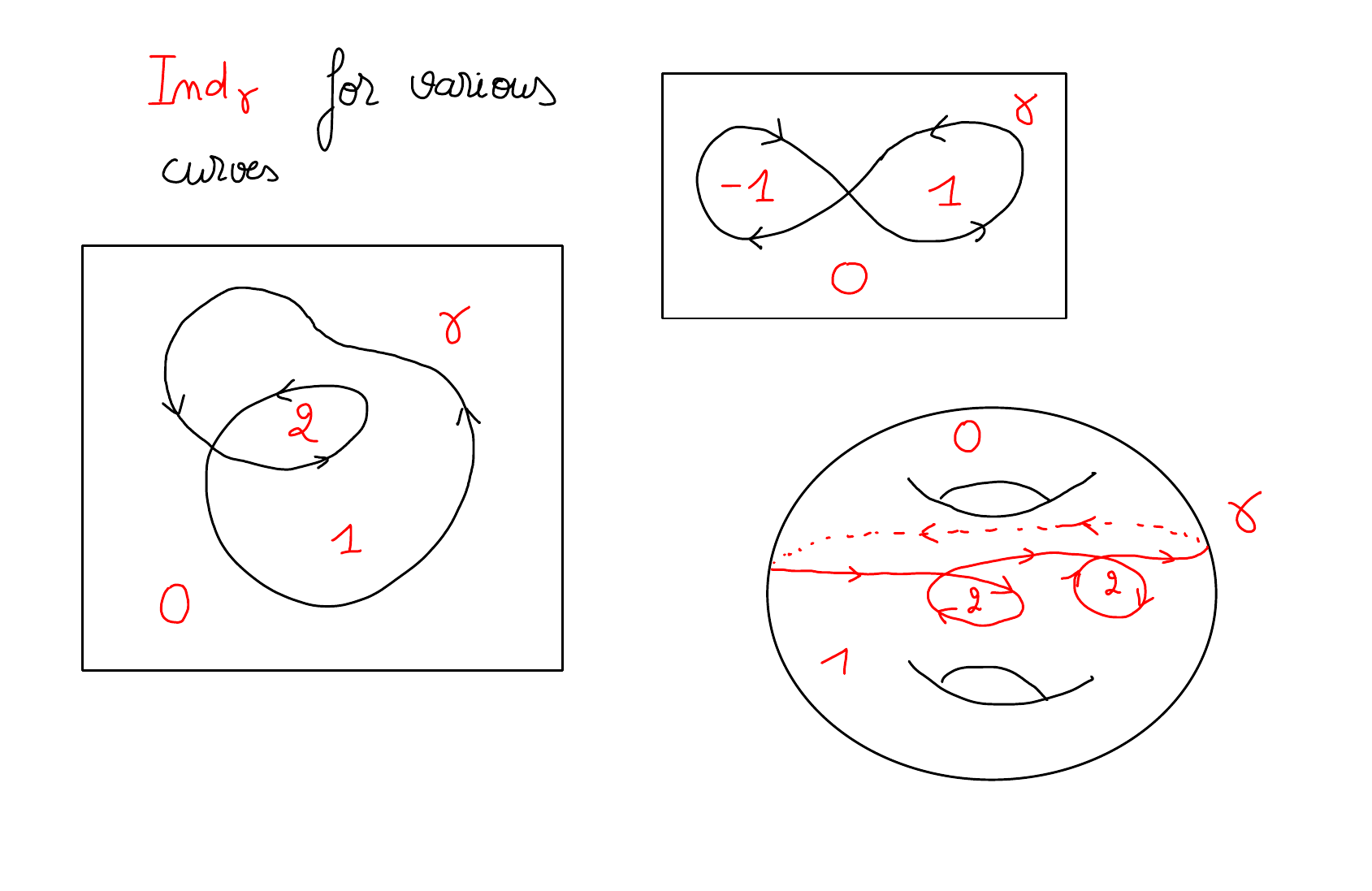} 
    \caption{Index for various Lipschitz curves.}
    \label{fig:index_lip}
\end{figure}

 When each curve $\gamma_j$ of type $\operatorname{I}$ or $\operatorname{II}$ used to define $\gamma$ is included in a level set of $f$, we shall say that $\gamma$ is of \emph{Tetris type}. These are typically the kind of curves appearing in the lattice approximation strategy that can be found in~\cite{DangNohra}. We refer to Figure \ref{fig:Tetris_curves_1} where we give examples of Tetris curves. As for the elementary case, we also introduce the notion of $V$-oriented curves with respect to the vector field $V$.
\begin{definition}[Admissible $V$-oriented curves]\label{def:admiscurves} Let $\gamma:[0,1]\rightarrow\Sigma$ be an admissible curve. We say that $\gamma$ is $V$-oriented with respect to $V$ if, for every $0\leq s<t\leq 1$, $[\blacktriangle(\gamma([s,t]))]$
has the same sign. 
\end{definition}
Observe that, along segments $[s,t]$ where $\gamma$ corresponds to a type $\operatorname{III}$-curve, the value of $[\blacktriangle(\gamma([0,\tau]))]$ is constant.  
\begin{definition}[Admissible primitive curves]  Let $\gamma:[0,1]\rightarrow\Sigma$ be an admissible and $V$-oriented curve. We say that $\gamma$ is primitive with respect to $V$ if, for every $0\leq t\leq 1$, $|[\blacktriangle(\gamma([0,t]))]|\leq 1$. 
\end{definition}
In other words, it means that the triangle drawn by the curve $\gamma$ and the flow lines of $V$ has no overlap. Both notions coincide with the ones used for elementary curves.

\begin{figure}
    \centering
    \includegraphics[scale=0.35]{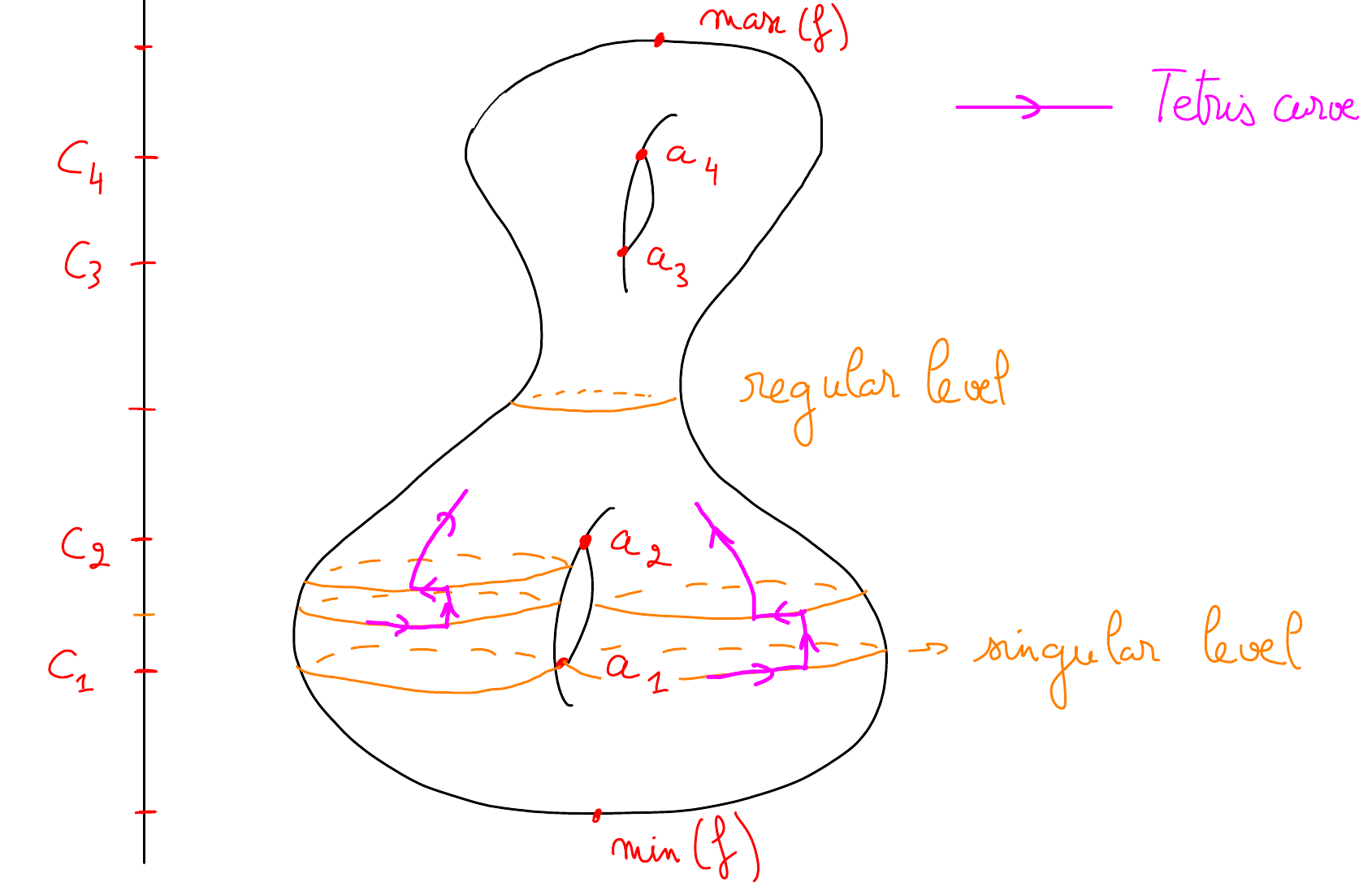} 
    \caption{Tetris curves.}
    \label{fig:Tetris_curves_1}
\end{figure}
We conclude this discussion with the following key lemma for our probabilistic constructions.
\begin{lemma}\label{l:holder-regularity} Let $\gamma$ be an admissible curve and let $\psi\in L^\infty(\Sigma,\mathbb{R})$ that is continuous on the support of $[\blacktriangle(\gamma([0,t]))]$. Then, the function
$$
t\in[0,1]\mapsto \int_{\Sigma}[\blacktriangle(\gamma([0,t]))]\psi\upsilon 
$$
is continuous and of class $\mathcal{C}^1$ outside finitely many points $t_1<t_2\ldots <t_M$ where $\gamma(t_m)\in W^u(a)$ for some critical point of index $1$. Moreover, near every $t_m$ with $1\leq m\leq M$, one has
$$
\partial_t\int_{\Sigma}[\blacktriangle(\gamma([0,t]))]\psi\upsilon=\mathcal{O}(|\ln |t-t_m||),\ \text{as}\ t\rightarrow t_m,
$$
where the constant in the remainder depends on $\psi$, $\gamma$, $\upsilon$ and $f$. Finally, if $\gamma$ is $V$-oriented and has no type $\operatorname{III}$ elementary curve, then there exists a constant $c_\gamma>0$ such that
$$
\forall t\in[0,1]\setminus\{t_1,\ldots,t_M\},\quad \left|\partial_t\int_{\Sigma}[\blacktriangle(\gamma([0,t]))]\upsilon\right|\geq c_\gamma
$$
\end{lemma}
This lemma shows that, for every smooth function $\psi$,  $t\mapsto \int_{\Sigma}[\blacktriangle(\gamma([0,t]))]\psi\upsilon$ belongs to $W^{1,p}([0,1])$ for every $1\leq p<\infty$. In particular, it has H\"older regularity for every $0\leq\gamma< 1$. 
\begin{proof} Up to replacing $\psi$ by $\frac{d\upsilon}{d\upsilon_f}\psi$, we can prove the result for the volume form $\upsilon_f:=\text{vol}_{h_f}$ induced by the metric $h_f$ used to define the Morse metric. By construction of admissible curves, it suffices to prove this regularity result for type $\operatorname{I}$ and $\operatorname{II}$-curves. Thanks to Lemma~\ref{lemm:triangles}, recall that, by construction,
\begin{equation}\label{eq_triangle_current}
  [\blacktriangle(\gamma([0,t]))]=\int_0^{\infty}\varphi_f^{\tau*}\iota_V([\gamma([0,t])])d\tau,  
\end{equation}
where the convergence of the integral holds in the sense of $L^p$ functions on $\Sigma$. Let us first deal with type $\operatorname{I}$ curves and write
$$
\int_{\Sigma}[\blacktriangle(\gamma([0,t]))]\psi \upsilon_f=\lim_{T\rightarrow+\infty}\int_0^T\int_{\Sigma}[\gamma([0,t])]\wedge \psi\circ\varphi_f^{-\tau}\iota_V(\upsilon_f) e^{-\int_{-\tau}^0(\text{div}_{\upsilon_f} V)\circ\varphi_f^sds}d\tau.
$$
Note that the integral on the right-hand side defines a $\mathcal{C}^1$ function of time $t$ and that it converges weakly as $T\rightarrow+\infty$ to the quantity we are interested in. In particular, the derivative with respect to time converges to the derivative (in the sense of distribution) of the continuous function $
t\mapsto \int_{\Sigma}[\blacktriangle(\gamma([0,t]))]\psi\upsilon_f$. Therefore,
\begin{equation}\label{e:formula-derivative}
\partial_t\int_{\Sigma}[\blacktriangle(\gamma([0,t]))]\psi\upsilon_f=\int_0^\infty e^{-\int_{-\tau}^0(\text{div}_{\upsilon_f} V)\circ\varphi_f^s(\gamma(t))ds}\psi\circ\varphi_f^{-\tau} (\gamma(t))\upsilon_f(V(\gamma(t)),\gamma'(t)) d\tau.
\end{equation}
When $\gamma$ is a type $\operatorname{I}$ (or away from the points $t_m$, $1\leq m\leq M$ for type $\operatorname{II}$ curves), this integral is indeed convergent as $\varphi_f^{-\tau}(\gamma(t))$ converges to $\text{argmin}(f)$ (where the divergence is positive). The corresponding function is also continuous with respect to time which establishes the first part of the lemma. The integral is a priori divergent when $t$ corresponds to the time $(t_m)_{1\leq m\leq M}$ and we want to understand the behavior of this function as $t\rightarrow t_m$ for some $1\leq m\leq M$. To do that, recall from Lemmas~\ref{l:crossingtime} and~\ref{l:outsidebasicsets} that the time spent by a trajectory outside the Morse charts near every critical point is uniformly bounded. Moreover, as $\gamma([0,t])$ does not contain $a_{2g+2}$, the integral in~\eqref{e:formula-derivative} can only blow up if the trajectory spent enough time near a critical point of index $1$. As $t\neq t_m$, one knows that $\gamma(t)\in W^u(a_1)$ and as $\gamma'(t)\neq 0$ for every $t\in[0,1]$ (by definition of admissible curves), one knows that the distance from $\gamma(t)$ to $W^u(a_j)$ (with $2\leq j\leq 2g+1$) is bounded from below by $c_0|t-t_m|$ where $c_0$ depends only on the curve $\gamma$ and the metric $h$. Recalling the expression~\eqref{e:local-expression-saddle} of the flow in the Morse chart, one can deduce that the amount of time spent near a saddle point is $\mathcal{O}(|\ln|t-t_m||)$ from which we deduce the expected behavior of the derivative near $t=t_m$. Finally, the lower bound on the derivative follows from~\eqref{e:formula-derivative} together with the fact that $|\upsilon_f(V(\gamma(t)),\gamma'(t))|$ is bounded from below by a positive constant $\tilde{c}_\gamma$ thanks to the transversality assumption made on elementary curves of type $\operatorname{I}$ and $\operatorname{II}$.  
\end{proof}

\subsubsection{Approximable currents} We now enlarge the above class of curves we are dealing with to allow for more general observables of the type $\int_{\gamma}A_\psi$ at the quantum level. We will call the corresponding currents \textbf{approximable}. This class contains all possible curves $\gamma$ on which we will be able to define the observables
$\int_\gamma A$ in a probabilistic manner. We consider the space of admissible curves from definition~\ref{def:admiscurves} and these define currents that we can endow with the weak topology. Recall that each admissible curve is oriented from $0$ to $1$ and thus defines a current of integration $[\gamma]$ of finite mass.

\begin{definition}[Currents with good approximation properties]\label{def:approx_currents}
 A current $\mathbf{T}$ is said to be approximable if 
\begin{itemize}
    \item there exists a sequence of admissible curves $(\gamma_n)_{n\geqslant 1}$ such that the sequence of currents $[\gamma_n]$ converge weakly to the current $\mathbf{T}$ (for the weak topology of $\mathcal{D}^{\prime,1}(\Sigma)$), 
    \item the corresponding sequence of index functions $\left([\blacktriangle (\gamma_n)]\right)_{n\in \mathbb{N}}$ is Cauchy in $L^2_\upsilon(\Sigma)$.
\end{itemize}
\end{definition}
In the following, we will denote by $\text{Ind}(\mathbf{T})$ the limit of $[\blacktriangle (\gamma_n)]$ in $L^2_\upsilon(\Sigma)$. Let us give two illustrative examples\footnote{Another example can be found in \cite[fig 4.3.4 p.~42]{Morgan}
which yields a set $E$ of finite area but with infinite perimeter.} to give the reader a sense of how large this class is.
\begin{example}
 We work in $[0,10]\times [0,10]$ that we view as a flowbox of our gradient flow.
For every $n$, consider the square $S_n$ whose vertices are $(2^{-n},2^{-n}),(2^{-n}+n^{-\alpha},2^{-n}), (2^{-n},2^{-n}+n^{-\alpha}), (2^{-n}+n^{-\alpha},2^{-n}+n^{-\alpha})$ for $\alpha>\frac{1}{2}$ and $\gamma_n$ is the straight segment connecting $(2^{-n},2^{-n}) $ and $(2^{-n-1},2^{-n-1})$. Then, each $S_n$ is admissible and one can show that
$$\mathbf{T}=\sum_{n\geqslant 0} [S_n]+[\gamma_n]$$
is approximable in the above sense. For $\alpha>1$, the current has finite mass since the union of squares has finite length but we see that, for $\alpha \in (\frac{1}{2},1]$,
the mass of $\mathbf{T}$ is infinite. 
\end{example}

\begin{example} We work on the plane $\mathbb{R}^2$.
Consider a collection of circles $C_n$ all oriented counterclockwise of radius $n^{-\alpha}$ centered at $(0,n^{-\alpha})$, $\alpha>\frac{1}{2}$ which are tangent at one given point. Then the corresponding current
$\mathbf{T}:=\sum_{n\geqslant 0} [C_n]$
can be approximated in the above sense. For $\alpha>1$, the current has finite mass since the union of circles has finite length but we see that, for $\alpha \in (\frac{1}{2},1]$, the mass of $\mathbf{T}$ is infinite.
\end{example}

Our definition of 
approximable currents can be related with classical concepts from geometric measure theory. In fact, in the above definition and as can be seen from the above examples, we do not even assume $\mathbf{T}$ to be an integral current meaning a rectifiable current of dimension $1$ with integer multiplicity whose boundary is also rectifiable with integer multiplicity~\cite[p.~39--41]{Morgan}. To ensure integral properties, one needs to ensure extra boundedness assumptions:

\begin{lemma}[Criterion for integrability]
 If we are given a sequence of admissible curves $(\gamma_n)_n$ such that the length of $(\triangle(\gamma_n))_n$ is uniformly bounded and such that the sequence $([\blacktriangle(\gamma_n)])_n$ is Cauchy in $L^2$ and $[\triangle(\gamma_n)]\rightarrow \mathbf{T}$ weakly in $\mathcal{D}^\prime$. Then $\mathbf{T}$ is an integral current.
\end{lemma}
 
  Since we are only dealing with $1$-dimensional objects, this simply means that $\partial\mathbf{T}$ is a finite sum of integration currents on points, so the support of $\partial\mathbf{T}$ is just a finite number of points.   
\begin{proof}
By assumption, $[\triangle (\gamma_n)]$ is a bounded sequence for the mass norm and also note that
$[\blacktriangle(\gamma_n)]$ is bounded in $L^2(\Sigma)$. Thus it is bounded in $L^1(\Sigma)$ since $\Sigma$ has finite area. By definition, the mass $\mathbf{M}( [\blacktriangle\gamma_n] )$ of the integer valued function $[\blacktriangle(\gamma_n)]$ is nothing but the 
$L^1$ norm $\Vert [\blacktriangle(\gamma_n)] \Vert_{L^1(\Sigma, \upsilon)}$. So we have a sequence $ ([\blacktriangle(\gamma_n)])_{n\geqslant 1} $ of integral currents whose mass and boundary mass remain bounded and such that $ ([\blacktriangle(\gamma_n)])_{n\geqslant 1} $ converges in $L^2(\Sigma)$ (by definition since it is Cauchy in $L^2$) hence in $L^1(\Sigma)$ (because $\Sigma$ has finite area) and for the weak topology. 
In other words, the $\flat$--norm~\cite[p.~41]{Morgan} 
of the sequence $([\triangle(\gamma_n)])_n$ remains bounded.
By the closure theorem~\cite[item (4) Thm 5.4 p.~62]{Morgan}, the sequence $ ([\blacktriangle(\gamma_n)])_n $ converges to some integral current denoted by $\text{Ind}(\mathbf{T})\in \mathcal{I}_2(\Sigma)$ which is an integer valued function in $L^1$ whose boundary $\partial \text{Ind}(\mathbf{T})=\mathbf{T}$ defines an
integral current $\mathbf{T}\in \mathcal{I}_1(\Sigma) $. In particular $\mathbf{T}$ is rectifiable of finite length.   
\end{proof}

\subsection{Integrating random connections and the area law}
\label{ss:arealawtransverse}

We now aim at making sense of $\int_\gamma A_\psi$ where $A_\psi$ is the random connection from Theorem~\ref{thm:main1}. We first deal with the case of admissible curves and, in the end, we show how it induces a random variable for approximable currents. In order to motivate our definition, recall from Theorem~\ref{thm:main1} that $A_\psi$ is obtained as the limit in $L^2(\Omega,H^{-1-\kappa})$ of the sequence   
$$
A_\psi^{(N)}:=\sum_{n=1}^N\sum_{\ell=1}^LX_{n,\ell}(\omega)\mathcal{L}_V^{-1}\left(\iota_V(\psi e_{n,\ell}\upsilon)\right),\quad N\geqslant 1.
$$  
Recall also from the proof of this Theorem (in fact from Theorem~\ref{t:contraction}) that $\mathcal{L}_V^{-1}\left(\iota_V(\psi e_{n,\ell}\upsilon)\right)$ is obtained as the limit in $\mathcal{D}^\prime(\Sigma,T^*\Sigma\otimes\mathfrak{g})$ of $\int_0^T\varphi_f^{-t*}\left(\iota_V(\psi e_{n,\ell}\upsilon)\right)$ and that the limit belongs to $L^q$ for every $1\leq q<2$. This last step is a purely deterministic argument. Hence, we set
$$
A_\psi^{(N,T)}:=\sum_{n=1}^N\sum_{\ell=1}^LX_{n,\ell}(\omega)\int_0^T\varphi_f^{-t*}\left(\iota_V(\psi e_{n,\ell}\upsilon)\right)dt,\quad N\geqslant 1,\ T>0.
$$
If we make the assumption that the orthonormal basis $(e_n)_{n\geqslant 1}$ used to construct the white noise is only made of \emph{continuous} functions, then one can define, for every type $\operatorname{I}$ or $\operatorname{II}$ curve $\gamma$, 
$$
W_{\psi,\gamma}^{(N,T)}:=\int_\gamma A_\psi^{(N,T)}:=\sum_{n=1}^N\sum_{\ell=1}^LX_{n,\ell}(\omega)\int_\gamma\left(\int_0^T\varphi_f^{-t*}\left(\iota_V(\psi e_{n,\ell}\upsilon)\right)dt\right),\quad N\geqslant 1,\ T>0.
$$
By duality, this rewrites as
$$
W_{\psi,\gamma}^{(N,T)}=\sum_{n=1}^N\sum_{\ell=1}^LX_{n,\ell}(\omega)\left\langle [\blacksquare_T(\gamma)],\psi e_{n,\ell}\upsilon\right\rangle,\quad N\geqslant 1,\ T>0.
$$
Note that this quantity is well defined even if $(e_n)_{n\geqslant 1}$ is only supposed to be made of $L^2$ functions. Thanks to Lemma~\ref{lemm:triangles}, this quantity converges deterministically, as $T\rightarrow+\infty$, to
$$
W_{\psi,\gamma}^{(N)}:=\sum_{n=1}^N\sum_{\ell=1}^LX_{n,\ell}(\omega)\left\langle [\blacktriangle(\gamma)],\psi e_{n,\ell}\upsilon\right\rangle,\quad N\geqslant 1.
$$
By definition of the white noise, this quantity converges in $L^2(\Omega,\mathfrak{g})$ to 
$$
\boxed{W_{\psi,\gamma}:=\left\langle \xi_{[\blacktriangle(\gamma)]\psi}\upsilon,1\right\rangle\ \in\ \mathfrak{g},}
$$
and \textbf{we pick this definition for the probabilistic version of $W_{\psi,\gamma}:=\int_\gamma A_\psi$ for any admissible curve $\gamma$}. It defines a $\mathfrak{g}$-valued random variable on the probability space of the white noise and it is not strictly speaking a function of $A_\psi$. Yet, when defining the Yang--Mills measure, we will verify that such random variables are allowed quantum observables for the Yang--Mills measure. Note that, if $\psi=\tilde{\psi}$ on the support of $[\blacktriangle(\gamma)]$, then $W_{\psi,\gamma}=W_{\tilde{\psi},\gamma}.$
\begin{rmk} If $\gamma=\gamma_1\star\gamma_2\star\ldots\gamma_J$ with all the $\gamma_j$ being admissible, then one can verify that 
$$
W_{\psi,\gamma}=\sum_{j=1}^JW_{\psi,\gamma_j}.
$$
\end{rmk}
By definition of the white noise $\xi_{\psi_1}$, one has
\begin{lemma}[Area law]\label{l:integration-connection} With the above conventions, one has, for every two admissible curves $\gamma$ and $\tilde{\gamma}$ and every $\psi$ and $\tilde{\psi}$ in $L^\infty(\Sigma)$,
\begin{equation}\label{e:keyrelation-proba}
\boxed{\mathbb{E}\left(\left\langle W_{\psi,\gamma},W_{\tilde{\psi},\tilde{\gamma}}\right\rangle_{\mathfrak{g}}\right)=\dim(\mathfrak{g})\langle\psi[\blacktriangle(\gamma)],\tilde{\psi}[\blacktriangle(\tilde{\gamma})]\rangle_{L^2_\upsilon(\Sigma)}.}
\end{equation}
\end{lemma}
Here, the scalar product of index functions describes
the variance of the Gaussian random variables $W_{\psi,\gamma}$. In particular, for $\psi=\tilde{\psi}$, the right-hand side  in~\eqref{e:keyrelation-proba} may have either positive or negative sign depending on the orientation issues involved in the definition of the index function $[\blacktriangle(\gamma)]$. Thanks to this property, we can turn to the extension of this definition to approximable currents in view of defining the probabilistic analogue of $\langle\mathbf{T},A_\psi\rangle$. In fact, if we fix an approximable current $\mathbf{T}$, then one can find a sequence of admissible curves $(\gamma_n)_{n\geqslant 1}$ such that $([\gamma_n])_{n\geqslant 1}$ converges weakly to $\mathbf{T}$ and such that $([\blacktriangle(\gamma_n)])_{n\geqslant 1}$ is a Cauchy sequence in $L^2_\upsilon(\Sigma)$. Hence, one has
$$
\forall n,m\geqslant 1,\quad \mathbb{E}\left(\left\| W_{\psi,\gamma_n}-W_{\psi,\gamma_m}\right\|^2_{\mathfrak{g}}\right)=\dim(\mathfrak{g})\left\|\psi\left([\blacktriangle(\gamma_n)]-[\blacktriangle(\gamma_m)]\right)\right\|_{L^2_{\upsilon}(\Sigma)},
$$
from which we infer that $(W_{\psi,\gamma_n})_{n\geqslant 1}$ is a Cauchy sequence in $L^2(\Omega,\mathfrak{g})$. The limit is denoted $W_{\psi,\mathbf{T}}$ and it is the probabilistic definition of $\langle\mathbf{T},A_\psi\rangle$. Formula~\eqref{e:keyrelation-proba} remains true with $[\blacktriangle(\gamma)]$ replaced by $\text{Ind}(\mathbf{T})$. Again $W_{\psi,\mathbf{T}}$ is a $\mathfrak{g}$-valued random variable and not strictly speaking a function of $A_\psi$. Yet, it will again be an admissible quantum observable of our Yang--Mills measure. We record the following continuity estimate:
\begin{lemma}
With the above conventions, one has, for every approximable currents $\mathbf{T}_1$ and $\mathbf{T}_2$ and for every $\psi\in L^\infty(\Sigma)$, 
\begin{equation}
\mathbb{E}\left(\left\|W_{\psi,\mathbf{T}_1}-W_{\psi,\mathbf{T}_2} \right\|^2 \right)=\operatorname{dim}\mathfrak{g} \left\Vert\psi\left(\operatorname{Ind}(\mathbf{T}_1)-\operatorname{Ind}(\mathbf{T}_2)\right)\right\Vert_{L^2_\upsilon(\Sigma)}^2.
\end{equation}
 \end{lemma}

\subsection{A stochastic process associated with \texorpdfstring{$(A_\psi,\gamma)$}{(A psi, gamma)}}
\label{ss:g-valued-BM}

We now fix $\psi\in L^\infty(\Sigma)$. We also let $\gamma$ be an admissible curve and we consider its restriction $\gamma_{[0,t]}$ to the interval $[0,t]$. This defines a $\mathfrak{g}$-valued stochastic process:
$$
\widetilde{\mathfrak{W}}_{\psi,\gamma}(t):=W_{\psi,\gamma_{[0,t]}}=\left\langle[\blacktriangle(\gamma_{[0,t]})]\psi\xi\upsilon,1\right\rangle\ \in\  L^2(\Omega,\mathfrak{g}),
$$
where we made a small abuse of notations and wrote $[\blacktriangle(\gamma_{[0,t]})]\psi\xi$ in place of $\xi_{[\blacktriangle(\gamma_{[0,t]})]\psi}$.
Recall that $(\mathfrak{b}_\ell)_{1\leq \ell\leq L}$ is an orthonormal basis of $\mathfrak{g}$. Hence, for $1\leq \ell\leq L$, the component of $\widetilde{\mathfrak{W}}_{\psi,\gamma}(t)$ along the $\mathfrak{b}_\ell$ direction is given explicitly by
$$
\widetilde{W}_{\psi,\gamma,\ell}(t):=\left\langle[\blacktriangle(\gamma_{[0,t]})]\psi\langle\xi,\mathfrak{b}_\ell\rangle_{\mathfrak{g}^*}\upsilon,1\right\rangle=\sum_{n\geqslant 1}X_{n,\ell}\left\langle[\blacktriangle(\gamma_{[0,t]})],\psi{e}_{\lambda}\upsilon\right\rangle\ \in\  L^2(\Omega,\mathbb{R}),
$$
where we recall that $(e_n)_{n\geqslant 1}$ is an orthonormal basis of $L^2_\upsilon(\Sigma,\R)$. One has then

\begin{lemma}\label{l:brownian-properties} Let $\gamma$ be an admissible curve and $\psi\in L^\infty(\Sigma)$. One has that $\widetilde{\mathfrak{W}}_{\psi,\gamma}(0)=0$ almost surely and, for every $1\leq \ell\leq L$ and for every $0\leq s\leq t\leq 1$, $\widetilde{W}_{\psi,\gamma,\ell}(t)-\widetilde{W}_{\psi,\gamma,\ell}(s)$ follows a centered normal law with variance $\|\psi[\blacktriangle(\gamma_{[s,t]})]\|_{L^2_\upsilon(\Sigma)}^2$. 

If, in addition, $\gamma$ is primitive with respect to $V$, then, for every $p\geq 1$ and for every $0=t_0<t_1<t_2 <\ldots<t_p\leq 1$, $\widetilde{W}_{\psi,\gamma,\ell}(t_1)$, $\widetilde{W}_{\psi,\gamma,\ell}(t_2)-\widetilde{W}_{\psi,\gamma,\ell}(t_1)$,$\ldots$, $\widetilde{W}_{\psi,\gamma,\ell}(t_p)-\widetilde{W}_{\psi,\gamma,\ell}(t_{p-1})$ are independent.
\end{lemma}

This Lemma shows that, up to reparametrization of time and up to the continuity with respect to time, each $\widetilde{W}_{\psi,\gamma,\ell}$ has the properties of a Brownian motion (or Wiener process) on $\mathfrak{g}$. While the independence requires a primitive curve, the first property in the Lemma holds for any admissible curve. Observe that, if there are type $\operatorname{III}$-curves in the decomposition of $\gamma$, then we use a small abuse of notations as the law of $\widetilde{W}_{\psi,\gamma,\ell}(t)-\widetilde{W}_{\psi,\gamma,\ell}(s)$ is the Dirac mass at $0$ (on the segments $[s,t]$ of $\gamma$ where it follows the flow lines). In the case where $\gamma$ is primitive, one has that
$$
\|\psi[\blacktriangle(\gamma_{[s,t]})]\|_{L^2_\upsilon(\Sigma)}^2
$$
is the area of the triangle $\blacktriangle(\gamma_{[s,t]})$. We say that our Wiener process verifies \emph{the area law}. 

\begin{proof}
Fix $1\leqslant\ell\leqslant L$, let $0=t_0<t_1<\ldots<t_p\leqslant 1$ and set $Z_j:=\widetilde{W}_{\psi,\gamma,\ell}(t_j)-\widetilde{W}_{\psi,\gamma,\ell}(t_{j-1})$ for $1\leqslant j\leqslant p$. Each $Z_j$ is centered and depends linearly on $\xi$, so that $(Z_1,\ldots,Z_p)$ is a jointly Gaussian vector and it suffices to compute its covariance matrix. The componentwise version of~\eqref{e:keyrelation-proba} gives
$$
\mathbb{E}\left(Z_jZ_k\right)=\left\langle \psi\left[\blacktriangle(\gamma_{[t_{j-1},t_j]})\right],\psi\left[\blacktriangle(\gamma_{[t_{k-1},t_k]})\right]\right\rangle_{L^2_\upsilon(\Sigma)}.
$$
Taking $j=k$ yields the first assertion. If moreover $\gamma$ is primitive with respect to $V$, the regions $\blacktriangle(\gamma_{[0,t]})$ increase with $t$ and their index functions take values in $\{0,\pm 1\}$; hence the currents $[\blacktriangle(\gamma_{[t_{j-1},t_j]})]$, $1\leqslant j\leqslant p$, have pairwise disjoint supports up to $\upsilon$-negligible sets. The above covariance matrix is then diagonal and $Z_1,\ldots,Z_p$ are independent.
\end{proof}

In order to use tools from stochastic differential equations, we also need the following continuity result with respect to time. 

\begin{lemma}\label{l:holder-process} Let $\gamma$ be an admissible curve and let $\psi\in L^\infty(\Sigma,\mathbb{R})$ that is continuous on the support of $[\blacktriangle(\gamma([0,t]))]$. Then, for every $1\leq \ell\leq L$ and for every $0\leq\beta<\frac{1}{2}$, there exists a \textbf{modification} $W_{\psi,\gamma,\ell}$ of $\widetilde{W}_{\psi,\gamma,\ell}$ such that $t\mapsto W_{\psi,\gamma,\ell}(t)$ is almost surely in $\mathcal{C}^\beta([0,1])$. 
\end{lemma}

By modification of the process, we mean that, for every $t\in[0,1]$,     
$\widetilde{W}_{\psi,\gamma,\ell}(t)=W_{\psi,\gamma,\ell}(t)$ almost surely. As in Lemma~\ref{l:brownian-properties}, we emphasize that we do not require the curve to be primitive here and this statement holds for any admissible curve which does not necessarily induce a reparametrization of a $\mathfrak{g}$-valued Brownian motion (due to the lack of independence if $\gamma$ is not primitive). We also note that $W_{\psi,\gamma,\ell}(t)$ verifies the property of Lemma~\ref{l:brownian-properties} as it is a modification of $\widetilde{W}_{\psi,\gamma,\ell}(t)$. And we define the following $\mathfrak{g}$-valued process:
$$
\mathfrak{W}_{\psi,\gamma}(t):=\sum_{\ell=1}^LW_{\psi,\gamma,\ell}(t)\mathfrak{b}_\ell.
$$

\begin{proof} Again the proof of such a result is standard once we are given~\eqref{e:keyrelation-proba} and we just record it for the sake of completeness. %First of all, we observe that 
%$$
%\mathbb{E}\left(\int_0^1|\widetilde{W}_{\psi,\gamma,\ell}(t)|^2dt\right)=\int_0^1\|[\blacktriangle(\gamma_{[0,t]})]\|_{L^2(\Sigma,|\psi|^2\upsilon)}^2dt<\infty.
%$$
%Hence, almost surely, $W_{\psi,\gamma,\ell}(t)$ belongs to $L^2_t([0,1])$. Our goal is to prove that it has in fact an H\"older continuous representative. 
To do that, we fix some integer $q\geq 2$ and, thanks to Lemma~\ref{l:brownian-properties}, we write, for $0\leq s\leq t\leq 1$,
$$
\mathbb{E}\left(|\widetilde{W}_{\psi,\gamma,\ell}(t)-\widetilde{W}_{\gamma,\ell}(s)|^{2q}\right)=\int_{\R}x^{2q}\frac{e^{-\frac{x^2}{2\int_{\Sigma}|[\blacktriangle(\gamma_{[s,t]})]|^2|\psi|^2\upsilon}}dx}{\left(2\pi\int_{\Sigma}|[\blacktriangle(\gamma_{[s,t]})]|^2|\psi|^2\upsilon\right)^{\frac12}}=C_q\left(\int_{\Sigma}|[\blacktriangle(\gamma_{[s,t]})]|^2|\psi|^2\upsilon\right)^q,
$$
for some constant $C_q>0$ depending only on $q$.
According to the remark  following Lemma~\ref{l:holder-regularity}, there exists, for every $0\leq \alpha<1$ a constant $C>0$ such that the upper bound is $\leq C|t-s|^{q\alpha}$. Hence, according to the Kolmogorov continuity theorem~\cite{Evans}, one can find $t\mapsto W_{\psi,\gamma,\ell}(t)$ which is almost surely in $\mathcal{C}^{\beta}([0,1])$ for every $0\leq\beta < \frac{\alpha}{2}-\frac{1}{2q}$. Moreover, for every $t\in[0,1]$, one has almost surely $W_{\psi,\gamma,\ell}(t)=\widetilde{W}_{\psi,\gamma,\ell}(t)$. %Let us finally verify that the almost sure property can be made independent of $t$. To see this, we fix $N\geq 1$ and we suppose that, for every $\omega\in\Omega_N$ (with $\Omega_N$ of full measure), one has $\tilde{W}_{\psi,\gamma,\ell}(j/N)=W_{\psi,\gamma,\ell}(j/N)$ for every $0\leq j\leq N$. Let us now set $\tilde{\Omega}=\cap_N\Omega_N$. One has
%\begin{multline*}
%\mathbb{E}\left(\int_0^1\left|\tilde{W}_{\psi,\gamma,\ell}(t)-W_{\psi,\gamma,\ell}(t)\right|^2dt\right)=\sum_{j=0}^{N-1}\int_{\frac{j}{N}}^{\frac{j+1}{N}}\mathbb{E}\left(\left|\tilde{W}_{\psi,\gamma,\ell}(t)-W_{\psi,\gamma,\ell}(t)\right|^2\right)dt\\
%\leq 
%2 \sum_{j=0}^{N-1}\int_{\frac{j}{N}}^{\frac{j+1}{N}}\mathbb{E}\left(\left|\tilde{W}_{\psi,\gamma,\ell}(t)-\tilde{W}_{\psi,\gamma,\ell}(j/N)\right|^2+\left|W_{\psi,\gamma,\ell}(t)-W_{\psi,\gamma,\ell}(j/N)\right|^2\right)dt.
%\end{multline*}
%Using the remark following Lemma~\ref{l:holder-regularity} together with~\eqref{e:keyrelation-proba}, we find that the second term in the upper bound is $\mathcal{O}(N^{-\beta})$ for some $\beta<\frac{1}{2}$. For the second term, one uses the H\"older regularity of $\tilde{W}_{\psi,\gamma,\ell}$ to reach a similar upper bound. This concludes the proof of the Lemma. 
\end{proof}

%\begin{rmk}
%Here, $\widetilde{W}_{\psi,\gamma,\ell}(t)$ is obtained as the limit in $L^2(\Omega,\mathbb{R})$ of 
%$$
%\widetilde{W}_{\psi,\gamma,\ell}^{(N)}(t,\omega):=\sum_{n=1}^NX_{n,\ell}(\omega)\langle[\blacktriangle(\gamma_{[0,t]})],\psi e_n\upsilon\rangle.
%$$ 
%This is a measurable function of $(t,\omega)$ and we have verified that, for every fixed $t\in[0,1]$, $\widetilde{W}_{\psi,\gamma,\ell}^{(N)}(t,\omega)$ converges to $\widetilde{W}_{\psi,\gamma,\ell}(t,\omega)$. Yet, it is not a priori true that this function is measurable of $(t,\omega)$. In particular, it does not a priori coincide with the limit of the sequence $\widetilde{W}_{\psi,\gamma,\ell}^{(N)}(t,\omega)$ in $L^2(\Omega,L^2([0,1]))$. This is why Kolmogorov Theorem provides only a modification of the process $\widetilde{W}_{\psi,\gamma,\ell}(t)$.
%\end{rmk}
%

\subsection{Wiener process at the maximum of \texorpdfstring{$f$}{f}}\label{ss:blowup-Wiener-max}

In view of defining the Yang-Mills measure, we will also need to consider the limit case where $\gamma$ is reduced to the point $a_{2g+2}$ where $f$ reaches its maximal value. To do that, we work in the Morse coordinates $(x_1,x_2)=(r\cos\theta,r\sin\theta)$ near $a_{2g+2}$. One knows that there exist $4g$ angles $0\leq\theta_1<\theta_2<\ldots \theta_{4g-1}<\theta_{4g}<2\pi$ corresponding to the unstable manifolds issued from a critical point of index $1$ and ending at $a_{2g+2}$. For every $r>0$, we define the following curve $\gamma_r(t)=(r\cos(2\pi t),-r\sin(2\pi t))$ which is an admissible curve obtained as the concatenation of type $\operatorname{I}$ and type $\operatorname{II}$ elementary curves (of Tetris type). Hence we can define $\mathfrak{W}_{\psi,\gamma_r}(t)$ and Lemmas~\ref{l:integration-connection},~\ref{l:brownian-properties} and~\ref{l:holder-process} apply to this random process.

In fact, the exact same proofs allow us to consider the blow-up case where $r=0$. Namely, we can set
\begin{equation}\label{e:process-max}
\widetilde{\mathfrak{W}}_{0}(t)=\sum_{n\geqslant 1}\sum_{\ell=1}^LX_{n,\ell}\langle [\blacktriangle_0(t)], e_{n,\ell}\upsilon\rangle, 
\end{equation}
where $[\blacktriangle_0(t)]=\lim_{r\rightarrow 0^+}[\blacktriangle(\gamma_r[0,t])].$ Again, $t\mapsto [\blacktriangle_0(t)]$ maps $\R$ to a function on $\Sigma$ taking only integer values. Note that we pick $\psi=1$ in this case. Recall that the key ingredient showing that $\mathfrak{W}_{\psi,\gamma_r}(t)$ has all the properties of a Wiener process was Lemma~\ref{l:holder-regularity} proving the H\"older regularity of the area functional corresponding to the function $[\blacktriangle(\gamma_r[0,t])]$. Hence, we only need to prove
\begin{lemma}\label{l:holder-regularity2} For every $0<\alpha<1$, the map
$$
t\in[0,1]\mapsto \int_{\Sigma}[\blacktriangle_0(t)]\upsilon
$$
belongs to $\mathcal{C}^\alpha([0,1])$. More precisely, it is continuous on $[0,1]$, of class $\mathcal{C}^1$ on $[0,1]\setminus\{\theta_j/(2\pi):1\leq j\leq 4g\},$ and as $t\rightarrow \theta_j/(2\pi)$, one has 
$$\int_{\Sigma}[\blacktriangle_0(t)]\upsilon=\mathcal{O}(|\ln|t-\theta_j/(2\pi)||).
$$ 
Finally, there exists a constant $c_0>0$ such that
$$
\forall t\in[0,1]\setminus\left\{\frac{\theta_1}{2\pi},\ldots,\frac{\theta_{4g}}{2\pi}\right\},\quad \partial_t\int_{\Sigma}[\blacktriangle_0(t)]\upsilon\geq c_0.
$$
\end{lemma}
As a corollary, we get following the lines of paragraph~\ref{ss:arealawtransverse}:
\begin{corollary}\label{c:brownian-maximum} The stochastic process $t\in\mathbb{R}_+\mapsto \widetilde{\mathfrak{W}}_0(t)\in\mathfrak{g}$ verifies the following properties:
\begin{itemize}
\item the one dimensional processes $\widetilde{W}_{0,\ell}:t\mapsto \operatorname{Re}\langle\widetilde{\mathfrak{W}}_0(t),\mathfrak{b}_\ell\rangle$, $\ell=1,\ldots, L$, are independent;
\item it has a modification $\mathfrak{W}_{0}(t)=\sum_{\ell}W_{0,\ell}(t)\mathfrak{b}_\ell$ which is almost surely in $\mathcal{C}^\alpha$ (for every $0<\alpha<\frac{1}{2}$);
\item for every $0\leq s\leq t<\infty$, $W_{0,\ell}(t)-W_{0,\ell}(s)$ follows a centered normal law with variance $\|[\blacktriangle_0(t)-\blacktriangle_0(s)]\|_{L^2_\upsilon(\Sigma)}^2$;
\item for every $p\geq 1$ and for every $0=t_0<t_1<t_2 <\ldots<t_n\leq 1$, $W_{0,\ell}(t_1)$, $W_{0,\ell}(t_2)-W_{0,\ell}(t_1)$,$\ldots$, $W_{0,\ell}(t_p)-W_{0,\ell}(t_{p-1})$ are independent.
\end{itemize}
\end{corollary}
Again, we emphasize that the independence property only holds on the interval $[0,1]$.
\begin{proof}[Proof of Lemma~\ref{l:holder-regularity2}] Let $r>0$. %Observe first that, for every $t\in\mathbb{R}$,
%\begin{multline*}
%\int_{\Sigma}[\blacktriangle(\gamma_r[0,t+2\pi])]\upsilon=\int_{\Sigma}[\blacktriangle(\gamma_r[0,t])]\upsilon+\int_{\Sigma}[\blacktriangle(\gamma_r[t,t+2\pi])]\upsilon\\
%=\int_{\Sigma}[\blacktriangle(\gamma_r[0,t])]\upsilon+\nu(\Sigma\setminus B_{h_f}(a_{2g+2},r)).
%\end{multline*}
%Hence, it is sufficient to describe the regularity of the function on the interval $[0,2\pi]$. 
By construction, the function $t\in[0,1]\mapsto \int_{\Sigma}[\blacktriangle(\gamma_r[0,t+2\pi])]\upsilon $  is continuous and Lemma~\ref{l:holder-regularity} shows that it is piecewise $\mathcal{C}^1$. Recall now from the proof of Lemma~\ref{l:holder-regularity} that
$$
\int_{\Sigma}[\blacktriangle(\gamma_r[0,t])]\upsilon=\lim_{T\rightarrow+\infty}\int_0^T\int_{\Sigma}e^{-\int_{-\tau}^0(\text{div}_{\upsilon_f}V)\circ\varphi_f^sds}\frac{d\upsilon}{d\upsilon_f}\circ\varphi_f^{-\tau}[\gamma_r([0,t])]\wedge\iota_V(\upsilon_f)d\tau.
$$
Hence, thanks to~\eqref{e:formula-derivative} and to the explicit expression of the metric in the Morse chart near $a_{2g+2}$, one has, for $t\notin\{\theta_j/(2\pi):1\leq j\leq 4g\}$,
$$
\partial_t\int_{\Sigma}[\blacktriangle(\gamma_r[0,t])]\upsilon=-2\pi r^2\int_0^\infty e^{-\int_{-\tau}^0(\text{div}_{\upsilon_f}V)\circ\varphi_f^s(\gamma_r(t))ds}\frac{d\upsilon}{d\upsilon_f}\circ\varphi_f^{-\tau}(\gamma_r(t))d\tau.
$$
We now let $0<r<r_0<1$ with $r_0$ fixed once and for all by the size of the Morse chart. We split the integral between the integration on the interval $\tau\in[0,\ln(r_0/r)]$ and the one for $\tau\in[\ln(r_0/r),\infty)$. For the first one, one has
\begin{multline*}
-2\pi r^2\int_0^{\ln\frac{r_0}{r}} e^{-\int_{-\tau}^0(\text{div}_{\upsilon_f}V)\circ\varphi_f^s(\gamma_r(t))ds}\frac{d\upsilon}{d\upsilon_f}\circ\varphi_f^{-\tau}(\gamma_r(t))d\tau\\
=-2\pi r^2\int_0^{\ln\frac{r_0}{r}} e^{2\tau}\frac{d\upsilon}{d\upsilon_f}\circ\varphi_f^{-\tau}(\gamma_r(t))d\tau\\
=-2\pi r_0^2\int_{-\ln\frac{r_0}{r}}^0 e^{2\tau}\frac{d\upsilon}{d\upsilon_f}\circ\varphi_f^{-\tau+\ln\frac{r_0}{r}}(\gamma_r(t))d\tau.
\end{multline*}
As $\frac{d\upsilon}{d\upsilon_f}$ is uniformly bounded, one can verify that the limit (as $r\rightarrow 0^+$) of this function is well defined and bounded on $[0,1]$. Moreover it is continuous on $[0,1]\setminus\{\theta_j/(2\pi):1\leq j\leq 4g\}.$ For the other part of the integral, we can write
\begin{multline*}
\int_{-\tau}^0(\text{div}_{\upsilon_f}V)\circ\varphi_f^s(\gamma_r(t))ds=\int_{-\tau}^{-\ln\frac{r_0}{r}}(\text{div}_{\upsilon_f}V)\circ\varphi_f^s(\gamma_r(t))ds+\int_{-\ln\frac{r_0}{r}}^0(\text{div}_{\upsilon_f}V)\circ\varphi_f^s(\gamma_r(t))ds\\
=\int_{-\tau}^{-\ln\frac{r_0}{r}}(\text{div}_{\upsilon_f}V)\circ\varphi_f^s(\gamma_{r}(t))ds-2\ln\left(\frac{r_0}{r}\right),
\end{multline*}
where we used the exact expression of the flow near the maximum given by~\eqref{e:local-expression-max}. Hence, one has
\begin{multline*}
 -2\pi r^2\int_{\ln\frac{r_0}{r}}^\infty e^{-\int_{-\tau}^0(\text{div}_{\upsilon_f}V)\circ\varphi_f^s(\gamma_r(t))ds}\frac{d\upsilon}{d\upsilon_f}\circ\varphi_f^{-\tau}(\gamma_r(t))d\tau\\
 =-2\pi r_0^2\int_{\ln\frac{r_0}{r}}^\infty e^{-\int_{-\tau}^{-\ln\frac{r_0}{r}}(\text{div}_{\upsilon_f}V)\circ\varphi_f^s(\gamma_r(t))ds}\frac{d\upsilon}{d\upsilon_f}\circ\varphi_f^{-\tau}(\gamma_r(t))d\tau\\
 =-2\pi r_0^2\int_{0}^\infty e^{-\int_{-\tau}^{0}(\text{div}_{\upsilon_f}V)\circ\varphi_f^s(\gamma_{r_0}(t))ds}\frac{d\upsilon}{d\upsilon_f}\circ\varphi_f^{-\tau+\ln\frac{r_0}{r}}(\gamma_r(t))d\tau.
\end{multline*}
One more time, the limit (as $r\rightarrow 0^+$) is a well defined function that is continuous $[0,1]\setminus\{\theta_j/(2\pi):1\leq j\leq 4g\}.$ The only difference with the first integral is that it is not anymore uniformly bounded on $[0,1]$. Arguing as in the proof of Lemma~\ref{l:holder-regularity}, one finds that, as $t\rightarrow \theta_j/(2\pi)$, this is in fact $\mathcal{O}(|\ln|t-\theta_j/(2\pi)||).$ From this, we can deduce that the weak derivative of $t\mapsto [\blacktriangle_0(t)]$ belongs to $L^p([0,1])$ for every $1\leq p<\infty$ which concludes the proof of the Lemma thanks to Sobolev embeddings.
\end{proof}

%It is important to note that the same proof carries over to regular level curves almost verbatim, which will be required for the definition of the YM measure, especially the conditioning part.

\section{Gaussian random holonomies}\label{s:randomholonomy}

In Section~\ref{s:integration-connection}, we defined a \emph{reparametrized} version of a Brownian motion on the Lie algebra through the definition of the stochastic process $t\mapsto \mathfrak{W}_{\psi,\gamma}(t)$ associated with the white noise $\xi_\psi$. Recall that this requires some primitive properties of the curve $\gamma$ -- see Section~\ref{s:integration-connection} for definitions and statements. In view of constructing the Yang-Mills measure and its holonomy processes, we need to solve the stochastic differential equation associated with $\mathfrak{W}_{\psi,\gamma}$. This solution will represent the holonomy of $A_\psi$ along $\gamma$.

All along this section, we let $\gamma$ be a primitive curve that we suppose without type $\operatorname{III}$-components. We will also suppose that the area functional
$$
\mathscr{A}_\gamma:t\in[0,1]\mapsto\int_\Sigma[\blacktriangle(\gamma([0,t]))]\upsilon
$$
is increasing (the decreasing case follows by reversing the time parametrization of the curve). 
In fact, we will make a small abuse of notations and allow the curve $\gamma$ to be formally reduced $a_{2g+2}=\text{argmax}(f)$ in order to include the stochastic process introduced in \S\ref{ss:blowup-Wiener-max}. Recall that, in the setting of \S\ref{ss:blowup-Wiener-max}, $[\blacktriangle(\gamma([0,t]))]$ is replaced by the limit function $[\blacktriangle_0(t)]$. In each case, it follows from Lemmas~\ref{l:holder-regularity} and~\ref{l:holder-regularity2} that $\mathscr{A}_\gamma$ belongs to every Sobolev space $W^{1,p}([0,1])$ with $1\leq p<\infty$ and that it is $\mathcal{C}^1$ except at finitely many points where $\mathscr{A}_\gamma'$ has singularities of logarithmic type. These singularities correspond to the points of the curve $\gamma$ where we switch from an elementary piece to another. They are indeed of logarithmic type when one has elementary curves of type $\operatorname{II}$, meaning the points where $\gamma$ intersects some $W^u(a)$ with $\text{ind}(a)=1$. 

\begin{rmk}
In \S\ref{ss:blowup-Wiener-max}, we also used the convention $\mathfrak{W}_0(t)$ when $\gamma$ is formally the blow-up of $\operatorname{argmax}(f)=a_{2g+2}$. In order to alleviate the notations, we use all along this section the same convention for both settings, i.e. we set $\gamma=0$ when we deal with the blow-up case.
\end{rmk}

\begin{rmk}\label{r:independence-reparametrization}
All along this section, we also make the assumption that $\psi\equiv 1$ as in later applications, we intend to apply the results from this Section with $\xi_U=\mathbf{1}_U\xi$, where $U$ contains the support of $[\blacktriangle(\gamma([0,t]))]$ for all $t\in[0,1]$. In that case, one has $\mathfrak{W}_{1,\gamma}=\mathfrak{W}_{\mathbf{1}_U,\gamma}$. In particular, the stochastic process depends only on $\xi_U$ where $U$ is the support of $[\blacktriangle(\gamma([0,1]))],$ the latter being equal to $[\blacktriangle(\gamma)]$ as $\gamma$ is $V$-oriented. In order to alleviate notations, we set
$$
\mathfrak{W}_{\gamma}:=\mathfrak{W}_{1,\gamma}.
$$
Also from Remark~\ref{r:independence-white-noise}, one finds that, if $\gamma$ and $\tilde{\gamma}$ verifies $[\blacktriangle(\gamma([0,t]))][\blacktriangle(\tilde{\gamma}([0,t']))]=0$ for all $t,t'\in[0,1]$, then the two stochastic processes $\mathfrak{W}_{\gamma}$ and $\mathfrak{W}_{\tilde{\gamma}}$ are independent.
\end{rmk}

This Section is organized as follows. First, in \S\ref{ss:def_lie_group}, we review a few properties of harmonic analysis on Lie groups. In particular, we introduce the heat kernel $p_t$ on $G$. Then, in \S\ref{ss:random-holo}, we define the holonomy of a random connection along a fixed curve $\gamma$ and gather its main properties in Theorem~\ref{t:local-holonomies}. Finally, in~\S\ref{ss:markov}, we describe the Markov properties of this holonomy process.

\subsection{A reminder on harmonic analysis on compact Lie groups}\label{ss:def_lie_group}

Before describing some properties of these random holonomies, let us record some other properties of the heat kernel $p_t$ for later use. We use conventions from harmonic analysis on Lie groups and we refer to~\cite[Def.~1.45, Th.~3.30]{Sepanski} and~\cite[Th.~4.20]{Knappbeyond} for more details. As before, we let $G$ be a compact and connected Lie group. We denote by $d\mu_G$ the normalized Haar measure on $G$, i.e. the only left invariant probability measure on $G$~\cite[Th.~1.47]{Sepanski}. Note that it is also right-invariant and invariant by inversion~\cite[Th.~1.46]{Sepanski}. Recall also that $(\mathfrak{b}_\ell)_{1\leq \ell\leq L}$ is an orthonormal basis of $\mathfrak{g}$ (viewed as a Euclidean vector space). We set 
\begin{equation}\label{e:def-constant-cg}
 C_\mathfrak{g}:=\sum_{\ell=1}^L\mathfrak{b}_\ell^2,
\end{equation}
where we recall that $G$ is a linear subgroup of $\text{GL}_N(\R)$~\cite[Cor.~4.22]{Knappbeyond}. One can verify that, for unitary groups of the form $U(N,\mathbb{K})$~\cite[Lemma 1.2]{LevyMaster}, $C_{\mathfrak{g}}=c_{\mathfrak{g}}\text{Id}_{\R^{N}}$ for some constant $c_\mathfrak{g}$.

Introduce now the following differential operator 
\begin{equation}\label{e:liederivative-group}
\forall\mathfrak{b}\in\mathfrak{g},\ \forall\mathrm{g}\in \text{M}_N(\C),\quad \mathcal{L}_{\mathfrak{b}}\Psi(\mathrm{g}):=\frac{d}{dt}\Psi(\mathrm{g}e^{t\mathfrak{b}})|_{t=0},
\end{equation}
where $\Psi$ is a $\mathcal{C}^1$ function. This induces a differential operator on $\mathcal{C}^1$ functions on $G$ and one defines the Laplacian (or Casimir) on $G$ as follows:
$$
\Delta_G:=\sum_{\ell=1}^L\mathcal{L}_{\mathfrak{b}_\ell}^2.
$$
A key role will be played by the kernel $p_t$ of the corresponding heat equation: 
\begin{equation}\label{e:12-heat}
\partial_tu=\frac{1}{2}\Delta_Gu,\quad u(t=0)=\Psi.
\end{equation}
In order to describe this kernel, we denote by  $\widehat{G}$ the set of equivalence classes of irreducible (unitary) representations of $G$, and we will let $(\chi_\rho)_{\rho\in \widehat G} $ be an orthonormal basis of 
$$
L^2(G, \mu_G)^G:=\left\{\Psi\in L^2(G,\mu_G):\forall\mathrm{g},\forall\mathrm{h},\ \Psi(\mathrm{h})=\Psi(\mathrm{g}\mathrm{h}\mathrm{g}^{-1})\right\},
$$
made of the corresponding characters. Recall also that, if $\chi_\rho$ and $\chi_{\rho'}$ are orthogonal, then 
\begin{equation}\label{e:orthogonality-character}
\chi_\rho *\chi_{\rho'}(\mathrm{g}):=\int_{G}\chi_\rho(\mathrm{g}\mathrm{h}^{-1})\chi_{\rho'}(\mathrm{h})d\mu_G(\mathrm{h})=0
\end{equation} 
while 
\begin{equation}\label{e:self-convolution-character}
\chi_\rho *\chi_\rho(\mathrm{g})=(\dim V_\rho)^{-1}\chi_\rho(\mathrm{g}). 
\end{equation}
For any $\rho\in\widehat{G}$, one has $\Delta_G\chi_\rho=-c_2(\rho)\chi_\rho$ so that the heat kernel on $G$ is given by 
\begin{equation}\label{e:heat-kernel-Fourier}
\forall \mathrm{g}\in G,\quad p_t(\mathrm{g})=\sum_{\chi\in \widehat{G}} \dim (V_\rho) e^{-\frac{c_2(\rho)}{2} t}  \chi_\rho (\mathrm{g}),
\end{equation}
with $\int_G p_t(\mathrm{g})d\mu_G(\mathrm{g})=1$. With these conventions, the solution to~\eqref{e:12-heat} writes 
$$
e^{\frac{t\Delta_G}{2}}\Psi(\mathrm{g})=p_t* \Psi(\mathrm{g})=\int_{G}p_t(\mathrm{g}\mathrm{h}^{-1})\Psi(\mathrm{h})d\mu_G(\mathrm{h}).
$$ 
Finally, we record the following useful lemma that will be used several times later on.
\begin{lemma}\label{l:easy-useful-heat-kernel} Let $\tau>0$ and let $\mathrm{h}_1,\mathrm{h}_2$ be two elements in $G$. Then, for any bounded and measurable function $\Psi:G\rightarrow\mathbb{C}$, one has
$$
\forall \mathrm{g}\in\mathbb{C},\quad\int_Gp_{\tau}\left(\mathrm{g}\mathrm{g}_1^{-1}\right)\Psi(\mathrm{h}_1\mathrm{g}_1\mathrm{h}_2)\mu_G(d\mathrm{g}_1)=e^{\frac{\tau\Delta_G}{2}}(\Psi)(\mathrm{h}_1\mathrm{g}\mathrm{h}_2).
$$
If, in addition, $\Psi$ is a central function, then one has 
$$
\forall \mathrm{g}\in\mathbb{C},\quad\int_Gp_{\tau}\left(\mathrm{g}\mathrm{g}_1^{-1}\right)\Psi(\mathrm{h}_1\mathrm{g}_1\mathrm{h}_2)\mu_G(d\mathrm{g}_1)=e^{\frac{\tau\Delta_G}{2}}(\Psi)(\mathrm{h}_2\mathrm{h}_1\mathrm{g}).
$$
\end{lemma}
The proof follows directly from the invariance properties of the Haar measure and we recall that central means that, for all $\mathrm{g},\mathrm{h}$ in $G$, one has $\Psi(\mathrm{h}\mathrm{g}\mathrm{h}^{-1})=\Psi(\mathrm{g})$.

%Note also that 
%\begin{equation}\label{e:identity-Fourier}
%\delta_{\operatorname{Id}}:=\sum_{\chi\in \widehat{G}} \dim (V_\chi)   \chi.
%\end{equation}%, where $c_2(\chi)$ is the eigenvalue of the Casimir operator. 

\subsection{Local gaussian random holonomies}
\label{ss:random-holo}
If we denote by $\mathcal{F}_\gamma(t)$ the $\sigma$-algebra generated\footnote{Here, we mean the $\sigma$-algebra generated by the set $(\mathfrak{W}_{\gamma}(s)^{-1}(B))_{0\leq s\leq t}$, where $B$ is a Borel set of $\mathfrak{g}$.} by $(\mathfrak{W}_{\gamma}(s))_{0\leq s\leq t}$ and containing all the set of $\mathbb{P}$-measure $0$, it is by construction independent of the $\sigma$-algebra generated by $(\mathfrak{W}_{\gamma}(s)-\mathfrak{W}_{\gamma}(t))_{ s\geq t}$ (see e.g. Remark~\ref{r:independence-reparametrization}). The family of $\sigma$-algebra $t\mapsto\mathcal{F}_\gamma(t)$ is a filtration (or also non-anticipating), and the continuous process $\mathfrak{W}_\gamma$ is a (local) martingale with respect to this filtration~\cite[Ch.~IV, Def.1.5]{RevuzYor} in the sense that
$$
\forall 0\leq s\leq t\leq 1,\quad\mathbb{E}\left(\mathfrak{W}_\gamma(t)|\mathcal{F}_\gamma(s)\right)=\mathfrak{W}_\gamma(s),
$$
which follows directly from Lemma~\ref{l:brownian-properties}. In order to define our random holonomies associated with the curve $\gamma$, we solve the following stochastic differential equation:
\begin{equation}\label{e:SDE-holonomy-Ito}
 \mathrm{g}_\gamma(t)=\text{Id}-\int_0^t\mathrm{g}_\gamma(\tau)d\mathfrak{W}_{\gamma,\tau}+\frac{1}{2}\int_0^t\mathscr{A}_\gamma'(\tau)\mathrm{g}_\gamma(\tau)C_{\mathfrak{g}}d\tau, \quad t\in[0,1],
\end{equation}
where the integral is understood in the sense of It\^o~\cite[\S IV.2]{RevuzYor} and where $C_{\mathfrak{g}}$ is defined by~\eqref{e:def-constant-cg}. This equation is the stochastic analogue of the classical parallel transport equation~\eqref{e:parallel-transport}. See Appendix~\ref{a:reparametrization-sde} for a brief reminder on It\^o's integral for reparametrized brownian motions. Equivalently, in the sense of Stratonovich integration~\cite[Ch.VI]{FranchiLeJan} (see also Appendix~\ref{a:reparametrization-sde}),  this equation reads
\begin{equation}\label{e:SDE-holonomy-Strato}
 \mathrm{g}_\gamma(t)=\text{Id}-\int_0^t\mathrm{g}_\gamma(\tau)\circ d\mathfrak{W}_{\gamma,\tau} , \quad t\in[0,1].
\end{equation}
The next theorem states the existence and uniqueness of these solutions together with some of their main properties.
\begin{thm}\label{t:local-holonomies} Let $\gamma$ be a primitive curve with no type $\operatorname{III}$ components and such that the area functional $t\mapsto \int_{\Sigma}[\blacktriangle(\gamma([0,t]))]$ is increasing on $[0,1]$ (including the case of \S\ref{ss:blowup-Wiener-max} where $\gamma$ is formally reduced to $\{a_{2g+2}\}$). Then, the following holds:
\begin{enumerate}
 \item there exists a unique solution $\mathrm{g}_\gamma$ to~\eqref{e:SDE-holonomy-Ito} which verifies $\mathbb{E}\left(\int_0^1\|\mathrm{g}_\gamma(t)\|^2\mathscr{A}_\gamma'(t)dt\right)<\infty,$ where $\|.\|$ is a norm on $\operatorname{M}_N(\C)$;
 \item the map $t\in[0,1]\mapsto \mathrm{g}_\gamma(t)$ is almost surely continuous and $\mathrm{g}_\gamma(t)$ is $\mathcal{F}_\gamma(t)$-measurable for every $t\in[0,1]$;
 \item almost surely, $\mathrm{g}_\gamma(t)\in G$ for all $t\in[0,1]$;
 \item the distribution of $\mathrm{g}_\gamma(t)$ is given by $p_{\mathscr{A}_\gamma(t)}(\mathrm{h})\mu_G(d\mathrm{h})$;
 \item for every $p\geqslant 1$ and for every $0\leqslant t_1< t_2<\ldots< t_p\leqslant 1$, $\mathrm{g}_{\gamma}(t_1)$, $\mathrm{g}_{\gamma}(t_1)^{-1}\mathrm{g}_{\gamma}(t_2)$, $\ldots$, $\mathrm{g}_{\gamma}(t_{p-1})^{-1}\mathrm{g}_{\gamma}(t_p)$ are independent;
 \item for all $0\leq s\leq t\leq 1$, $\mathrm{g}_{\gamma}(s)^{-1}\mathrm{g}_{\gamma}(t)$ has the same distribution as $\mathrm{g}_{\gamma}(t-s)$;
 \item for every $\mathrm{h}$ in $G$, $\mathrm{g}_{\gamma}(t)$ and $\mathrm{h}\mathrm{g}_{\gamma}(t)\mathrm{h}^{-1}$ have the same distribution.
\end{enumerate}

\end{thm}

In other words, this theorem shows that $t\in[0,1]\mapsto \mathrm{g}_\gamma(t)\in G$ is a reparametrization of the Brownian motion on $G$. The solution $\mathrm{g}_\gamma(t,\omega)$ at time $t$ is regarded as the holonomy of the random connection $A(\omega)$ along the piece of curve $\gamma|_{[0,t]}$. The proof of such a result is classical in stochastic differential equations. We refer to~\cite[\S IX.2]{RevuzYor} for general results on the resolution of stochastic differential equations and to~\cite[Ch.VI-VII]{FranchiLeJan} for the specific case of $G$-valued Brownian with $\mathscr{A}_\gamma'=1$. We also refer to Theorem~\ref{t:reparametrized-brownian} in Appendix~\ref{a:reparametrization-sde} for a brief reminder on the existence and uniqueness of solutions to equation~\eqref{e:SDE-holonomy-Ito}. In particular, the reader less familiar with stochastic integration will also find in this appendix proofs and references for the different statements in this theorem. Here, the main difference with references like~\cite{FranchiLeJan} is that the derivative of the reparametrization $\mathscr{A}_\gamma'$ is not equal to $1$ and may even have finitely many logarithmic singularities. 

\begin{rmk}
The general theory~\cite[\S IV, \S IX]{RevuzYor} in fact allows one to deal with settings where $\mathscr{A}_\gamma$ is a continuous increasing function with a nice enough derivative. Indeed, we are considering in~\eqref{e:SDE-holonomy-Ito} a continuous process $\mathfrak{X}_\gamma=-\mathfrak{W}_{\gamma}+\frac{C_\mathfrak{g}}{2}\mathscr{A}_\gamma$, with $t\in[0,1]\mapsto \mathfrak{W}_{\gamma}(t)\in\mathfrak{g}$ being a continuous $\mathcal{F}_\gamma$-local martingale. In the terminology of this reference, $\mathscr{A}_\gamma$ is an $\mathscr{F}_\gamma$-adapted and continuous process that is  increasing (hence of finite variation). Thus one can verify that $\mathfrak{X}_\gamma$ is a semimartingale in the sense of~\cite[\S IV, Def. 1.17]{RevuzYor} and it has finite quadratic variation~\cite[\S IV, Prop.~1.18]{RevuzYor} which makes it amenable to It\^o's integration. In our case, this quadratic variation can be explicitly expressed in terms of $\mathscr{A}_\gamma$. Namely, for all $1\leq \ell\leq L$,
$$
\forall 0\leq s\leq t\leq 1,\quad \lim_{n\rightarrow+\infty}\sum_{k=0}^{m_n-1}\left(W_{\gamma,\ell}(t_{k+1}^n)-W_{\gamma,\ell}(t_k^n)\right)^2=\mathscr{A}_\gamma(t)-\mathscr{A}_\gamma(s)=\int_{s}^t\mathscr{A}_\gamma'(\tau)d\tau,
$$
where $s=t_0^n<t_1^n<\ldots<t_{m_n}^n=t$ with $\lim_{n\rightarrow+\infty}\max|t_{k+1}^n-t_k^n|=0$ and where the limit is taken in $L^2(\Omega)$. As $\mathfrak{X}_\gamma$ is a \emph{continuous semi-martingale}, classical theorems on the resolution of stochastic differential equations like~\cite[\S IX.2, Th.~2.1]{RevuzYor} apply to~\eqref{e:SDE-holonomy-Ito}. The fact that the solution lies indeed in $G$ requires more arguments~\cite[Ch.~VII]{FranchiLeJan} -- see also Appendix~\ref{a:reparametrization-sde} for a brief reminder from this reference.
\end{rmk}

Before defining the Yang-Mills measure using these holonomy processes, let us record several statements on them.
\begin{lemma}\label{l:decomposition-curve} Let $\gamma,\tilde{\gamma}$ be two curves as in Theorem~\ref{t:local-holonomies} such that $\gamma([0,1])=\tilde{\gamma}([0,1])$ and such that $\gamma(0)=\tilde{\gamma}(0)$. Then, one has almost surely
 $$
 \mathrm{g}_{\gamma}(1)=\mathrm{g}_{\tilde{\gamma}}(1).
 $$
 Moreover, if one decomposes $\gamma=\gamma_1\star\left(\gamma_2\star\left(\ldots\star\gamma_J\right)\right)$ into elementary pieces (including the blowup case of \S\ref{ss:blowup-Wiener-max}), then
 $$
  \mathrm{g}_{\gamma}(1)=\mathrm{g}_{\gamma_J}(1)\ldots\mathrm{g}_{\gamma_2}(1)\mathrm{g}_{\gamma_1}(1),
 $$
 and $(\mathrm{g}_{\gamma_J}(1),\ldots,\mathrm{g}_{\gamma_2}(1),\mathrm{g}_{\gamma_1}(1))$ are independent random variables.
\end{lemma}
 In particular, the random holonomy along a curve $\gamma$ is independent of the choice of parametrization and it can be written as the product of random holonomies along the elementary pieces defining $\gamma$. The fact that local holonomies $\mathrm{g}_{\gamma_j}(1)$ are independent uses crucially the property that $\gamma$ is primitive -- see Lemma~\ref{l:brownian-properties} where the primitive assumption is instrumental to ensure that $\mathfrak{W}_\gamma$ is indeed a reparametrized Brownian motion.
\begin{proof} We begin with the case where $\gamma$ is also an elementary curve and we consider a $\mathcal{C}^1$ function $\theta:[a,b]\rightarrow [0,1]$ such that $\theta'>0$ and $\theta$ is onto. Letting $\tilde{\gamma}=\gamma\circ\theta$, one can solve~\eqref{e:SDE-holonomy-Ito} with $\tilde{\gamma}$ replacing $\theta$ and verify from the definition of It\^o's integral (see Appendix~\ref{a:reparametrization-sde} for a brief reminder) that $\mathrm{g}_{\tilde{\gamma}}(t)=\mathrm{g}_\gamma\circ\theta(t)$ so that $\mathrm{g}_{\tilde{\gamma}}(b)=\mathrm{g}_\gamma(1)$. Moreover, if $\gamma$ is an elementary curve that is decomposed as $\gamma_1\star\gamma_2$, then, almost surely, one has
$$
\forall t\in\left[0,\frac12\right],\quad \mathrm{g}_{\gamma}(t)=\mathrm{g}_{\gamma_1}\left(2t\right)\ \text{and}\ \forall t\in\left[\frac12,1\right],\quad \mathrm{g}_{\gamma}(t)=\mathrm{g}_{\gamma_1}\left(2t-1\right)\mathrm{g}_{\gamma_1}\left(1\right).
$$
This follows from the construction of $\mathfrak{W}_\gamma$ and more precisely from the fact that
$$
\forall t\in\left[0,\frac12\right],\ \mathfrak{W}_\gamma(t)=\mathfrak{W}_{\gamma_{1}}\left(2t\right)\ \text{and}\ \forall t\in\left[\frac12,1\right],\quad \mathfrak{W}_{\gamma}(t)=\mathfrak{W}_{\gamma_2}\left(2t-1\right)+\mathfrak{W}_{\gamma_1}\left(1\right).
$$
For more general primitive curves $\gamma$ and $\tilde{\gamma}$ as in the statement of this Lemma, we use the property that $\gamma$ can be decomposed into elementary pieces. Up to reparametrization and splitting of the elementary pieces, each curve can be decomposed into $J$ pieces as $\gamma=\gamma_1\star\left(\gamma_2\star(\ldots\star\gamma_J)\right)$. One can then find $0=t_0<t_1<\ldots<t_J=1$ such that, on each interval $\tilde{\gamma}([t_{j-1},t_{j}])=\gamma_j([0,1])$ with $\gamma_j(0)=\tilde{\gamma}(t_j)$. Applying the discussion on elementary curves to each $j$, one gets the expected result. Independence follows from the independence property in Theorem~\ref{t:local-holonomies}.
\end{proof}

\begin{rmk}
Our holonomy process is well-defined for any curve $\gamma$ which is primitive without type $\text{III}$ components. In this case, $\mathfrak{W}_\gamma(t)$ is in fact a reparametrization of the standard Brownian motion as we have seen above. Hence, it is amenable to the standard theory of stochastic differential equations~\cite[Ch.~IX]{RevuzYor} thanks to the local martingale properties of $\mathfrak{W}_\gamma$. In fact, we could define holonomies for more general curves that are concatenation of primitive curves as in this paragraph by taking products of these local holonomies (with the holonomy being equal to the identity along flow lines) as we did in the previous proof. This will be the content of \S~\ref{ss:randomholonomy} where we will also include the contribution of the unstable manifolds $(W^u(a))_{\operatorname{ind}(a)=1}$.
\end{rmk}

We also have the following independence property:
\begin{lemma}\label{l:independence-SDE} Let $\gamma$ and $\tilde{\gamma}$ be two primitive curves with no type $\operatorname{III}$ components and such that the area functional $t\mapsto \int_{\Sigma}[\blacktriangle(\gamma([0,t]))]$ is increasing on $[0,1]$. Suppose that
$$
[\blacktriangle(\gamma)][\blacktriangle(\tilde{\gamma})]=0.
$$
 Then, for all $t,t'$, $\mathcal{F}_{\gamma}(t)$ and $\mathcal{F}_{\tilde{\gamma}}(t')$ are independent. In particular, 
 $\tilde{\mathrm{g}}_{\gamma,t}$ and $\tilde{\mathrm{g}}_{\tilde{\gamma},t'}$ are independent. 
\end{lemma}
\begin{proof}
The independence of the two $\sigma$-algebras follows from Remark~\ref{r:independence-reparametrization}. Now, as $\tilde{\mathrm{g}}_{\gamma,t}$ and $\tilde{\mathrm{g}}_{\tilde{\gamma},t'}$ are respectively $\mathcal{F}_{\gamma}(t)$ and $\mathcal{F}_{\tilde{\gamma}}(t')$ measurable by Theorem~\ref{t:local-holonomies}, they are also independent.
\end{proof}

Finally, we have the following relations between the different $\sigma$-algebra involved in our analysis
\begin{lemma}\label{l:inclusion-sigma-algebra} Let $\gamma$ be a primitive curve with no type $\operatorname{III}$ components and such that the area functional $t\mapsto \int_{\Sigma}[\blacktriangle(\gamma([0,t]))]$ is increasing on $[0,1]$. Then, one has, for every $0\leq t\leq 1$,
 $$
  \sigma((\mathrm{g}_{\gamma,s})_{0\leq s\leq t})\subset\mathcal{F}_{\gamma}(t)\subset\sigma\left(\left(\xi_{[\blacktriangle(\gamma[0,s])]}\right)_{0\leq s\leq t}\right),
 $$
 where the last $\sigma$-algebra is the one generated by $\left(\langle \xi_{[\blacktriangle(\gamma[0,s])]},\psi\rangle_{\mathfrak{g}}\right)_{0\leq s\leq t}$, with $\psi$ running over elements in $\mathcal{C}^\infty(\Sigma,\mathfrak{g})$, and by all the sets of $\mathbb{P}$-measure zero.
\end{lemma}
We emphasize that $A_{[\blacktriangle(\gamma)]}$ is measurable with respect to the last $\sigma$-algebra for $t=1$. Indeed, it is obtained as a limit (under the gradient flow) of the random variables used to define this $\sigma$-algebra -- see Theorem~\ref{thm:main1}. In order to alleviate notations, we will also set
$$
\widetilde{\mathcal{F}}_\gamma(t):=\sigma\left(\left(\xi_{[\blacktriangle(\gamma[0,s])]}\right)_{0\leq s\leq t}\right)=\sigma\left(\xi_{[\blacktriangle(\gamma[0,t])]}\right),
$$
where the last equality follows from Lemma~\ref{r:whitenoise-multiplication}. It is also a filtration (or non-anticipating) with respect to the $\mathfrak{g}$-valued Brownian motion $\mathfrak{W}_\gamma(t)$. Recall from~\cite[Ch.~4]{Evans} that it means that this is a nondecreasing sequence of $\sigma$-algebra, that $\mathcal{F}_\gamma(t)\subset \widetilde{\mathcal{F}}_\gamma(t)$ and that it is independent of $(\mathfrak{W}_{\gamma}(s)-\mathfrak{W}_{\gamma}(t))_{ s\geq t}$. Roughly speaking, $\widetilde{\mathcal{F}}_\gamma(t)$ contains all the information of the white noise $\xi$ on $\blacktriangle(\gamma[0,t])$ while $\mathcal{F}_\gamma(t)$ contains only the information of the $\mathfrak{g}$-valued random variable $\langle\xi_{[\blacktriangle(\gamma[0,t])]},1\rangle$.

\begin{proof}
 The first inclusion is a direct consequence of our construction and of the properties from Theorem~\ref{t:local-holonomies}. Regarding the last inclusion, it follows from the fact that the elements used to generate $\mathcal{F}_{\gamma}(t)$ are by construction measurable with respect to $\sigma\left((\mathfrak{W}_\gamma(s))_{0\leq s\leq t}\right)\subset \widetilde{\mathcal{F}}_\gamma(t).$ 
\end{proof}

\subsection{Markov properties of gaussian random holonomies}
\label{ss:markov}

We conclude this section by describing a key property of these random processes on $G$. This property (or more specifically a variant of it) will be at the heart of the definition of the Yang-Mills measure as well as the law of its random holonomies. In this paragraph, we illustrate this mechanism that we call abelianization in law in its most elementary form. For the random holonomies of Theorem~\ref{t:local-holonomies}, this property reads:
\begin{thm}\label{t:abelianization-form0} Let $\gamma_1,\tilde{\gamma}$ be two elementary curves which are both of type $\operatorname{I}$ or $\operatorname{II}_\pm$ such that $\tilde{\gamma}\preccurlyeq\gamma_1$. Set 
$$
\gamma(t)=\gamma_1\circ\tau(t),
$$
where $\tau$ is the function appearing in Definition~\ref{d:included-curve}. Then, for any bounded and measurable function $\Psi$ on $G$, one has
$$
\forall t\in[0,1],\quad\mathbb{E}\left(\Psi\left(\mathrm{g}_\gamma(t)\right)|\widetilde{\mathcal{F}}_{\tilde{\gamma}}(1)\right)=e^{\left(\mathscr{A}_\gamma(t)-\mathscr{A}_{\tilde{\gamma}}(t)\right)\frac{\Delta_G}{2}}\Psi\left(\mathrm{g}_{\tilde{\gamma}}(t)\right).
$$
\end{thm}
In particular, for $t=1$, one can argue as in the proof of Lemma~\ref{l:decomposition-curve} to deal with the reparametrization and this equality reads, for $\gamma_1,\tilde{\gamma}$ two elementary curves which are both of type $\operatorname{I}$ or $\operatorname{II}_\pm$,
\begin{equation}\label{e:abelianization-easy}
\tilde{\gamma}\preccurlyeq\gamma_1 \ \Longrightarrow   \mathbb{E}\left(\Psi\left(\mathrm{g}_{\gamma_1}(1)\right)|\widetilde{\mathcal{F}}_{\tilde{\gamma}}(1)\right)=e^{\left(\mathscr{A}_{\gamma_1}(1)-\mathscr{A}_{\tilde{\gamma}}(1)\right)\frac{\Delta_G}{2}}\Psi\left(\mathrm{g}_{\tilde{\gamma}}(1)\right).
\end{equation}
In the following, we will use this Theorem (and more specifically the strategy to prove it) in a refined form involving the contribution of the unstable manifolds but we already state it here in its most elementary form to emphasize the mechanism at work behind the proofs in the next sections. Recall that the $\sigma$-algebra $\widetilde{\mathcal{F}}_{\tilde{\gamma}}(1)$ contains all the information of the white noise on the support of $\blacktriangle(\tilde{\gamma})$ (hence on the backward trajectory of $\tilde{\gamma}$ under the gradient flow). Hence, this Theorem describes the law of the holonomy along a curve which is on the forward orbit of $\tilde{\gamma}$ knowing this information.

\begin{proof}
We fix $t\in[0,1]$ and we want to determine the conditional law
$ 
   \mathcal{L}\left( \mathrm{g}_\gamma(t) \ | 
   \widetilde{\mathcal{F}}_{\tilde{\gamma}}(1) \right).
$
We will perform the following computations in the Stratonovich convention~\eqref{e:SDE-holonomy-Strato}, which will take advantage of the fact that the chain rule and the Leibniz rule take the usual forms in that convention. See Appendix~\ref{a:reparametrization-sde} for a brief reminder and references. As such the symbol $\circ$ is omitted in this proof for the sake of simpler notations.
Because of the chain rule
$$
   d\left( (\mathrm{g}_{\tilde{\gamma}})^{-1} \right) = 
   - (\mathrm{g}_{\tilde{\gamma}})^{-1}
   d\mathrm{g}_{\tilde{\gamma}}
   (\mathrm{g}_{\tilde{\gamma}})^{-1} \ .
$$
As such, using the Leibniz rule
\begin{align*}
   d\left( \mathrm{g}_{\gamma}(\mathrm{g}_{\tilde{\gamma}})^{-1} \right) 
   = & \
   \mathrm{g}_{\gamma}d\left( (\mathrm{g}_{\tilde{\gamma}})^{-1}  \right)
   +
   d\mathrm{g}_{\gamma}(\mathrm{g}_{\tilde{\gamma}})^{-1}  
   \\
   = & \ 
   - \mathrm{g}_{\gamma}(\mathrm{g}_{\tilde{\gamma}})^{-1} 
   d\mathrm{g}_{\tilde{\gamma}}
   (\mathrm{g}_{\tilde{\gamma}})^{-1}  
   -    \mathrm{g}_{\gamma}d\mathfrak{W}_\gamma (\mathrm{g}_{\tilde{\gamma}})^{-1} \\
   = & \  \mathrm{g}_{\gamma}d\mathfrak{W}_{\tilde{\gamma}} (\mathrm{g}_{\tilde{\gamma}})^{-1} 
      - \mathrm{g}_{\gamma}  d\mathfrak{W}_\gamma (\mathrm{g}_{\tilde{\gamma}})^{-1}   \\
   = & \ -\mathrm{g}_{\gamma}(\mathrm{g}_{\tilde{\gamma}})^{-1} \Ad_{\mathrm{g}_{\tilde{\gamma}} } \left( d\mathfrak{W}_\gamma-d\mathfrak{W}_{\tilde{\gamma}}\right)  \\
   = & \ \mathrm{g}_{\gamma}(\mathrm{g}_{\tilde{\gamma}})^{-1}d\mathfrak{W}_{\tilde{\gamma},\gamma}.
\end{align*}
In the end, we find that if $\mathrm{g}_{\tilde{\gamma},\gamma}: =\mathrm{g}_{\gamma}(\mathrm{g}_{\tilde{\gamma}})^{-1}$, and $d\mathfrak{W}_{\tilde{\gamma},\gamma} =- \Ad_{\mathrm{g}_{\tilde{\gamma}} } \left( d\mathfrak{W}_\gamma-d\mathfrak{W}_{\tilde{\gamma}}\right)$, then $\mathrm{g}_{\tilde{\gamma},\gamma}$ solves the Stratonovich SDE
$$
d \mathrm{g}_{\tilde{\gamma},\gamma}  
   = d\mathfrak{W}_{\tilde{\gamma},\gamma} \mathrm{g}_{\tilde{\gamma},\gamma} \ .
$$
Recall that $X\mapsto \mathbb{E}\left(X|\widetilde{\mathcal{F}}_{\tilde{\gamma}}(1)\right)$ is a random variable which is measurable with respect to the reduced $\sigma$-algebra $\widetilde{\mathcal{F}}_{\tilde{\gamma}}(t)$. We now determine the law of $s\in[0,t]\mapsto\mathfrak{W}_{\tilde{\gamma},\gamma}(t)$ with respect to this conditional measure. To do that, we fix some $\mathfrak{b}\in\mathfrak{g}$ and we compute
$$
\forall 0\leq t\leq 1,\quad \mathbb{E}\left(e^{i\langle\mathfrak{b},\mathfrak{W}_{\tilde{\gamma},\gamma}(t)\rangle_{\mathfrak{g}}}|\widetilde{\mathcal{F}}_{\tilde{\gamma}}(1)\right)=\mathbb{E}\left(e^{i\langle\mathrm{g}_{\tilde{\gamma}}(t)^{-1}\mathfrak{b}\mathrm{g}_{\tilde{\gamma}}(t),\widetilde{\mathfrak{W}}_{\tilde{\gamma},\gamma}(t)\rangle_{\mathfrak{g}}}|\widetilde{\mathcal{F}}_{\tilde{\gamma}}(1)\right),
$$
where $\widetilde{\mathfrak{W}}_{\tilde{\gamma},\gamma}(t)=-\mathfrak{W}_\gamma(t)+\mathfrak{W}_{\tilde{\gamma}}(t)$. We now use that $\mathrm{g}_{\tilde{\gamma}}(t)$ is $\widetilde{\mathcal{F}}_{\tilde{\gamma}}(1)$-measurable and we view it as a deterministic variable. Using that $\widetilde{\mathfrak{W}}_{\tilde{\gamma},\gamma}(t)$ is by construction independent of $\widetilde{\mathcal{F}}_{\tilde{\gamma}}(1)$ (it only depends on the white noise on $\blacktriangle(\gamma)\setminus \blacktriangle(\tilde{\gamma})$ and its image under the gradient flow), we finally find that, for all $0\leq t\leq 1$,
$$
\mathbb{E}\left(e^{i\langle\mathfrak{b},\mathfrak{W}_{\tilde{\gamma},\gamma}(t)\rangle_{\mathfrak{g}}}|\widetilde{\mathcal{F}}_{\tilde{\gamma}}(1)\right)=\mathbb{E}\left(e^{i\langle\mathrm{g}_{\tilde{\gamma}}(t)^{-1}\mathfrak{b}\mathrm{g}_{\tilde{\gamma}}(t),\widetilde{\mathfrak{W}}_{\tilde{\gamma},\gamma}(t)\rangle_{\mathfrak{g}}}\right)=\frac{e^{-\frac{\|\mathrm{g}_{\tilde{\gamma}}(t)^{-1}\mathfrak{b}\mathrm{g}_{\tilde{\gamma}}(t)\|^2_{\mathfrak{g}}}{2\mathscr{A}_{\tilde{\gamma},\gamma}(t)}}}{(2\pi)^{\frac{L}{2}}}=\frac{e^{-\frac{\|\mathfrak{b}\|^2_{\mathfrak{g}}}{2\mathscr{A}_{\tilde{\gamma},\gamma}(t)}}}{(2\pi)^{\frac{L}{2}}},
$$
where $\mathscr{A}_{\tilde{\gamma},\gamma}(t):= \mathscr{A}_{\gamma}(t)-\mathscr{A}_{\tilde{\gamma}}(t)$, where $\mathrm{g}_{\tilde{\gamma}}(t)$ is understood as a deterministic variable in the second expectation and where we used Lemma~\ref{l:brownian-properties} (with $\psi=\mathbf{1}_{\blacktriangle(\gamma)\setminus \blacktriangle(\tilde{\gamma})}$) to write the second equality. In fact, arguing like this, we can verify that all the properties of Lemma~\ref{l:brownian-properties} remain true for $\mathfrak{W}_{\tilde{\gamma},\gamma}(t)$ conditionally to the $\sigma$-algebra $\widetilde{\mathcal{F}}_{\tilde{\gamma}}(1)$. Hence, conditionally to $\widetilde{\mathcal{F}}_{\tilde{\gamma}}(1)$, $\mathfrak{W}_{\tilde{\gamma},\gamma}(t)$ is a reparametrized $\mathfrak{g}$-valued Brownian motion. In particular, conditionally to $\widetilde{\mathcal{F}}_{\tilde{\gamma}}(1)$, $\mathrm{g}_{\tilde{\gamma},\gamma}$ solves the corresponding stochastic equation and it is a reparametrized Brownian motion on $G$. In particular, conditionally to this $\sigma$-algebra, its law can be expressed in terms of the solution to the heat equation as in Theorem~\ref{t:local-holonomies}. Thus, for any bounded and measurable function $\Psi:G\rightarrow\mathbb{C},$ one has
$$
\E\left( 
\Psi( \mathrm{g}_\gamma(t)) \ | \ \widetilde{\mathcal{F}}_{\tilde\gamma}(1)
\right)
=
\E\left( 
\Psi( \mathrm{g}_{\tilde{\gamma},\gamma}(t)\mathrm{g}_{\tilde{\gamma}}(t)) \ | \ \widetilde{\mathcal{F}}_{\tilde\gamma}(1)
\right)
=
\left( \exp\left(  \mathscr{A}_{\tilde{\gamma},\gamma}(t) \frac12 \Delta_G \right)\Psi \right)
\left( \mathrm{g}_{\tilde{\gamma}}(t) \right) \ .
$$
\end{proof}

As a direct corollary, one finds
\begin{corollary}\label{c:markov-holonomy-elementary}
  Let $\gamma,\tilde{\gamma}$ be two elementary curves which are both of type $\operatorname{I}$ or $\operatorname{II}_\pm$ such that $\tilde{\gamma}\preccurlyeq\gamma$.
Then, for any bounded and measurable functions $\Psi_1,\Psi_2$ on $G$, one has
$$
\mathbb{E}\left(\Psi_1(\mathrm{g}_{\tilde{\gamma}}(1))\Psi_2(\mathrm{g}_{\gamma}(1))\right)=e^{\frac{\mathscr{A}_{\tilde{\gamma}}(1)\Delta_G}{2}}\left(\Psi_1e^{\frac{\left(\mathscr{A}_\gamma(1)-\mathscr{A}_{\tilde{\gamma}}(1)\right)\Delta_G}{2}}(\Psi_2)\right)(\operatorname{Id}).
$$ 
Equivalently, for every $\mathrm{g}_1,\mathrm{g}_2$ in $G$, one has
$$
\mathbb{P}\left(\mathrm{g}_{\tilde{\gamma}}(1)\in d\mathrm{g}_1,\mathrm{g}_{\gamma}(1)\in d\mathrm{g}_2\right)=p_{\mathscr{A}_{\tilde{\gamma}}(1)}(\mathrm{g}_1)p_{\mathscr{A}_\gamma(1)-\mathscr{A}_{\tilde{\gamma}}(1)}(\mathrm{g}_2\mathrm{g}_1^{-1})\mu_G^{\otimes 2}(d\mathrm{g}_1,d\mathrm{g}_2).
$$
More generally, for every $m\geqslant1$ elementary curves $\gamma_i,$, $i= 1, \ldots, m$, all of type $\operatorname{I}$ or $\operatorname{II}_\pm$,  with $\gamma_1\preccurlyeq\gamma_2\preccurlyeq\ldots \preccurlyeq\gamma_m$, one has
\begin{align*}
   &\mathbb{P}\left(\mathrm{g}_{\gamma_1}(1)\in d\mathrm{g}_1, \ldots, \gamma_m(1)\in d\mathrm{g}_m\right)\\
   & \quad =p_{\mathscr{A}_{1}(1)}(\mathrm{g}_1)p_{\mathscr{A}_2(1)-\mathscr{A}_1(1)}(\mathrm{g}_2\mathrm{g}_1^{-1}) \ldots p_{\mathscr{A}_{m}(1)-\mathscr{A}_{m-1}(1)}(\mathrm{g}_m\mathrm{g}_{m-1}^{-1}) \mu_G^{\otimes m}(d\mathrm{g}_1,\ldots,d\mathrm{g}_m),
\end{align*}
where $\mathscr{A}_i(1):=\mathscr{A}_{\gamma_i}(1)$
\end{corollary}
\begin{proof}
 We give the proof for $m=2$, the general case follows similarly by induction. 
 
 \begin{align*}
     \mathbb{E}\left(\Psi_1(\mathrm{g}_{\gamma_1}(1))\Psi_2(\mathrm{g}_{\gamma_2}(1))\right) &= \mathbb{E}\left(      \mathbb{E}\left( \Psi_1(\mathrm{g}_{{\gamma}_1}(1))\Psi_2(\mathrm{g}_{\gamma_2}(1))|  \widetilde{\mathcal{F}}_{\gamma_1}(1)\right)\right)\\
     %&=\mathbb{E}\left(      \mathbb{E}\left( \Psi_1(\mathrm{g}_{{\gamma}_1}(1))\Psi_2(\mathrm{g}_{\gamma_2}(1))|  \widetilde{\mathcal{F}}_{\gamma_1}(1)\right)\right)\\
     &= \mathbb{E}\left(  \Psi_1(\mathrm{g}_{{\gamma}_1}(1))    \mathbb{E}\left( \Psi_2(\mathrm{g}_{\gamma_2}(1))|  \widetilde{\mathcal{F}}_{\gamma_1}(1)\right)\right)\\
     &=\mathbb{E}\left(  \Psi_1(\mathrm{g}_{{\gamma}_1}(1))e^{\left(\mathscr{A}_{\gamma_2}(1)-\mathscr{A}_{{\gamma_1}}(1)\right)\frac{\Delta_G}{2}}\Psi_2\left(\mathrm{g}_{{\gamma_1}}(1)\right)\right)\\
     &=e^{\frac{\mathscr{A}_{{\gamma_1}}(1)\Delta_G}{2}}\left(\Psi_1e^{\frac{\left(\mathscr{A}_{\gamma_2}(1)-\mathscr{A}_{{\gamma_1}}(1)\right)\Delta_G}{2}}(\Psi_2)\right)(\operatorname{Id})
 \end{align*}
 where we used Theorem \ref{t:abelianization-form0} in the third equality.
 
\end{proof}

We also record the following elementary consequence on these Gaussian random holonomies:
\begin{corollary}\label{c:markov-holonomy-elementary2}
  Let $\gamma,\tilde{\gamma}$ be two elementary curves which are both of type $\operatorname{I}$ or $\operatorname{II}_\pm$ such that $\tilde{\gamma}\preccurlyeq\gamma$.
Then, for any bounded and measurable function $\Psi$ on $G$, one has
$$
\mathbb{E}\left(\Psi(\mathrm{g}_{\gamma}(1))|\widetilde{\mathcal{F}}_{\tilde{\gamma}}(1)\right)=\mathbb{E}\left(\Psi(\mathrm{g}_{\gamma}(1))|\mathrm{g}_{\tilde{\gamma}}(1)\right).
$$
\end{corollary}
Again, we recall that $\widetilde{\mathcal{F}}_{\tilde{\gamma}}(1)$ is the $\sigma$-algebra from Lemma~\ref{l:inclusion-sigma-algebra}, i.e. the one containing all the information of the white noise below the curve $\tilde{\gamma}$, i.e. inside the triangle $\blacktriangle(\tilde{\gamma})$. Thanks to Lemma~\ref{r:whitenoise-multiplication}, this is exactly the $\sigma$-algebra generated by $\xi_{[\blacktriangle(\tilde{\gamma})]}$.

\begin{proof} It follows from~\eqref{e:conditional-expectation-largest-sigma-algebra} that
\begin{equation}\label{e:conditional-expectation-largest-sigma-algebra}
\mathbb{E}\left(\Psi(\mathrm{g}_{\gamma}(1))|\widetilde{\mathcal{F}}_{\tilde{\gamma}}(1)\right)=e^{\left(\mathscr{A}_\gamma(1)-\mathscr{A}_{\tilde{\gamma}}(1)\right)\frac{\Delta_G}{2}}\Psi\left(\mathrm{g}_{\tilde{\gamma}}(1)\right). 
\end{equation}
 Hence, one has, by the usual rules for conditional expectation:
 $$
 \mathbb{E}\left(\Psi(\mathrm{g}_{\gamma}(1))|\mathrm{g}_{\tilde{\gamma}}(1)\right)=\mathbb{E}\left(\mathbb{E}\left(\Psi(\mathrm{g}_{\gamma}(1))|\widetilde{\mathcal{F}}_{\tilde{\gamma}}(1)\right)|\mathrm{g}_{\tilde{\gamma}}(1)\right)=e^{\left(\mathscr{A}_\gamma(1)-\mathscr{A}_{\tilde{\gamma}}(1)\right)\frac{\Delta_G}{2}}\Psi\left(\mathrm{g}_{\tilde{\gamma}}(1)\right),
 $$
 where one uses~\eqref{e:conditional-expectation-largest-sigma-algebra} together with the fact that $e^{\left(\mathscr{A}_\gamma(1)-\mathscr{A}_{\tilde{\gamma}}(1)\right)\frac{\Delta_G}{2}}\Psi\left(\mathrm{g}_{\tilde{\gamma}}(1)\right)$ is $\mathrm{g}_{\tilde{\gamma}}(1)$-measurable to write the last equality. 
\end{proof}

%%%%%%%%%%%%%%%%%%%%%%%%%%%%%%%%%%%%%%%
%%%%%%%%%%%%%%%%%%%%%%%%%%%%%%%%%%%%%%%

\section{Definition of the Yang--Mills measure by conditioning}
\label{s:definition}

We will now gather the results of the previous sections to define the Yang-Mills measure on spaces of connections but also on a large set of observables. In that manner, we will achieve the main goals of this article and prove Theorems~\ref{t:def-YM-general} and~\ref{t:def-YM} from the introduction. This is organized in several steps. First, we define the free boundary Yang-Mills measure where the contribution of the white noise (encoding the curvature) and of the unstable manifolds $([W^u(a)])_{\text{ind}(a)=1}$ (encoding the topology of $\Sigma$) are independent. This is the content of \S\ref{ss:free-boundary-YM}. In order to define a proper Yang-Mills measure capturing the topological properties of the surface, we need to couple the information from the white noise with the information coming from the $1$-dimensional unstable manifolds. To do that, we introduce random holonomies near the maximum and gather some of their properties in \S\ref{sss:holonomy_convention}. Once this is settled, we discuss in~\S\ref{ss:quantum-observables} the quantum observables we aim at integrating against the Yang--Mills measure. Then, we define the Yang-Mills measure for large set of observables by conditioning the free boundary measure with the requirement that holonomies near the maximum are roughly equal to the identity. This is done in~\S\ref{ss:def-YM-measure} which contains the main Theorem of this section, namely Theorem~\ref{t:existence-YM-measure_Viet}. Gathered with classical tools from measure theory and functional analysis~\cite{GelfandVilenkin4,KoralovSinai}, this theorem implies Theorems~\ref{t:def-YM-general} and~\ref{t:def-YM} except for the law of random holonomies that will be described in~\S\ref{ss:random-holonomies-YM}. More precisely, while Theorem~\ref{t:def-YM-general} is a direct consequence of Theorem~\ref{t:existence-YM-measure_Viet} (up the laws of random holonomies), the proof of Theorem~\ref{t:def-YM} requires more work which is done in~\S\ref{sss:proof-theo-YM}.

\subsection{The free boundary Yang--Mills measure}\label{ss:free-boundary-YM}
In order to define the Yang-Mills measure, we start by defining the \emph{free boundary Yang--Mills measure} where we do not put any coupling between the components of the connection $A$ coming from the white noise $\xi$ and from the one coming from the unstable components $([W^u(a)])_{\text{ind}(a)=1}$. 
\begin{definition} The free boundary Yang-Mills measure is defined as
$$
\int_{\Omega\times G^{2g}} \Phi(\omega,\mathrm{b}) \mathbb{P}_{\operatorname{YM}}^{{\rm free}}(d\omega,d\mathrm{b}) :=\int_{\Omega\times G^{2g}}\Phi\left(\omega,(\mathrm{b}_a)_{\operatorname{ind}(a)=1})\right)\mathbb{P}(d\omega)\mu_G^{\otimes 2g}(d\mathrm{b}),
$$
where $\mathbb{P}$ is the white noise probability measure, $\Omega$ is the probability space of the white noise, and $\Phi:\Omega\times G^{2g}\rightarrow\R$ is an integrable function with respect to the $\sigma$-algebra induced by the product of the one on $\Omega$ and the Borel sets of $G^{2g}$.
\end{definition}
This generalizes Definition~\ref{def:free_YM} which only defined a measure on connections while we allow here a larger set of observables. The goal of this section is to proceed to appropriate conditioning of this measure to define the Yang--Mills measure on a large set of observables including for instance random connections and random holonomies. In view of Theorems~\ref{t:normalform} and~\ref{t:random-cohomological}, we set
\begin{equation}\label{e:random-connection}
A(\xi,(\mathrm{b}_a)_{\text{ind}(a)=1})= \mathcal{L}_V^{-1}\left(\xi\iota_V(\upsilon)\right)+\sum_{\text{ind}(a)=1}\log \left(\mathrm{b}_a \right) [W^u(a)].
\end{equation}
Recall that the logarithm map is not a priori continuous but it is bounded and measurable with respect to the Borel $\sigma$-algebra on $G$. With the conventions of Definition~\ref{def:free_YM}, one has
$$
\int_{\text{ker}(\iota_V)} \Phi(A) \mu_{\operatorname{YM}}^{{\rm free}}(dA)=
\int_{\Omega\times G^{2g}} \Phi(A(\xi,\mathrm{b})) \mathbb{P}_{\operatorname{YM}}^{{\rm free}}(d\xi,d\mathrm{b}).
$$
For later use, we distinguish the Gaussian component   
\begin{equation}\label{e:gaussian-component-connection}
A_{\mathcal{N}}(\xi):=\mathcal{L}_V^{-1}\left(\xi\iota_V(\upsilon)\right),
\end{equation}
carrying the curvature of our random connection and the unstable one
\begin{equation}\label{e:unstable-component-connection}
A_u(\mathrm{b}):=\sum_{\text{ind}(a)=1}\log \left(\mathrm{b}_a \right) [W^u(a)],
\end{equation}
carrying the topology of $\Sigma$.

\subsection{Holonomies near the maximum of \texorpdfstring{$f$}{f}}\label{sss:holonomy_convention}

In order to define the holonomies needed for our conditioning, we work with the Morse coordinates $(x_1,x_2)=(r\cos (2\pi t),r\sin(2\pi t))$ near $a_{2g+2}=\text{argmax}(f)$ as in paragraph~\ref{ss:blowup-Wiener-max}. 

\begin{rmk}
In the following, the curve $\gamma_r(t):=(r\cos (2\pi t),r\sin(2\pi t))$ is oriented in such a way that $\mathscr{A}_{\gamma_r}(t):=\int_{\Sigma}[\blacktriangle(\gamma_r[0,t])]\upsilon$ is increasing.
\end{rmk}

\subsubsection{Topological preliminaries}
Recall that there exist consecutive ``angles'' $0\leq t_1<t_2<\ldots<t_{4g-1}<t_{4g}\leq 1$ such that the piece of curve $r\mapsto (r\cos (2\pi t_j),r\sin(2\pi t_j))$ is a piece of unstable manifolds corresponding to a critical point $a$ of index $1$. We denote by $a(t_j)$ the corresponding critical point and we observe that each critical point of index $1$ appears exactly twice. Equivalently, for every critical point $a$ of index $1$, there exist $1\leq j<k\leq 4g$ such that $a=a(t_j)=a(t_k)$. This allows us to define the following permutation of $\mathbb{Z}/4g\mathbb{Z}$:
$$
\varrho_f:j\in\mathbb{Z}/4g\mathbb{Z}\mapsto k+1\in \mathbb{Z}/4g\mathbb{Z},
$$
where $k$ is the unique integer such that $a(t_j)=a(t_k)$. The following holds:
\begin{lemma}\label{l:combinatorics-unstable} With the above conventions, the permutation $\varrho_f$ consists of one cycle of length $4g$.
\end{lemma}
\begin{proof} This topological argument was indicated to us by Baptiste Chantraine and St\'ephane Guillermou. In the following, we suppose that $g\geqslant 1$. We consider the Morse chart around $a_{2g+2}$ and the small disk $D_{r_0}:=\{x_1^2+x_2^2\leq r_0^2\}$ centered at $a_{2g+2}$. It follows from~\cite[Th.~3.1, Th.~3.2]{Milnor} that, for $\varepsilon>0$ small enough, $\{f>f(a_{2g+1})-\varepsilon\}$ has the same homotopy type as a small strip attached to the disk $D_{r_0}$. In fact, the proof from this reference shows that the retraction on $D_{r_0}$ to which we have attached a strip can be chosen to be a retraction on $D_{r_0}\cup W^u(a_{2g+1},\delta)$, where $W^u(a_{2g+1},\delta)$ is a $\delta$-neighborhood of $W^u(a_{2g+1})$. By the same argument, $\{f>f(a_{2g})-\varepsilon\}$ has the same homotopy type as a small strip that has been attached to $\{f>f(a_{2g+1})-\varepsilon\}$ and thus to a small strip attached to $D_{r_0}\cup W^u(a_{2g+1},\delta)$. Again, the construction shows that $\{f>f(a_{2g})-\varepsilon\}$ retracts to  $D_{r_0}\cup W^u(a_{2g+1},\delta)\cup W^u(a_{2g},\delta)$. By induction, one finds that $\{f>f(a_{2g})-\varepsilon\}$ retracts to a small neighborhood of $D_{r_0}\bigcup\cup_{\text{ind}(a)=1}W^u(a,\delta)$. Hence, a last application of~\cite[Th.~3.1]{Milnor} shows that $\{f>f(a_{1})+\varepsilon\}$ has the same homotopy type as a small neighborhood of $D_{r_0}\bigcup\cup_{\text{ind}(a)=1}W^u(a,\delta)$. This means that this manifold with boundary has a single boundary component which is diffeomorphic to a circle. Without loss of generality, these attached strips correspond to small disjoint neighborhoods $(t_j-\delta,t_j+\delta)$ of the points $t_j$ inside $\{x_1^2+x_2^2=r_0^2\}$. Then starting from an arbitrary point on the boundary component, one will encounter successively the points $t_{j}-\delta$. See Figure~\ref{fig:Ribbon}. The order in which these points appear corresponds exactly to cycle defining $\varrho_f$.
\begin{figure}
    \centering
    \includegraphics[scale=0.2]{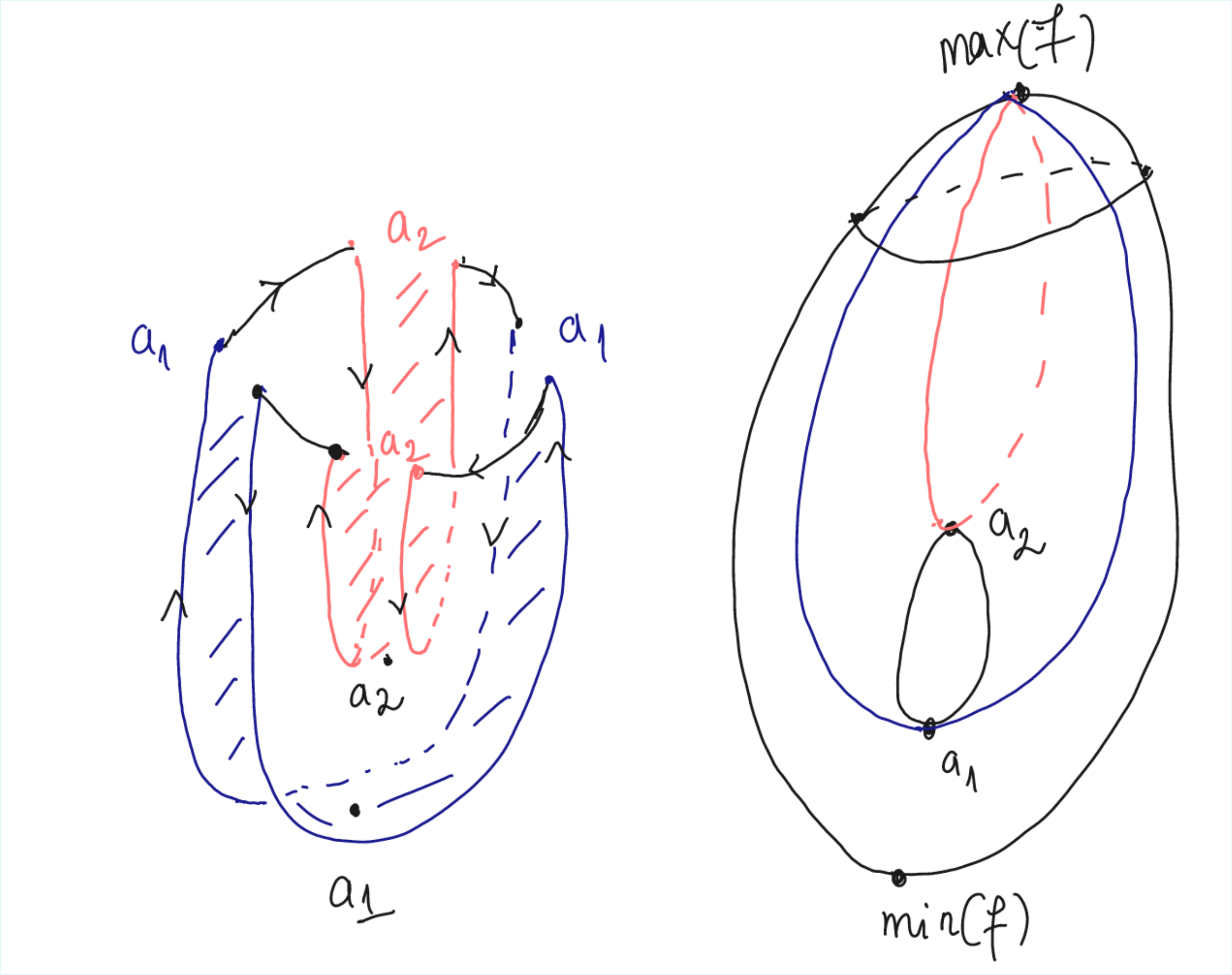} 
    \caption{Example for Ribbon graph}
    \label{fig:Ribbon}
\end{figure}
\end{proof}

\begin{rmk}
We give a more dynamical interpretation of the above proof and we hope this will help the reader get another visualization of $\varrho_f$. Note that $D_{r_0}\bigcup\cup_{\operatorname{ind}(a)=1}W^u(a,\delta)$ is a small neighborhood
of the union $\cup_{\operatorname{ind}(a)=1}W^u(a)$ of unstable curves. This neighborhood can be viewed as a Ribbon graph and can be retracted on some bouquet of $2g$ circles which are the $2g$--unstable curves connected at $\argmax(f)$. Observe that any point $x\in\Sigma\setminus \overline{\cup_{\operatorname{ind}(a)=1}W^u(a)}$ is in the unstable manifold $W^u(a_1)$ of $a_1=\operatorname{argmin}(f)$ and therefore our surface $\Sigma$ is diffeomorphic to some disc $\mathbb{D}$ which is glued with the Ribbon graph  $D_{r_0}\bigcup\cup_{\operatorname{ind}(a)=1}W^u(a,\delta)$ along its boundary which is therefore diffeomorphic to some circle $\partial\mathbb{D}\simeq \mathbb{S}^1$. 
\end{rmk}

\begin{rmk} The set $\Sigma_{r_0}^u:=D_{r_0}\bigcup\cup_{\operatorname{ind}(a)=1}W^u(a)$ can be deformed through a retraction to the wedge of $2g$-circles given by $\cup_{\operatorname{ind}(a)=1}\overline{W^u(a)}$. According to~\cite[Prop.1.17]{Hatcher}, $\Sigma_{r_0}^u$ (and thus $\Sigma$ minus a small neighborhood of $a_{1}$) has the same fundamental group of the wedge of $2g$-circles. Thanks to Van Kampen Theorem~\cite[Th.~1.20]{Hatcher}, this fundamental group is thus the free group $\mathbb{F}_{2g}:=*_{\operatorname{ind}(a)=1}\mathbb{Z}$. In order to get the fundamental group of $\Sigma$, one uses one more time Van Kampen Theorem and the fact that the fundamental group of a disk is trivial to get that $\pi_1(\Sigma)$ is the free group $\mathbb{F}_{2g}=\langle a_2,a_3,\ldots, a_{2g+1}\rangle$ that we quotient by the relation
\begin{equation}\label{e:presentation}
a(t_{\varrho_f^{4g}(1)})^{\varepsilon(\varrho_f^{4g}(1))}\ldots a(t_{\varrho_f(1)})^{\varepsilon(\varrho_f(1))}a(t_{1})^{\varepsilon(1)}=\text{e},
\end{equation}
where, for each critical point $a=a(t_j)=a(t_k)$, $j\neq k$, one has $\varepsilon(j)+\varepsilon(k)=0$ and $\varepsilon(j)\in\{\pm 1\}$ for all $1\leq j\leq 4g$. In other words, $\varrho_f$ is the map associated with the presentation of the fundamental group induced by critical points of index $1$.
\end{rmk}

\begin{rmk}
 Let us now give another interpretation of the relation~\eqref{e:presentation} in terms of intersection with stable curves. The boundary 
of $D_{r_0}\bigcup\cup_{\operatorname{ind}(a)=1}W^u(a,\delta)$ is a simple closed curve and it satisfies the following properties:
\begin{itemize}
\item it is homotopic to a small closed curve around the minimum $a_1$ just by considering its image under the gradient flow in backward time,
\item this image is transverse to the union $\cup_{\operatorname{ind}(a)=1}W^s(a)$ with $4g$ intersection points. 
\end{itemize}
Then~\eqref{e:presentation} describes exactly the $4g$ intersections of a small circle around $a_1$ with the $2g$ stable curves $W^s(a)$ for $a\in \text{Crit}(f)_1$ with the intersection denoted by either $a_j$ or $a_j^{-1}$ depending on the orientation at the intersection.
\end{rmk}

\subsubsection{Definition of random holonomies near the maximum}

Recall also from~\S\ref{ss:blowup-Wiener-max} that we defined curves $\gamma_r(t)=(r\cos(2\pi t),r\sin(2\pi t))$, $t\in[0,1]$, surrounding $a_{2g+2}$. By construction, these curves are primitive and have no type $\text{III}$ components and we showed how to associate to them a Wiener process. Indeed, using Theorem~\ref{t:local-holonomies}, we set 
$$
\mathrm{g}_r(t)=\mathrm{g}_{\gamma_r}(t),\quad t\in[0,1],
$$ 
to be the holonomy of the Gaussian part $\mathcal{L}_V^{-1}(\xi\iota_V(\upsilon))$ of our connection along the curve $\gamma_r$. Similarly, we have defined $\mathrm{g}_0(t)$ in that same theorem when the curve is formally reduced to the point $a_{2g+2}=\text{argmax}(f)$.

Hence, given $A=A(\omega,(\mathrm{b}_a)_{\text{ind}(a)=1})$ and $r\geq 0$ small enough (to use the Morse chart), we now define the \textbf{random holonomy along $\gamma_r$ of the connection} $A=A(\xi,(\mathrm{b}_a)_{\text{ind}(a)=1})$:
\begin{equation}\label{e:holonomy-max}
\textbf{Hol}_{r}:=\mathrm{g}_r(1)\mathrm{g}_r(t_{4g})^{-1}\mathrm{b}_{a(t_{4g})}^{\varepsilon_{4g}}\mathrm{g}_r(t_{4g})\mathrm{g}_r(t_{4g-1})^{-1}\ldots\mathrm{b}_{a(t_{2})}^{\varepsilon_{2}}\mathrm{g}_r(t_{2})\mathrm{g}_r(t_{1})^{-1}\mathrm{b}_{a(t_{1})}^{\varepsilon_{1}}\mathrm{g}_r(t_{1}),
\end{equation}
where, for $j\neq k$ such that $a(t_j)=a(t_k)$, one has $\varepsilon_j+\varepsilon_k=0$ (with $\varepsilon_k$ independent of $r\geq 0$). More precisely, $\varepsilon_k=1$ if the flow line from $\gamma_r(t_k)$ (with $r>0$) to the point $a_{2g+2}$ has the same orientation as $W^u(a)$ and $\varepsilon_k=-1$ otherwise. See Figure \ref{fig:word} for an example.   
\begin{figure}
    \centering
    \includegraphics[scale=0.17]{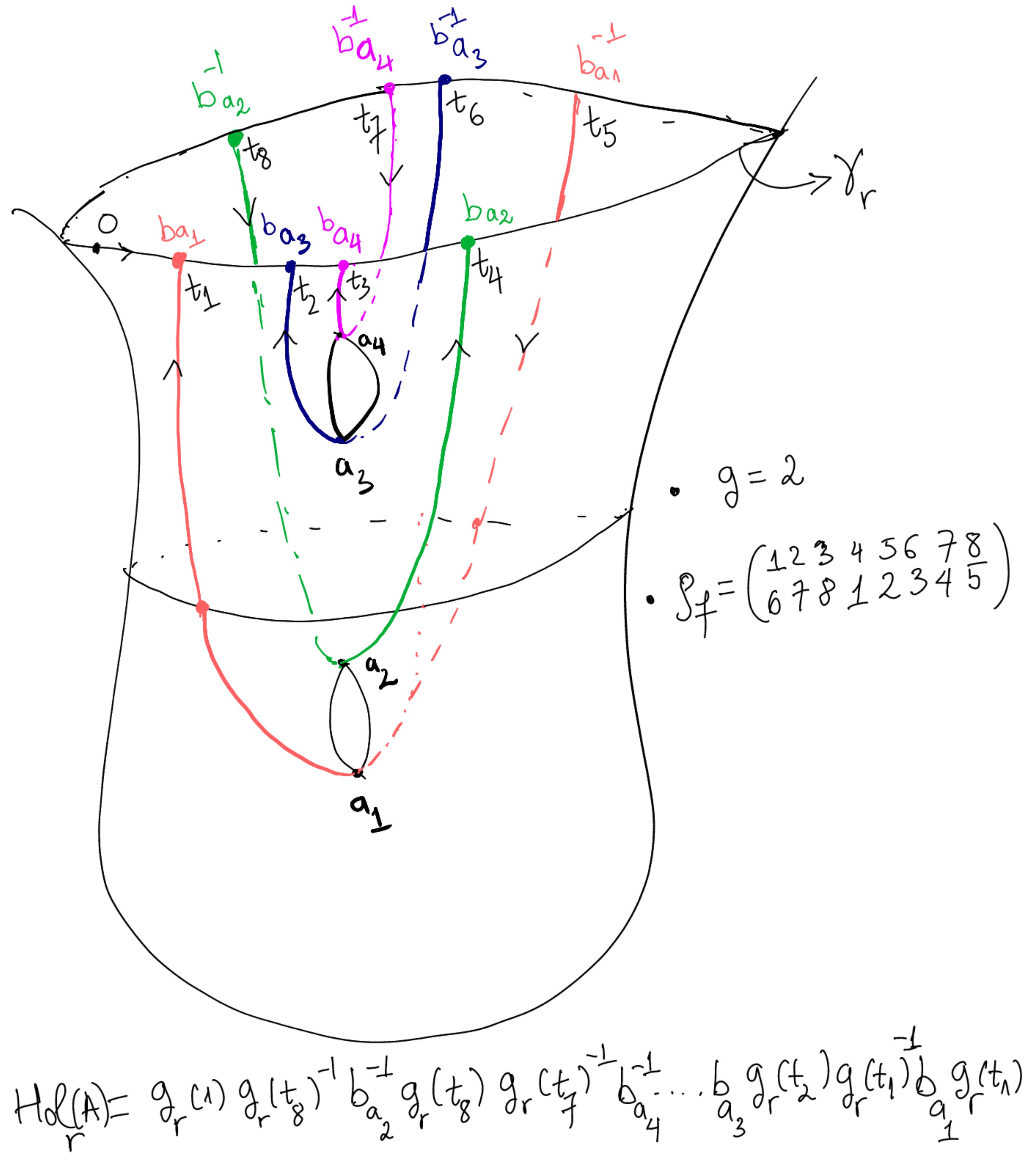} 
    \caption{Example for an order of $4g$ intersection points}
    \label{fig:word}
\end{figure}
Recall that the case $r=0$ corresponds formally to the case where the curve is reduced to a point as in~\S\ref{ss:blowup-Wiener-max}. Equivalently, if we denote by $\mathcal{S}$ the bordered surface 
obtained by blowing up the initial closed surface $\Sigma$ at $a_{2g+2}$, the surface $\mathcal{S}$ has one boundary component $\partial\mathcal{S}$ corresponding to the case $r=0$ introduced in~\S\ref{ss:blowup-Wiener-max}. In the following, we will sometimes make the small abuse of notations that consists in writing $\partial\mathcal{S}=\gamma_0$.

\begin{rmk}
Even if we define this quantity to be the random holonomy associated with $A$ along the small curve $\gamma_r$ surrounding the maximal value of $f$, we emphasize that it is not strictly speaking a function of $A$. It is rather a function on $\Omega\times G^{2g}$, thus a random variable.
\end{rmk}
\begin{rmk} In the case of $\mathbb{S}^2$ or in the case where $G$ is Abelian, one has only
 $$
\operatorname{\mathbf{Hol}}_{r}:=\mathrm{g}_r(1,\omega).
 $$
\end{rmk}

This holonomy 
can in fact be obtained by solving a stochastic parallel transport of the random connection $A$ along $\gamma_r$ (for $r>0$) or along the boundary of the blow-up surface $\partial\mathcal{S}$ (for $r=0$). Indeed, for $r\geq 0$ small enough, we can set
\begin{equation}\label{e:random-process-jump}
\mathrm{h}_r(t):=\left\{ \begin{array}{l}\mathrm{g}_r(t)\ \text{if}\ t\in[0,t_1)\\
                          \mathrm{g}_r(t)\mathrm{g}_r(t_{j})^{-1}\mathrm{b}_{a(t_{j})}^{\varepsilon_{j}}\mathrm{g}_r(t_{j})\mathrm{g}_r(t_{j-1})^{-1}\ldots\mathrm{b}_{a(t_{2})}^{\varepsilon_{2}}\mathrm{g}_r(t_{2})\mathrm{g}_r(t_{1})^{-1}\mathrm{b}_{a(t_{1})}^{\varepsilon_{1}}\mathrm{g}_r(t_{1})\ \text{if}\ t\in[t_{j},t_{j+1})\\
                          \mathrm{g}_r(t)\mathrm{g}_r(t_{4g})^{-1}\mathrm{b}_{a(t_{4g})}^{\varepsilon_{j}}\mathrm{g}_r(t_{4g})\mathrm{g}_r(t_{4g-1})^{-1}\ldots\mathrm{b}_{a(t_{2})}^{\varepsilon_{2}}\mathrm{g}_r(t_{2})\mathrm{g}_r(t_{1})^{-1}\mathrm{b}_{a(t_{1})}^{\varepsilon_{1}}\mathrm{g}_r(t_{1})\ \text{if}\ t\in[t_{4g},1].
                         \end{array}
\right.
\end{equation}
This is a $G$-valued random process which is continuous in time except at finitely many points $t_1<t_2<\ldots<t_{4g}$ in $(0,1)$, where one has a jump corresponding to the intersection with a one-dimensional unstable manifold. 

\begin{figure}
    \centering
    \includegraphics[scale=0.35]{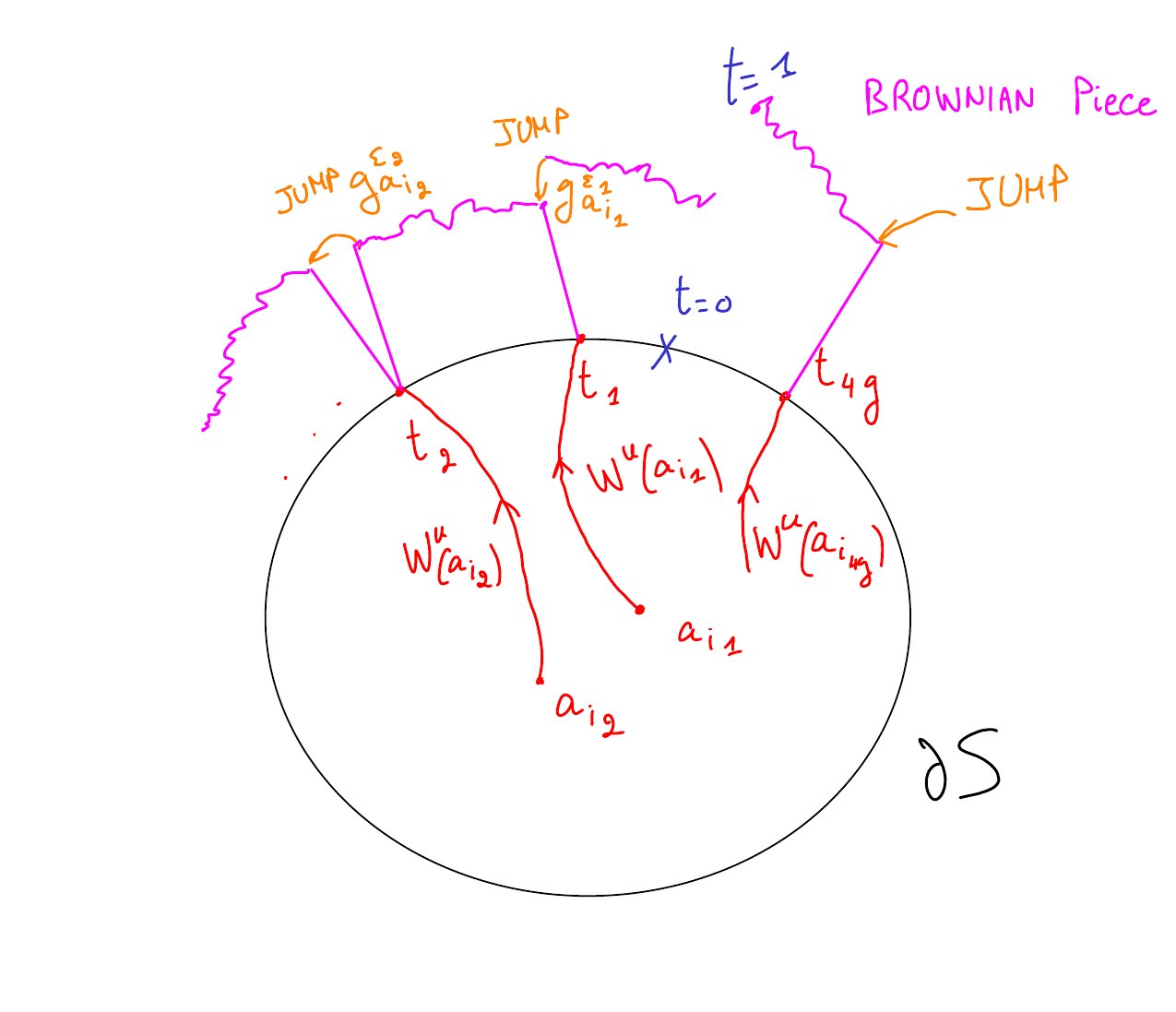} 
    \caption{Picture of the holonomy process $\mathbf{Hol}_{r } $ }
    \label{fig:holonomy_boundarySigmatilde}
\end{figure} 

We give in Figure~\ref{fig:holonomy_boundarySigmatilde} an illustration of the holonomy process at the closed curve $\gamma_r$ or $\partial\mathcal{S}$ where we represent pieces of Brownian evolution in pink and the jumps 
that come from contributions of group elements indexed by the unstable curves in orange.

Let us explain why this quantity can indeed be understood as the random holonomy associated with $A(\xi,\mathrm{b})$. Strictly speaking, this discussion is not necessary to define the Yang--Mills measure but it is worth recording that these holonomies arise as solutions to stochastic differential equations associated with $A$. More precisely, we can formally write
\begin{equation}
\label{eq:holgammaout}
d\mathrm{h}_r= \mathrm{h}_r\circ d\widetilde{\mathfrak{W}}_r,\quad \widetilde{\mathfrak{W}}_r(t)=-\int_0^t \gamma_{r}^*A \ , \quad \mathrm{h}_r(t=0)=\text{Id}_G,   
\end{equation}
where one more time the case $r=0$ corresponds to the blow-up case from~\S\ref{ss:blowup-Wiener-max}.
A first observation is that the expression~\eqref{eq:holgammaout} is formal and that it can be decomposed formally as a sum 
$$
\widetilde{\mathfrak{W}}_r(t)
        = -\int_0^t \gamma_{r}^*A_{\mathcal{N}} - \int_0^t \gamma_{r}^*A_u \ ,
$$
with $A_\mathcal{N}$ being defined in~\eqref{e:gaussian-component-connection} and $A_u$ in~\eqref{e:unstable-component-connection}.
On the one hand, the random process $\mathfrak{W}_{\gamma_r}(t)=\int_0^t \gamma_{r}^*A_{\mathcal{N}}$ was already rigorously defined in~\S\ref{ss:g-valued-BM} for $r>0$ and in~\S\ref{ss:blowup-Wiener-max} for $r=0$. We explained there that it yields a reparametrized $\mathfrak{g}$--valued Brownian motion and, in~\S\ref{s:randomholonomy}, $\mathrm{g}_r(t)$ was defined to be the solution to the corresponding stochastic differential equation. On the other hand, $\int_0^t \gamma_{r}^*A_{u}$ is a finite sum of $\mathfrak{g}$--valued Heaviside functions. Therefore, the driver $\widetilde{\mathfrak{W}}_r(t)$ is a semi-martingale with a reparametrized Brownian for the martingale part and pure jumps for the finite variation part~\cite[\S IV]{RevuzYor} with respect to the filtration $\mathcal{F}_{\gamma_r}^G(t)$ which lifts the filtration $\mathcal{F}_{\gamma_r}(t)$ (appearing in Lemma~\ref{l:inclusion-sigma-algebra}) to $\Omega\times G^{2g}$. Itô Lemma
for (non continuous) semimartingales still holds in the present setting~\cite[Chapter 2 \S7]{Protter}.
\begin{rmk} Here, for the $G^{2g}$ component, we take the Borel sets on $G^{2g}$. In the case $r=0$, this corresponds again to the blow-up curve from~\S\ref{ss:blowup-Wiener-max}.
\end{rmk}

Equivalently, the stochastic differential equation takes the Marcus canonical form as in~\cite[Marcus canonical integral p.~272]{ApplebaumLevy},\cite[eq (2.1) p.~5]{Albeverio}, in the setting with deterministic jumps as in ~\cite{BechererSun}:
\begin{align}
\label{eq:holgammaout_2}
d\mathrm{h}_r & = \mathrm{h}_r \circ\left(  d\mathfrak{W}_{\gamma_r} 
   +\left(\sum_{j=1}^{4g} \varepsilon_j\log(\mathrm{b}_{a(t_j)}) d \mathds{1}_{\{ t_j \leq t \}}   \right) 
   \right)  \ , \ \mathrm{h}_r(t=0) = \text{Id}_G \ .
\end{align}
Thanks to \cite[Th. 2.3 p.~6]{Albeverio} this stochastic equation has a unique strong solution which is càd-làg (``continue à droite, limitée à gauche''). This is exactly the random process defined in~\eqref{e:random-process-jump}. Indeed, the jumps occur at exactly the times $(t_j)_{j=1,\ldots,4g}$ and the jump values are given by
\begin{align}
\label{eq:jump_value}
\mathrm{h}_r(t_j) \mathrm{h}_r(t_j^-)^{-1} & = \mathrm{b}_{a(t_j)}^{\varepsilon_j} \ .
\end{align}
Moreover, recalling that $\mathrm{g}_r(t)$ solves~\eqref{e:SDE-holonomy-Strato} with $\gamma=\gamma_r$ for $r>0$ and $\gamma$ replaced by the blow-up curve of \S\ref{ss:blowup-Wiener-max} for $r=0$, one has that $\mathrm{h}_r$ and $\mathrm{g}_r$ solve the same stochastic differential equation between the jump times $(t_j)_{1\leq j\leq 4g}$. One has also
\begin{align}
\label{eq:U_to_UNc}
\forall s<t \in [t_j, t_{j+1}), \ 
\mathrm{h}_r(t) \mathrm{h}_r(s)^{-1} = \mathrm{g}_r(t) \mathrm{g}_r(s)^{-1} \ .
\end{align}
In summary, $\mathrm{h}_r$ follows the continuous flow $\mathrm{g}_r$ unless it is  interrupted every time the curve $\gamma_r$ crosses an unstable curve. The SDE contribution of the singular part $\sum_{\ind(a)=1} \log(\mathrm{b}_a) [W^u(a)] $ yields the multiplication by a group element from the group elements $(\mathrm{b}_a)_{\ind(a)=1}$.
\begin{lemma}[Brownian motion with jumps]
\label{lemma:BM_with_jumps}
The stochastic differential equation~\eqref{eq:holgammaout_2} is solved for $t_j \leq t < t_{j+1}$ by $\mathrm{h}_r(t)$ as defined in~\eqref{e:random-process-jump}.
In particular, $\operatorname{\mathbf{Hol}}_r$ is $\mathcal{F}_{\gamma_r}^G(1)$-measurable. Furthermore, one has $\operatorname{\mathbf{Hol}}_r \eqlaw \mathrm{b}_{a(t_{4g})}^{\varepsilon_{4g}}\ldots\mathrm{b}_{a(t_{2})}^{\varepsilon_{2}}\mathrm{b}_{a(t_{1})}^{\varepsilon_{1}}\mathrm{g}_r(1)$. 
\end{lemma}
In Section~\ref{ss:random-holonomies-YM}, we will define $\mathbf{Hol}(\gamma)$ for more general curves and, through the exact same discussion, one could verify that they satisfy similar stochastic differential equations with deterministic jumps.
\begin{proof}
The above discussion proves the form of the solution explicitly (and thus the first part of the Lemma). Regarding the final identification in law, it is an algebraic rearrangement which we now detail. Because the increments of $\mathrm{g}_r$ are conjugation invariant (see Theorem~\ref{t:local-holonomies}), we have the equality in law for $t_j \leq t < t_{j+1}$
\begin{align*}
\mathrm{h}_r(t) 
& = \mathrm{g}_r(t) \left( \mathrm{g}_r(t_j) \right)^{-1}
    \mathrm{b}_{a(t_j)}^{\varepsilon_j} \ 
    \mathrm{g}_r(t_j) \left( \mathrm{g}_r(t_{j-1}) \right)^{-1}
    \dots \ \mathrm{b}_{a(t_{2})}^{\varepsilon_{2}}\mathrm{g}_r(t_2) \left( \mathrm{g}_r(t_1) \right)^{-1}
    \mathrm{b}_{a(t_1)}^{\varepsilon_1} \mathrm{g}_r(t_1)  \\
& = \mathrm{b}_{a(t_j)}^{\varepsilon_j} \ 
    \mathrm{b}_{a(t_j)}^{-\varepsilon_j} \mathrm{g}_r(t) \left( \mathrm{g}_r(t_j) \right)^{-1}
    \mathrm{b}_{a(t_j)}^{\varepsilon_j} \ 
    \mathrm{g}_r(t_j) \left( \mathrm{g}_r(t_{j-1}) \right)^{-1}
    \dots \ \mathrm{b}_{a(t_{2})}^{\varepsilon_{2}}\mathrm{g}_r(t_2) \left( \mathrm{g}_r(t_1) \right)^{-1}
    \mathrm{b}_{a(t_1)}^{\varepsilon_1} \mathrm{g}_r(t_1)  \\
& \eqlaw \mathrm{b}_{a(t_j)}^{\varepsilon_j} \ 
     \mathrm{g}_r(t) \left( \mathrm{g}_r(t_j) \right)^{-1}
     \ 
    \mathrm{g}_r(t_j) \left( \mathrm{g}_r(t_{j-1}) \right)^{-1}
    \dots \ \mathrm{b}_{a(t_{2})}^{\varepsilon_{2}}\mathrm{g}_r(t_2) \left( \mathrm{g}_r(t_1) \right)^{-1}
    \mathrm{b}_{a(t_1)}^{\varepsilon_1} \mathrm{g}_r(t_1) \\
& =  \mathrm{b}_{a(t_j)}^{\varepsilon_j} \ 
     \mathrm{g}_r(t)  \left( \mathrm{g}_r(t_{j-1}) \right)^{-1}
    \dots \ \mathrm{b}_{a(t_{2})}^{\varepsilon_{2}}\mathrm{g}_r(t_2) \left( \mathrm{g}_r(t_1) \right)^{-1}
    \mathrm{b}_{a(t_1)}^{\varepsilon_1} \mathrm{g}_r(t_1) \ .
\end{align*}
Here, we also used the independence property in Theorem~\ref{t:local-holonomies} to write the equality in law on the third line. By induction, we have
$$
\mathrm{h}_r(t) \eqlaw \mathrm{b}_{a(t_{j})}^{\varepsilon_{j}}\ldots\mathrm{b}_{a(t_{2})}^{\varepsilon_{2}}\mathrm{b}_{a(t_{1})}^{\varepsilon_{1}}\ \mathrm{g}_r(t) \ .
$$
The same identity in law for $t=1$ yields the result.
\end{proof}

\subsubsection{Law of random holonomies around $a_{2g+2}$.}

Finally, we would like to compute the law (with respect to the free Yang--Mills measure) of the holonomy $\textbf{Hol}_{r}$ for $r\geq 0$ small enough. Before doing that, let us show the following corollary of Lemma~\ref{l:combinatorics-unstable}:
\begin{lemma}\label{c:multi-orthogonality}
Let $\rho\in\widehat{G}$. Then, one has
$$
\forall\mathrm{h}\in G,\quad \int_{G^{2g}}\chi_\rho\left(\mathrm{b}_{a(t_{4g})}^{\varepsilon_{4g}}\ldots\mathrm{b}_{a(t_{2})}^{\varepsilon_{2}}\mathrm{b}_{a(t_{1})}^{\varepsilon_{1}}\mathrm{h}\right)d\mu_G^{\otimes 2g}((\mathrm{b}_a)_{\operatorname{ind}(a)=1})=\frac{1}{\dim(V_\rho)^{2g}}\chi_\rho(\mathrm{h}).
$$
\end{lemma}
\begin{proof}
In order to simplify this integral, we will use the orthogonality properties of irreducible representations and we first write
$$
\chi_\rho\left(\mathrm{b}_{a(t_{4g})}^{\varepsilon_{4g}}\ldots\mathrm{b}_{a(t_{2})}^{\varepsilon_{2}}\mathrm{b}_{a(t_{1})}^{\varepsilon_{1}}\mathrm{h}\right)=\sum_{k_0,k_1,\ldots,k_{4g}=1}^{\dim V_\rho}\rho(\mathrm{b}_{a(t_{4g})}^{\varepsilon_{4g}})_{k_{0},k_{4g}}\ldots\rho(\mathrm{b}_{a(t_{1})}^{\varepsilon_{1}})_{k_{2},k_{1}}\rho(\mathrm{h})_{k_{1},k_{0}}.
$$
where $\rho$ is the unitary irreducible representation corresponding to $\chi$. Recall now from \cite[Theorem 3.3]{Sepanski} that
$$ \int_{G}\rho(\mathrm{b})_{ij} \rho(\mathrm{b}^{-1})_{k\ell}d\mu_G(\mathrm{b})= \int_{G}\rho(\mathrm{b})_{ij} \overline{\rho(\mathrm{b})}_{\ell k}d\mu_G(\mathrm{b})= \dim(V_\rho)^{-1}\delta_{i\ell}\delta_{jk}.$$
Fix now some critical point $a=a(t_j)=a(t_i)$ with both $j,i\notin\{4g\}$. This orthogonality relation reads 
$$
\dim(V_\rho)\rho(\mathrm{b}_{a(t_j)}^{\varepsilon_j})_{k_{j+1},k_j}\rho(\mathrm{b}_{a(t_i)}^{\varepsilon_i})_{k_{i+1},k_i}=\delta(k_i,k_{j+1})\delta(k_j,k_{i+1})=\delta(k_i,k_{\varrho_f(i)})\delta(k_j,k_{\varrho_f(j)}),
$$
where $\delta(k,k')=1$ if $k=k'$ (and $0$ otherwise). In the case where $a=a(t_{4g})=a(t_j)$, the orthogonality relation reads
$$
\dim(V_\rho)\rho(\mathrm{b}_{a(t_j)}^{\varepsilon_j})_{k_{j+1},k_j}\rho(\mathrm{b}_{a(t_{4g})}^{\varepsilon_{4g}})_{k_{0},k_{4g}}=\delta(k_{4g},k_{j+1})\delta(k_0,k_{j})=\delta(k_{4g},k_{\varrho_f(4g)})\delta(k_0,k_{\varrho_f^{-1}(1)}).
$$
Hence, one finds
$$
\chi_\rho\left(\mathrm{b}_{a(t_{4g})}^{\varepsilon_{4g}}\ldots\mathrm{b}_{a(t_{2})}^{\varepsilon_{2}}\mathrm{b}_{a(t_{1})}^{\varepsilon_{1}}\mathrm{h}\right)=\dim(V_\rho)^{-2g}\sum_{k_1,\ldots,k_{4g}=1}^{\dim V_\rho}\rho(\mathrm{h})_{k_1,k_{\varrho_f^{-1}(1)}}\prod_{j=1}^{4g}\delta(k_j,k_{\varrho_f(j)}),
$$
from which the conclusion follows thanks to Lemma~\ref{l:combinatorics-unstable}.
\end{proof}

We are now ready to compute the law of $\textbf{Hol}_r$ for $r\geq 0$ small enough:
\begin{lemma}[Law of $\mathbf{Hol}_{r}$]
\label{lemm:lawHolcombi}
The law of $\mathbf{Hol}_{r}  \eqlaw \mathrm{b}^{\varepsilon_{4g}}_{a(t_4g)} \dots \mathrm{b}^{\varepsilon_1}_{a(t_1)} \mathrm{g}_r(t=1) $ (see Lemma~\ref{lemma:BM_with_jumps}) is given by the formula
\begin{align*}
    \P_{\operatorname{YM}}^{{\rm free}}\left( \mathbf{Hol}_{r} \in d\mathrm{h} \right)
= & \ \P_{\operatorname{YM}}^{{\rm free}}\left(\mathrm{b}^{\varepsilon_{4g}}_{a(t_4g)} \dots \mathrm{b}^{\varepsilon_1}_{a(t_1)}\mathrm{g}_r(1)  \in d \mathrm{h} \right) \\
= &  \sum_{\rho\in \widehat{G} } e^{-{\frac{c_2(\rho)}{2}\mathscr{A}_r(1) }}    
      \frac{ \chi_\rho(\mathrm{h})}{\dim(V_{\rho})^{2g-1}}\, \mu_G(d\mathrm{h})=: p_{ {\bf Hol},r}(\mathrm{h}) \mu_G(d \mathrm{h}),
\end{align*}
where $\mathscr{A}_r(1):=\mathscr{A}_{\gamma_r}(1)$ for $r>0$ and $\mathscr{A}_0(1)=\upsilon(\Sigma)$. 
\end{lemma}
Recall that $\mathscr{A}_\gamma(t)$ was defined in~\S\ref{s:randomholonomy} as the area functional (or variance) associated with the $\mathfrak{g}$-valued Brownian motions from~\S\ref{ss:g-valued-BM} (case $r>0$) and~\S\ref{ss:blowup-Wiener-max} (case $r=0$).
\begin{proof}
On the one hand, we set
\begin{align*}
   \label{eq_push_forward_Haar}
   \mu_1(d\mathrm{h}) :=
   \P_{\text{YM}}^{{\rm free}}\left( \mathrm{b}^{\varepsilon_{4g}}_{a(t_{4g})} \dots \mathrm{b}^{\varepsilon_1}_{a(t_1)} \in d\mathrm{h} \right),
\end{align*}
that is, for every bounded and measurable function $\Psi:G\rightarrow\C$,
$$
\int_{G}\Psi(\mathrm{h})\mu_1(d\mathrm{h})%:=\int_{\Omega\times G^{2g}}\Psi\left(\mathrm{b}^{\varepsilon_{4g}}_{a(t_{4g})} \dots \mathrm{b}^{\varepsilon_1}_{a(t_1)}\right)\P(d\xi)\mu_G^{\otimes 2g}(d\mathrm{b})
=\int_{ G^{2g}}\Psi\left(\mathrm{b}^{\varepsilon_{4g}}_{a(t_{4g})} \dots \mathrm{b}^{\varepsilon_1}_{a(t_1)}\right)\mu_G^{\otimes 2g}(d\mathrm{b}).
$$
On the other hand, Theorem~\ref{t:local-holonomies} implies that
\begin{align*}
  \mu_2(d \mathrm{h}) := \P_{\text{YM}}^{{\rm free}}\left( \mathrm{g}_r(1) \in d\mathrm{h} \right)
  = & \sum_{\rho\in \widehat{G} } e^{-c_2(\rho)\mathscr{A}_r(1) } \dim(V_\rho) \chi_\rho(\mathrm{h}) \mu_G(d\mathrm{h})  =:  p_{\mathscr{A}_r(1)}(\mathrm{h})\mu_G(d\mathrm{h})\ .
\end{align*}
Since  $\mathrm{b}^{\varepsilon_{4g}}_{a(t_{4g})} \dots \mathrm{b}^{\varepsilon_1}_{a(t_1)}$ and  $\mathrm{g}_r(1)$ are independent with respect to the free boundary Yang-Mills measure $\P_{\text{YM}}^{{\rm free}}$, the law of  $\mathrm{b}^{\varepsilon_{4g}}_{a(t_{4g})} \dots \mathrm{b}^{\varepsilon_1}_{a(t_1)} \mathrm{g}_r(1)$ is given, for any Borel set $B\subset G$, by
\begin{align*}
\P_{\text{YM}}^{{\rm free}} (\mathbf{Hol}_r\in B)&:= \int_G \int_{G}  {\bf 1}_{B}( \mathrm{b}\mathrm{h}) \ p_{\mathscr{A}_r(1)}(\mathrm{h}) \mu_G(d\mathrm{h}) \ \mu_1(d\mathrm{b}) \\
& = \int_G \int_{G}  {\bf 1}_{A}(\mathrm{h}) \ p_{\mathscr{A}_r(1)}(\mathrm{b}^{-1} \mathrm{h}) \  \mu_G(d\mathrm{h})\  \mu_1(d\mathrm{b})  \\
& = \int_G {\bf 1}_{A}(\mathrm{h}) \left( \int_{G}  p_{\mathscr{A}_r(1)}(\mathrm{b}^{-1} \mathrm{h}) \mu_1 (d\mathrm{b}) \right) \mu_G(d\mathrm{h}) \ .
\end{align*}
Therefore the density we are looking for is given with respect to the Haar measure by
\begin{align*}
p_{ {\bf Hol},r}(\mathrm{h}) 
= & \ \int_{G}p_{\mathscr{A}_r(1)}(\mathrm{b}^{-1} \mathrm{h})\mu_1 (d \mathrm{b}) \\
= & \ \int_{G}  \left( \sum_{\rho\in \widehat{G} } e^{-{c_2(\rho)\mathscr{A}_r(1) }} \dim(V_\rho) \chi_\rho(\mathrm{b}^{-1}\mathrm{h}) \right) \mu_1(d\mathrm{b}) \ .
\end{align*}  
Since
\begin{align*}
    \int_{G}  \chi_\rho(\mathrm{b}^{-1} \mathrm{h})  \mu_1(d\mathrm{b}) =\int_{G^{2g}} \chi_\rho \left(  (\mathrm{b}^{\varepsilon_{4g}}_{a(t_{4g})} \dots \mathrm{b}^{\varepsilon_1}_{a(t_1)})^{-1}\mathrm{h}\right) \prod_{\text{ind}(a)=1}\mu_G(d\mathrm{b}_{a}),
\end{align*}
Corollary \ref{c:multi-orthogonality} implies that 
\begin{align*}
   \int_{G}  \chi_\rho(\mathrm{b}^{-1} \mathrm{h})  \mu_1(d\mathrm{b}) =\int_{G}  \overline{\chi_\rho(\mathrm{b} \mathrm{h}^{-1})}  \mu_1(d\mathrm{b}) =\frac{\chi_\rho(\mathrm{h})}{\dim (V_\rho)^{2g}}.
\end{align*}
Therefore the density we are looking for with respect to the normalized Haar measure on $G$ is
$$ p_{ {\bf Hol},r}(\mathrm{h})=\sum_{\rho\in \widehat{G} } e^{-{\frac{c_2(\rho)}{2}\mathscr{A}_r(1)  }}    
      \frac{ \chi_\rho(\mathrm{h})}{\dim(V_{\rho})^{2g-1}} \ ,
$$
as required.
\end{proof}

\subsection{Admissible observables}\label{ss:quantum-observables}

Before stating the main result of this section, we will discuss the main observables we want to integrate on our probability space. Roughly speaking, we aim at covering (measurable) functions of the following random variables:
\begin{enumerate}
\item $\langle A,\psi\rangle$ for any $\psi\in\Omega^1_c(\Sigma\setminus \{a_{2g+2}\},\mathfrak{g})$; 
\item random holonomies $\textbf{Hol}(\gamma)$ along admissible closed curves $\gamma$;
\item $W_{\psi,\mathbf{T}}$ where $\psi\in L^{\infty}(\Sigma)$ and where $\mathbf{T}$ is an approximable current. 
\end{enumerate}
The last two items can be viewed as the probabilistic analogues of the classical observables from Corollary~\ref{c:classical-observables}, namely $\text{Hol}_\gamma(A)$ and $\int_\gamma A$ (when $\mathbf{T}=[\gamma]$). 
\begin{rmk}
Note that, for a general approximable current $\mathbf{T}$, the unstable component $\langle\mathbf{T},A_u\rangle$ is a priori ill-defined even if the Gaussian part $W_{\psi,\mathbf{T}}:=\langle \mathbf{T},A_\mathcal{N}\rangle$ was properly defined in a probabilistic way. A manner to handle this problem is to make an extra assumption on the wavefront properties of $\mathbf{T}$ or to suppose that $\mathbf{T}$ is compactly supported in $W^u(a_1)$ as in Corollary~\ref{c:classical-observables}. In the case where $\mathbf{T}=[\gamma]$ for some admissible $\gamma$, the transversality assumption involved in the definition of admissible curves ensures that $\langle\mathbf{T},A_u\rangle$ is well-defined.
\end{rmk}

To define these quantum observables, we introduce the $\sigma$-algebra $\widetilde{\mathcal{F}}_{\gamma_r}^G(1)$ defined as the lift\footnote{For the $G^{2g}$ component, we take the Borel sets on $G^{2g}$.} to $\Omega\times G^{2g}$ of the $\sigma$-algebra $\widetilde{\mathcal{F}}_{\gamma_r}(1)$ appearing in Lemma~\ref{l:inclusion-sigma-algebra}. Recall that $\widetilde{\mathcal{F}}_{\gamma_r}(1)$ is exactly the $\sigma$-algebra $\sigma(\xi_{[\blacktriangle(\gamma_r)]})$ generated by the restriction of the white noise $\xi$ to $\blacktriangle(\gamma_r)$, i.e. the $\sigma$-algebra generated by the subsets
$$
\{\langle \xi_{[\blacktriangle(\gamma_r)]}\upsilon,\psi\rangle^{-1}(B)\},\ B\ \text{Borel set of}\ \R\ \text{and}\ \psi\in\Omega^1_c(\Sigma,\mathfrak{g})\}.
$$
With this convention, one has
\begin{lemma}\label{l:admissible-random-variable} Let $K$ be a compact subset of $\Sigma\setminus\{a_{2g+2}\}$ Then, there exists $r_0>0$ such that, for every $0<r<r_0$, the following holds
\begin{enumerate}
 \item for any $\psi\in\Omega^1_c(\Sigma\setminus \{a_{2g+2}\},\mathfrak{g})$ supported in $K$, $\langle A,\psi\rangle$ is $\widetilde{\mathcal{F}}_{\gamma_r}^G(1)$ measurable;
 \item for any type $\operatorname{I}$ or $\operatorname{II}$ curve $\gamma$ supported in $K$, $\mathrm{g}_\gamma(1)$ is $\widetilde{\mathcal{F}}_{\gamma_r}^G(1)$ measurable;
 \item for any $\psi\in L^\infty(\Sigma)$ and for any approximable current $\mathbf{T}$ supported in $K$, $W_{\psi,\mathbf{T}}$ is $\widetilde{\mathcal{F}}_{\gamma_r}^G(1)$ measurable.
\end{enumerate}
\end{lemma}

\begin{proof}
For the first item, we observe that $\langle A(\xi,\mathrm{b}),\psi\rangle=\langle A(\xi_{[\blacktriangle(\gamma_r)]},\mathrm{b}),\psi\rangle$ for $r>0$ small enough (in a way that depends only on $K$). Indeed recall that 
$$
\iota_V\mathcal{L}_V^{-1} \left( \xi\upsilon \right) =\iota_V\mathcal{L}_V^{-1}\left(\xi_{\blacktriangle(\gamma_r)}\upsilon)\right)+\iota_V\mathcal{L}_V^{-1}\left(\xi_{\Sigma\setminus\blacktriangle(\gamma_r)}\upsilon\right),
$$
and that, by construction, the second term in the right-hand side is supported in the set $\{x_1^2+x_2^2\leq r^2\}$ (in the Morse chart). Hence, the function $\langle A(\xi,\mathrm{b}),\psi\rangle$ is $\widetilde{\mathcal{F}}_{\gamma_r}^G(1)$-measurable. The second item is a direct consequence of Lemma~\ref{l:inclusion-sigma-algebra} together with Lemma~\ref{r:whitenoise-multiplication}. For the last item, observe that $\mathbf{T}$ is obtained as the limit in the sense of currents of $[\gamma_n]$, where $\gamma_n$ is a sequence of admissible curves such that $[\blacktriangle(\gamma_n)]$ converges in $L^2(\Sigma)$. In particular, for $n\geqslant 1$ large enough, $\gamma_n([0,1])$ is supported in a compact subset $K'$ of $\Sigma\setminus\{a_{2g+2}\}$ containing $K$ and depending only on $K$. Recalling that $W_{\psi,\gamma_n}=\langle\xi_{\psi[\blacktriangle(\gamma_n)]}\upsilon, 1\rangle $ and Lemma~\ref{r:whitenoise-multiplication}, we find that, for $n\geq 1$ large enough, this random variable is $\widetilde{\mathcal{F}}_{\gamma_r}^G(1)$-measurable as $K'\subset\blacktriangle(\gamma_r)$. As $W_{\psi,\gamma_n}$ converges to $W_{\psi,\mathbf{T}}$, we find that the limit is also measurable with respect to this $\sigma$-algebra. 
\end{proof}

\subsection{Conditioning of the free boundary Yang--Mills measure near \texorpdfstring{$a_{2g+2}$}{a(2g+2)}}
\label{ss:def-YM-measure}

Before defining the Yang-Mills measure properly, let us describe intuitively how the measure is formally obtained by conditioning the free boundary Yang--Mills measure near the maximum of $f$. Indeed, now that we have defined the holonomy process $\mathbf{Hol}_{0}$ for a random connection $A$, we can disintegrate the free Yang--Mills measure according to this observable:
$$
 \P^{{\rm free}}_{\text{YM}} = \int_{\mathrm{h}\in G} \P_{\text{YM}}^{{\rm free}}\left( \ . \ | \mathbf{Hol}_{0} = \mathrm{h} \right) p_{ {\bf Hol},0}(\mathrm{h})  \mu_G(d\mathrm{h}),
$$
where $p_{ {\bf Hol},0}$ was defined in Lemma~\ref{lemm:lawHolcombi} and
where the conditional probabilities  $\P_{\text{YM}}^{{\rm free}}\left( \ . \ | \mathbf{Hol}_{0} = \mathrm{h} \right)$  are defined \emph{only for $\mu_G$ almost all} $\mathrm{h}\in G$ by the Rokhlin disintegration Theorem.
In particular, the measure $p_{ {\bf Hol},0}  \mu_G$ on $G$ is the pushforward of  $\P^{{\rm free}}_{\text{YM}}$ under $\mathbf{Hol}_{0}$.
\begin{rmk}
 We refer the reader to \cite[section 5.1.2 p.~144]{Vianaergodic}, \cite[5.5.1 p.~85]{VianaLyapunov} and \cite[C.6 p.~305]{BonattiViana} for discussions and proofs of the disintegration Theorem. Here it can be applied because of the measurability of $\mathbf{Hol}_{0}:\Omega\times G^{2g}\rightarrow G$ with respect to $(\omega,\mathrm{b})$ chosen randomly under $\P\times\mu_G^{\otimes 2g}$ and because $G$ is compact.
 \end{rmk}

Concretely, this means that, for every bounded and measurable observable $\Phi$,
\begin{align*}
\mathbb{E}\left( \Phi\right)=  \int_{\mathrm{h}\in G} \mathbb{E}\left( \Phi | \mathbf{Hol}_{0} = \mathrm{h} \right) p_{ {\bf Hol},0}(\mathrm{h})  \mu_G(d\mathrm{h}),  
\end{align*}
where the expectation is taken with respect to the free boundary Yang-Mills measure and where $\mathbb{E}\left( . | \mathbf{Hol}_{0} = \mathrm{h}\right)$ is the usual conditional expectation with respect to the $\sigma$--algebra generated by the random holonomy $\mathbf{Hol}_{0}$, i.e. the subalgebra generated by $\{\mathbf{Hol}_{0}^{-1}(B):B \subset G \ \text{Borel set}\}$. In particular, it follows from Lemma \ref{lemm:lawHolcombi}  that, $\mu_G$-almost surely,
$$ \vert \P_{\text{YM}}^{{\rm free}}\left( \,  . \, | \mathbf{Hol}_{0}=\mathrm{h} \right) p_{ {\bf Hol},0}(\mathrm{h}) \vert =\sum_{\rho\in \widehat{G}} e^{-\frac{c_2(\rho)}{2}\upsilon(\Sigma)} \frac{\chi_\rho(\mathrm{h})}{\dim(V_\rho)^{2g-1}}.$$
Now, \emph{only at the intuitive level}, we would like to define the Yang--Mills  measure $\mu_{\text{YM}}$ on the closed surface
$\Sigma$ by the simple formula
\begin{equation}\label{eq_mu_ym}
\mathbb{M}_{\text{YM}} := \P_{\text{YM}}^{{\rm free}}\left( \  . \ | \  \mathbf{Hol}_{0} = \text{Id}_G \right)p_{ {\bf Hol},0}(\text{Id}_G),
\end{equation}
and the corresponding probability measure 
$ \frac{ \mathbb{M}_{\text{YM}} }{Z_\upsilon(\Sigma,G)} $  would be the conditional probability
\begin{align*}
\P_{\text{YM}} := 
\frac{\mathbb{M}_{\text{YM}}}{Z_\upsilon(\Sigma,G)} = \P_{\text{YM}}^{{\rm free}}\left( \  . \ | \  \mathbf{Hol}_{0} = \text{Id}_G \right) \ . 
\end{align*}
Here, we recognize the Yang-Mills partition function (see for instance \cite[Eq.~$(2.51)$]{Witten1})
$$Z_\upsilon(\Sigma,G)= p_{ {\bf Hol},0}(\text{Id}_G)=\sum_{\rho\in \widehat{G}}  \frac{e^{-\frac{c_2(\rho)}{2}\upsilon(\Sigma)}}{\dim(V_\rho)^{2g-2}}.$$
Unfortunately, these last steps are only formal. Indeed, regular conditioning is well-defined only for almost every conditioning value $\mathrm{h}$ (and not for a fixed value). However, we have proved in Lemma \ref{lemm:lawHolcombi} that the law of $\mathbf{Hol}_{0}$ has density with respect to the Haar measure. As such, we need to prove that regular conditioning is defined \emph{pointwise} (for all $\mathrm{h}\in G$ instead as for almost all $\mathrm{h}\in G$) and that it coincides with the naive definition which consists in conditioning $\mathbf{Hol}_{0}$ to belong to smaller and smaller neighborhoods of $\text{Id}_G$. More precisely, up to normalization by $Z_\upsilon(\Sigma,G)$, we would like to define the \emph{normalized} Yang-Mills measure as the limit as $\delta\rightarrow 0^+$ of
$$
 \mathbb{P}_{\text{YM}}^{{\rm free}}\left(\ .\ |\operatorname{Hol}_0\in B_\delta(\operatorname{Id})\right):=\frac{\mathbf{1}_{B_\delta(\operatorname{Id})}\circ\operatorname{Hol}_0}{\mathbb{P}_{\text{YM}}^{{\rm free}}\left(\mathbf{1}_{B_\delta(\operatorname{Id})}\circ\operatorname{Hol}_0\right)}\mathbb{P}_{\text{YM}}^{{\rm free}},
 $$
 where $B_\delta(\operatorname{Id})$ is the ball of radius $\delta$ centered at $\text{Id}_G$. In order to make sense of this limit, the key idea is to introduce a certain regularization by disintegrating along the closed loop $\gamma_r$ rather than the blow-up curve from \S\ref{ss:blowup-Wiener-max}. We combine this with the Markov property of the measures $\P^{{\rm free}}_{\text{YM}}$ and with 
a certain property which is a variant of the abelianization in law property from Theorem~\ref{t:abelianization-form0}. Gathering all these elements will lead us to the proper definition of the Yang--Mills measure (under Morse gauge) as we will now explain it.

\subsubsection{Definition of the Yang-Mills measure for large class of observables}

We can now formulate the main statement of the present paper in view of defining the Yang--Mills measure for large class of observables.

\begin{thm}[Construction of the Yang--Mills measure]
\label{t:existence-YM-measure_Viet}  There exists $r_0>0$ such that, for every $0<r<r_0$ and for every function $\Phi:\Omega\times G^{2g}\rightarrow \mathbb{C}$ that is bounded and $\widetilde{\mathcal{F}}_{\gamma_r}^G(1)$-measurable, one has
$$
\lim_{\delta\rightarrow 0^+} Z_\upsilon(\Sigma,G) \mathbb{P}_{\operatorname{YM}}^{{\rm free}}\left(\Phi|\mathbf{Hol}_0\in B_\delta(\operatorname{Id}_G)\right)
$$
exists and is equal to
$$
\int_{\Omega\times G^{2g}}\Phi d\mathbb{M}_{\operatorname{YM},r}:=\int_G\P_{\operatorname{YM}}^{{\rm free}}\left(\Phi|\operatorname{\mathbf{Hol}}_r=\mathrm{g}\right)p_{\mathbf{Hol},r}(\mathrm{g})p_{\upsilon(\Sigma)-\mathscr{A}_{r}(1)}(\mathrm{g})\mu_G(d\mathrm{g}),
$$
where $p_t$ is the heat kernel on $G$, $p_{\mathbf{Hol},r}$ is defined in Lemma~\ref{lemm:lawHolcombi} and $\mathscr{A}_{r}(1)=\int_\Sigma[\blacktriangle(\gamma_r)]\upsilon.$
\end{thm}

Recall that $\widetilde{\mathcal{F}}_{\gamma_r}^G(1)\subset \widetilde{\mathcal{F}}_{\gamma_{r'}}^G(1)$ for $0<r'<r$ so that this result implies that 
$$
\int_{\Omega\times G^{2g}}\Phi d\mathbb{M}_{\operatorname{YM},r}=\int_{\Omega\times G^{2g}}\Phi d\mathbb{M}_{\operatorname{YM},r'},
$$ 
for $0<r'\leqslant r$ and for $\Phi$ that is $\widetilde{\mathcal{F}}_{\gamma_r}^G(1)$-measurable. Except for the definition and the law of random holonomy that will be proved in~\S\ref{ss:law-random-holonomy}, this theorem implies Theorem~\ref{t:def-YM-general} from the introduction. Indeed, thanks to Lemma~\ref{r:whitenoise-multiplication}, the $\sigma$-algebra $\mathcal{B}_n^G$ used in the filtration from Theorem~\ref{t:def-YM-general} is a sub $\sigma$-algebra of $\widetilde{\mathcal{F}}_{\gamma_r}^G(1)$ for $r>0$ small enough. This follows from the fact that $K_n$ is a compact subset of $\blacktriangle(\gamma_r)$ for $r>0$ small enough (depending on $n$). Hence, for any measurable set $A$ in $\mathcal{B}_n^G$, we define
$$
\mathbb{M}_{\text{YM}}(A):=\int_G\P_{\operatorname{YM}}^{{\rm free}}\left(\mathbf{1}_A|\operatorname{\mathbf{Hol}}_r=\mathrm{g}\right)p_{\mathbf{Hol},r}(\mathrm{g})p_{\upsilon(\Sigma)-\mathscr{A}_{r}(1)}(\mathrm{g})\mu_G(d\mathrm{g}),
$$
which is independent of $0<r<r_0$ according to the previous theorem. This defines the measure on $\mathcal{B}_n^G$ for every $n\geqslant 1$ and thus the finitely additive functional from Theorem~\ref{t:def-YM-general}. Note that, for $\Phi=1$, one recovers that 
$$
\mathbb{M}_{\text{YM}}(\Omega\times G^{2g})=Z_{\upsilon}(\Sigma,G),
$$
as expected. This measure depends also a priori on the choice of our gauge, thus on the Morse function used to make our dynamical and probabilistic constructions. Yet, as we shall see in Section~\ref{ss:random-holonomies-YM}, the law of the corresponding random holonomies are independent of these choices as expected by the classical results on the Yang--Mills measure as a random holonomy process. At this stage, Theorem~\ref{t:existence-YM-measure_Viet} does not define yet the Yang--Mills measure $\mu_{\operatorname{YM}}$ from Theorem~\ref{t:def-YM} on spaces of connections. In \S\ref{sss:proof-theo-YM}, we will explain how to construct this measure using Theorem~\ref{t:existence-YM-measure_Viet} and classical results from functional analysis.

%Let us add a few words about observables. For any approximable current $[\gamma]$ in the sense of definition~\ref{def:approx_currents} whose support does not intersect unstable curves, the observable $[\mathbf{W}_\gamma]$ (it is valued in adjoint orbits in $\mathfrak{g}$) defined in Lemma ~\ref{l:integration-connection}
%is such that $\mathbf{W}_\gamma$ is in $L^p$ for all $p\in [2,+\infty)$ under $\mathbb{P}_{{\rm YM}}^{{\rm free}}$ since it is Gaussian. Then after conditioning, we still have random variables $\mathbf{W}_\gamma$ in $L^p$ for all $p\in [2,+\infty)$ under $\mathbb{P}_{{\rm YM}}$.
  
\begin{rmk} Here we work with the normalized Haar measure $\mu_G= {\rm Vol} (G)^{-1}d{\rm g}$ where  the Haar measure $d{\rm g}$  is the volume element corresponding to the bi-invariant metric of $G$ and ${\rm Vol} (G)= \int_G d{\rm g}$. Thus, the Yang--Mills partition function appears to be independent of ${\rm Vol}(G)$. In general, however, the Yang--Mills partition function may depend on the choice of disintegration, and two partition functions may differ by a factor depending on ${\rm Vol}(G)$.  For example, if we define the {\sl non-normalized}
free boundary Yang-Mills measure  $\mathbb{M}_{\operatorname{YM}}^{{\rm free}}$ as
$$
\int_{\Omega\times G^{2g} } \Phi d\mathbb{M}_{\operatorname{YM}}^{{\rm free}}:=\int_{\Omega\times G^{2g}}\Phi d\mathbb{P}\otimes(d\mathrm{b})^{\otimes 2g},
$$
where $\Phi:\Omega\times G^{2g}\rightarrow\C$ is a bounded and measurable function. Thus $\mathbb{M}_{\operatorname{YM}}^{{\rm free}}={\rm Vol(G)}^{2g} \mathbb{P}_{\operatorname{YM}}^{{\rm free}}$.  Then we disintegrate it along values of $\mathbf{Hol}_{0}$ and with respect to the Haar measure $d{\rm g }$, 
$$
 \mathbb{M}_{\operatorname{YM}}^{{\rm free}} = \int_{\mathrm{h}\in G} \mathbb{M}_{\operatorname{YM}}^{{\rm free}}\left( \ . \ | \mathbf{Hol}_{0} = \mathrm{g} \right) p_{ {\bf Hol},0}(\mathrm{h})  d{\rm g},
$$
where $p_{ {\bf Hol},0}$ was defined in Lemma~\ref{lemm:lawHolcombi}. By definition, we have, $d{\rm g}$-almost everywhere
 $$ \mathbb{M}_{\operatorname{YM}}^{{\rm free}}\left( \ . \ | \mathbf{Hol}_{0} = \mathrm{g} \right) = {\rm Vol}(G)^{2g-1}\P_{\text{YM}}^{{\rm free}}\left( \,  . \, | \mathbf{Hol}_{0}=\mathrm{g} \right).$$
Now we can define the corresponding Yang--Mills  measure $\mathbf{M}_{\operatorname{YM}}$ on the closed surface
$\Sigma$ (intuitively, $\mathbf{M}_{\operatorname{YM}}=\mathbb{M}_{\operatorname{YM}}^{{\rm free}}\left( \ . \ | \mathbf{Hol}_{0} = \operatorname{Id}_G \right) p_{ {\bf Hol},0}(\operatorname{Id}_G)$)  by
$$\int_{\Omega\times G^{2g}}\Phi d\mathbf{M}_{\operatorname{YM}}:=\int_G\mathbb{M}_{\operatorname{YM}}^{{\rm free}}\left(\Phi|\operatorname{\mathbf{Hol}}_r=\mathrm{g}\right)p_{\mathbf{Hol},r}(\mathrm{g})p^G_{\upsilon(\Sigma)-\mathscr{A}_{r}(1)}(\mathrm{g}) d\mathrm{g}$$ where  $p^G_t= {\rm Vol }(G)^{-1} p_t $ is the heat kernel with respect to the  bi-invariant metric of $G$ and the latter integration is independent of $r$ for all $r$ sufficiently small.
In particular, $\mathbf{M}_{\operatorname{YM}}=  {\rm Vol (G)}^{2g-1} \mathbb{M}_{\operatorname{YM}} $ and we can identify the Yang--Mills partition function in this convention:
\begin{align*}
   \tilde Z_\upsilon(\Sigma,G)
   &= {\rm Vol (G)}^{2g-1}Z_\upsilon(\Sigma,G) \\
   &={\rm Vol (G)}^{2g-1}\sum_{\rho\in \widehat{G}} e^{-\frac{c_2(\rho)}{2}\upsilon(\Sigma)} \frac{1}{\dim(V_\rho)^{2g-2}}.
\end{align*}
This Yang--Mills partition function appears in the Witten's formula for the symplectic volume of the moduli space of flat connections (cf. \cite[Section 4]{Witten1}, \cite[Proposition 6.6.6]{Labourie}). 
\end{rmk}

\subsubsection{Proof of Theorem \ref{t:existence-YM-measure_Viet}}

We now give the proof of Theorem \ref{t:existence-YM-measure_Viet}. It relies on a property that we call abelianization in law for random holonomies satisfied by the free boundary Yang--Mills measure $\P^{\rm free}_{\rm YM}$. See Lemma~\ref{l:abelianization-max} for a precise statement. This property will be proved along the same lines as its analogue for type $\operatorname{I}$ and $\operatorname{II}$ curves (namely Theorem~\ref{t:abelianization-form0}) and it will be at the heart of the following instrumental lemma:
\begin{lemma}\label{l:markov-maximum}There exists $r_0>0$ such that, for every $0<r<r_0$, for every bounded and measurable function $\Psi:G\rightarrow \R$, and for every bounded function $\Phi:\Omega\times G^{2g}\rightarrow \mathbb{R}$ that is $\widetilde{\mathcal{F}}_{\gamma_r}^G(1)$-measurable, one has
\begin{eqnarray*}
\int_{\Omega\times G^{2g}}\Phi(\omega,\mathrm{b}) \Psi\circ \mathbf{Hol}_0\mathbb{P}(d\omega)\mu_G^{\otimes 2g}(d\mathrm{b})\hspace{5cm}\\
=\int_{\Omega\times G^{2g}}\Phi(\omega,\mathrm{b}) e^{\frac{(\upsilon(\Sigma)-\mathscr{A}_{r}(1))}{2}\Delta_G}\Psi\left(\mathbf{Hol}_r\right)\mathbb{P}(d\omega)\mu_G^{\otimes 2g}(d\mathrm{b}).
\end{eqnarray*}
\end{lemma}
This lemma will be proved in paragraph~\ref{sss:markov-YM} and let us first show how it implies Theorem~\ref{t:existence-YM-measure_Viet}. We fix $\Phi$ as in this theorem and such that Lemma~\ref{l:markov-maximum} applies. We want to compute the limit as $\delta\rightarrow 0^+$ of
$$
\mathbb{P}_{\text{YM}}^{{\rm free}}\left(\Phi|\mathbf{Hol}_0\in B_\delta(\text{Id})\right)=\frac{\int_{\Omega\times G^{2g}}\Phi \mathbf{1}_{B_\delta(\text{Id})}\circ \mathbf{Hol}_0\P(d\omega)\mu_G^{\otimes 2g}(d\mathrm{b})}{\int_{\Omega\times G^{2g}} \mathbf{1}_{B_\delta(\text{Id})}\circ \mathbf{Hol}_0\P(d\omega)\mu_G^{\otimes 2g}(d\mathrm{b})}.
$$
Thanks to Lemma~\ref{l:markov-maximum}, one gets
$$
\mathbb{P}_{\text{YM}}^{{\rm free}}\left(\Phi|\mathbf{Hol}_0\in B_\delta(\text{Id})\right)
=\frac{\int_{\Omega\times G^{2g}}\Phi(\omega,\mathrm{b}) e^{\frac{(\upsilon(\Sigma)-\mathscr{A}_{r}(1))}{2}\Delta_G}(\mathbf{1}_{B_\delta(\text{Id})})\left(\mathbf{Hol}_r\right)\mathbb{P}(d\omega)\mu_G^{\otimes 2g}(d\mathrm{b})}{\int_{\Omega\times G^{2g}} e^{\frac{(\upsilon(\Sigma)-\mathscr{A}_{r}(1))}{2}\Delta_G}(\mathbf{1}_{B_\delta(\text{Id})})\left(\mathbf{Hol}_r\right)\mathbb{P}(d\omega)\mu_G^{\otimes 2g}(d\mathrm{b})},
$$
from which we infer thanks to Lemma~\ref{lemm:lawHolcombi}
\begin{equation}\label{e:measure-delta-radius}
\mathbb{P}_{\text{YM}}^{{\rm free}}\left(\Phi|\mathbf{Hol}_0\in B_\delta(\text{Id})\right)
=\frac{\int_{G}\mathbb{P}_{\text{YM}}^{{\rm free}}(\Phi|\mathbf{Hol}_r=\mathrm{g}) e^{\frac{(\upsilon(\Sigma)-\mathscr{A}_{r}(1))}{2}\Delta_G}(\mathbf{1}_{B_\delta(\text{Id})})\left(\mathrm{g}\right)p_{\mathbf{Hol},r}(\mathrm{g})\mu_G(d\mathrm{g})}{\int_{G}e^{\frac{(\upsilon(\Sigma)-\mathscr{A}_{r}(1))}{2}\Delta_G}(\mathbf{1}_{B_\delta(\text{Id})})\left(\mathrm{g}\right)p_{\mathbf{Hol},r}(\mathrm{g})\mu_G(d\mathrm{g})}.
\end{equation}
Hence, in order to conclude, we can study the convergence of  
$$
\frac{1}{\mu_G(B_\delta(\text{Id}))}e^{\frac{t\Delta_G}{2}}(\mathbf{1}_{B_\delta(\text{Id})})\left(\mathrm{g}\right)=\frac{\int_G p_{t}( \mathrm{g} \mathrm{h}^{-1})  \mathds{1}_{B_\delta(\id_G)}\left(\mathrm{h} \right) \mu_G(d\mathrm{h}) }{\int_{B_\delta(\id_G)
} \mu_G(d\mathrm{h}) }
$$ 
when $\delta\rightarrow 0^+$ for a fixed value of $t>0$ (here $t=\upsilon(\Sigma)-\mathscr{A}_{r}(1)$). Since $p_t$ is smooth, one gets that this quantity converges uniformly to $p_{t}( \mathrm{g})$. Hence, by dominated convergence, one gets the proof of Theorem~\ref{t:existence-YM-measure_Viet}.

\subsubsection{Proof of Lemma~\ref{l:markov-maximum}}\label{sss:markov-YM}
Hence, the proof of Theorem~\ref{t:existence-YM-measure_Viet} boils down to the proof of Lemma~\ref{l:markov-maximum} which is where we will use the abelianization in law property. More precisely, we write 
\begin{equation}\label{e:conditional-expectation-filtration-brownien}
\mathbb{E}\left(\Phi F\circ \mathbf{Hol}_0\right)=\mathbb{E}\left(\mathbb{E}\left(\Phi \Psi\circ \mathbf{Hol}_0|\widetilde{\mathcal{F}}_{\gamma_r}^G(1)\right)\right),
\end{equation}
where expectation is understood here with respect to $\mathbb{P}\times\mu_G^{\otimes 2g}$ (with $\mathbb{P}$ being the probability measure on the space defining the white noise). As $\Phi$ is supposed to be $\widetilde{\mathcal{F}}_{\gamma_r}^G(1)$-measurable, one can deduce from the usual rule for conditional expectations that
\begin{equation}\label{e:point-using-largest-filtration}
\mathbb{E}\left(\Phi \Psi\circ \mathbf{Hol}_0|\widetilde{\mathcal{F}}_{\gamma_r}^G(1)\right)=
\Phi\mathbb{E}\left( \Psi\circ \mathbf{Hol}_0|\widetilde{\mathcal{F}}_{\gamma_r}^G(1)\right).
\end{equation}
Hence the proof of Lemma~\ref{l:markov-maximum} follows from the next property:
\begin{lemma}[Abelianization in law near the maximum]\label{l:abelianization-max} With the above conventions, one has
 almost surely,
\begin{equation}\label{e:abelianization-max}
 \mathbb{E}\left(\Psi\circ \mathbf{Hol}_0|\widetilde{\mathcal{F}}_{\gamma_r}^G(1)\right)=e^{\frac{(\upsilon(\Sigma)-\mathscr{A}_{r}(1))}{2}\Delta_G}\Psi\left(\mathbf{Hol}_r\right).
\end{equation}
\end{lemma}
Compared with Theorem~\ref{t:abelianization-form0} which stated abelianization in law for elementary curves, we emphasize that the random holonomies involved in this Lemma include the contribution of the one dimensional unstable manifolds. Despite that, we will see that the argument remains true in that case. Note also that, thanks to the usual rule for conditional expectation, this Lemma also implies that
$$
\boxed{\mathbb{E}\left(\Psi\circ \mathbf{Hol}_0|\widetilde{\mathcal{F}}_{\gamma_r}^G(1)\right)=\mathbb{E}\left(\Psi\circ \mathbf{Hol}_0|\mathbf{Hol}_r\right)=e^{\frac{(\upsilon(\Sigma)-\mathscr{A}_{r}(1))}{2}\Delta_G}\Psi\left(\mathbf{Hol}_r\right),}
$$
where expectation is understood with respect to the free boundary Yang--Mills measure. Recall that $\widetilde{\mathcal{F}}_{\gamma_r}^G(1)$ is the $\sigma$-algebra (lifted to $\Omega\times G^{2g}$) carrying all the information of the white noise $\xi$ outside the neighborhood $\{x_1^2+x_2^2<r^2\}$ of $a_{2g+2}$. 
\begin{proof}
 Recall that we defined a random process $\mathrm{h}_r(t)$ near the maximum by formula~\eqref{e:random-process-jump} which is the solution to the stochastic differential equation~\eqref{eq:holgammaout_2}. By construction, one knows that, almost surely, $\mathrm{h}_r(t)$ is continuous on $[0,1]\setminus\{t_1,\ldots,t_{4g}\}$. Yet, one can verify that 
$$
\mathrm{h}_{0,r}:t\in[0,1]\mapsto \mathrm{h}_0(t)\mathrm{h}_r(t)^{-1}
$$ 
is almost surely continuous. As in the proof of Theorem~\ref{t:abelianization-form0}, one finds that, in the Stratonovich sense, $\mathrm{h}_{0,r}$ solves the following stochastic differential equation:
$$
d\mathrm{h}_{0,r}=\mathrm{h}_{0,r}\circ d\mathfrak{W}_{0,r},\quad\mathrm{h}_{0,r}(0)=\text{Id},
$$
where
$$
\mathfrak{W}_{0,r}:=\mathrm{h}_r(t)\left(\mathfrak{W}_0(t)-\mathfrak{W}_r(t)\right)\mathrm{h}_r(t)^{-1}.
$$
Arguing as in the proof of Theorem~\ref{t:abelianization-form0}, one gets that, conditionally to $\widetilde{\mathcal{F}}_{\gamma_r}^G(1)$, $\mathfrak{W}_{0,r}(t)$ is a $\mathfrak{g}$-valued Brownian motion with variance $\mathscr{A}_{0,r}(t):=\mathscr{A}_0(t)-\mathscr{A}_r(t)$. Hence, one can apply Theorem~\ref{t:local-holonomies} conditionally to the $\sigma$-algebra $\widetilde{\mathcal{F}}_{\gamma_r}^G(1)$. In particular, the law of $\mathrm{h}_{0,r}(t)$ (conditionally to this $\sigma$-algebra) is given by $p_{\mathscr{A}_{0,r}(t)}(\mathrm{h})\mu_G(d\mathrm{h})$. Hence, almost surely,~\eqref{e:abelianization-max} holds true.
\end{proof}

\subsection{Definition of the Yang--Mills measure}\label{sss:proof-theo-YM}

We will now explain how to construct the measure from Theorem~\ref{t:def-YM} as a by-product of Theorem~\ref{t:existence-YM-measure_Viet}. We begin by defining the distributional spaces adapted to our dynamical constructions in view of applying Prokhorov Theorem. In fact, after singular conditioning, we lost information on the connection $A$ at the critical point $a_{2g+2}=\argmax(f)$ and therefore it is natural to work in functional spaces of the pointed space $\Sigma\setminus \{a_{2g+2}\}$.
This explains why it is more natural to work with a weighted version of the Sobolev space $H^{-1-\kappa}_{\text{loc}}(\Sigma\setminus \{a_{2g+2}\})$ used so far. In some sense, the weight measures the growth of the distribution when we approach $a_{2g+2}$. This is reminiscent of the weighted Lebesgue spaces used in Theorem~\ref{t:contraction}. We are very grateful to Elias Nohra for suggesting to introduce also weighted Sobolev spaces here rather than working on the larger space $\mathcal{D}^\prime(\Sigma\setminus \{a_{2g+2}\})$.

\subsubsection{Weighted Sobolev spaces}

Given some nonnegative integer $p$ and some $q\geqslant 0$, one can define
 $
H^{p,q}(\mathbb{R}^2\setminus\{0\})
$ as the completion of $C^\infty_c  \left( \mathbb{R}^2\setminus \{0\} \right)$
for the Hilbert norm defined as:
$$
\left\|T\right\|_{H^{p,q}(\R^2 \setminus \{0\} )}^2=\sum_{|\alpha|\leqslant p}\left\|\|.\|^{-q} \partial_x^\alpha T\right\|_{L^2(\R^2)}^2
$$
where $\|.\|$ is the standard Euclidean norm on $\R^2$. 
We now fix a smooth partition of $\Sigma$:
$$
\forall x\in\Sigma,\quad\theta_1(x)+\theta_2(x)=1,
$$
where $\theta_j\in\mathcal{C}^{\infty}(\Sigma,[0,1])$ and where $\theta_1$ is equal to $1$ in a small neighborhood of $a_{2g+2}$ and compactly supported in the Morse chart $\{x_1^2+x_2^2<r_0^2\}$ near this critical point. In this chart, any smooth $\mathfrak{g}$-valued $1$-form $\psi$ is of the form $\kappa^*(\psi)(x,dx)=\sum_{\ell=1}^L\left(\psi^{1,\ell}dx_1+\psi^{2,\ell}dx_2\right)\mathfrak{b}_\ell$.
Hence, given $(p,q)\in\Z_+^2$ and $\psi\in\Omega^1(\Sigma\setminus\{a_{2g+2}\},\mathfrak{g})$, one can define the corresponding weighted Sobolev norm as
$$
\|\psi\|_{H^{p,q}(\Sigma)}^2:=\left\|(1+\Delta_h)^{\frac p2}\theta_2\psi\right\|_{L^2(\Sigma,T^*\Sigma\otimes\mathfrak{g})}^2+\sum_{\ell=1}^L\sum_{j=1}^2\left\|\theta_1\left(\psi^{j,\ell}\right)\right\|^2_{H^{p,q}(\R^2 \setminus \{0\})}.
$$
We consider the corresponding space of distributions in $\mathcal{D}^\prime(\Sigma\setminus\{a_{2g+2}\},T^*\Sigma\otimes\mathfrak{g})$ that we denote by $H^{p,q}(\Sigma\setminus\{a_{2g+2}\},T^*\Sigma\otimes\mathfrak{g})$. This is a Hilbert space with weighted Sobolev regularity. We also denote by $\mathcal{S}\left(\Sigma\setminus\{a_{2g+2}\},T^*\Sigma\otimes\mathfrak{g}\right)$ the space of smooth $\mathfrak{g}$-valued $1$-forms $\psi$ in $\Sigma\setminus\{a_{2g+2}\}$ such that, in the Morse chart near $a_{2g+2}$, one has, for every $j=1,2$, $1\leq\ell\leq L$, for every $\alpha\in\Z_+^2$ and for every $N\geqslant 1$, $\partial^\alpha_x\psi^{j,\ell}(x)=\mathcal{O}(\|x\|^N)$ as $x\rightarrow 0$. Its topological dual is denoted by $\mathcal{S}^\prime\left(\Sigma\setminus\{a_{2g+2}\},T^*\Sigma\otimes\mathfrak{g}\right)$ and one has the following continuous embeddings:
$$
\mathcal{S}\left(\Sigma\setminus\{a_{2g+2}\},T^*\Sigma\otimes\mathfrak{g}\right)\subset H^{p,q}\left(\Sigma\setminus\{a_{2g+2}\},T^*\Sigma\otimes\mathfrak{g}\right)\subset\mathcal{S}^\prime\left(\Sigma\setminus\{a_{2g+2}\},T^*\Sigma\otimes\mathfrak{g}\right).
$$
The algebra $\mathcal{S}\left(\Sigma\setminus\{a_{2g+2}\},\mathbb{R}\right)$ is the ideal of smooth functions vanishing at infinite order at $a_{2g+2}$. By a classical result whose proof can be found in \cite[paragraph 4 p.~10]{Malgrangeideals}, this forms a closed ideal in $C^\infty \left(\Sigma,\mathbb{R}\right) $
and therefore any element $\ell$ in $ \mathcal{S}^\prime\left(\Sigma\setminus\{a_{2g+2}\},T^*\Sigma\otimes\mathfrak{g}\right)$ extends (a priori non uniquely) as a distribution
in $\mathcal{D}^\prime \left(\Sigma,T^*\Sigma\otimes\mathfrak{g}\right) $ by the Hahn--Banach Theorem.
We refer the reader to~\cite[section 2 p.~157--163]{Casselman}, \cite[Section 5]{AizenbudSchwartz}
and specially to the~\cite[Example (2.3) p.~159]{Casselman} and \cite[introduction]{AizenbudSchwartz} which explain why the classical notions of Schwartz spaces $\mathcal{S}$ of test functions and $\mathcal{S}^\prime$ of tempered distributions are only special cases of the above spaces of functions vanishing at infinite order and extendible distributions.
These spaces satisfy the following properties:
\begin{lemma}\label{l:basic-properties-weighted-sobolev} Let $(q,q')\in \R_+^2$ and $(p,p')\in \Z_+^2$. One has
\begin{enumerate}
\item if $p\leqslant p'$ and $q\leqslant q'$, then $H^{p',q'}(\Sigma\setminus\{a_{2g+2}\},T^*\Sigma\otimes\mathfrak{g})$ is continuously embedded in $H^{p,q}(\Sigma\setminus\{a_{2g+2}\},T^*\Sigma\otimes\mathfrak{g})$;
\item if $p< p'$ and $q+p< q'$, then $H^{p',q'}(\Sigma\setminus\{a_{2g+2}\},T^*\Sigma\otimes\mathfrak{g})$ is compactly embedded in $H^{p,q}(\Sigma\setminus\{a_{2g+2}\},T^*\Sigma\otimes\mathfrak{g})$.
%\item the space $H^{p,q}(\Sigma\setminus\{a_{2g+2}\},T^*\Sigma\otimes\mathfrak{g})$ is separable;
%\item the subspace $\Omega^1_c(\Sigma\setminus\{a_{2g+2}\},\mathfrak{g})$ is dense in $H^{p,q}(\Sigma\setminus\{a_{2g+2}\},T^*\Sigma\otimes\mathfrak{g})$.
\item the Hilbert space $H^{p,q} \left( \Sigma \setminus \{a_{2g+2}\}, T^*\Sigma\otimes \mathfrak{g}  \right) $ is separable and with a countable dense family in $\Omega^1_c \left( \Sigma \setminus \{a_{2g+2}\},  \mathfrak{g}  \right) $.
\end{enumerate}
\end{lemma}
\begin{proof} The first item is immediate by definition of the weighted Sobolev norm on $\Sigma\setminus\{a_{2g+2}\}.$ For the second item, we fix a sequence $(\psi_n)_{n\geqslant 1}$ which is bounded in $H^{p',q'}(\Sigma\setminus\{a_{2g+2}\})$ and we want to extract a converging subsequence for the $H^{p,q}(\Sigma\setminus\{a_{2g+2}\})$ topology. To see this, we let $\tilde{d}:\Sigma\rightarrow[0,1]$ be a smooth function such that in the Morse chart near $a_{2g+2}$, $\tilde{d}(x)=x_1^2+x_2^2$ and such that outside the Morse chart $\tilde{d}(x)\geq \delta_0>0$. We also fix $\tilde{\chi}:\mathbb{R}_+\rightarrow [0,1]$ be a smooth nondecreasing function which is equal to $0$ on $[0,1]$ and to $1$ outside $[0,2)$. We then set $\tilde{\chi}_0(s)=\tilde{\chi}(x)$ and, for $k\geqslant 1$, $\tilde{\chi}_k(s)=\tilde{\chi}(2^{k} s)-\tilde{\chi}(2^{k-1} s)$ so that, for every $s>0$, $1=\sum_{k\geqslant 0}\tilde{\chi}_k(s)$. We can decompose $\psi_n$ as follows:
$$
\psi_n=\sum_{k\geqslant 0}\chi_k\psi_n,
$$
where $\chi_k=\tilde{\chi}_k\circ\tilde{d}$ are functions supported on dyadic coronas. By definition of the weighted Sobolev norms, one has $\|\chi_k\psi_n\|_{H^{p,q}}\leqslant C2^{k(q+p-q')} \|\psi_n\|_{H^{p',q'}}$. Now, thanks to the compact embeddings between standard Sobolev spaces and by a diagonal extraction, we can extract a subsequence such that, for every $ k\geqslant 0$, $\psi_{n'}\chi_k$ converges to some limit $u_k$ in the $H^{p}(\Sigma)$ topology (hence in the $H^{p,q}$ one). By the above discussion, $\sum_{k\geqslant 0}u_k$ is convergent in the $H^{p,q}(\Sigma\setminus\{a_{2g+2}\})$ norm. By construction, this is the limit of the subsequence $(\psi_{n'})$ for this norm.

The last item is proved as follows. Consider a sequence $(\psi_n)_{n\geqslant 0}, 0\leqslant \psi_n\leqslant 1$ of cut--off functions which are equal to $1$ on a ball of radius $\frac{1}{n}$ around $a_{2g+2}$ 
and vanish outside a ball of radius $\frac{2}{n}$ around $a_{2g+2}$. Fix a sequence $(e_\lambda)_{\lambda\in \sigma(\Delta)}$ of Laplace eigenforms of degree $1$ for the Laplacian $\Delta$ acting on $\Omega^1\left(\Sigma, \mathfrak{g} \right)$. Then, we consider the sequence $((1-\psi_n)e_\lambda)_{n\geqslant 0, \lambda\in \sigma(\Delta)}$ of smooth compactly supported sections in $\Omega^1_c \left(\Sigma\setminus\{a_{2g+2}\},\mathfrak{g} \right) $. We would like to establish that the vector space spanned by the above family  $((1-\psi_n)e_\lambda)_{n\geqslant 0, \lambda\in \sigma(\Delta)}$ is everywhere dense inside $H^{p,q}(\Sigma\setminus\{a_{2g+2}\},T^*\Sigma\otimes\mathfrak{g})$. Assume by contradiction that $\overline{ \text{Span}((1-\psi_n)e_\lambda)_{n\geqslant 0, \lambda\in \sigma(\Delta)} }$ is a strict closed subspace of   $H^{p,q}(\Sigma\setminus\{a_{2g+2}\},T^*\Sigma\otimes\mathfrak{g})$. Then by the Hahn-Banach Theorem, there exists a continuous linear form $\ell\neq 0$ on $H^{p,q}(\Sigma\setminus\{a_{2g+2}\},T^*\Sigma\otimes\mathfrak{g})$ such that 
$$\overline{ \text{Span}((1-\psi_n)e_\lambda)_{n\geqslant 0, \lambda\in \sigma(\Delta)} } \subset \ker(\ell) .$$
Since the norm of $H^{p,q}(\Sigma\setminus\{a_{2g+2}\},T^*\Sigma\otimes\mathfrak{g})$ is a continuous norm on the topology of $\Omega^1_c(\Sigma\setminus\{a_{2g+2}\},\mathfrak{g})$, the continuity of $\ell$ immediately implies that it can be identified with a distribution still denoted by $\ell$ in $\mathcal{D}^\prime\left(\Sigma\setminus\{a_{2g+2}\}  \right)$.
By definition, $\left\langle \ell,  (1-\psi_n)e_\lambda\right\rangle=0 \implies \left\langle \ell(1-\psi_n),e_\lambda\right\rangle=0 $ for all $n\geqslant 0$, $\lambda\in \sigma(\Delta)$. Since $\ell (1-\psi_n)$ is a distribution in $\mathcal{D}^\prime(\Sigma)$, this implies that $\ell (1-\psi_n)=0 $ as distribution in  $\mathcal{D}^\prime(\Sigma)$ for all $n\geqslant 0$. Since the above holds true for all $n\geqslant 0$, this implies that $\ell=0\in\mathcal{D}^\prime \left(\Sigma\setminus\{a_{2g+2}\}  \right)$ 
which contradicts the non triviality of $\ell$. 
\begin{comment}
A more complete proof is that the continuous linear map $\ell$ can be represented (again using Hahn--Banach) as
$\ell(T)= \sum_{\vert \alpha\vert \leqslant p } \left\langle \ell_\alpha, \Vert x\Vert^{-q} \partial_x^\alpha T\right\rangle_{L^2}    $ for a sequence $(\ell_\alpha)_{\vert \alpha\vert \leqslant p }$ of $L^2$ sections. But $\ell=0$ in $\mathcal{D}^\prime(\Sigma\setminus \{ a_{2g+2} \} )$ implies that the $L^2$ sections $\ell_\alpha$ are only supported at $a_{2g+2}$ hence they vanish in $L^2$. 
\end{comment}
\end{proof}
One also defines their dual spaces
$H^{-p,-q}(\Sigma\setminus\{a_{2g+2}\},T^*\Sigma\otimes\mathfrak{g})
$ that are naturally endowed with the norm
$$
\left\|A\right\|_{H^{-p,-q}(\Sigma)}:=\sup\left\{|A(\psi)|:\ \|\psi\|_{H^{p,q}(\Sigma)}\leqslant 1\right\}.
$$
These spaces are continuously embedded in $\mathcal{S}^\prime\left(\Sigma\setminus\{a_{2g+2}\},T^*\Sigma\otimes\mathfrak{g}\right)$, hence distributions in $\Sigma\setminus\{a_{2g+2}\}$ which can be extended as distributions in $\mathcal{D}^\prime(\Sigma)$. Since we proved that $H^{p,q}$ is a separable Hilbert space, its topological dual $H^{-p-q}$ is also a separable Hilbert space by the Riesz representation Theorem.

One has the following embedding property with respect to the standard Sobolev spaces (corresponding to $q=0$):
\begin{lemma}\label{l:embedding-classical-weighted} Let $(\kappa,p)\in\R_+\times\Z_+$ such that $0\leqslant\kappa\leqslant p$ and let $q\geqslant 0$. There exists a constant $C>0$ such that, for every $A \in H^{-\kappa}(\Sigma,T^*\Sigma\otimes\mathfrak{g})$,
$$
\psi\in H^{p,q}(\Sigma\setminus\{a_{2g+2}\},T^*\Sigma\otimes\mathfrak{g})\mapsto\langle A,\psi\rangle\in \R
$$
belongs to $H^{-p,-q}(\Sigma\setminus\{a_{2g+2}\},T^*\Sigma\otimes\mathfrak{g})$ and
$$
\left\|A\right\|_{H^{-p,-q}(\Sigma)}\leq C\| A\|_{H^{-\kappa}(\Sigma)}.
$$
More precisely, there exists a constant $C_1>0$ such that, for every $\chi\in\mathcal{C}^{\infty}(\Sigma,\R)$ and for every $A \in H^{-\kappa}(\Sigma,T^*\Sigma\otimes\mathfrak{g})$ that is supported in $\{\chi=1\}$, one has
$$
\left\|A\right\|_{H^{-p,-q}(\Sigma)}\leq C_1(1+\|\chi\|_{\mathcal{C}^p})\max\{d\left(a_{2g+2},x\right)^{q}:\chi(x)\neq 0\}\| A\|_{H^{-\kappa}(\Sigma)}
$$
\end{lemma}

This Lemma is important since it ensures that expectations of our new weighted Sobolev norms are well--defined. Moreover the second estimate will be useful since we will need to control our weighted Sobolev norms of $A$ by classical Sobolev norms of our distributions suitably localized to small dyadic coronas.

\begin{proof}
 Let $\psi\in H^{p,q}(\Sigma\setminus\{a_{2g+2}\},T^*\Sigma\otimes\mathfrak{g})$ and $\chi\in\mathcal{C}^{\infty}(\Sigma,\R)$. Without loss of generality, we can suppose, thanks to Lemma~\ref{l:basic-properties-weighted-sobolev}, that $\psi$ belongs to $\Omega^1_c(\Sigma\setminus\{a_{2g+2}\},\mathfrak{g})$. We need to bound 
 $$
 \langle A,\psi\rangle=\langle A,\chi\psi\rangle,
 $$
 where the pairing is understood for the duality between $\Omega_c^1(\Sigma\setminus\{a_{2g+2}\},\mathfrak{g})$ and $\mathcal{D}^\prime(\Sigma\setminus\{a_{2g+2}\},T^*\Sigma\otimes\mathfrak{g})$. By the definition of the weighted Sobolev norms, one finds some constant $C>0$ (that is independent of $A$, $\psi$ and $\chi$) such that
 $$
 \left|\langle A,\psi\rangle\right|\leq C \max\{d\left(a_{2g+2},x\right)^{q}:\chi(x)\neq 0\}\|A\|_{H^{-p}(\Sigma)} \|\chi\psi\|_{H^{p,q}(\Sigma)},
 $$
 which implies the expected upper bound as $0\leqslant \kappa\leqslant p$.
 \end{proof}

\subsubsection{Proof of Theorem~\ref{t:def-YM}}

In order to define the Yang--Mills measure on the space of connections, we will start from the measures appearing in Theorem~\ref{t:existence-YM-measure_Viet}. More precisely, for $\delta>0$ small enough (say $\delta<\delta_0$) and for $\Phi:H^{-1-\kappa}(\Sigma,T^\Sigma\otimes\mathfrak{g})\rightarrow \R$ a bounded and measurable\footnote{Here measurability is understood with respect to the Borel $\sigma$-algebra of the Hilbert space.} function, we set
\begin{equation}\label{e:measure-YM-delta}
\int_{H^{-1-\kappa}(\Sigma)}\Phi(A)\frac{\mu_{\text{YM}}^\delta(dA)}{Z_\upsilon(\Sigma,G)}:=  \int_{\Omega\times G^{2g}}\Phi(A(\omega,\mathrm{b}))\mathbf{1}_{B_\delta(\operatorname{Id})}\circ\operatorname{Hol}_0(\omega,\mathrm{b})\frac{\mathbb{P}_{\operatorname{YM}}^{{\rm free}}(d\omega,d\mathrm{b})}{\mathbb{P}_{\text{YM}}^{{\rm free}}\left(\mathbf{1}_{B_\delta(\operatorname{Id})}\circ\operatorname{Hol}_0\right)} .
\end{equation}
This defines a measure on the space $H^{-1-\kappa}(\Sigma)$ and, thanks to Lemma~\ref{l:embedding-classical-weighted}, a measure on $H^{-p,-q}(\Sigma\setminus\{a_{2g+2}\})$ for $p> 1+\kappa$. The following holds:

\begin{lemma}\label{l:Tightness} Let $q'> 4+\frac{\operatorname{dim} G}{2}$. Then, for every $\epsilon>0$, one can find a compact set $K_\epsilon\subset H^{-3,-q'}(\Sigma\setminus\{a_{2g+2}\})$ such that, for every $0<\delta<\delta_0$,
$$
\mu_{\operatorname{YM}}^\delta\left(K_\epsilon\right)\geqslant Z_\upsilon(\Sigma,G)-\epsilon.
$$ 
\end{lemma}
In other words, the family of measures $(\mu_{\text{YM}}^\delta)_{0<\delta<\delta_0}$ is tight. Hence, according to Prokhorov Theorem~\cite[Th.~8.10]{KoralovSinai}, it is a weakly compact family of measures and one can extract converging subsequences. A short yet important remark: usually the Prokhorov Theorem
applies only to sequences of probability measures, but here the family of measures 
$\left( \mu_{\operatorname{YM}}^\delta \right)_{\delta>0}$ has fixed mass 
equals to $Z_\upsilon(\Sigma,G)$ that does not depend on $\delta>0$.

\begin{proof} 
We let $0<\kappa<\min\{1,\frac{\text{dim} G}{4}\}$ and $q=1+\frac{\operatorname{dim} G}{2}<q'-3$. Thanks to Lemma~\ref{l:embedding-classical-weighted} and to the fact that $\mu^\delta_{\text{YM}}(H^{-1-\kappa})=Z_\upsilon(\Sigma,G)$ (for $\kappa>0$), one can write that, for every $R>0$,
\begin{eqnarray*}
\mu_{\text{YM}}^\delta\left(\overline{\left\{A\in H^{-2,-q}:\|A\|_{H^{-2,-q}}\leqslant R\right\}}^{H^{-3,-q'}}\right)\\
=\mu_{\text{YM}}^\delta\left(\left\{A\in H^{-2,-q}:\|A\|_{H^{-2,-q}}\leqslant R\right\}\right).
\end{eqnarray*}
As $H^{-2,-q}$ is compactly embedded in $H^{-3,-q'}$, it is thus sufficient to prove that the right-hand side of this equality can be made arbitrarily close to $Z_\upsilon(\Sigma,G)$ by picking $R$ large enough. %To see this, we let $\tilde{\chi}:\mathbb{R}_+\rightarrow [0,1]$ be a smooth nondecreasing function which is equal to $0$ on $[0,1]$ and to $1$ outside $[0,2)$. We then set $\tilde{\chi}_0(s)=\tilde{\chi}(x)$ and, for $k\geqslant 1$, $\tilde{\chi}_k(s)=\tilde{\chi}(2^{k} s)-\tilde{\chi}(2^{k-1} s)$ so that, for every $s>0$, $1=\sum_{k\geqslant 0}\tilde{\chi}_k(s)$. 
For every $R>0$, one has by Markov's inequality
$$
\mu_{\text{YM}}^{\delta}\left(\left\{A:\|A\|_{H^{-2,-q}}>R\right\}\right)\leqslant \frac{1}{R}\sum_{k\geqslant 0}\int_{H^{-2,-q}}\|A\chi_k\|_{H^{-2,-q}}\mu_{\text{YM}}^\delta(dA),
$$
where $\chi_k$ is the same function as in the proof of Lemma~\ref{l:basic-properties-weighted-sobolev}. By Cauchy--Schwarz inequality, one finds for some $C>0$:
$$
\mu_{\text{YM}}^{\delta}\left(\left\{A:\|A\|_{H^{-2,-q}}>R\right\}\right)\leqslant \frac{C}{R}\sum_{k\geqslant 0}\left(\int_{H^{-2,-q}}\|A\chi_k\|_{H^{-2,-q}}^2\mu_{\text{YM}}^\delta(dA)\right)^{\frac{1}{2}}.
$$
Hence, thanks to Lemma~\ref{l:embedding-classical-weighted}, one can find some constant $C_0>0$ such that
$$
\mu_{\text{YM}}^{\delta}\left(\left\{A:\|A\|_{H^{-2,-q}}>R\right\}\right)\leqslant \frac{C_0}{R}\sum_{k\geqslant 0}2^{-k q}\left(\int_{H^{-1-\kappa}}\|A\chi_k\|_{H^{-1-\kappa}}^2\mu_{\text{YM}}^\delta(dA)\right)^{\frac{1}{2}}.
$$
Observe now that $\|A\chi_k\|_{H^{-1-\kappa}}^2=\sum_{\lambda,\ell}(1+\lambda)^{-1-\kappa}\left|\langle A,\chi_k\tilde{\mathbf{e}}_{\lambda,\ell}\rangle\right|^2$ where $(\tilde{\mathbf{e}}_{\lambda,\ell})_{\lambda,\ell}$ is an orthonormal basis of Laplace eigenfunctions. In particular, thanks to Lemma~\ref{l:admissible-random-variable}, each term in this sum is $\widetilde{\mathcal{F}}_{\gamma_{r_k}}^G(1)$ measurable for $r_k=c_0 2^{-k}$ with $c_0>0$ independent of $k\geqslant 0$. Hence, thanks to~\eqref{e:measure-delta-radius}, we find that
 \begin{eqnarray*}
 \mu_{\text{YM}}^{\delta}\left(\left\{A:\|A\|_{H^{-2,-q}}>R\right\}\right)\hspace{7cm}\\ \leqslant \frac{C_0Z_{\upsilon}(\Sigma,G)^{\frac12}}{R}
 \sum_{k\geqslant 0}\frac{2^{-kq}\left(\int_{H^{-1-\kappa}}\|A\chi_k\|_{H^{-1-\kappa}}^2\mathbb{P}_{\text{YM}}^{\operatorname{free}}(dA)\right)^{\frac12}}{\left(\int_{G}e^{\frac{(\upsilon(\Sigma)-\mathscr{A}_{r_k}(1))}{2}\Delta_G}(\mathbf{1}_{B_\delta(\text{Id})})\left(\mathrm{g}\right)p_{\mathbf{Hol},r_k}(\mathrm{g})\mu_G(d\mathrm{g})\right)^{\frac12}}\\
 \times \mu_G(B_{\delta}(\text{Id})^{\frac12}\max_{\mathrm{g}\in G}\left\{\left(e^{\frac{(\upsilon(\Sigma)-\mathscr{A}_{r_k}(1))}{2}\Delta_G}(\mathbf{1}_{B_\delta(\text{Id})})\right)^{\frac12}\right\}.
 \end{eqnarray*}
 Thanks to Theorem~\ref{thm:main1}, one knows that $\int_{H^{-1-\kappa}}\|A\|_{H^{-1-\kappa}}^2\mathbb{P}_{\text{YM}}^{\operatorname{free}}(dA)<\infty$. Hence, there exists some constant $C>0$ such that
 \begin{eqnarray*}
 \mu_{\text{YM}}^{\delta}\left(\left\{A:\|A\|_{H^{-2,-q}}>R\right\}\right)\hspace{7cm}\\
 \leqslant \frac{C}{R}\sum_{k\geqslant 0}\frac{2^{-k(q-1+\kappa)}\mu_G(B_{\delta}(\text{Id})^{\frac12}\max_{\mathrm{g}\in G}\left\{\left(e^{\frac{(\upsilon(\Sigma)-\mathscr{A}_{r_k}(1))}{2}\Delta_G}(\mathbf{1}_{B_\delta(\text{Id})})\right)^{\frac12}\right\}}{\left(\int_{G}e^{\frac{(\upsilon(\Sigma)-\mathscr{A}_{r_n}(1))}{2}\Delta_G}(\mathbf{1}_{B_\delta(\text{Id})})\left(\mathrm{g}\right)p_{\mathbf{Hol},r_k}(\mathrm{g})\mu_G(d\mathrm{g})\right)^{\frac12}}.
 \end{eqnarray*}
 Thanks to the estimates for the heat kernel as $t\rightarrow 0^+$, one finally finds, up to modifying the constant $C>0$:
  $$
 \mu_{\text{YM}}^{\delta}\left(\left\{A:\|A\|_{H^{-2,-q}}>R\right\}\right)\leqslant \frac{C}{R}\sum_{n\geqslant 0}\frac{2^{-k(q-1+\kappa-\frac{\text{dim} G}{4})}\mu_G(B_{\delta}(\text{Id}))^{\frac12}}{\left(\int_{G}e^{\frac{(\upsilon(\Sigma)-\mathscr{A}_{r_k}(1))}{2}\Delta_G}(\mathbf{1}_{B_\delta(\text{Id})})\left(\mathrm{g}\right)p_{\mathbf{Hol},r_k}(\mathrm{g})\mu_G(d\mathrm{g})\right)^{\frac12}}.
 $$
 Thanks to Lemma~\ref{lemm:lawHolcombi}, one also has
 $$
 \frac{1}{\mu_G(B_{\delta}(\text{Id})}\int_{G}e^{\frac{(\upsilon(\Sigma)-\mathscr{A}_{r_k}(1))}{2}\Delta_G}(\mathbf{1}_{B_\delta(\text{Id})})\left(\mathrm{g}\right)p_{\mathbf{Hol},r_k}(\mathrm{g})\mu_G(d\mathrm{g})= Z_{\upsilon}(\Sigma,G).
 $$
 Hence, recalling that $q=1+\frac{\text{dim} G}{2}$ and that $\kappa\in(0,\text{dim} G/4)$, we finally get the upper bound
 $$
 \mu_{\text{YM}}^{\delta}\left(\left\{A:\|A\|_{H^{-2,-q}}>R\right\}\right)\leqslant \frac{C}{R},
 $$
 for some constant $C>0$ that is independent of $R>0$ and $0<\delta<\delta_0$. As already explained at the beginning of the proof, the conclusion follows by picking $R>0$ large enough.
\end{proof}

We are now ready to prove our last main theorem. We will in fact prove something slightly stronger as the Yang--Mills measure will be defined on the Hilbert space $H^{-3,-q}$ (with $q$ large enough). As already explained in the introduction, the parameters $(3,q)$ are probably not optimal and the question of optimizing the metric space involved is a natural and subtle question. Here our argument gives $q=4+\text{dim}(G)$.

\begin{proof}[Proof of Theorem~\ref{t:def-YM}] Thanks to Lemma~\ref{l:Tightness} and to Prokhorov Theorem~\cite[Th.8.10]{KoralovSinai} applied in $H^{-3,-4-\text{dim} G}(\Sigma)$, one has that the family $(\mu_{\text{YM}}^\delta)_{0<\delta<\delta_0}$ is weakly compact. Hence, we can extract converging subsequences as $\delta\rightarrow 0$ and we let $\mu_{\text{YM}}$ and $\overline{\mu}_{\text{YM}}$ be two accumulation points. We will verify that these two measures coincide along with the properties stated in Theorem~\ref{t:def-YM}. First, we let $p\geqslant 1$ and $\psi_1,\ldots,\psi_p$ be elements in $H^{3}(\Sigma, T^*\Sigma\otimes\mathfrak{g})$ that are compactly supported in $\Sigma\setminus\{a_{2g+2}\}$. Then, for any bounded and continuous function $u:\R^p\rightarrow\R$, the map
$$
A\in H^{-3,-4-\text{dim} G}(\Sigma\setminus\{a_{2g+2}\})\mapsto u\left(\langle A,\psi_1\rangle,\ldots\langle A,\psi_p\rangle\right)
$$
is bounded and continuous. The same argument as in the proof of Lemma~\ref{l:admissible-random-variable} (which was stated for smooth compactly supported one forms) can be applied and the map 
$$
(\omega,\mathrm{b})\in\Omega\times G^{2g}\mapsto u\left(\langle A(\omega,\mathrm{b}),\psi_1\rangle,\ldots\langle A(\omega,\mathrm{b}),\psi_p\rangle\right)
$$
is also $\widetilde{\mathcal{F}}^G_{\gamma_r}(1)$-measurable for $r>0$ small enough. Hence, one can apply Theorem~\ref{t:existence-YM-measure_Viet} and one finds that
\begin{eqnarray*}
\int_{H^{-3,-4-\text{dim} G}}u\left(\langle A,\psi_1\rangle,\ldots\langle A,\psi_p\rangle\right)\mu_{\text{YM}}(dA)\hspace{5cm}\\
=\int_{\Omega\times G^{2g}}u\left(\langle A(\omega,\mathrm{b}),\psi_1\rangle,\ldots\langle A(\omega,\mathrm{b}),\psi_p\rangle\right)\mathbb{M}_{\text{YM}}(d\omega,d\mathrm{b})
\\
=\int_{H^{-3,-4-\text{dim} G}}u\left(\langle A,\psi_1\rangle,\ldots\langle A,\psi_p\rangle\right)\overline{\mu}_{\text{YM}}(dA).
\end{eqnarray*}
By Riesz representation Theorem (applied to the pushforward measures on $\R^p$), this remains true if $u$ is replaced by $\mathbf{1}_B$ where $B$ is any Borel set of $\R^p$.

We now let $\mathbf{B}$ be the unit ball of $H^{-3,-4-\text{dim} G}$. We want to show that $\mu_{\text{YM}}\left(\mathbf{B}\right)=\overline{\mu}_{\text{YM}}\left(\mathbf{B}\right)$. This ball is defined as
$$
\mathbf{B}:=\left\{A:\sup_{\|\psi\|_{H^{3,4+\text{dim}G}}= 1}\left| A(\psi)\right|\leqslant 1\right\}.
$$
Recall now from Lemma~\ref{l:basic-properties-weighted-sobolev} that $H^{3,4+\text{dim}G}$ is separable (with a dense countable family $(\psi_m)_{m\geqslant 1}$ compactly supported in $\Sigma\setminus\{a_{2g+2}\}$). Hence, one has
$$
\mathbf{B}=\bigcap_{m\geqslant 1}\left\{A\in H^{-3,-q}:\left| A(\psi_m)\right|\leqslant \|\psi_m\|_{H^{3,4+\text{dim} G}}\right\},
$$
or equivalently
$$
\mathbf{1}_{\mathbf{B}}=\prod_{m\geqslant 1}\mathbf{1}_{[-1,1]}\left(\left| A\left(\frac{\psi_m}{\|\psi_m\|_{H^{3,4+\text{dim} G}}}\right)\right|\right).
$$
From the dominated convergence Theorem and from the above equality on finite cylinders, one can conclude that $\mu_{\text{YM}}(\mathbf{B})=\overline{\mu}_{\text{YM}}(\mathbf{B}).$ As $H^{-3,-4-\text{dim}G}$ is separable, the $\sigma$-algebra of closed balls is the same as the $\sigma$-algebra of Borel sets. Hence, one can conclude that $\mu_{\text{YM}}=\overline{\mu}_{\text{YM}}.$
\textbf{This is the Yang--Mills measure we were aiming at}. By construction, one has $\iota_V(A)=0$ for $\mu_{\text{YM}}$ almost every $A$.

Finally, it remains to describe the Sobolev regularity of these distributions. We let $\kappa>0$ and $\chi\in\mathcal{C}^\infty_c(\Sigma\setminus\{a_{2g+2}\},[0,1])$. Recall that
$$
\left\|\chi A\right\|_{H^{-1-\kappa}}^2=\sum_{\lambda,\ell}(1+\lambda)^{-1-\kappa}\left|\langle A,\chi\mathbf{e}_{\lambda,\ell}\rangle \right|^2.
$$
By construction of the $\sigma$-algebra on $\mathcal{S}^\prime(\Sigma\setminus\{a_{2g+2}\})$, the map $\left\|\chi A\right\|_{H^{-1-\kappa}}^2$ is measurable for the $\sigma$-algebra generated by the cylinder sets. Moreover, according to Lemma~\ref{l:admissible-random-variable} and by similar observation, $(\omega,\mathrm{b})\mapsto\left\|\chi A(\omega,\mathrm{b})\right\|_{H^{-1-\kappa}}^2$ is $\widetilde{\mathcal{F}}^G_{\gamma_r}(1)$-measurable for $r>0$ small enough. We also let $\theta_R:\R_+\rightarrow[0,R]$ be a continuous function which is the identity on $[0,R]$ and which is identically equal to $R$ on $[0,\infty]$. Hence, thanks to Theorem~\ref{t:existence-YM-measure_Viet}, one has 
\begin{eqnarray*}
\int_{H^{-3,-4-\text{dim} G}}\theta_R\left(\left\|\chi A\right\|_{H^{-1-\kappa}}^2\right)\mu_{\text{YM}}(dA)\hspace{5cm}\\
\leqslant \|p_{\upsilon(\Sigma)-\mathscr{A}_r(1)}\|_{\infty} \int_{\Omega\times G^{2g}} \theta_R\left(\left\|\chi A(\omega,\mathrm{b})\right\|_{H^{-1-\kappa}}^2\right) \mathbb{P}_{\text{YM}}^{\operatorname{free}}(d\omega,d\mathrm{b})\\
\leqslant \|p_{\upsilon(\Sigma)-\mathscr{A}_r(1)}\|_{\infty} \int_{\Omega\times G^{2g}} \left\|\chi A(\omega,\mathrm{b})\right\|_{H^{-1-\kappa}}^2 \mathbb{P}_{\text{YM}}^{\operatorname{free}}(d\omega,d\mathrm{b})<\infty,
\end{eqnarray*}
where the finiteness of the last integral is again a consequence of Theorem~\ref{thm:main1}. By the monotone convergence Theorem, we can let $R\rightarrow+\infty$ and we deduce that $\|\chi A\|_{H^{-1-\kappa}}<\infty$ for $\mu_{\text{YM}}$-almost every $A$.
\end{proof}

\section{Random holonomies for the Yang--Mills measure}
\label{ss:random-holonomies-YM}

In this section, we define random holonomies for general admissible curves (see \S\ref{ss:randomholonomy}) and we explain how to integrate these quantities with respect to the Yang-Mills functional as it appears in Theorem~\ref{t:existence-YM-measure_Viet}. This is explained in~\S\ref{ss:law-random-holonomy} and we give in~\S\ref{ss:law-holonomy-free} an example of computation of this law for small loops with respect to the free boundary measure and in~\S\ref{ss:warmup-example} another example for the Yang--Mills measure. Once these definitions are settled, the main result of this Section is Theorem~\ref{t:holonomy-disk-regular} where we compute the law for random holonomies along curves which are the boundaries of small admissible disks contained in flow boxes of the gradient flow $\varphi_f^t$ used to construct the Yang--Mills measure. As we shall see, the law for random holonomy along such curves is independent of the choice of the Morse function and it matches with the results from the literature for instance as~\cite[subsection 1.8.2 p.~15]{Levyphd}.
Actually, such a formula would immediately follow if we could prove that our Yang--Mills measure recovers the more general 
Driver--Sengupta 
formulas~\cite[equations $(1.1), (1.2), (1.3)$ p.~6]{Levyphd}. We leave the proof of such a general formula as a question for future investigations.

\subsection{Random holonomies for admissible curves}\label{ss:randomholonomy}

Given an elementary curve $\gamma$, we define 
$$
\overline{\gamma}(t)=\gamma(1-t),\quad t\in[0,1].
$$
With this convention, we can introduce the elementary (random) holonomy of $A$ along $\gamma$. More precisely, thanks to Theorem~\ref{t:local-holonomies} and to Lemma~\ref{l:decomposition-curve}, we would like to define the holonomy of $\gamma$ with respect to $A=A(\omega,(\mathrm{b}_a))$ as the product of the elementary holonomies. Yet, we need to take into account the contribution of the unstable manifolds $([W^u(a)])_{\text{ind}(a)=1}$ and this can be handled in the following way.
\begin{definition} Let $\gamma$ be an elementary curve. The holonomy along $\gamma$ of the random connection $A$ is defined as follows:
\begin{enumerate}
 \item If $\gamma$ is of type $\operatorname{III}$, then $\mathbf{Hol}(\gamma)=\operatorname{Id}$.
 \item If $\gamma$ is of type $\operatorname{I}$ or $\operatorname{II}$ and if $t\in[0,1]\mapsto\mathscr{A}_\gamma(t)$ is increasing, then 
 $$
\mathbf{Hol}(\gamma):=\left\{ \begin{array}{l}\mathrm{g}_{\gamma}(1)\quad \text{if}\ \gamma(0)\in W^u(a_1),\\
                         \mathrm{g}_{\gamma}(1)\mathrm{b}_a^{\varepsilon(\gamma)}\quad \text{if}\ \gamma(0)\in W^u(a)\ \text{for some}\ a\ \text{of index}\ 1 ,
                         \end{array}
\right.
$$
where $\mathrm{g}_{\gamma}(1)$ is the random holonomy defined in Theorem~\ref{t:local-holonomies} and where $\varepsilon(\gamma)=1$ if the flowline from $\gamma(0)$ to $a_{2g+2}$ has the same orientation as $W^u(a)$ and $\varepsilon(\gamma)=-1$ otherwise.
\item If $\gamma$ is of type $\operatorname{I}$ or $\operatorname{II}$ and if $t\in[0,1]\mapsto\mathscr{A}_\gamma(t)$ is decreasing, then 
 $$
\mathbf{Hol}(\gamma):=\mathbf{Hol}(\overline{\gamma})^{-1}.$$
\end{enumerate}
\end{definition}
This leads to the definition of random holonomy along a general admissible curve.
\begin{definition} Let $\gamma=\gamma_1\star\left(\gamma_2\star\left(\ldots\star\gamma_J\right)\ldots\right)$ be an admissible curve. Then, one defines the holonomy of the random connection $A$ along $\gamma$ as
$$
\mathbf{Hol}(\gamma):=\mathbf{Hol}(\gamma_J)\ldots\mathbf{Hol}(\gamma_2)\mathbf{Hol}(\gamma_1).
$$
\end{definition}
Thanks to Lemma~\ref{l:decomposition-curve}, this definition is independent of the decomposition of $\gamma$ in elementary curves.

\subsection{Definition of the law of random holonomies}
\label{ss:law-random-holonomy}

By definition of an admissible curve $\gamma$, one knows that $\gamma([0,1])$ does not intersect critical points of $f$. In particular, by compactness, $\gamma([0,1])$ does not intersect the small neighborhood $\{x_1^2+x_2^2<r^2\}$ of $a_{2g+2}$. Hence, by Lemma~\ref{l:admissible-random-variable}, $\mathbf{Hol}(\gamma)$ is $\widetilde{\mathcal{F}}_{\gamma_r}^G(1)$-measurable for $r>0$ small enough. In particular, thanks to Theorem~\ref{t:existence-YM-measure_Viet}, we can set the following definition for the law of a random holonomy with respect to the Yang--Mills measure:
\begin{definition}[Law of random holonomies]\label{d:law-random-holonomy}
 Let $\gamma$ be an admissible curve. Then, one defines the law of the random holonomy along $\gamma$ as follows:
 \begin{eqnarray*}
\int_{\operatorname{ker}(\iota_V)}\Psi(\mathbf{Hol}(\gamma))d\mathbb{M}_{\operatorname{YM}}:=\int_G\P_{\operatorname{YM}}^{{\rm free}}\left(\Psi(\mathbf{Hol}(\gamma))|\operatorname{\mathbf{Hol}}_r=\mathrm{g}\right)p_{\mathbf{Hol},r}(\mathrm{g})p_{\upsilon(\Sigma)-\mathscr{A}_{r}(1)}(\mathrm{g})\mu_G(d\mathrm{g})\\
=\int_{\operatorname{ker}(\iota_V)}\Psi(\mathbf{Hol}(\gamma))p_{\upsilon(\Sigma)-\mathscr{A}_{r}(1)}(\mathbf{Hol}_r) d\P_{\operatorname{YM}}^{{\rm free}},
 \end{eqnarray*}
where $\Psi:G\rightarrow\R$ is a bounded and measurable function and where $r>0$ is small enough.
\end{definition}
Recall from Theorem~\ref{t:existence-YM-measure_Viet} that this definition is independent of $r>0$ which is small enough in a way that depends only on $\gamma$ (and on the Morse chart near the maximum). Note that the second equality follows from Lemma~\ref{lemm:lawHolcombi}. Even if $\mathbf{Hol}(\gamma)$ is not strictly speaking a function of $A$, it is a random variable on the probability space defining the Yang--Mills. Finally, if $\Psi$ is a central function (meaning that $\Psi(\mathrm{h}^{-1}\mathrm{g}\mathrm{h})=F(\mathrm{g})$, then the above quantity defines the law of the conjugacy class $[\mathbf{Hol}(\gamma)]$ of $\mathbf{Hol}(\gamma)$ in $G$, i.e.
$$
\int_{\operatorname{ker}(\iota_V)}\Psi([\mathbf{Hol}(\gamma)])d\mathbb{M}_{\operatorname{YM}}:=\int_{\operatorname{ker}(\iota_V)}\Psi(\mathbf{Hol}(\gamma))p_{\upsilon(\Sigma)-\mathscr{A}_{r}(1)}(\mathbf{Hol}_r) d\P_{\operatorname{YM}}^{{\rm free}}
$$

\subsection{Law of random holonomies for the free boundary Yang-Mills measure}\label{ss:law-holonomy-free}
As an example of computation for the law of random holonomies, we first derive a simple version of the Migdal formula under $\mathbb{P}_{{\rm YM}}^{{\rm free}}$ along a closed curve $\gamma$ that forms the boundary of an admissible flow box, generalizing the corresponding result for the unit disk in the plane~\cite[Th.~4.8]{Sengupta97}.
\begin{definition}
 We say $\blacksquare\subset \Sigma$ is an admissible flow rectangle if $\partial \blacksquare$ is the union of two (possibly empty) flow lines and two transverse curves $\gamma, \tilde \gamma  $ of same type (either $\operatorname{I}$ or $\operatorname{II}$) with $\tilde{\gamma}\preccurlyeq\gamma$  (with $\mathscr{A}_\gamma(t)$ increasing) and if $\blacksquare\subset \Sigma$ does not contain any critical points.
\end{definition}
%Recall \begin{equation}
%A(\xi,(\mathrm{b}_a)_{\text{ind}(a)=1})= A_{\mathcal{N}}+\sum_{\text{ind}(a)=1}\log \left(\mathrm{b}_a \right) [W^u(a)],
%\end{equation}
%with $A_{\mathcal{N}}= \iota_V\mathcal{L}_V^{-1}\left(\xi\upsilon\right)$.
\begin{lemma}
Let $\blacksquare$ be an admissible flow rectangle. Then,  for every $\Psi\in \mathcal{C}^0(G)$, we have 
\begin{equation}
\mathbb{E}^{\rm free}\left( \Psi(\mathbf{Hol}(\partial \blacksquare) ) \right) = \int_{G}\Psi\left(\mathrm{g}\right)p_{\upsilon(\blacksquare)}(\mathrm{g})\mu_G(d\mathrm{g}),
%\sum_{\rho\in \widehat G} \dim V_\rho e^{-(\int_\blacksquare \upsilon )c_2(\rho)/2 }   \langle \Psi, \chi_\rho\rangle_{L^2(G, \mu_G)}. 
\end{equation}
or equivalently
$$
\mathbb{P}_{\operatorname{YM}}^{\operatorname{free}}\left(\mathbf{Hol}(\partial \blacksquare)\in d\mathrm{g}\right)=p_{\upsilon(\blacksquare)}(\mathrm{g})\mu_G(d\mathrm{g}).
$$
\end{lemma}
We observe that this expectation is independent of the gradient flow used to define the Yang--Mills measure. 
\begin{proof} Let us assume first that $\gamma $ and $\tilde \gamma$ are of type $I$. By definition of admissible flow rectangle and of the random holonomy, we have $\mathbf{Hol}(\partial \blacksquare)=\mathbf{Hol}(\gamma)^{-1}\mathbf{Hol}(\tilde{\gamma})$. According to Corollary~\ref{c:markov-holonomy-elementary}, one has
$$
\int_{\Omega\times G^{2g}}\Psi(\mathbf{Hol}(\partial \blacksquare))\mathbb{P}_{\text{YM}}^{\rm free}(d\omega,d\mathrm{b})=\int_{G^2}\Psi\left(\mathrm{g}_2^{-1}\mathrm{g}_1\right)p_{\mathscr{A}_{\tilde{\gamma}}(1)}(\mathrm{g}_1)p_{\upsilon(\blacksquare)}(\mathrm{g}_2\mathrm{g}_1^{-1})\mu_G^{\otimes 2}(d\mathrm{g}_1,d\mathrm{g}_2),
$$
which can be simplified as
$$
\int_{\Omega\times G^{2g}}\Psi(\mathbf{Hol}(\partial \blacksquare))\mathbb{P}_{\text{YM}}^{\rm free}(d\omega,d\mathrm{b})=\int_{G}\Psi\left(\mathrm{g}\right)p_{\upsilon(\blacksquare)}(\mathrm{g})\mu_G(d\mathrm{g}).
$$
Now let $\gamma$ and $\tilde \gamma$ be of type $II$. In that case, one has either $\mathbf{Hol}(\partial \blacksquare)=\mathbf{Hol}(\gamma)^{-1}\mathbf{Hol}(\tilde{\gamma})$ or $\mathbf{Hol}(\partial \blacksquare)=\mathrm{b}_a\mathbf{Hol}(\gamma)^{-1}\mathbf{Hol}(\tilde{\gamma})\mathrm{b}_a^{-1}$ (for some $\mathrm{b}_a\in G$). The same calculation goes along and yields the proof.
\end{proof}

\subsection{A warm-up example}
\label{ss:warmup-example}

Before computing the law of random holonomies along relevant curves for the Yang--Mills measure, let us first compute it in the case of elementary curves of type $\operatorname{I}$ as a warm-up:
\begin{lemma}\label{l:law-elementary}
 Let $\gamma$ be an elementary curve of type $\operatorname{I}$ such that $t\in[0,1]\mapsto \mathscr{A}_\gamma(t)$ is increasing and such that, for every $t\in[0,1]$, $\gamma(t)\in W^s(a_{2g+2})$. Then, one has
 $$
 \mathbb{M}_{\operatorname{YM}}\left(\mathbf{Hol}(\gamma)\in d\mathrm{g}\right)=p_{\mathbf{Hol},\gamma}(\mathrm{g})p_{\upsilon(\Sigma)-\mathscr{A}_\gamma(1)}(\mathrm{g})\mu_G(d\mathrm{g}),
 $$
 where
 $$
 p_{\mathbf{Hol},\gamma}(\mathrm{g}):=\sum_{\rho\in\widehat{G}} e^{-\frac{c_2(\rho)}{2}\mathscr{A}_\gamma(1)}\frac{\chi_\rho(\mathrm{g})}{\dim (V_\rho)^{2g-1}}.
 $$
\end{lemma}
Recall that, by construction, $\mathbf{Hol}(\gamma)$ can be viewed as the holonomy along the closed curve composed by $\gamma([0,1])$ and by the flow lines joining from $\gamma(0)$ and $\gamma(1)$ to the critical point $a_1$. Recall that these curves were denoted by $\mathcal{L}_{\gamma(0)}$ and $\mathcal{L}_{\gamma(1)}$. In the case of type $\operatorname{I}$ curves, the area delimited by these three curves is given by $\mathscr{A}_\gamma(1)$. The assumption that $\gamma(t)\in W^s(a_{2g+2})$ makes the combinatorics slightly simpler but it could be removed. 
\begin{proof} We let $\Psi:G\rightarrow\R$ be a bounded and measurable function. By definition, one has
$$
\int_{\operatorname{ker}(\iota_V)} \Psi(\mathbf{Hol}(\gamma))d\mathbb{M}_{\operatorname{YM}}=\int_{G^{2g}}\left(\int_{\Omega}\Psi(\mathbf{Hol}(\gamma))p_{\upsilon(\Sigma)-\mathscr{A}_r(1)}(\mathbf{Hol}(\gamma_r))\P(d\omega)\right)\mu_G^{\otimes 2g}(d\mathrm{b}),
$$
where both $\mathbf{Hol}(\gamma)$ and $\mathbf{Hol}(\gamma_r)$ are functions of $(\omega,\mathrm{b})$. In view of this expression, one can decompose $\gamma_r=\gamma_r^0\star \widetilde{\gamma}_r$ with $\gamma\preccurlyeq\widetilde{\gamma}_r$ and $\widetilde{\gamma}_r$ being also of type $\operatorname{I}$ (this is where we use that $\gamma(t)\in W^{s}(a_{2g+2})$ for every $t\in[0,1]$). See for example Figure \ref{fig:holonomy_law_1}.
\begin{figure}
    \centering
    \includegraphics[scale=0.2]{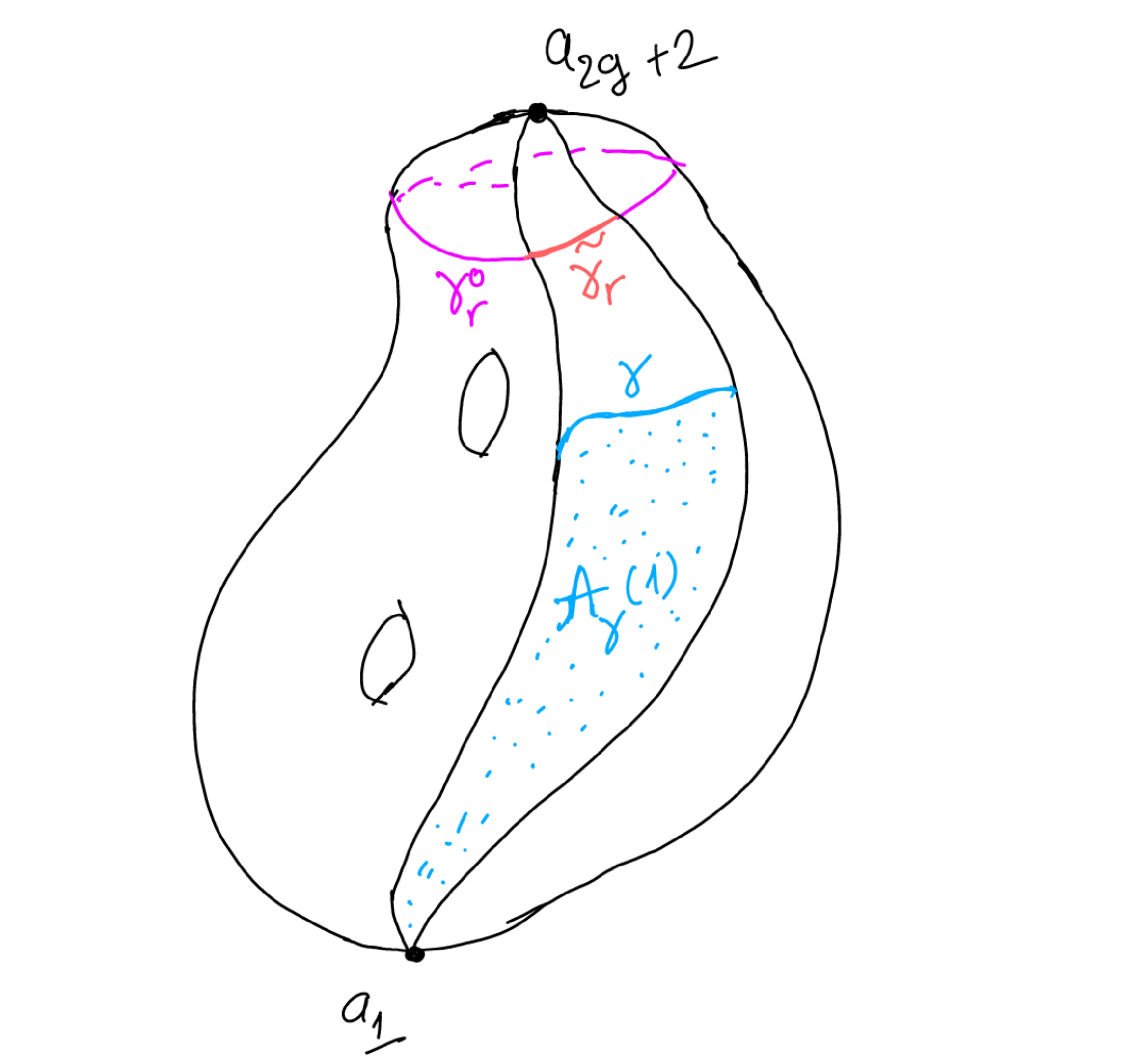} 
    \caption{Example for $\gamma,\gamma_r^0$ and $\widetilde{\gamma}_r$.}
    \label{fig:holonomy_law_1}
\end{figure} 
 By construction, one has $\mathbf{Hol}(\gamma_r)=\mathbf{Hol}(\widetilde{\gamma}_r)\mathbf{Hol}(\gamma_r^0)$ and the random variables $\mathbf{Hol}(\widetilde{\gamma}_r)$ and $\mathbf{Hol}(\gamma_r^0)$ are independent with respect to the probability space $\Omega$. Hence, one can split the integral over $\Omega$ as follows, for every fixed value $\mathrm{b}=(\mathrm{b}_a)_{\text{ind}(a)=1}\in G^{2g}$,
\begin{eqnarray*}
 \int_{\Omega}\Psi(\mathbf{Hol}(\gamma))p_{\upsilon(\Sigma)-\mathscr{A}_r(1)}(\mathbf{Hol}(\gamma_r))\P(d\omega)\hspace{7cm}\\
 =\int_{\Omega^2}\Psi(\mathbf{Hol}(\gamma)(\omega_2,\mathrm{b}))p_{\upsilon(\Sigma)-\mathscr{A}_r(1)}(\mathbf{Hol}(\widetilde{\gamma}_r)(\omega_2,\mathrm{b})\mathbf{Hol}(\gamma_r^0)(\omega_1,\mathrm{b}))\P(d\omega_1)\P(d\omega_2).
\end{eqnarray*}
Recall that $\mathbf{Hol}$ is constructed from the solutions of the stochastic differential equation~\eqref{e:SDE-holonomy-Strato} whose law is given by the heat kernel as stated in Theorem~\ref{t:local-holonomies}. Hence, one finds that, for every $\mathrm{h}\in G$ and for every $\mathrm{b}=(\mathrm{b}_a)_{\text{ind}(a)=1}\in G^{2g}$,
 \begin{eqnarray*}
  \int_{\Omega}p_{\upsilon(\Sigma)-\mathscr{A}_r(1)}(\mathrm{h}\mathbf{Hol}(\gamma_r^0)(\omega_1,\mathrm{b}))\P(d\omega_1)\hspace{7cm}\\
  =
  \int_{G^{4g}}p_{\upsilon(\Sigma)-\mathscr{A}_r(1)}\left(\mathrm{h}\mathrm{g}_{4g}\mathrm{b}_{\widetilde{a}(4g)}^{\widetilde{\varepsilon}(4g)}\ldots \mathrm{g}_1\mathrm{b}_{\widetilde{a}(1)}^{\widetilde{\varepsilon}(1)} \right) \prod_{j=1}^{4g}p_{\tau_j}(\mathrm{g}_j) \mu_G^{\otimes 4g}(d\mathrm{g}),
 \end{eqnarray*}
where $\sum_{j=1}^{4g}\tau_j=\mathscr{A}_r(1)-\mathscr{A}_{\widetilde{\gamma}_r}(1)$, and where, for every $1\leq j\leq 4g$, $\widetilde{a}(j)$ is a critical point of index $1$ and $\widetilde{\varepsilon}(j)=\pm 1$. These parameters are obtained from the decomposition of $\gamma_r^0$ into elementary pieces. From the invariance of the heat kernel by conjugation, this is also equal to
 \begin{eqnarray*}
  \int_{\Omega}p_{\upsilon(\Sigma)-\mathscr{A}_r(1)}(\mathrm{h}\mathbf{Hol}(\gamma_r^0)(\omega_1,\mathrm{b}))\P(d\omega_1)\hspace{7cm}\\
  =
  \int_{G^{4g}}p_{\upsilon(\Sigma)-\mathscr{A}_r(1)}\left(\mathrm{h}\mathrm{g}_{4g}\ldots \mathrm{g}_1\mathrm{b}_{\widetilde{a}(4g)}^{\widetilde{\varepsilon}(4g)}\ldots \mathrm{b}_{\widetilde{a}(1)}^{\widetilde{\varepsilon}(1)} \right) \prod_{j=1}^{4g}p_{\tau_j}(\mathrm{g}_j) \mu_G^{\otimes 4g}(d\mathrm{g}).
 \end{eqnarray*}
 Hence, using Lemma~\ref{l:easy-useful-heat-kernel} $4g$ times, one finds
 $$
 \int_{\Omega}p_{\upsilon(\Sigma)-\mathscr{A}_r(1)}(\mathrm{h}\mathbf{Hol}(\gamma_r^0)(\omega_1,\mathrm{b}))\P(d\omega_1)
  =
  p_{\upsilon(\Sigma)-\mathscr{A}_{\widetilde{\gamma}_r}(1)}\left(\mathrm{h}\mathrm{b}_{\widetilde{a}(4g)}^{\widetilde{\varepsilon}(4g)}\ldots \mathrm{b}_{\widetilde{a}(1)}^{\widetilde{\varepsilon}(1)} \right),
 $$
 from which we infer
 \begin{eqnarray*}
 \int_{\Omega}\Psi(\mathbf{Hol}(\gamma))p_{\upsilon(\Sigma)-\mathscr{A}_r(1)}(\mathbf{Hol}(\gamma_r))\P(d\omega)\hspace{7cm}
 \\
 =\int_{\Omega}\Psi(\mathbf{Hol}(\gamma)(\omega_2,\mathrm{b}))p_{\upsilon(\Sigma)-\mathscr{A}_{\widetilde{\gamma}_r}(1)}\left(\mathbf{Hol}(\widetilde{\gamma}_r)(\omega_2,\mathrm{b})\mathrm{b}_{\widetilde{a}(4g)}^{\widetilde{\varepsilon}(4g)}\ldots \mathrm{b}_{\widetilde{a}(1)}^{\widetilde{\varepsilon}(1)}\right)\P(d\omega_2),
 \end{eqnarray*}
which is valid for any $\mathrm{b}=(\mathrm{b}_a)_{\text{ind}(a)=1}\in G^{2g}$. Recalling the definition of the holonomy of an elementary curve of type $\operatorname{I}$, one knows that $\mathbf{Hol}(\widetilde{\gamma}_r)$ (resp. $\mathbf{Hol}(\gamma)$) is of the form $\mathrm{g}_{\widetilde{\gamma}_r}(1)$ (resp. $\mathrm{g}_{\gamma}(1)$). Hence by construction, $\mathbf{b}:=\mathrm{b}_{\widetilde{a}(4g)}^{\widetilde{\varepsilon}(4g)}\ldots \mathrm{b}_{\widetilde{a}(1)}^{\widetilde{\varepsilon}(1)}$ is equal to $\mathrm{b}_{a(t_{\ell+4g})}^{\varepsilon_{\ell+4g}}\ldots\mathrm{b}_{a(t_{\ell+2})}^{\varepsilon_{\ell+2}}\mathrm{b}_{a(t_{\ell+1})}^{\varepsilon_{\ell+1}}$, where the indices $\ell+1\leq k\leq \ell+4g$ are understood modulo $4g$. Letting $\mathbf{b}:=\mathrm{b}_{\widetilde{a}(4g)}^{\widetilde{\varepsilon}(4g)}\ldots \mathrm{b}_{\widetilde{a}(1)}^{\widetilde{\varepsilon}(1)}$, one gets
$$
\int_{\Omega}\Psi(\mathbf{Hol}(\gamma))p_{\upsilon(\Sigma)-\mathscr{A}_r(1)}(\mathbf{Hol}(\gamma_r))\P(d\omega)
 =\int_{\Omega}\Psi(\mathrm{g}_{\gamma}(1))p_{\upsilon(\Sigma)-\mathscr{A}_{\widetilde{\gamma}_r}(1)}\left(\mathrm{g}_{\widetilde{\gamma}_r}(1)\mathbf{b}\right)\P(d\omega).
$$
We can now apply Corollary~\ref{c:markov-holonomy-elementary} to compute this integral. It yields, for every $\mathrm{b}\in G^{2g}$,
\begin{eqnarray*}
\int_{\Omega}\Psi(\mathbf{Hol}(\gamma))p_{\upsilon(\Sigma)-\mathscr{A}_r(1)}(\mathbf{Hol}(\gamma_r))\P(d\omega)
 \hspace{7cm}\\
 =\int_{G^2}\Psi(\mathrm{g}_1)p_{\upsilon(\Sigma)-\mathscr{A}_{\widetilde{\gamma}_r}(1)}\left(\mathrm{g}_2\mathbf{b}\right)p_{\mathscr{A}_\gamma(1)}(\mathrm{g}_1)p_{\mathscr{A}_{\widetilde{\gamma}_r}(1)-\mathscr{A}_\gamma(1)}(\mathrm{g}_2\mathrm{g}_1^{-1})\mu_G^{\otimes 2}(d\mathrm{g}_1,d\mathrm{g}_2)\\
 =\int_{G}\Psi(\mathrm{g}_1)p_{\upsilon(\Sigma)-\mathscr{A}_{\gamma}(1)}\left(\mathrm{g}_1\mathbf{b}\right)p_{\mathscr{A}_\gamma(1)}(\mathrm{g}_1)\mu_G(d\mathrm{g}_1),
\end{eqnarray*}
where we applied one more time Lemma~\ref{l:easy-useful-heat-kernel} to get the second equality. In order to conclude, one needs to integrate over the $\mathrm{b}$-variable and this can be achieved using Lemma~\ref{c:multi-orthogonality}.
\end{proof}

\subsection{Holonomies along boundaries of small disks with no critical points}

We will now compute the law of random holonomies along continuous and piecewise $\mathcal{C}^1$ curves $\gamma:[0,1]\rightarrow \Sigma$ which are the concatenation of elementary curves and which are the \emph{boundary of a small open set homeomorphic to a disk not containing any critical point}. These correspond to the admissible disks from Theorem~\ref{t:def-YM-general}. More precisely, we will make the following simplifying assumptions:
\begin{itemize}
 \item $\gamma^1,\gamma^2:[0,1]\rightarrow \Sigma$ are $\mathcal{C}^1$ curves such that $\gamma^j([0,1])$ does not contain any critical point;
 \item $\gamma^1$ and $\gamma^2$ are primitive and verify that $t\in[0,1]\mapsto\mathscr{A}_j(t):=\mathscr{A}_{\gamma^j}(t)$ is increasing for each $j\in\{1,2\}$;
 \item $\gamma^1(0)=\gamma^2(0)$ and $\gamma^1(1)=\gamma^2(1)$.
\end{itemize}
Our goal is to compute the random holonomy along the piecewise $\mathcal{C}^1$ curve
\begin{equation}\label{e:decomposition-regular}
\gamma=\gamma^1\star\overline{\gamma^2}.
\end{equation}
We also suppose that $\gamma$ is small enough so that $\gamma$ is the oriented boundary of a small domain $D_\gamma$ that is homeomorphic to a disk not containing any critical point of $f$. We make the assumption that 
$$
\upsilon(D_\gamma):=\mathscr{A}_2(1)-\mathscr{A}_1(1)>0.
$$
Finally, we will suppose that $\gamma$ satisfies one of the following three hypothesis: 
\begin{enumerate}
 \item[(H1)] Both $\gamma^j$ are of type $\operatorname{I}$ and do not intersect any $W^s(a)$ with $a$ a critical point of index $1$. In that case, one has $\gamma^1\preccurlyeq\gamma^2$ and one can find a subinterval $I_\gamma$ of $\R/\Z$ such that $\gamma^2\preccurlyeq\gamma_r|_{I_\gamma}=:\widetilde{\gamma}_r$. Observe that $I_\gamma$ contains no point $t_1<t_2<\ldots<t_{4g}$ corresponding (in the Morse chart) to the intersection of $\gamma_r$ with some unstable manifold of dimension $1$. In that case 
 $$
 \mathbf{Hol}(\gamma)=\mathrm{g}_{\gamma^2}(1)^{-1}\mathrm{g}_{\gamma^1}(1).
 $$
 See Figure~\ref{fig:holonomy_h1} for an example.
 \begin{figure}
    \centering
    \includegraphics[scale=0.4]{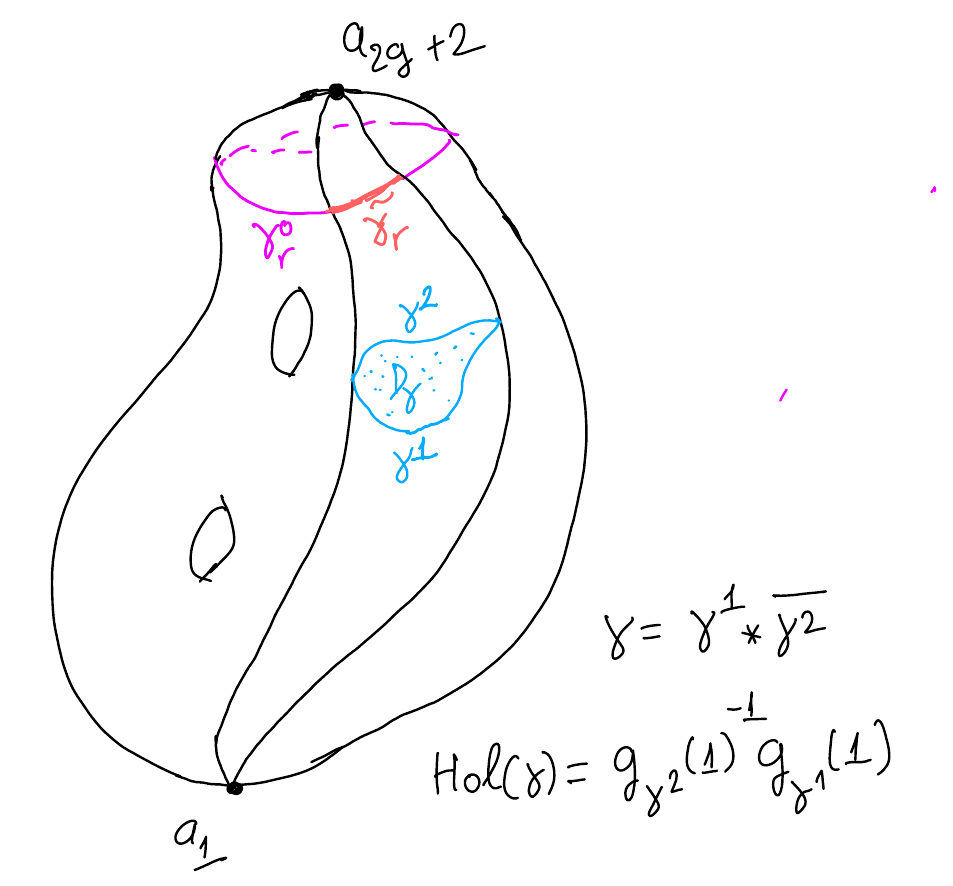} 
    \caption{Example for (H1).}
    \label{fig:holonomy_h1}
\end{figure} 

 \item[(H2)] Both $\gamma^j$ are of type $\operatorname{I}$ and they intersect in their interior some $W^s(a(\gamma))$ with $a(\gamma)$ a critical point of index $1$. In that case, one also has $\gamma^1\preccurlyeq\gamma^2$. One can write $a(\gamma)=a(t_j)=a(t_k)$ for some $1\leq k\neq \ell\leq 4g$. Then, one verifies that each $\gamma^j$ can be decomposed as $\gamma^j_k\star\gamma_\ell^j$ with $\gamma^1_k\preccurlyeq \gamma^2_k\preccurlyeq \gamma_r|_{I_{\gamma,k}}=:\widetilde{\gamma}_{r,k}$ and $\gamma^1_\ell\preccurlyeq\gamma^2_\ell\preccurlyeq \gamma_r|_{I_{\gamma,\ell}}=:\widetilde{\gamma}_{r,\ell}$ where the interval $I_{\gamma,\ell}$ (resp. $I_{\gamma,k}$) contains $t_{\ell}$ (resp. $t_k$) as an endpoint, say left (resp. right). In that case, one has 
 $$
 \mathbf{Hol}(\gamma)=\mathrm{g}_{\gamma^2}(1)^{-1}\mathrm{g}_{\gamma^1}(1)=\mathrm{g}_{\gamma_k^2}(1)^{-1}\mathrm{g}_{\gamma_\ell^2}(1)^{-1}\mathrm{g}_{\gamma_\ell^1}(1)\mathrm{g}_{\gamma_k^1}(1).
 $$
  See Figure~\ref{fig:holonomy_h2} for an example.
 \begin{figure}
    \centering
    \includegraphics[scale=0.4]{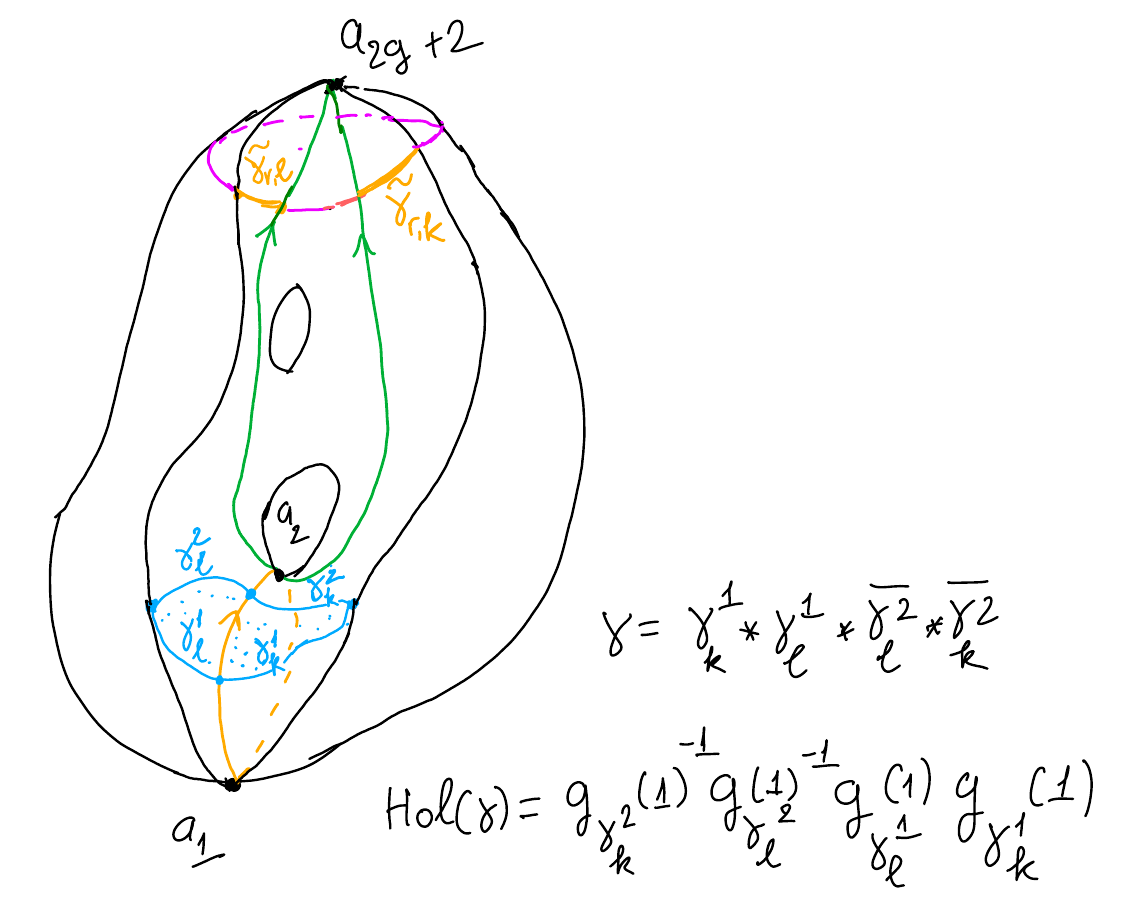} 
    \caption{Example for (H2).}
    \label{fig:holonomy_h2}
\end{figure} 
  \item[(H3)] Both $\gamma^j$ are the concatenation of two type $\operatorname{II}$ curves $\gamma_1^j\star\gamma_2^j$ such that $\gamma_1^j(1)=\gamma_2^j(0)$ belongs to some $W^u(a(\gamma))$ with $a(\gamma)$ a critical point of index $1$ (that does not depend on $j$). In that case, one has still $\gamma_1^1\preccurlyeq\gamma_1^2$ and $\gamma_2^1\preccurlyeq\gamma_2^2$. Moreover, there exists some $1\leq k\leq 4g$ such that $a(t_k)=a(\gamma)$ and $\gamma_r(t_k)$ lies on the forward orbit of $\gamma_1^j(1)$ under the gradient flow. In particular, one can find an interval $I_{\gamma,1}$ (resp. $I_{\gamma,2}$) containing $t_k$ as a right (resp. left) endpoint such that $\gamma_i^1\preccurlyeq\gamma_i^2\preccurlyeq\gamma_r|_{I_{\gamma,i}}=:\widetilde{\gamma}_{r,i}$. In that case, one has
 $$
 \mathbf{Hol}(\gamma)=\mathrm{g}_{\gamma_1^2}(1)^{-1}\mathrm{b}_{a(t_k)}^{-\varepsilon_k}\mathrm{g}_{\gamma_2^2}(1)^{-1}\mathrm{g}_{\gamma_2^1}(1)\mathrm{b}_{a(t_k)}^{\varepsilon_k}\mathrm{g}_{\gamma_1^1}(1).
 $$
  See Figure~\ref{fig:holonomy_h3} for an example.
 \begin{figure}
    \centering
    \includegraphics[scale=0.4]{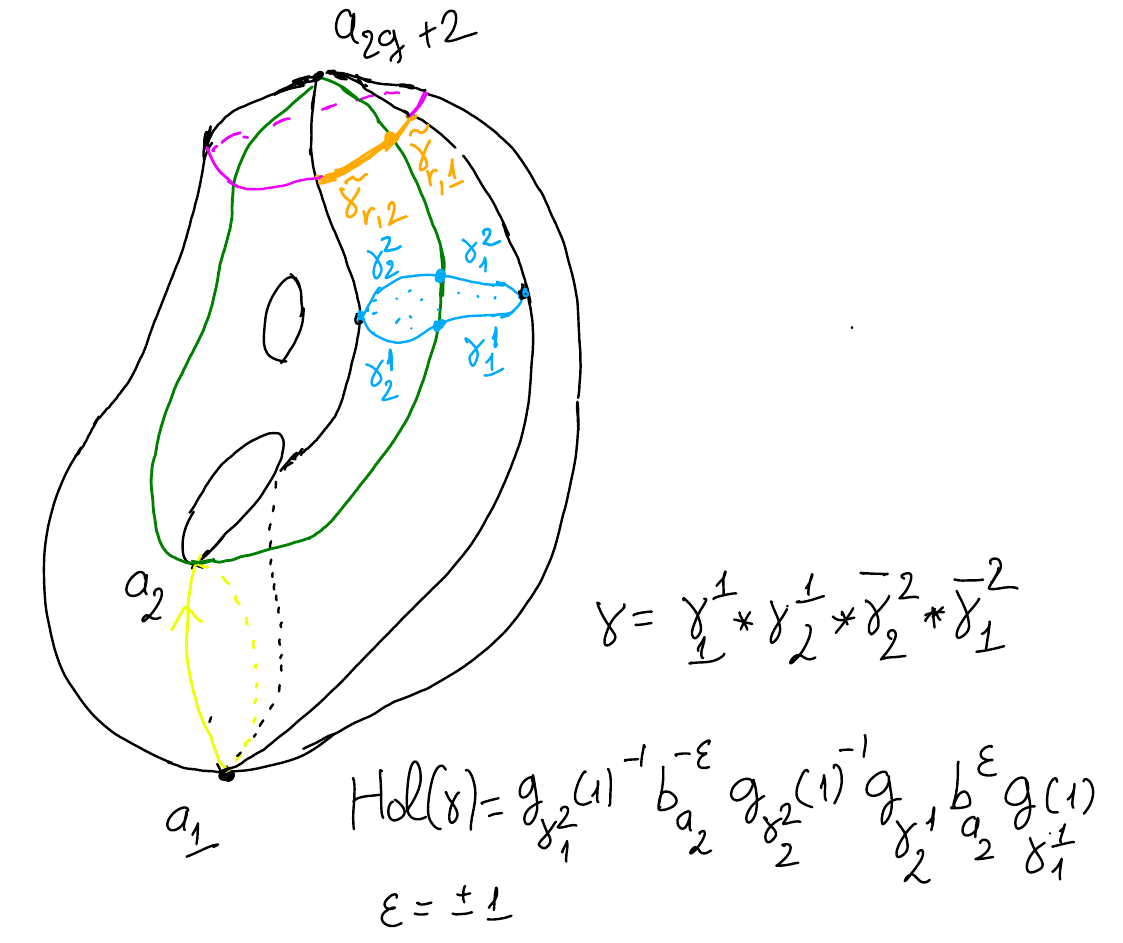} 
    \caption{Example for (H3).}
    \label{fig:holonomy_h3}
\end{figure} 
\end{enumerate}

With these conventions at hand, we are ready to prove the main result of this section:
\begin{thm}\label{t:holonomy-disk-regular} Let $\gamma:[0,1]\rightarrow\Sigma$ be a curve verifying the above properties, meaning either $(\operatorname{H1})$, $(\operatorname{H2})$ or $(\operatorname{H3})$. Then, for any bounded and measurable function $\Psi:G\rightarrow\R$, one has
 $$
 \int_{\operatorname{ker}(\iota_V)}\Psi\left(\mathbf{Hol}(\gamma)\right)d\mathbb{M}_{\operatorname{YM}}=\int_{G}\Psi(\mathrm{g}) p_{\upsilon(D_\gamma)}(\mathrm{g})p_{\mathbf{Hol},D_\gamma}(\mathrm{g})\mu_G(d\mathrm{g}),
 $$
 where
 $$
 p_{\mathbf{Hol},D_\gamma}:=\sum_{\rho\in\widehat{G}}(\operatorname{dim}(V_\rho))^{1-2g}e^{-\frac{c_2(\rho)}{2}\upsilon(\Sigma\setminus D_\gamma)}\chi_\rho.
 $$
\end{thm}
In a more compact way, we simply write
$$
\boxed{\mathbb{M}_{\operatorname{YM}}\left(\mathbf{Hol}(\gamma)\in d\mathrm{g}\right)=p_{\upsilon(D_\gamma)}(\mathrm{g})p_{\mathbf{Hol},D_\gamma}(\mathrm{g})\mu_G(d\mathrm{g}).}
$$
We emphasize that we compute here the law of $\mathbf{Hol}(\gamma)$ and not only of its conjugacy class. As we shall see in the proof, the argument relies on (heavy) Gaussian integration computations. Yet, everything is very explicit by construction of the Yang-Mills measure and of the properties of Gaussian holonomies as stated in Theorem~\ref{t:local-holonomies} and in Corollary~\ref{c:markov-holonomy-elementary}. Therefore, one could in principle deal with more general curves (including curves surrounding critical points) at the expense of performing much more involved computations. 
\begin{proof}
 We begin with the case where $\gamma$ satisfies property (H1) which is the closest situation to the warm-up Lemma~\ref{l:law-elementary}. In fact, given a bounded and measurable function 
 $\Psi:G\rightarrow\R$,  we can argue as in the proof of Lemma~\ref{l:law-elementary} to get, for all $\mathrm{b}\in G^{2g}$,
\begin{equation}\label{e:first-simplification-law-holonomy}
 \int_\Omega \Psi(\mathbf{Hol}(\gamma))p_{\upsilon(\Sigma)-\mathscr{A}_{\gamma_r}(1)}(\mathbf{Hol}(\gamma_r))\P(d\omega)=\int_\Omega \Psi(\mathbf{Hol}(\gamma))p_{\upsilon(\Sigma)-\mathscr{A}_{\widetilde{\gamma}_{r}}(1)}(\widetilde{\mathbf{Hol}}(\widetilde{\gamma}_r))\P(d\omega),
\end{equation}
where $\widetilde{\mathbf{Hol}}(\widetilde{\gamma}_r)=\mathrm{b}_{a(t_{4g})}^{\varepsilon_{4g}}\ldots\mathrm{b}_{a(t_{2})}^{\varepsilon_{2}}\mathrm{b}_{a(t_{1})}^{\varepsilon_{1}}\mathrm{g}_{\widetilde{\gamma}_r}(1).$ We can now make use of Corollary~\ref{c:markov-holonomy-elementary} which tells us that
\begin{eqnarray*}
\mathbb{P}\left(\mathrm{g}_{\gamma^1}(1)\in d\mathrm{g}_1,\mathrm{g}_{\gamma^2}(1)\in d\mathrm{g}_2, \mathrm{g}_{\widetilde{\gamma}_{r}}(1)\in d\mathrm{g}_3\right)=p_{\mathscr{A}_{1}}(\mathrm{g}_1) p_{\mathscr{A}_{2}(1)-\mathscr{A}_{1}(1)}(\mathrm{g}_2\mathrm{g}_1^{-1})\\
\times p_{\mathscr{A}_{\widetilde{\gamma}_{r}}(1)-\mathscr{A}_{2}(1)}(\mathrm{g}_3\mathrm{g}_2^{-1})\mu_G^{\otimes 3}(d\mathrm{g}_1,d\mathrm{g}_2,d\mathrm{g}_3).
\end{eqnarray*}
This yields 
\begin{eqnarray*}
\int_{G^{2g}}\int_\Omega \Psi(\mathbf{Hol}(\gamma))p_{\upsilon(\Sigma)-\mathscr{A}_{\gamma_r}(1)}(\mathbf{Hol}(\gamma_r))\P(d\omega)\mu_G^{\otimes 2g}(d\mathrm{b}) \hspace{5cm}\\
=\int_{G^{2g}}\int_{G^3} \Psi(\mathrm{g}_2^{-1}\mathrm{g}_1)p_{\mathscr{A}_{1}}(\mathrm{g}_1) p_{\mathscr{A}_{2}(1)-\mathscr{A}_{1}(1)}(\mathrm{g}_2\mathrm{g}_1^{-1})p_{\mathscr{A}_{\widetilde{\gamma}_{r}}(1)-\mathscr{A}_{2}(1)}(\mathrm{g}_3\mathrm{g}_2^{-1})\\
\times p_{\upsilon(\Sigma)-\mathscr{A}_{\widetilde{\gamma}_{r}}(1)}\left(\mathrm{b}_{a(t_{4g})}^{\varepsilon_{4g}}\ldots\mathrm{b}_{a(t_{2})}^{\varepsilon_{2}}\mathrm{b}_{a(t_{1})}^{\varepsilon_{1}}\mathrm{g}_3\right)\mu_G^{\otimes 2g}(d\mathrm{b})\mu_G^{\otimes 3}(d\mathrm{g}_1,d\mathrm{g}_2,d\mathrm{g}_3).
\end{eqnarray*}
Hence, letting $\mathrm{g}=\mathrm{g}_2^{-1}\mathrm{g}_1$, the law $\mathbb{M}_{\operatorname{YM}}(\mathbf{Hol}(\gamma)\in d\mathrm{g})$ is given by
$$
p_{\upsilon(D_\gamma)}(\mathrm{g})\int_{G^{2g+2}} p_{\mathscr{A}_{1}}(\mathrm{g}_2\mathrm{g}) p_{\mathscr{A}_{\widetilde{\gamma}_{r}}(1)-\mathscr{A}_{2}(1)}(\mathrm{g}_3\mathrm{g}_2^{-1}) p_{\upsilon(\Sigma)-\mathscr{A}_{\widetilde{\gamma}_{r}}(1)}\left(\mathrm{b}_{a(t_{4g})}^{\varepsilon_{4g}}\ldots\mathrm{b}_{a(t_{2})}^{\varepsilon_{2}}\mathrm{b}_{a(t_{1})}^{\varepsilon_{1}}\mathrm{g}_3\right)d\mu_G^{\otimes 2g+2},
$$
where we used that $\upsilon(D_\gamma)=\mathscr{A}_2(1)-\mathscr{A}_1(1)$. Using Lemma~\ref{l:easy-useful-heat-kernel} twice to integrate over $\mathrm{g}_2$ and $\mathrm{g}_3$, one finds
$$
\mathbb{M}_{\operatorname{YM}}(\mathbf{Hol}(\gamma)\in d\mathrm{g})=p_{\upsilon(D_\gamma)}(\mathrm{g})\int_{G^{2g}}p_{\upsilon(\Sigma\setminus D_\gamma)}\left(\mathrm{b}_{a(t_{4g})}^{\varepsilon_{4g}}\ldots\mathrm{b}_{a(t_{2})}^{\varepsilon_{2}}\mathrm{b}_{a(t_{1})}^{\varepsilon_{1}}\mathrm{g}\right)d\mu_G^{\otimes 2g}(d\mathrm{b}).
$$
Using Lemma~\ref{c:multi-orthogonality}, one finally obtains the expected formula:
$$
\mathbb{M}_{\operatorname{YM}}(\mathbf{Hol}(\gamma)\in d\mathrm{g})=p_{\upsilon(D_\gamma)}(\mathrm{g})\sum_{\rho\in\widehat{G}}(\text{dim}(V_\rho))^{1-2g}e^{-\frac{c_2(\rho)}{2}\upsilon(\Sigma\setminus D_\gamma)}\chi_\rho(\mathrm{g}).
$$

We now deal with the case (H2) which is slightly more involved but which goes through along similar lines. Indeed, one finds one more time that
\begin{eqnarray*}
\int_{G^{2g}}\int_\Omega \Psi(\mathbf{Hol}(\gamma))p_{\upsilon(\Sigma)-\mathscr{A}_{\gamma_r}(1)}(\mathbf{Hol}(\gamma_r))\P(d\omega)\mu_G^{\otimes 2g}(d\mathrm{b}) \hspace{5cm}\\
=\int_{G^{2g+6}}\Psi(\mathrm{g}_{2,k}^{-1}\mathrm{g}_{2,\ell}^{-1}\mathrm{g}_{1,\ell}\mathrm{g}_{1,k})p_{\mathscr{A}_{\gamma_k^1}(1)}(\mathrm{g}_{1,k})p_{\mathscr{A}_{\gamma_\ell^1}(1)}(\mathrm{g}_{1,\ell}) p_{\mathscr{A}_{\gamma_k^2}(1)-\mathscr{A}_{\gamma_k^1}(1)}(\mathrm{g}_{2,k}\mathrm{g}_{1,k}^{-1})
\\
\times p_{\mathscr{A}_{\gamma_\ell^2}(1)-\mathscr{A}_{\gamma_\ell^1}(1)}(\mathrm{g}_{2,\ell}\mathrm{g}_{1,\ell}^{-1}) p_{\mathscr{A}_{\widetilde{\gamma}_{r,k}}(1)-\mathscr{A}_{\gamma^2_k}(1)}(\mathrm{g}_{3,k}\mathrm{g}_{2,k}^{-1})p_{\mathscr{A}_{\widetilde{\gamma}_{r,\ell}}(1)-\mathscr{A}_{\gamma^2_\ell}(1)}(\mathrm{g}_{3,\ell}\mathrm{g}_{2,\ell}^{-1})\\
\times p_{\upsilon(\Sigma)-\mathscr{A}_{\widetilde{\gamma}_{r}}(1)}\left(\mathrm{b}_{a(t_{4g})}^{\varepsilon_{4g}}\ldots\ldots\mathrm{b}_{a(t_{\ell})}^{\varepsilon_{\ell}}\mathrm{g}_{3,\ell}\mathrm{b}_{a(t_{\ell-1})}^{\varepsilon_{\ell-1}}\ldots \mathrm{g}_{3,k}\mathrm{b}_{a(t_{k})}^{\varepsilon_{k}}\ldots\mathrm{b}_{a(t_{1})}^{\varepsilon_{1}}\right)\mu_G^{\otimes 2g+6}(d\mathrm{b},d\mathrm{g}_1,d\mathrm{g}_2,d\mathrm{g}_3).
\end{eqnarray*}
Recall now that $a(\gamma)=a(t_k)=a(t_\ell)$ and that $\varepsilon_k+\varepsilon_\ell=0$ so that, if one makes the change of variable $\mathrm{g}_{3,k}\mathrm{b}_{a(t_{k})}^{\varepsilon_{k}}=\widetilde{\mathrm{b}}_{a(t_{k})}^{\varepsilon_{k}}$, one obtains
\begin{eqnarray*}
\int_{G^{2g}}\int_\Omega \Psi(\mathbf{Hol}(\gamma))p_{\upsilon(\Sigma)-\mathscr{A}_{\gamma_r}(1)}(\mathbf{Hol}(\gamma_r))\P(d\omega)\mu_G^{\otimes 2g}(d\mathrm{b}) \hspace{5cm}\\
=\int_{G^{2g+6}}\Psi(\mathrm{g}_{2,k}^{-1}\mathrm{g}_{2,\ell}^{-1}\mathrm{g}_{1,\ell}\mathrm{g}_{1,k})p_{\mathscr{A}_{\gamma_k^1}(1)}(\mathrm{g}_{1,k})p_{\mathscr{A}_{\gamma_\ell^1}(1)}(\mathrm{g}_{1,\ell}) p_{\mathscr{A}_{\gamma_k^2}(1)-\mathscr{A}_{\gamma_k^1}(1)}(\mathrm{g}_{2,k}\mathrm{g}_{1,k}^{-1})
\\
\times p_{\mathscr{A}_{\gamma_\ell^2}(1)-\mathscr{A}_{\gamma_\ell^1}(1)}(\mathrm{g}_{2,\ell}\mathrm{g}_{1,\ell}^{-1}) p_{\mathscr{A}_{\widetilde{\gamma}_{r,k}}(1)-\mathscr{A}_{\gamma^2_k}(1)}(\mathrm{g}_{3,k}\mathrm{g}_{2,k}^{-1})p_{\mathscr{A}_{\widetilde{\gamma}_{r,\ell}}(1)-\mathscr{A}_{\gamma^2_\ell}(1)}(\mathrm{g}_{3,\ell}\mathrm{g}_{2,\ell}^{-1})\\
\times p_{\upsilon(\Sigma)-\mathscr{A}_{\widetilde{\gamma}_{r}}(1)}\left(\mathrm{b}_{a(t_{\ell-1})}^{\varepsilon_{\ell-1}}\ldots\mathrm{b}_{a(t_{1})}^{\varepsilon_{1}}\mathrm{b}_{a(t_{4g})}^{\varepsilon_{4g}}\ldots\mathrm{b}_{a(t_{\ell})}^{\varepsilon_{\ell}}\mathrm{g}_{3,k}\mathrm{g}_{3,\ell}\right)\mu_G^{\otimes 2g+6}(d\mathrm{b},d\mathrm{g}_1,d\mathrm{g}_2,d\mathrm{g}_3),
\end{eqnarray*}
where we also used the invariance of the heat kernel by conjugation. Integrating over the $\mathrm{g}_3$ variables and using Lemma~\ref{l:easy-useful-heat-kernel}, this can be further simplified as
\begin{eqnarray*}
\int_{G^{2g}}\int_\Omega \Psi(\mathbf{Hol}(\gamma))p_{\upsilon(\Sigma)-\mathscr{A}_{\gamma_r}(1)}(\mathbf{Hol}(\gamma_r))\P(d\omega)\mu_G^{\otimes 2g}(d\mathrm{b}) \hspace{5cm}\\
=\int_{G^{2g+4}}\Psi(\mathrm{g}_{2,k}^{-1}\mathrm{g}_{2,\ell}^{-1}\mathrm{g}_{1,\ell}\mathrm{g}_{1,k})p_{\mathscr{A}_{\gamma_k^1}(1)}(\mathrm{g}_{1,k})p_{\mathscr{A}_{\gamma_\ell^1}(1)}(\mathrm{g}_{1,\ell})  p_{\mathscr{A}_{\gamma_k^2}(1)-\mathscr{A}_{\gamma_k^1}(1)}(\mathrm{g}_{2,k}\mathrm{g}_{1,k}^{-1})p_{\mathscr{A}_{\gamma_\ell^2}(1)-\mathscr{A}_{\gamma_\ell^1}(1)}(\mathrm{g}_{2,\ell}\mathrm{g}_{1,\ell}^{-1}) \\
\times p_{\upsilon(\Sigma)-\mathscr{A}_{\gamma^2(1)}}\left(\mathrm{b}_{a(t_{\ell-1})}^{\varepsilon_{\ell-1}}\ldots\mathrm{b}_{a(t_{1})}^{\varepsilon_{1}}\mathrm{b}_{a(t_{4g})}^{\varepsilon_{4g}}\ldots\mathrm{b}_{a(t_{\ell})}^{\varepsilon_{\ell}}\mathrm{g}_{2,k}\mathrm{g}_{2,\ell}\right)\mu_G^{\otimes 2g+4}(d\mathrm{b},d\mathrm{g}_1,d\mathrm{g}_2).
\end{eqnarray*}
We now make the change of variables $\mathrm{g}_{2,\ell}\mathrm{g}_{2,k}=\widetilde{\mathrm{g}}_{2,k}$ and $\widetilde{\mathrm{g}}_{1,k}=\mathrm{g}_{1,\ell}\mathrm{g}_{1,k}$. After simplification, we obtain
\begin{eqnarray*}
\int_{G^{2g}}\int_\Omega \Psi(\mathbf{Hol}(\gamma))p_{\upsilon(\Sigma)-\mathscr{A}_{\gamma_r}(1)}(\mathbf{Hol}(\gamma_r))\P(d\omega)\mu_G^{\otimes 2g}(d\mathrm{b}) \hspace{5cm}\\
=\int_{G^{2g+4}}\Psi(\widetilde{\mathrm{g}}_{2,k}^{-1}\widetilde{\mathrm{g}}_{1,k})p_{\mathscr{A}_{\gamma_\ell^1}(1)}(\mathrm{g}_{1,\ell})p_{\mathscr{A}_{\gamma_k^1}(1)}(\mathrm{g}_{1,\ell}^{-1}\widetilde{\mathrm{g}}_{1,k}) p_{\mathscr{A}_{\gamma_k^2}(1)-\mathscr{A}_{\gamma_k^1}(1)}(\mathrm{g}_{2,\ell}^{-1}\widetilde{\mathrm{g}}_{2,k}\widetilde{\mathrm{g}}_{1,k}^{-1}\mathrm{g}_{1,\ell}) \\
\times p_{\mathscr{A}_{\gamma_\ell^2}(1)-\mathscr{A}_{\gamma_\ell^1}(1)}(\mathrm{g}_{2,\ell}\mathrm{g}_{1,\ell}^{-1}) p_{\upsilon(\Sigma)-\mathscr{A}_{\gamma^2(1)}}\left(\mathrm{b}_{a(t_{\ell-1})}^{\varepsilon_{\ell-1}}\ldots\mathrm{b}_{a(t_{1})}^{\varepsilon_{1}}\mathrm{b}_{a(t_{4g})}^{\varepsilon_{4g}}\ldots\mathrm{b}_{a(t_{\ell})}^{\varepsilon_{\ell}}\widetilde{\mathrm{g}}_{2,k}\right)\mu_G^{\otimes 2g+4}(d\mathrm{b},d\mathrm{g}_1,d\mathrm{g}_2).
\end{eqnarray*}
%where we used the invariance of the Haar measure by conjugation along the $\mathrm{b}$ variable. We now integrate over the variables $\mathrm{g}_{2,\ell}$ and $\mathrm{g}_{1,\ell}$ (in this order). 
Thanks to Lemma~\ref{l:easy-useful-heat-kernel} one more time, we obtain the simplified expression
\begin{eqnarray*}
\int_{G^{2g}}\int_\Omega \Psi(\mathbf{Hol}(\gamma))p_{\upsilon(\Sigma)-\mathscr{A}_{\gamma_r}(1)}(\mathbf{Hol}(\gamma_r))\P(d\omega)\mu_G^{\otimes 2g}(d\mathrm{b}) \hspace{5cm}\\
=\int_{G^{2g+2}}\Psi(\widetilde{\mathrm{g}}_{2}^{-1}\widetilde{\mathrm{g}}_{1})p_{\mathscr{A}_{\gamma^1}(1)}(\widetilde{\mathrm{g}}_{1})p_{\mathscr{A}_{\gamma^2}(1)-\mathscr{A}_{\gamma^1}(1)}(\widetilde{\mathrm{g}}_{2}\widetilde{\mathrm{g}}_{1}^{-1}) \hspace{4cm}\\
\times p_{\upsilon(\Sigma)-\mathscr{A}_{\gamma^2(1)}}\left(\mathrm{b}_{a(t_{\ell-1})}^{\varepsilon_{\ell-1}}\ldots\mathrm{b}_{a(t_{1})}^{\varepsilon_{1}}\mathrm{b}_{a(t_{4g})}^{\varepsilon_{4g}}\ldots\mathrm{b}_{a(t_{\ell})}^{\varepsilon_{\ell}}\widetilde{\mathrm{g}}_{2}\right)\mu_G^{\otimes 2g+2}(d\mathrm{b},d\widetilde{\mathrm{g}}_1,d\widetilde{\mathrm{g}}_2).
\end{eqnarray*}
Hence, we are left with the same calculation as in (H1) which leads to the expected formula for the law of the holonomy.

Finally, we are left with the case of (H3) which is of slightly different nature as the disk $D_\gamma$ is now crossed by a one-dimensional unstable manifold which has to be taken into account in the computation. Again, one has 
\begin{eqnarray*}
\int_{G^{2g}}\int_\Omega \Psi(\mathbf{Hol}(\gamma))p_{\upsilon(\Sigma)-\mathscr{A}_{\gamma_r}(1)}(\mathbf{Hol}(\gamma_r))\P(d\omega)\mu_G^{\otimes 2g}(d\mathrm{b}) \hspace{5cm}\\
=\int_{G^{2g+6}}\Psi(\mathrm{g}_{2,1}^{-1}\mathrm{b}_{a(t_{k})}^{-\varepsilon_{k}}\mathrm{g}_{2,2}^{-1}\mathrm{g}_{1,2}\mathrm{b}_{a(t_{k})}^{\varepsilon_{k}}\mathrm{g}_{1,1})p_{\mathscr{A}_{\gamma_1^1}(1)}(\mathrm{g}_{1,1})p_{\mathscr{A}_{\gamma_2^1}(1)}(\mathrm{g}_{1,2}) p_{\mathscr{A}_{\gamma_1^2}(1)-\mathscr{A}_{\gamma_1^1}(1)}(\mathrm{g}_{2,1}\mathrm{g}_{1,1}^{-1})
\\
\times p_{\mathscr{A}_{\gamma_2^2}(1)-\mathscr{A}_{\gamma_2^1}(1)}(\mathrm{g}_{2,2}\mathrm{g}_{1,2}^{-1}) p_{\mathscr{A}_{\widetilde{\gamma}_{r,1}}(1)-\mathscr{A}_{\gamma^2_1}(1)}(\mathrm{g}_{3,1}\mathrm{g}_{2,1}^{-1})p_{\mathscr{A}_{\widetilde{\gamma}_{r,2}}(1)-\mathscr{A}_{\gamma^2_2}(1)}(\mathrm{g}_{3,2}\mathrm{g}_{2,2}^{-1})\\
\times p_{\upsilon(\Sigma)-\mathscr{A}_{\widetilde{\gamma}_{r}}(1)}\left(\mathrm{b}_{a(t_{4g})}^{\varepsilon_{4g}}\ldots \mathrm{g}_{3,2}\mathrm{b}_{a(t_{k})}^{\varepsilon_{k}}\mathrm{g}_{3,1}\ldots\mathrm{b}_{a(t_{1})}^{\varepsilon_{1}}\right)\mu_G^{\otimes 2g+6}(d\mathrm{b},d\mathrm{g}_1,d\mathrm{g}_2,d\mathrm{g}_3).
\end{eqnarray*}
One more time, we integrate with respect to the $\mathrm{g}_{3,1}$ and $\mathrm{g}_{3,2}$ variables and use Lemma~\ref{l:easy-useful-heat-kernel} to derive the following equality:
\begin{eqnarray*}
\int_{G^{2g}}\int_\Omega \Psi(\mathbf{Hol}(\gamma))p_{\upsilon(\Sigma)-\mathscr{A}_{\gamma_r}(1)}(\mathbf{Hol}(\gamma_r))\P(d\omega)\mu_G^{\otimes 2g}(d\mathrm{b}) \hspace{5cm}\\
=\int_{G^{2g+4}}\Psi(\mathrm{g}_{2,1}^{-1}\mathrm{b}_{a(t_{k})}^{-\varepsilon_{k}}\mathrm{g}_{2,2}^{-1}\mathrm{g}_{1,2}\mathrm{b}_{a(t_{k})}^{\varepsilon_{k}}\mathrm{g}_{1,1})p_{\mathscr{A}_{\gamma_1^1}(1)}(\mathrm{g}_{1,1})p_{\mathscr{A}_{\gamma_2^1}(1)}(\mathrm{g}_{1,2}) p_{\mathscr{A}_{\gamma_1^2}(1)-\mathscr{A}_{\gamma_1^1}(1)}(\mathrm{g}_{2,1}\mathrm{g}_{1,1}^{-1})\\
\times p_{\mathscr{A}_{\gamma_2^2}(1)-\mathscr{A}_{\gamma_2^1}(1)}(\mathrm{g}_{2,2}\mathrm{g}_{1,2}^{-1})  p_{\upsilon(\Sigma)-\mathscr{A}_{\gamma^2}(1)}\left(\mathrm{b}_{a(t_{4g})}^{\varepsilon_{4g}}\ldots \mathrm{g}_{2,2}\mathrm{b}_{a(t_{k})}^{\varepsilon_{k}}\mathrm{g}_{2,1}\ldots\mathrm{b}_{a(t_{1})}^{\varepsilon_{1}}\right)\mu_G^{\otimes 2g+4}(d\mathrm{b},d\mathrm{g}_1,d\mathrm{g}_2).
\end{eqnarray*}
We now make the change of variable $\widetilde{\mathrm{g}}_{1,1}=\mathrm{b}_{a(t_{k})}^{\varepsilon_{k}}\mathrm{g}_{1,1}$ and $\widetilde{\mathrm{g}}_{2,1}=\mathrm{b}_{a(t_{k})}^{\varepsilon_{k}}\mathrm{g}_{2,1}$. This yields
\begin{eqnarray*}
\int_{G^{2g}}\int_\Omega \Psi(\mathbf{Hol}(\gamma))p_{\upsilon(\Sigma)-\mathscr{A}_{\gamma_r}(1)}(\mathbf{Hol}(\gamma_r))\P(d\omega)\mu_G^{\otimes 2g}(d\mathrm{b}) \hspace{5cm}\\
=\int_{G^{2g+4}}\Psi(\widetilde{\mathrm{g}}_{2,1}^{-1}\mathrm{g}_{2,2}^{-1}\mathrm{g}_{1,2}\widetilde{\mathrm{g}}_{1,1})p_{\mathscr{A}_{\gamma_1^1}(1)}(\mathrm{b}_{a(t_{k})}^{-\varepsilon_{k}}\widetilde{\mathrm{g}}_{1,1})p_{\mathscr{A}_{\gamma_2^1}(1)}(\mathrm{g}_{1,2}) \\
\times p_{\mathscr{A}_{\gamma_1^2}(1)-\mathscr{A}_{\gamma_1^1}(1)}(\widetilde{\mathrm{g}}_{2,1}\widetilde{\mathrm{g}}_{1,1}^{-1})p_{\mathscr{A}_{\gamma_2^2}(1)-\mathscr{A}_{\gamma_2^1}(1)}(\mathrm{g}_{2,2}\mathrm{g}_{1,2}^{-1})\\
\times   p_{\upsilon(\Sigma)-\mathscr{A}_{\gamma^2}(1)}\left(\mathrm{b}_{a(t_{4g})}^{\varepsilon_{4g}}\ldots \mathrm{b}_{a(t_{k+1})}^{\varepsilon_{k+1}}\mathrm{g}_{2,2}\widetilde{\mathrm{g}}_{2,1}\mathrm{b}_{a(t_{k-1})}^{\varepsilon_{k-1}}\ldots\mathrm{b}_{a(t_{1})}^{\varepsilon_{1}}\right)\mu_G^{\otimes 2g+4}(d\mathrm{b},d\mathrm{g}_1,d\mathrm{g}_2).
\end{eqnarray*}
We now set $\widetilde{\mathrm{g}}_{1,2}=\mathrm{g}_{1,2}\widetilde{\mathrm{g}}_{1,1}$ and $\widetilde{\mathrm{g}}_{2,2}=\mathrm{g}_{2,2}\widetilde{\mathrm{g}}_{2,1}$ so that the integral becomes
\begin{eqnarray*}
\int_{G^{2g}}\int_\Omega \Psi(\mathbf{Hol}(\gamma))p_{\upsilon(\Sigma)-\mathscr{A}_{\gamma_r}(1)}(\mathbf{Hol}(\gamma_r))\P(d\omega)\mu_G^{\otimes 2g}(d\mathrm{b}) \hspace{5cm}\\
=\int_{G^{2g+4}}\Psi(\widetilde{\mathrm{g}}_{2,2}^{-1}\widetilde{\mathrm{g}}_{1,2})p_{\mathscr{A}_{\gamma_1^1}(1)}(\mathrm{b}_{a(t_{k})}^{-\varepsilon_{k}}\widetilde{\mathrm{g}}_{1,1})p_{\mathscr{A}_{\gamma_2^1}(1)}(\widetilde{\mathrm{g}}_{1,2}\widetilde{\mathrm{g}}_{1,1}^{-1})  
\\
\times p_{\mathscr{A}_{\gamma_1^2}(1)-\mathscr{A}_{\gamma_1^1}(1)}(\widetilde{\mathrm{g}}_{2,1}\widetilde{\mathrm{g}}_{1,1}^{-1}) p_{\mathscr{A}_{\gamma_2^2}(1)-\mathscr{A}_{\gamma_2^1}(1)}(\widetilde{\mathrm{g}}_{2,2}\widetilde{\mathrm{g}}_{2,1}^{-1}\widetilde{\mathrm{g}}_{1,1}\widetilde{\mathrm{g}}_{1,2}^{-1})\\
\times   p_{\upsilon(\Sigma)-\mathscr{A}_{\gamma^2}(1)}\left(\mathrm{b}_{a(t_{4g})}^{\varepsilon_{4g}}\ldots \mathrm{b}_{a(t_{k+1})}^{\varepsilon_{k+1}}\widetilde{\mathrm{g}}_{2,2}\mathrm{b}_{a(t_{k-1})}^{\varepsilon_{k-1}}\ldots\mathrm{b}_{a(t_{1})}^{\varepsilon_{1}}\right)\mu_G^{\otimes 2g+4}(d\mathrm{b},d\widetilde{\mathrm{g}}_1,d\widetilde{\mathrm{g}}_2).
\end{eqnarray*}
We can now integrate with respect to $\widetilde{\mathrm{g}}_{2,1}$ and to $\widetilde{\mathrm{g}}_{1,1}$ (in this order). Thanks to Lemma~\ref{l:easy-useful-heat-kernel}, we obtain
\begin{eqnarray*}
\int_{G^{2g}}\int_\Omega \Psi(\mathbf{Hol}(\gamma))p_{\upsilon(\Sigma)-\mathscr{A}_{\gamma_r}(1)}(\mathbf{Hol}(\gamma_r))\P(d\omega)\mu_G^{\otimes 2g}(d\mathrm{b}) \hspace{5cm}\\
=\int_{G^{2g+2}}\Psi(\widetilde{\mathrm{g}}_{2}^{-1}\widetilde{\mathrm{g}}_{1})p_{\mathscr{A}_{\gamma^1}(1)}(\widetilde{\mathrm{g}}_{1}\mathrm{b}_{a(t_{k})}^{-\varepsilon_{k}})   p_{\mathscr{A}_{\gamma^2}(1)-\mathscr{A}_{\gamma^1}(1)}(\widetilde{\mathrm{g}}_{2}^{-1}\widetilde{\mathrm{g}}_{1})\\
\times   p_{\upsilon(\Sigma)-\mathscr{A}_{\gamma^2}(1)}\left(\mathrm{b}_{a(t_{4g})}^{\varepsilon_{4g}}\ldots \mathrm{b}_{a(t_{k+1})}^{\varepsilon_{k+1}}\widetilde{\mathrm{g}}_{2}\mathrm{b}_{a(t_{k-1})}^{\varepsilon_{k-1}}\ldots\mathrm{b}_{a(t_{1})}^{\varepsilon_{1}}\right)\mu_G^{\otimes 2g+2}(d\mathrm{b},d\widetilde{\mathrm{g}}_1,d\widetilde{\mathrm{g}}_2).
\end{eqnarray*}
Hence, one gets
\begin{eqnarray*}
\mathbb{M}_{\operatorname{YM}}(\mathbf{Hol}(\gamma)\in d\mathrm{g})=p_{\upsilon(D_\gamma)}(\mathrm{g})\int_{G^{2g+1}}p_{\upsilon(\Sigma)-\mathscr{A}_{\gamma^2}(1)}\left(\mathrm{b}_{a(t_{4g})}^{\varepsilon_{4g}}\ldots \mathrm{b}_{a(t_{k+1})}^{\varepsilon_{k+1}}\widetilde{\mathrm{g}}_{2}\mathrm{b}_{a(t_{k-1})}^{\varepsilon_{k-1}}\ldots\mathrm{b}_{a(t_{1})}^{\varepsilon_{1}}\right)\\
p_{\mathscr{A}_{\gamma^1}(1)}(\widetilde{\mathrm{g}}_{2}\mathrm{g}\mathrm{b}_{a(t_{k})}^{-\varepsilon_{k}})      
\mu_G^{\otimes 2g+2}(d\mathrm{b},d\widetilde{\mathrm{g}}_2).
\end{eqnarray*}
Applying Lemma~\ref{l:easy-useful-heat-kernel} and~\ref{c:multi-orthogonality} (in this order), one obtains the expected formula.
\end{proof}

We end this section by mentioning an open question, namely recover the full Driver--Sengupta formula at some level of generality close to what is formulated in the work of Lévy~\cite[Eq.~$(1.1), (1.2), (1.3)$ p.~6]{Levyphd}, \cite[p.~289]{levysurvey} using our version of the closed Yang--Mills measure. Our calculations for small loops are a first step towards this direction. Since we believe this might be a difficult combinatorial problem, it is possible that one would need to modify the present setting to make such proofs natural.

\section{Abelian Yang--Mills (Maxwell) theory on surfaces via Morse gauge. }
\label{s:MaxwellMorse}

The goal of the present section is to discuss how our quantization works in the $G=U(1)$ case as well as to describe certain subtleties of the induced $YM_2$ measure related to topologies of line bundles.
\subsection{U(1)-principal bundle and associated complex Line bundles}

We recall the correspondence between principal circle bundles and line bundles.
Let $M$ be a manifold and $\pi: P\rightarrow M$ be a principal $U(1)$ bundle. Let $\rho: U(1)\rightarrow \mathbb{C}^*$ be a unitary representation,  then we can construct the associated complex line bundles $L:= P\times_{\rho}\mathbb{C}:=P\times\mathbb{C}/\sim $, where the equivalence relation is $(p, z)\sim (p g, \rho(g)^{-1}z), \, g\in U(1)$.  
The one-dimensional unitary representations of $U(1)$ are classified by integers $n\in \mathbb{Z}$. Given $n\in \mathbb{Z}$, define $\rho_n: U(1)\rightarrow \mathbb{C}^*, \rho_n(e^{i\theta})=e^{in\theta}, n\in \mathbb{Z}$. Then we obtain the associated complex line bundle $L_n:= P\times_{\rho_n}\mathbb{C}:=P\times\mathbb{C}/\sim $. Also observe that the bundle $L_n$ can be realized as the $n$-th tensor power of $L_1$, namely $L_n=L_1^{\otimes n} $.

Conversely, let  $(L, h)$  be a Hermitian line bundle, then $S(L):= \{ u \in L: \, |u |_h=1 \}$ is a $U(1)$-principal bundle with  the action given by $e^{i\theta} \cdot u= e^{i\theta} u.$ And we have natural isomorphisms:
\begin{align*}
&S(L) \times_\rho \mathbb{C}\rightarrow L, \, [u, z]\mapsto z u\\
 &P\rightarrow S(P \times_\rho \mathbb{C}), u\mapsto [u, 1].
\end{align*} 
A connection on $P$ is a Lie algebra valued one form $A\in \Omega^1(P, i\mathbb{R})$. Given any connection $A_0$ on $P$, the space of connections on $P$ is an affine space modeled on $\Omega^1(M, \mathbb{R} )$ given explicitly by $\mathcal{A}(P)= \{A_0+ i \pi^*\alpha:\, \alpha\in \Omega^1(M, \mathbb{R} )  \}$.  
A connection (or covariant derivative) of $L$ is a linear map $\nabla: \Gamma^0(M, L) \rightarrow \Omega^1(M,L)$ satisfying $\nabla(fs) = df \otimes s + f\nabla s$ for any $s\in C^\infty(M, L)$ and $f\in C^\infty (M)$. Given any connection $\nabla^0$ on $L$, then we have the space of connections of $L$ is given by $\mathcal{A}(L)= \{ \nabla^0 +\alpha:\, \alpha\in \Omega^1(M, \mathbb{R}) \}$. 
We summarize what we just described in the following:
\begin{thm} Let $L=P\times_\rho \mathbb{C}$, where $\rho: U(1)\rightarrow \mathbb{C}^*, \rho(e^{i\theta})=e^{i\theta}$ is the standard representation of $U(1)$. 
There is a $1-1$ correspondence space of connections on $P$ and the covariant derivatives on an associated vector bundle $L$. In particular, the curvatures are the same via this correspondence.
\end{thm}
We refer the reader to \cite[Proposition 15.1]{Duistermaat} for the details of the proof. From this correspondence, we can consider $\mathcal{A}$ the space of all covariant derivatives $\nabla$ on an associated complex line bundle $L$ of $P$, and the Yang–Mills action $$S_{YM} :=\frac{1}{2} \int_M |F_\nabla|^2 d\upsilon,$$
for some volume form $\upsilon$. Recall that any complex line bundle is determined up to $C^\infty$ isomorphism by its Chern class~\cite[l.~15 p.~140]{GH}. If the line bundle is holomorphic, then one can define the Chern class in terms of divisors~\cite[p.~139 l.~13]{GH} or equivalently in terms of the curvature of some connection $1$--form by~\cite[Proposition p.~141]{GH} (the connection need not be the Chern connection). We recall the definition:

\begin{definition} 
The Chern class of a line bundle $ L$ on $M$ is $c_1(L)\in H^2(M, \mathbb{Z})$, which can be defined as $c_1(L):=\left[\frac{i}{2\pi} F_\nabla\right]\in H^2_{dR}(M)$, where $F_\nabla$ is the curvature for any connection $\nabla$ on $L$. We also define $c_1(L, \nabla) = \frac{i}{2\pi} F_\nabla$ to be the Chern form of $\nabla$. If $M$ is a compact surface, the degree of $L$ is defined by $\deg(L):=\int_M c_1(L, \nabla)$ which is independent of $\nabla$.
\end{definition}
We can define $c_1(P):= c_1(L_1)$, then $P$ is trivial if and only if  $c_1(P)=0$. If $c_1(P)\neq 0$, then $c_1(L_n)= n c_1(P)$ for all $n\in \mathbb{Z}$.

\medskip
 We give an example in genus $0$ where we build a 
non trivial line bundle on $\mathbb{CP}^1\simeq \mathbb{S}^2$. Consider the Hopf bundle $P= S^3:=\{ (z_1, z_2)\in \mathbb{C}^2|\,  |z_1|^2+|z_2|^2 =1\}$ on $\mathbb{CP}^1$. This model describes a singly-charged Dirac monopole (cf. \cite[p. 272]{EGH}). The group $U(1)$ acts on $S^3$ by $(z_1, z_2)\cdot e^{i\theta} = (e^{i\theta} z_1, e^{i\theta} z_2)$.  The quotient is $S^3/U(1)= \mathbb{C}P^1\simeq  S^2$.  Then we can see that the standard representation $\rho_1: U(1)\rightarrow \mathbb{C}^*, \rho_1(e^{i\theta})(z)= e^{i\theta} z$ defines the associated line bundle $L_1\simeq  \mathcal{O}(-1) $ which is the tautological line bundle over $\mathbb{C}P^1$.  Then the first Chern class $c_1(P):= c_1(L_1)= [\omega_{FS}]$ with the Chern number $\int_X \omega_{FS}=1$, where $\omega_{FS}$ is the Fubini-Study metric.  This implies that the space of connections on $P$ is equivalent to the space of connections of arbitrary  Chern number in $\mathbb{Z}$.

Now we will consider only connections on a complex line bundle $L$ on a compact surface $\Sigma$ instead of the principal bundle $P$. 

\paragraph{From connections on non trivial line bundles to singular connections on the trivial line bundle.}

Our construction of the Yang--Mills measure suggests a change of perspective on how to deal with connections on non-trivial line bundles $L\rightarrow \Sigma$. Starting from any non trivial line bundle $L\rightarrow \Sigma$, for instance when we constructed the free boundary measure, we first remove the critical point $a_{2g+2}=\argmax(f)$ and then consider the restriction of the line bundle $L\rightarrow \Sigma\setminus \{a_{2g+2}\} $. However 
by doing so topology is lost since it is well--known that the restricted bundle $L\rightarrow \Sigma\setminus \{a_{2g+2}\} $ becomes trivial.
\begin{comment}
Hence any random connection $A$ under the free boundary measure $\mu_{\mathrm{YM}}^{\mathrm{free}}$ should be thought of as connections on a trivial line bundle. However, when we condition
\end{comment}

Then in our probabilistic constructions, all the objects live on $\Sigma\setminus \argmax(f)$ and can be extended as currents on the whole $\Sigma$. A natural question is: how can we recover non triviality of line bundles on $\Sigma$ (capture the topology) by working only with connections living on the pointed surface $\Sigma\setminus \argmax(f) $?

The lemma below shows that we can do this using {\bf singular} connections of $L$ on $\Sigma\setminus \max(f)$.  

\begin{definition}[Singular connection]
Let $L$ be a complex line bundle over a compact surface $\Sigma$. A singular connection of  $L$ is a smooth connection $d+\alpha$ on the trivial line bundle $\Sigma\setminus \{p_1, \ldots, p_n\}\times \mathbb{C} $ such that the connection $1$--form $\alpha$ extends as a current of degree $1$ to $\Sigma$.
\end{definition}

\begin{lemma}\label{lem_sing_connection}
There is a 1-1 correspondence between gauge equivalence classes of pairs
$(L,\nabla)$, where $L\to\Sigma$ is a nontrivial complex line bundle of
degree $k$ equipped with a smooth connection $\nabla$, and gauge
equivalence classes of pairs $(\Sigma\times\mathbb C,\nabla^s)$, where
$\nabla^s$ is a singular connection whose connection $1$--form is a
current of degree $1$ on $\Sigma$ satisfying
\[
\alpha = k\,\frac{dz}{z}+\beta,
\]
in a holomorphic coordinate system $z$ centered at $\argmax(f)$, with
$\beta$ smooth. Moreover, every such singular connection is smooth on
$\Sigma\setminus\{\argmax(f)\}$ and its connection $1$--form belongs
globally to $H^{-\varepsilon}(\Sigma)$ for every $\varepsilon>0$.

In
particular, we have the fundamental residue equation: 
\begin{equation}\label{eq:residue}
 \boxed{ c_1(L,\nabla)= \frac{i}{2\pi}  F_\nabla = -\frac{i}{2\pi}  d\alpha + k[\argmax (f)] }
\end{equation}  
on $\Sigma$, where the singular contributions cancel. 
\end{lemma}

This is a classical lemma in gauge theory sometimes called smooth Poincar\'e--Lelong as in ~\cite[Section 9]{HarveyLawson93}. 
We want to explain that the reader should forget about non trivial line bundles, we only deal with trivial bundles in this work, but with \textbf{singular connections} and the topological information is encoded in the singularity of these connections at $\argmax(f)$ summarized by the residue equation~\ref{eq:residue}. In particular,
this motivates the following definition which recovers the degree of the line bundle $L$ from the singular connection:

\begin{definition}[Chern number of a singular connection]\label{def:Chernsingnabla}
Let $L$ be a complex line bundle over a compact surface $\Sigma$ and $\nabla:=d+\alpha$ be a singular connection of $L$ with the singularity at $p\in \Sigma$, then the degree of $L$ can be defined as
\begin{equation}
\deg(L)=-\frac{i}{2\pi}\lim_{\varepsilon\rightarrow 0^+}\int_{\Sigma\setminus D_\varepsilon} d\alpha
\end{equation}
where $D_\varepsilon$ is a small disc of radius $\varepsilon>0$ around $p$.

From now on, we will refer to $\lim_{\varepsilon\rightarrow 0^+}\int_{\Sigma\setminus D_\varepsilon} -id\alpha \in 2\pi\mathbb{Z} $ as the \textbf{Chern number} $c(\nabla)$ of the singular connection $\nabla:= d+\alpha$.
\end{definition}

\paragraph{Proof of Lemma~\ref{lem_sing_connection}.}

\begin{proof}
We can choose a holomorphic structure on $L\rightarrow \Sigma$ such that there exists a holomorphic section $s_0$ of $L$ having a zero of degree $k=\deg(L)$ at the prescribed point $\argmax(f)$ (see for example \cite[p 101-102]{Hitchin}). We trivialize the holomorphic bundle $L$ on $U_0=\Sigma\setminus \argmax(f) $ by $s_0$ which is smooth outside $\argmax(f)$. Set $\alpha\in \Omega^1 (\Sigma\setminus \argmax(f))$ the unique $1$--form such that $\nabla s_0=\alpha s_0$. Now, taking $U_1$ a small disc centered at $\argmax(f)$, we trivialize $L$ on $U_1$ using another local holomorphic section $s_1$ which is non zero at $\argmax(f)$. On $U_0\cap U_1$, we have $s_0=f_0s_1$ where $f_0$ is a holomorphic function that vanishes at order $k$ at $\argmax(f)$, we can choose $f_0(z)=z^k$ in local holomorphic chart and denote by $\alpha_1$ the connection $1$--form in the trivialization by $s_1$: $\nabla s_1=\alpha_1s_1$ where $\alpha_1$ is smooth near $\argmax(f)$. An immediate calculation yields 
$\nabla s_0=\alpha s_0= \nabla (z^k s_1)=(k z^{k-1}) dz  s_1+\alpha_1z^k s_1  $
which implies that near $\argmax(f)$,
$$\alpha= k \left(\frac{dz}{z} \right)+\alpha_1 = k\partial_z( \log|z|^2)dz+\alpha_1.$$
Now we have $\frac{1}{z}\in L^p_{loc}$ for every $p\in [1, 2)$ since 
$$
    \int_{|z|<1}\frac{1}{|z|^p}dz = 2\pi\int_0^1 \frac{1}{r^{p-1}} dr<\infty \Longleftrightarrow p<2.
$$
The Sobolev embedding implies that  $\alpha\in H^{-\varepsilon}$ near $\argmax(f)$ as required. In particular, $\alpha$ extends as a current to a neighborhood of $\argmax(f)$ on $\Sigma$ which defines a singular connection of $L$ with pole at $\argmax(f)$.
Moreover, we have $$ \frac{i}{2\pi} d\alpha=k \frac{i}{2\pi}  \partial_{\bar z}\partial_z( \log|z|^2) dz\wedge d\bar z + \frac{i}{2\pi}  d\alpha_1 = k [\argmax(f)] + \frac{i}{2\pi} d\alpha_1$$  by Poincaré-Lelong formula.
Remark that on $U_0\cap U_1$, $d\alpha=d\alpha_1$, and the smooth curvature form $F_\nabla$ is defined by  $F_\nabla =- d\alpha$ on $U_0$ and $ F_\nabla= -d\alpha_1$ on $U_1$, therefore we get $$c_1(L,\nabla)= \frac{i}{2\pi}  F_\nabla = -\frac{i}{2\pi}  d\alpha + k[\argmax (f)]$$
on all $\Sigma$.
\end{proof}
\begin{rmk}
We also refer to \cite[Proposition p.~141]{GH} for the same result formulated in slightly different language. In complex algebraic geometry, if $L$ is a holomorphic line bundle, one would pick some Hermitian metric $\langle. , .\rangle_L$ on the line bundle $L$ and associated Chern connection $\nabla$ of the corresponding metric $\langle. , .\rangle_L$. 
 It follows from \cite[Theorem 9.5]{HarveyLawson93} that $\alpha= \partial (\log|s_0|^2_L )$ and  
\begin{equation}\label{eq:residue2}
  c_1(L, \nabla)=   - \frac{i}{2\pi} \bar \partial \partial  \log|s_0|^2_L  + k[\argmax(f)],
\end{equation}
where $d= \partial+\bar \partial$.
The left hand side of equation (\ref{eq:residue2}) 
is the curvature form $\omega\in \Omega^{1,1}(\Sigma)$ of the Chern connection and the degree $2$ current $k[\argmax (f)]$ should be interpreted as a representative of the \textbf{divisor} of the line bundle $L\rightarrow \Sigma$.

Our Lemma can also be viewed as a $2d$ version of the Dirac monopole construction~\cite{Bottsurvey} and also appears in the recent work~\cite[p.~601-604, section 4.2]{GKR} under the name of \textbf{magnetic charges}.
\end{rmk}

\subsection{Morse gauge for nontrivial \texorpdfstring{$U(1)$}{U(1)} bundle}

In this final short section, we derive formulas in the Abelian case as a toy model of the machinery developed in the present paper and we also explain
how our formalism applied to abelian case
yields a Yang--Mills measure on connections which contain topological information in the sense of Lemma~\ref{lem_sing_connection}. 
This should indicate how one could attack the case of nontrivial bundles modulo some extra work. In the case where $G=U(1)\simeq\mathbb{S}^1$, the theory developed above is indeed slightly simpler to deal with even if several technical aspects remain unchanged (e.g. application of Theorem~\ref{t:contraction} to prove Theorem~\ref{t:random-cohomological} or conditioning at the maximum of $f$). 

\paragraph{Singular connections in the Morse gauge and the residue equation.}

When $G=U(1)$, we start with a singular connection $d+\alpha$ as given in Lemma~\ref{lem_sing_connection}  on the trivial bundle $\Sigma\times\mathbb{C}$. 
In particular, one has $\operatorname{Hol}_\gamma(A):=e^{i\int_\gamma\alpha}$. Hence, once we are able to define a probabilistic version of $\int_\gamma\alpha$, the corresponding random holonomy is directly defined without using tools from stochastic differential equations which was the content of \S\ref{s:randomholonomy}. On top of that, the curvature of $A=-i\alpha$ has no quadratic terms, i.e. $F(i\alpha)=-id\alpha$ for any gauge choice. This results into a simpler exposition of the classical gauge from Theorem~\ref{t:normalform} following directly the results from~\cite{DR19}. More precisely, the Morse gauge at time $T>0$ can be explicitly written as
$$
\mathrm{g}_T(x)=e^{i\int_{\varphi_f^{-T}(x)\rightarrow x}\alpha},
$$
where the path is taken over the flow line joining $\varphi_f^{-t}(x)$ to $x$. By definition, one has then
$$
\alpha_T:=\alpha -d\left(\int_{\varphi_f^{-T}(x)\rightarrow x}\alpha\right)=\alpha -\int_{-T}^0 d\iota_V\varphi_f^{t*}(\alpha) dt.
$$
Thanks to Cartan formula, this can be rewritten as
$$
\alpha_T=\varphi_f^{-T*}(\alpha)+\int_{-T}^0 \iota_Vd\varphi_f^{t*}(\alpha) dt.
$$
Letting $\Pi_0(\alpha)=\sum_{\text{ind}(a)=1}\left(\int_{W^s(a)}\alpha\right)[W^u(a)]$ and recalling that $d[W^u(a)]=0$, one finds $d\alpha=(\text{Id}-\Pi_0)(d\alpha)$ and
$$
\alpha_T=\Pi_0(\alpha)+\varphi_f^{-T*}(\text{Id}-\Pi_0)(\alpha)+\int_{-T}^0 \iota_V\varphi_f^{t*}(\text{Id}-\Pi_0)(d\alpha) dt.
$$
Using the fact that $\alpha$ is smooth outside $\argmax(f)$ and belongs to $H^{-\varepsilon}(\Sigma)$ for all $\varepsilon>0$ then
an application of~\cite[Prop.~5.7]{DR19} (which is a refinement of Theorem~\ref{t:harveylawson}) then shows that, as $T\rightarrow\infty$,
$$
\alpha_T\rightharpoonup \alpha_\infty:=\Pi_0(\alpha)+\iota_V \mathcal{L}_V^{-1} (d\alpha).
$$
 For each fixed $T>0$, by Lemma  \ref{lem_sing_connection}, we have the relation $d\alpha_T +2\pi i\deg(L)[\argmax(f)]=-F_\nabla$ independently of $T$.  Because $\alpha_T$ converges weakly to $\alpha_\infty$ in the sense of current (in fact the convergence holds true in stronger topologies given by anisotropic Sobolev spaces) and because the de Rham $d$ differential is continuous for the weak topology of currents, we have $d\alpha_\infty +2\pi i\deg(L)[\argmax(f)]=-F_\nabla$. So the current $\alpha_\infty$ still satisfies the residue equation~\ref{eq:residue} which encodes the topological information about the degree of the line bundles we started with and the singular connection $d+\alpha_\infty$ satisfies $\iota_V\alpha_\infty=0$ by construction.     

\subsection{The quantum measure and conditioning on Chern numbers.}\label{sss:YMChern}

The Abelian case is in fact particularly convenient when computing the law of random holonomies as we will now illustrate with a (formal) computation. We let $\lambda_J<\ldots<\lambda_1$ be $p$ regular values of $f$. In particular, it intersects every unstable manifold either two times or not at all. Hence, thanks to the commuting properties of $U(1)$ and using Stokes Theorem, the random holonomy along $\gamma$ has the simple expression
$$
\forall 1\leqslant j\leqslant J,\quad \mathbf{Hol}(\gamma_j)
= 
\exp\left( i\int_{f\leqslant \lambda_j}\xi\upsilon \right) \ .
$$
When $\{f=\lambda_j\}$ is not connected, the holonomy is taken to be the product of the holonomies along each connected component (with the same orientation). Similarly, the random holonomy used to perform our conditioning is defined as
$$
\mathbf{Hol}_0=e^{i\int_\Sigma\xi\upsilon}.
$$
We now fix a bounded and continuous function $\Psi:U(1)^J\rightarrow \R$ and the joint law for these holonomies is (formally) given by
$$
\int_{U(1)^J}\Psi d\mathbb{M}_{\text{YM}}^{\gamma_1,\ldots,\gamma_J}:=\int_{H^{-1-\kappa}}\Psi\left(\int_{f\leqslant \lambda_1}\xi\upsilon,\ldots,\int_{f\leqslant \lambda_J}\xi\upsilon\right)\delta_{\Z}\left(\int_\Sigma\xi\upsilon\right)d\xi,
$$
where $\delta_{\Z}(x)=\sum_{k_0\in\Z} e^{ik_0x}=\sum_{c\in 2\pi\Z}\delta_0(x-c)$ (recall that we picked the normalized Haar measure $\int_{\mathbb{R}} \delta_0(\theta) \frac{d\theta}{2\pi}=1$ which explains why our normalization for the Poisson summation formula differs from the usual one in the literature).
The random connection under the free boundary measure $\mu_{\text{YM}}^{\text{free}}$ writes
$A=\mathcal{L}_V^{-1}\left(\xi( \iota_V\upsilon ) \right)+\sum_{\ind(a)=1} \theta_a [W^u(a)] $ where the pair $(\xi,(\theta_a)_{\ind(a)=1})$ is randomly chosen under $\mathbb{P}(d\xi)\otimes \prod_{ \ind(a)=1} \frac{d\theta_a}{2\pi} $ where $\theta_a$ plays the role of $\log(\mathsf{b}_a)$ for some $\mathsf{b}_a$ on the unit circle.  
So the  Yang--Mills measure $\mathbb{M}_{\text{YM}}$ is defined from the free boundary Yang--Mills measure by
$\mathbb{M}_{\text{YM}}=\mathbb{P}_{\text{YM}}^{\text{free}}\left( \,\ . \,\ | \int_\Sigma \xi\upsilon \in 2\pi\mathbb{Z} \right) p_{\upsilon(\Sigma)}\left(\text{Id}_G \right) $, where we used the crucial fact that in the abelian case the probability under $\mathbb{P}_{\text{YM}}^{\text{free}}$ that the holonomy equals $\text{Id}_G$ is given by $p_{\upsilon(\Sigma)}\left(\text{Id}_G \right)$ where $p_t(.)$ is the heat kernel on $\mathbb{S}^1$.
Therefore the Yang--Mills measure $\mathbb{M}_{\text{YM}}$ decomposes as a series of measures indexed by Chern numbers $c\in 2\pi\mathbb{Z}$:
\begin{equation}
 \mathbb{M}_{\text{YM}}=\sum_{c\in 2\pi\mathbb{Z}} \mathbb{M}_{\text{YM},c} 
\end{equation} 
where the
measure $\mathbb{M}_{\text{YM},c}$
is defined by 
\begin{equation}
\mathbb{M}_{\text{YM},c}(F):=\mathbb{P}_{\text{YM}}^{\text{free}}\left( \,\ . \,\ | \int_\Sigma \xi\upsilon = 2\pi c \right) \frac{e^{-\frac{c^2}{4\upsilon(\Sigma)} }}{\left(4\pi \upsilon(\Sigma) \right)^{\frac{1}{2}}}
\end{equation}
and averages of holonomies are given by the simple formula: 
$$
\int_{U(1)^J}\Psi d\mathbb{M}_{\text{YM},c}^{\gamma_1,\ldots,\gamma_J}:=\int_{H^{-1-\kappa}}\Psi\left(\int_{f\leqslant \lambda_1}\xi\upsilon,\ldots,\int_{f\leqslant \lambda_J}\xi\upsilon\right)\delta_{0}\left(\int_\Sigma\xi\upsilon-c\right)d\xi.
$$
The above discussion immediately implies that the induced measure $\mu_{\text{YM}}$ on
distributional connections writes as a series $\mu_{\text{YM}}=\sum_{c\in 2\pi\mathbb{Z}} \mu_{\text{YM},c} $.

The Yang-Mills probability measure is defined by $\P_{\text{YM}}:=\frac{\mathbb{M}_{\text{YM}}}{Z_\upsilon(\Sigma, U(1))}$, where $$Z_\upsilon(\Sigma, U(1))= p_{\upsilon(\Sigma)}\left(\text{Id}_G \right) =\frac{\sum_{k_0\in \mathbb{Z}}e^{-\frac{\pi^2k_0^2}{\upsilon(\Sigma)} }}{\left(4\pi \upsilon(\Sigma) \right)^{\frac{1}{2}}} $$ is the partition function for $\M_{\text{YM}}$. The reader will immediately notice by definition~\ref{def:Chernsingnabla} of the Chern number of a singular connection that 
random connections in the support of  
$\mathbb{M}_{\text{YM},c}$ have Chern number $c$. Therefore the conditional probability measure 
$\mathbb{P}_{\text{YM}} \left( \,\ \cdot \,\ | \text{Chern number}(\nabla)=c \right) $ writes
\begin{equation}
\boxed{ 
    \mathbb{P}_{\text{YM}} \left( \,\ \cdot \ | \ \text{Chern number}(\nabla)= c \right) = 
    \frac{\mathbb{M}_{\text{YM},c}( \cdot )}
         {   e^{-\frac{c^2}{4\upsilon(\Sigma)} } / \left(4\pi \upsilon(\Sigma) \right)^{\frac{1}{2}} }    } 
\end{equation}
and the probability to pick a random connection under $\mathbb{P}_{\text{YM}}$ that has Chern number $c$ equals: 
\begin{equation} 
\boxed{\mathbb{P}_{\text{YM}} \left( \text{Chern number}(\nabla)=c \right)=  \frac{e^{-\frac{c^2}{4\upsilon(\Sigma)}}}{  \sum_{k_0\in \mathbb{Z}  } e^{-\frac{4\pi^2k_0^2}{4\upsilon(\Sigma)}}  } .} 
\end{equation}
This is the probability that a Brownian bridge on $\mathbb{S}^1$ at time $\upsilon(\Sigma)$ winds $\frac{c}{2\pi}$ times around $\mathbb{S}^1$. 

\paragraph{The conceptual relation with geometric quantization.}

Recall that under $\mu_{\text{YM}}^{\text{free}}$, the curvature $ \lim_{\varepsilon\rightarrow 0^+} (1-1_{D_\varepsilon})  dA$ of a random connection is distributed as a white noise $\xi\upsilon$ and therefore $\int_\Sigma \xi\upsilon$ is real valued Gaussian hence it fails to satisfy the integrality condition from line bundles that $\int_\Sigma F\in 2\pi\mathbb{Z}$. It means that the curvature $\lim_{\varepsilon\rightarrow 0^+} (1-1_{D_\varepsilon})  dA$ \textbf{cannot come from a connection on a line bundle $L$ on the closed surface $\Sigma$}. Therefore, exactly in the spirit of geometric quantization, the conditioning procedure aims \textbf{to 
restore this integrality condition} by \textbf{imposing} that $ \int_\Sigma F\in 2\pi\mathbb{Z} $. This can be viewed as the union $\cup_{c\in 2\pi\mathbb{Z}} \{ \int_\Sigma F=c \} $ where the measure charges a countable number of components. In each component $\{ \int_\Sigma F=c \} $, the Chern number of the corresponding random connections is exactly $c$.

\paragraph{An explicit expression for $\mathbb{M}_{\text{YM}}$.}

 The function $\Psi$ can be decomposed in Fourier series so that it is sufficient by density to consider $\Psi$ of the form $\Psi_{\mathbf{k}}(x):=e^{i(k_1x_1+\ldots +k_Jx_J)}$. The integral reads
$$
\int_{U(1)^J}\Psi_{\mathbf{k}}d\mathbb{M}_{\text{YM}}^{\gamma_1,\ldots,\gamma_J}=\sum_{k_0\in\mathbb{Z}}\int_{H^{-1-\kappa}}e^{ik_0\int_\Sigma\xi\upsilon} e^{i\sum_{j=1}k_j\int_{f\leqslant \lambda_j}\xi\upsilon}d\xi.
$$
As $\xi$ is a white noise, one finds that this integral is equal to
$$
\int_{U(1)^J}\Psi_{\mathbf{k}}d\mathbb{M}_{\text{YM}}^{\gamma_1,\ldots,\gamma_J}=\sum_{k_0\in\mathbb{Z}}\prod_{j=0}^Je^{-\frac{1}{2}(k_0+k_1+\ldots+k_j)^2\upsilon(\lambda_{j+1}\leqslant f\leqslant \lambda_j)},
$$
with the convention $\lambda_0=\max f$ and $\lambda_{J+1}=\min f$. In particular, the partition function is given by
$$
Z_{\upsilon}(\Sigma,G)=\sum_{k_0\in\mathbb{Z}}e^{-\frac{k_0^2}{2}\upsilon(\Sigma)}.
$$
Writing down the Fourier decomposition $\Psi(x)=\sum_{\mathbf{k}\in\mathbb{Z}^J}\widehat{\Psi}_{\mathbf{k}} e^{i\mathbf{k}\cdot x}$, one finds
$$
\int_{U(1)^J}\Psi d\mathbb{M}_{\text{YM}}^{\gamma_1,\ldots,\gamma_J}=\sum_{k_0\in\mathbb{Z}}\sum_{\mathbf{k}\in\mathbb{Z}^J}\widehat{\Psi}_{k_1-k_0,k_2-k_1,\ldots,k_{J}-k_{J-1}}\prod_{j=0}^Je^{-\frac{1}{2}k_j^2\upsilon(\lambda_{j+1}\leqslant f\leqslant \lambda_j)}.
$$
From this, we infer that
\begin{multline*}
\mathbb{M}_{\text{YM}}\left(\mathbf{Hol}(\gamma_1)\in d\mathrm{g}_1,\ldots, \mathbf{Hol}(\gamma_J)\in d\mathrm{g}_J\right)=p_{\upsilon(\lambda_1\leqslant f\leqslant \lambda_0)}(\mathrm{g}_1) \dots p_{\upsilon(\lambda_{J+1}\leqslant f\leqslant \lambda_J)}(\mathrm{g}_J)\\
\times\left(\prod_{j=1}^{J-1}p_{\upsilon(\lambda_{j+1}\leqslant f\leqslant \lambda_j)}(\mathrm{g}_{j+1}^{-1}\mathrm{g}_j)\right)\mu_G^{\otimes J}(d\mathrm{g}).
\end{multline*}

\appendix

\newpage

\section{Stochastic differential equations for continuous reparametrizations of the Brownian motion}
\label{a:reparametrization-sde}

In this appendix, we review It\^o's calculus and resolution of stochastic differential equations when one considers reparametrization of the Brownian motion by an \emph{increasing} map $
\mathscr{A}\in\mathcal{C}^0([0,1],\mathbb{R}_+)
$ which is also of class $\mathcal{C}^1$ except at finitely many points $(T_j)_{j\in J}$ in $[0,1]$ where $\mathscr{A}'(t)=\mathcal{O}(|\ln |t-T_j||)$. The purpose of this appendix is to describe the basic theory for readers less familiar with stochastic calculus who will find here some material and references on this topic. In particular, we explain that, despite the low regularities properties of $\mathscr{A}$, It\^o's integration (and its application to stochastic differential equations) indeed makes sense~\cite[Ch.~IV]{RevuzYor}. 

\begin{rmk}
For the $1$-dimensional case, we follow the presentation of~\cite[Ch.~4,5]{Evans} while, for the case of compact Lie groups, we rather follow~\cite[Ch.~6-7]{FranchiLeJan}. See also~\cite[Ch.~IV]{RevuzYor} for a presentation of stochastic integration in a general setting handling the integration of more general $W$ (and thus $\mathscr{A}$). Finally, discussion on stochastic differential equations is a mixture of these three references.
\end{rmk}

More precisely, we proceed in two steps. First, we review what It\^o's integral is for a reparametrized Brownian motion. Then, we discuss the resolution of stochastic differential equations for such stochastic processes. All along this appendix and once $\mathscr{A}$ is fixed, we suppose that $W:t\in[0,1]\mapsto\R$ is a stochastic process that is almost surely $\mathcal{C}^{\beta}([0,1])$ and that satisfies the following properties:
\begin{enumerate}
 \item $W(0)=0$;
 \item for all $0\leq s\leq t$, $W(t)-W(s)$ follows a normal law with variance $\int_{s}^t\mathscr{A}'(\tau)d\tau$;
 \item for all $p\geqslant 1$ and for all $0=t_0< t_1< t_2<\ldots<t_n\leqslant 1$, $W(t_1)$, $W(t_2)-W(t_1)$,$\ldots$, $W(t_p)-W(t_{p-1})$ are independent random variables.
\end{enumerate}
Recall also that under these assumptions, one has 
$$
\forall s,t\in[0,1],\ \mathbb{E}\left(W(t)W(s)\right)=\min\{\mathscr{A}(t),\mathscr{A}(s)\}.
$$

\subsection{Background on It\^o's integrals}

Following the lines of~\cite[Ch.~4]{Evans}, one can first verify that
$$
\lim_{n\rightarrow +\infty}\sum_{k=0}^{m_n-1}\left(W(t_{k+1}^n)-W(t_k^n)\right)^2=\int_a^b\mathscr{A}'(t)dt.
$$
where $P_n:=\{a=t_0^n<t_1^n<\ldots<t_{m_n}^n=b\}$ verifies $|P_n|=\max_k\{t_{k+1}^n-t_k^n\}\rightarrow 0$ as $n\rightarrow+\infty$ and where the limit is taken in $L^2(\Omega)$. This is the fundamental step in the construction of It\^o's integral and, following~\cite[Ch.~IV]{RevuzYor}, this is referred to as the quadratic variation of the process. From this, one can deduce that:
\begin{equation}\label{e:derivative-square}
\lim_{n\rightarrow +\infty}\sum_{k=0}^{m_n-1}W(t_k^n)\left(W(t_{k+1}^n)-W(t_k^n)\right)=\frac{1}{2}W^2(b)-\frac{1}{2}W^2(a)-\frac12\int_a^b\mathscr{A}'(t)dt.
\end{equation}
We now denote by $\mathcal{F}(t)$ the $\sigma$-algebra generated by $(W(s))_{0\leq s\leq t}$ and by $\mathcal{F}^+(t)$ the one generated by $(W(s)-W(t))_{t\leq s\leq 1}$. From the properties of $W$, these are independent $\sigma$-algebra and we say that $\mathcal{F}(t)$ is non-anticipating. A stochastic process $H(t)$ on $[0,1]$ is said to be non-anticipating with respect to $\mathcal{F}(t)$ if, for every $t\in[0,1]$, $H(t)$ is $\mathcal{F}(t)$-measurable. We say that $H$ belongs to $\mathbb{L}^2([0,1])$ if $\mathbb{E}\left(\int_0^1H^2\mathscr{A}'(t)dt\right)<\infty.$ Following~\cite[Ch.~4]{Evans}, we say that $H$ is a step process if it is nonanticipating and if one can find a partition $0=t_0<t_1<\ldots<t_m=1$ such that, on each $[t_k,t_{k+1})$, $H$ is equal to some $H_k$ in $L^2(\Omega)$. For step processes, one defines then the It\^o integral as
$$
\int_0^1HdW=\sum_{k=0}^{m-1}(W(t_{k+1})-W(t_k))H_k,
$$
and can verify that it is linear in $H$ and that
\begin{equation}\label{e:basic-properties-Ito}
\mathbb{E}\left(\int_0^1HdW\right)=0\quad \text{and}\quad\mathbb{E}\left(\int_0^1H^2\mathscr{A}'(t)dt\right)=\mathbb{E}\left(\left(\int_0^1HdW\right)^2\right).
\end{equation}
We now want to extend this definition to more general non-anticipating process $H$ in $\mathbb{L}^2(0,1)$. This is ensured by the following lemma stating that $H$ can be approximated by a step process:
\begin{lemma}\label{:approx-step-process} Let $H$ in $\mathbb{L}^2([0,1])$ which is non anticipating with respect to $\mathcal{F}(t)$. Then, there exists a sequence of step processes $(H_n)_{n\geq 1}$ such that
$$
\lim_{n\rightarrow+\infty}\mathbb{E}\left(\int_0^1|H(t)-H_n(t)|^2\mathscr{A}'(t)dt\right)=0.
$$
\end{lemma}

\begin{comment}
\begin{proof}
 One starts by dealing with the case where $H$ is a continuous and bounded nonanticipating process. In that case we can fix a partition $P_n:=\{0=t_0^n<t_1^n<\ldots<t_{m_n}^n=1\}$ verifying $|P_n|=\max_k\{t_{k+1}^n-t_k^n\}\rightarrow 0$ and for all $k,n$, $t_k^n$ does not belong to the singularities of $\mathscr{A}'$. One can set 
$$
H_n(t):=\sum_{k=0}^{m_n-1}H(t_k^n)\mathbf{1}_{[t_k^n,t_{k+1}^n)}(t),
$$
which is a step process. One can verify that, for almost every $\omega$, $\int_0^1(H_n-H)^2\mathscr{A}'(t)dt$ tends to $0$ by the dominated convergence theorem. Using boundedness one more time, one finds that $\mathbb{E}\left(\int_0^1(H_n-H)^2\mathscr{A}'dt\right)\rightarrow 0$. If $H$ is only supposed to be bounded, we can set 
$H_n(t)=n\int_{t-\frac1n}^tH(\tau)d\tau,$
which is continuous, bounded and nonanticipating. Hence, the same calculation shows that, we can approximate $H$ by a sequence $H_n$ which has the properties made in the first part of the proof. Hence, $H$ can also be approximated by a step process in that case. If now $H$ is only supposed to be $L^2$, we set $H_n(t)=\min\{n,\max\{H(t),-n\}\}$ and, in order to conclude, we need to verify that it approximates $H$ in the expected sense. One has
$\int_0^1(H_n-H)^2\mathscr{A}'dt=\int_{|H|>n}(n-|H|)^2\mathscr{A}'dt$ which goes to $0$ as well as its expectation by the dominated convergence Theorem.
\end{proof}
\end{comment}

With this lemma at hand, one can thus define $\int_0^1 H dW$ as an element in $L^2(\Omega)$ by approximating $H$ by step processes. This is the so-called It\^o integral (for the reparametrized Brownian $W$) and it verifies as well~\eqref{e:basic-properties-Ito}. For our application, we also define the map $X:t\mapsto \int_0^tHdW$ for $H$ a nonanticipating element in $\mathbb{L}^2([0,1])$. This is as well nonanticipating and one can verify that it is almost surely continuous and that it defines a martingale, meaning
\begin{equation}\label{e:martingale}
\forall 0\leq s\leq t\leq 1,\quad X(s)=\mathbb{E}(X(t)|\mathcal{F}_s). 
\end{equation}
More generally, suppose now that, for all $0\leq s\leq t\leq 1$,
$$
X(t)=X(s)+\int_s^tF(\tau)d\tau+\int_s^tHdW,
$$
with $F\in\mathbb{L}^1$ and $H\in\mathbb{L}^2$. For short, one usually writes $dX=Fdt+HdW$ and we would like to explain how the classical It\^o's formula reads in our setting. Again, the singularities of $\mathscr{A}$ do not play any role at this stage. Assume that $u:\mathbb{R}^2\rightarrow\R$ is a smooth function. It\^o's formula describes the stochastic differential of $Y(t)=u(t,X(t))$. To begin with, let us remark that, from~\eqref{e:derivative-square}, one has $d(W^2)=2WdW+\mathscr{A}'(t)dt$. One can also verify from the definitions that $d(tW)=tdW+Wdt$ as in the usual case. A key step in proving It\^o's formula is to compute $d(X_1X_2)$ when $dX_i=F_idt+H_idW$. The same calculation as in~\cite[Ch.~4]{Evans} shows that
$$
d(X_1X_2)=X_1dX_2+X_2dX_1+\mathscr{A}'(t)H_1H_2 dt.
$$
Even if $\mathscr{A}'$ is not defined everywhere, we emphasize that this formula has to be understood in an integral sense with test functions $H_1$ and $H_2$ in $\mathbb{L}^2([0,1])$. Once this product formula is proved, one can derive the so-called It\^o's formula:
\begin{equation}\label{e:reparametrized-Ito-formula}
\forall u\in\mathcal{C}^{\infty}([0,1]\times\mathbb{R}),\quad d\left(u(t,X(t))\right)=\partial_tu dt+\partial_xu(t,X) dX+\frac{\mathscr{A}'(t)}{2}\partial_x^2u(t,X) H^2 dt.
\end{equation}
We refer to~\cite{Evans} for detailed explanations when $\mathscr{A}'=1$ and to~\cite{RevuzYor} for the general setting.

\subsection{The Lie algebra setting}

This construction can be generalized in the multidimensional case. For instance, we can set that $\mathfrak{W}=\sum_{\ell=1}^LW_\ell\mathfrak{b}_\ell$, where $W_\ell$ are independent reparametrized Brownian motions verifying the assumption of previous section with respect to the same function $\mathscr{A}$ and where $(\mathfrak{b}_\ell)_{1\leq \ell\leq L}$ is an orthonormal basis of $\mathfrak{g}$. We say that $\mathfrak{W}$  is a reparametrized $\mathfrak{g}$-valued Brownian motion. It satisfies in particular the same properties as the geometric processes $\mathfrak{W}_{\psi,\gamma}$ appearing in Section~\ref{s:integration-connection}.

Recall that we denoted by $L$ the real dimension of $\mathfrak{g}$ and that the corresponding compact Lie group $G$ is a group of matrices i.e. included in a real vector space. In particular, $\mathfrak{W}$ can be identified with an element in $\text{M}_N(\R)^{ L}$. We are then interested in $\text{M}_N(\mathbb{R})$-valued stochastic processes. The corresponding reference $\sigma$-algebra $\mathcal{F}_t$ for our processes is the one generated by $\{\mathfrak{W}_s^{-1}(B):\ B\ \text{Borel set of } \mathfrak{g},\ 0\leq s\leq t\}$. Following the construction in the $1$-dimensional case, we can make sense
$$
\mathbf{X}(t)=\mathbf{X}_0+\int_0^t\mathbf{F}d\tau+\sum_{\ell=1}^L\int_0^t \mathbf{H}_\ell dW_{\ell,\tau} ,
$$
where $\mathbf{X}_0\in \text{M}_N(\mathbb{R})$ is a $\text{M}_N(\R)$-valued random variable, where, $\mathbf{H}=(\mathbf{H}_\ell)_{1\leq \ell\leq L}\in \mathbb{L}^2([0,1],\text{M}_{N}(\mathbb{R})^{ L})$ and where $\mathbf{F}\in \mathbb{L}^1([0,1],\text{M}_{N}(\mathbb{R}))$. Again, the definition of the $\mathbb{L}^p$ space involves the presence of the weight $\mathscr{A}'$ in the integral defining the norm and one uses the short notation:
\begin{equation}\label{e:admissible-sto-integral}
d\mathbf{X}=\mathbf{F}dt+\sum_{\ell=1}^L\mathbf{H}_{\ell} dW_\ell,\quad\mathbf{X}(0)=\mathbf{X}_0.
\end{equation}
For $u\in\mathcal{C}^\infty([0,1]\times\text{M}_N(\R),\mathbb{R})$, It\^o's formula takes following form:
\begin{equation}\label{e:ito-formula-higher-dimension}
d\left(u(t,\mathbf{X}(t))\right)=\partial_tu dt+ \partial_xu(t,\mathbf{X})\left(d\mathbf{X}\right)+\frac{\mathscr{A}'(t)}{2}\sum_{\ell=1}^L\partial^2_xu(t,\mathbf{X}) (\mathbf{H}_{\ell},\mathbf{H}_\ell) dt,
\end{equation}
where $\partial_x$ and $\partial_x^2$ indicates that we pick the standard derivatives with respect to $x$.

Finally, we indicate that it is also useful to introduce the so-called Stratonovich stochastic differential (or integral) which, contrary to It\^o's one satisfies the chain rule (at the expense of losing the martingale property~\eqref{e:martingale}). To do that, we set
$$
\mathbf{X}_j(t)=\mathbf{X}_{j,0}+\int_0^t\mathbf{F}_jd\tau+\sum_{\ell=1}^L\int_0^t\mathbf{H}_{j,\ell} dW_{\ell,\tau},\quad j=1,2.
$$
Using It\^o's formula for each coefficient of the resulting matrices, one has
$$
d(\mathbf{X}_1\mathbf{X}_2)=\mathbf{X}_1d\mathbf{X}_2+d\mathbf{X}_1 \mathbf{X}_2+d\langle\mathbf{X}_1,\mathbf{X}_2\rangle,\quad (\mathbf{X}_1\mathbf{X}_2)(0)=\mathbf{X}_{1,0}\mathbf{X}_{2,0},
$$
where
$$
d\langle \mathbf{X}_1,\mathbf{X}_2\rangle:=\frac{\mathscr{A}'}{2}\left(\sum_{\ell=1}^L \textbf{H}_{1,\ell} \textbf{H}_{2,\ell}\right)dt.
$$
Following~\cite[Def. VI.6.1]{FranchiLeJan}, one defines the Stratonovich integral
$$
\int_0^t \mathbf{X}_1\circ d\mathbf{X}_2:=\int_0^t \mathbf{X}_1 d\mathbf{X}_2+\langle \mathbf{X}_1,\mathbf{X}_2\rangle.
$$
Using It\^o's formula, one finds that this alternative stochastic differential verifies the usual chain rule:
\begin{equation}\label{e:chain-rule-strato}
 u(t,\mathbf{X}(t))=u(0,\mathbf{X}_0)+\int_0^t\partial_\tau u(\tau,\mathbf{X}(\tau)) d\tau+\int_0^t\partial_xu(\tau,\mathbf{X}(\tau))\circ d\mathbf{X},
\end{equation}
for every $u\in\mathcal{C}^\infty([0,1]\times\text{M}_N(\R))$.

\subsection{Resolution of stochastic differential equations with rough coefficients}\label{ss:solution-rough-SDE}

We now discuss the existence of solutions to the following equation:
\begin{equation}\label{e:reparametrized-brownian-general}
\mathbf{X}(t)=\text{Id}+\frac{1}{2}\int_{0}^t\mathscr{A}'(\tau)\mathbf{X}(\tau)C_{\mathfrak{g}}d\tau-\sum_{\ell=1}^L\int_0^t \mathbf{X}(\tau) \mathfrak{b}_\ell dW_{\ell,\tau},
\end{equation}
where $C_\mathfrak{g}$ is the constant matrix $\sum_{\ell=1}^L\mathfrak{b}_\ell^2.$
\begin{rmk}\label{r:SDE-strato} If $\mathbf{X}$ solves~\eqref{e:reparametrized-brownian-general}, then $\mathbf{X}\mathfrak{b}_\ell $ is of the form~\eqref{e:admissible-sto-integral} and the corresponding $\mathbf{H}_\ell$ is given by $\mathbf{X}\mathfrak{b}_\ell^2$. Hence, recalling the definition of $c_{\mathfrak{g}}$, one has
$$
\mathbf{X}(t)=\operatorname{Id}-\sum_{\ell=1}^L\int_0^t \mathbf{X}(\tau) \mathfrak{b}_\ell\circ dW_{\ell,\tau},
$$
or, in short,
$$
d\mathbf{X}=\mathbf{X}\circ d\mathfrak{W},\quad\mathbf{X}(0)=\mathbf{X}_0.
$$
\end{rmk}

%More specifically, we aim at solving~\eqref{e:reparametrized-brownian-general} using an iteration procedure and following the presentation in~\cite[Ch.~5]{Evans}. Compared with that reference, we only consider this linear equation which makes certain aspects of the proof simpler but we pay attention to the presence of the weight $\mathscr{A}'$ while this reference only considers the case $\mathscr{A}'=1$. See~\cite[\S IX.2]{RevuzYor} for the most general setting. The precise statement is the following:
\begin{thm}\label{t:reparametrized-brownian}There is a unique solution $\mathbf{X}(t)$ to~\eqref{e:reparametrized-brownian-general} which belongs to $\mathbb{L}^2([0,1],\operatorname{M}_N(\R))$. Moreover, this solution is almost surely continuous and nonanticipating (with respect to $\mathcal{F}_t$). 
\end{thm}

\begin{proof}
For the sake of completeness for readers less familiar with stochastic differential equations, we review the classical argument used to prove such a theorem. We closely follow the proof of~\cite[Ch.5]{Evans} which deals with the case $\mathscr{A}(t)=t$ and we refer to~\cite[\S IX.2]{RevuzYor} for more general continuous semi-martingales. We mostly pay attention to the differences due to the presence of $\mathscr{A}$ and we refer to~\cite[Ch.~5]{Evans} for the details that are identical.

Let us begin with uniqueness. Suppose we are given two nonanticipating solutions $\mathbf{X}$ and $\widetilde{\mathbf{X}}$ to~\eqref{e:reparametrized-brownian-general} that belong to $\mathbb{L}^2([0,1],\text{M}_N(\R))$. One has
$$
\mathbf{X}(t)-\widetilde{\mathbf{X}}(t)=\frac{1}{2}\int_{0}^t\mathscr{A}'(\tau)\left(\mathbf{X}(\tau)-\widetilde{\mathbf{X}}(\tau)\right)C_{\mathfrak{g}}d\tau-\sum_{\ell=1}^L\int_0^t\left(\mathbf{X}(\tau)-\widetilde{\mathbf{X}}(\tau)\right)\mathfrak{b}_\ell dW_{\ell,\tau}
$$
One has
\begin{multline*}
\mathbb{E}\left(\left\|\mathbf{X}(t)-\widetilde{\mathbf{X}}(t)\right\|^2\right)\leq C_0\mathbb{E}\left(\left(\int_{0}^t\mathscr{A}'(\tau)\left\|\mathbf{X}(\tau)-\widetilde{\mathbf{X}}(\tau)\right\|d\tau\right)^2\right)\\
+C_0\int_0^t\mathbb{E}\left(\left\|\mathbf{X}(\tau)-\widetilde{\mathbf{X}}(\tau)\right\|^2\mathscr{A}'(\tau)\right)d\tau,
\end{multline*}
where $C_0>0$ is some positive constant that depends only on $G$. Here and compared with the standard case, $\mathscr{A}'$ is not bounded but we can still apply Cauchy-Schwarz inequality to get that
$$
\left(\int_{0}^t\mathscr{A}'(\tau)\left\|\mathbf{X}(\tau)-\widetilde{\mathbf{X}}(\tau)\right\|d\tau\right)^2\leq \mathscr{A}(1)\int_{0}^t\mathscr{A}'(\tau)\left\|\mathbf{X}(\tau)-\widetilde{\mathbf{X}}(\tau)\right\|^2d\tau.
$$
From that, we find 
$$
\mathbb{E}\left(\left\|\mathbf{X}(t)-\widetilde{\mathbf{X}}(t)\right\|^2\right)\leq C_0\int_0^t\mathscr{A}'(\tau)\mathbb{E}\left(\left\|\mathbf{X}(\tau)-\widetilde{\mathbf{X}}(\tau)\right\|^2\right)d\tau,
$$
where $C_0>0$ depends now on $G$ and $\mathscr{A}$. We now set $\phi(t):=\mathbb{E}\left(\left\|\mathbf{X}(t)-\widetilde{\mathbf{X}}(t)\right\|^2\right)$ so that $\phi(t)\leq C_0\int_0^t\mathscr{A}'(\tau)\phi(\tau)d\tau.$ As usual, one sets $\psi(t)=\left(C_0\int_0^t\mathscr{A}'(\tau)\phi(\tau)d\tau\right)e^{-C_0\mathscr{A}(t)}\geq 0$ which is a continuous function. Outside the singular points of $\mathscr{A}$, one has $\psi'(t)\leq 0$. Hence, $\psi$ is nonincreasing on the intervals where $\mathscr{A}'$ is continuous. As $\psi(0)=0$ and as $\psi$ is continuous, one can conclude that $\phi\equiv 0$ and thus the uniqueness of the solutions to~\eqref{e:reparametrized-brownian-general}.

Let us now discuss the existence. We set
$$
\mathbf{X}_0=\text{Id},\ \text{and}\ \forall n\geq 0,\ \mathbf{X}_{n+1}(t):=\text{Id}+\frac{1}{2}\int_{0}^t\mathscr{A}'(\tau)\mathbf{X}_n(\tau)C_{\mathfrak{g}}d\tau+\sum_{\ell=1}^L\int_0^t\mathbf{X}_n(\tau)\mathfrak{b}_\ell dW_{\ell,\tau}
$$
As in~\cite{Evans}, we define
$$
d^n(t):=\mathbb{E}\left(\left\|\mathbf{X}_{n+1}(t)-\mathbf{X}_{n}(t)\right\|^2\right)
$$
Let us show by induction that there exists a constant $C_0$ such that, for all $n\geq 0$,
\begin{equation}\label{e:induction-resolution-SDE}
 d_n(t)\leq\frac{(C_0\mathscr{A}(t))^{n+1}}{(n+1)!}.
\end{equation}
We begin with the case $n=0$. One has
$$
d^0(t)\leq\mathbb{E}\left(\left(\frac{1}{2}\int_0^t\|C_{\mathfrak{g}}\|\mathscr{A}'(\tau)d\tau+\sum_{\ell=1}^L\|\mathfrak{b}_\ell\||W_\ell(t)|\right)^2\right)\leq C_0\mathscr{A}(t),
$$
where we use that $W(t)$ has variance equal to $\mathscr{A}(t)$. We now suppose that the result holds true for some $n\geq 0$ and we get
\begin{multline*}
d^{n+1}(t)\leq C_1\mathbb{E}\left(\left(\int_0^t\mathscr{A}'(\tau)\left\|\mathbf{X}_{n+1}(\tau)-\mathbf{X}_{n}(\tau)\right\|d\tau\right)^2\right)\\+C_1\mathbb{E}\left(\left\|\int_0^t\sum_{\ell=1}^L \left(\mathbf{X}_{n+1}(\tau)-\mathbf{X}_{n}(\tau)\right)\mathfrak{b}_\ell dW_{\ell,\tau}\right\|^2\right),
\end{multline*}
for some constant $C_1>0$ depending both on $G$ and $\mathscr{A}$ (but not on $n\geq 0$). Again, we deal with the first term using the Cauchy-Schwarz inequality. Together with the induction hypothesis, this yields an upper bound of the form $\frac{C_1\mathscr{A}(1)}{(n+1)!}\int_0^t\mathscr{A}'(\tau)(C_0\mathscr{A}(\tau))^{n+1}d\tau$ which yields the expected bound by picking $C_0>C_1\mathscr{A}(1)$ from the start. For the second term, we used the properties of our It\^o's integrals and we also obtain the expected upper bound by a similar computation. 

Now we observe that 
\begin{multline*}
\max_{0\leq t\leq 1}\|\mathbf{X}_{n+1}(t)-\mathbf{X}_n(t)\|^2\leq C_1\int_0^1\mathscr{A}'(\tau)\|\mathbf{X}_{n+1}(\tau)-\mathbf{X}_n(\tau)\|^2d\tau\\
+2\max_{0\leq t\leq T}\left\|\sum_{\ell=1}^L\int_0^t(\mathbf{X}_{n+1}(\tau)-\mathbf{X}_n(\tau))\mathfrak{b}_\ell dW_{\ell,\tau}\right\|^2,
\end{multline*}
where we used one more time the Cauchy-Schwarz inequality and where $C_1>0$ depends only on $G$ and $\mathscr{A}$. By construction, It\^o's integrals give rise to martingales so that we can apply the martingale's inequality here. When combined with inequality~\eqref{e:induction-resolution-SDE}, we get that there exists a constant $C_2>0$ such that
$$
\forall n\geq 0,\quad \mathbb{E}\left(\max_{0\leq t\leq 1}\|\mathbf{X}_{n+1}(t)-\mathbf{X}_n(t)\|^2\right)\leq C_2\frac{(C_0\mathscr{A}(1))^n}{n!}.
$$
We can then apply the Borel-Cantelli Lemma (see~\cite[Ch.~5]{Evans} for details) and we find that, almost surely, 
$$
\mathbf{X}_n(t)=\mathbf{X}_0(t)+\sum_{k=0}^{n-1}\left(\mathbf{X}_{k+1}(t)-\mathbf{X}_k(t)\right)
$$
converges uniformly on $[0,1]$. We denote the limit process by $\mathbf{X}(t)$ and it solves~\eqref{e:reparametrized-brownian-general}. By construction, $\mathbf{X}$ is non-anticipating and it remains to check that it belongs to $\mathbb{L}^2$. To see this, we use~\eqref{e:induction-resolution-SDE} one more time to verify that $(\mathbf{X}_n(t))_{n\geq 1}$ is a Cauchy sequence in $L^2(\Omega)$. Hence, for every $t\in[0,1]$, $\mathbf{E}(\|\mathbf{X}(t)\|^2)\leq C$ for some constant $C>0$ that is independent of $t\in[0,1]$. Integrating against $\mathscr{A}'(t)$ gives the expected regularity. 
\end{proof}

\subsection{Properties of the solutions to~\texorpdfstring{\eqref{e:reparametrized-brownian-general}}{the reparametrized Brownian SDE}}

Let us now discuss properties of the solutions to~\eqref{e:reparametrized-brownian-general}.
\begin{lemma}\label{l:reparametrized-law} Let $f:\operatorname{M}_N(\R)\rightarrow\C$ be a smooth function and let $\mathbf{X}(t)$ be the solution to~\eqref{e:reparametrized-brownian-general}. One has, for all $t\in[0,1]$,
$$
f(\mathbf{X}(t))=f(\operatorname{Id})-\sum_{\ell=1}^L\int_0^t\mathcal{L}_{\mathfrak{b}_\ell}f(\mathbf{X}(\tau))dW_{\ell,\tau}+\frac12\sum_{\ell=1}^L\int_0^t\mathscr{A}'(\tau)\mathcal{L}_{\mathfrak{b}_\ell}^2f(\mathbf{X}(\tau))d\tau.
$$
\end{lemma}
The proof of this last property follows an application of It\^o's formula and it is given in~\cite[Th.~VII.2.2]{FranchiLeJan} when $\mathscr{A}'=1$. The proof in our generalized setup is the same. Taking the expectation in the conclusion of Lemma~\ref{l:reparametrized-law}, one gets:
\begin{equation}\label{e:sto-PDE}
\mathbb{E}\left(f(\mathbf{X}(t))\right)=f(\operatorname{Id})+\frac12\int_0^t\mathscr{A}'(\tau)\mathbb{E}\left(\sum_{\ell=1}^L\mathcal{L}_{\mathfrak{b}_\ell}^2f(\mathbf{X}(\tau))\right)d\tau.
\end{equation}
As a consequence of this Lemma, one also finds:
\begin{corollary}\label{c:solution-in-G} Let $\mathbf{X}(t)$ be the solution to~\eqref{e:reparametrized-brownian-general}. Almost surely, $\mathbf{X}(t)$ belongs to $G$ for every $t\in[0,1]$.
\end{corollary}
\begin{proof}
In the Stratonovich convention, the result of Lemma~\ref{l:reparametrized-law} reads, for all $0\leq s\leq t\leq 1$,
$$
f(\mathbf{X}(t))=f(\mathbf{X}(s))-\sum_{\ell=1}^L\int_s^t\mathcal{L}_{\mathfrak{b}_\ell}f(\mathbf{X}(\tau))\circ dW_{\ell,\tau}.
$$
In particular, if $f$ is
such that, for every $1\leq \ell\leq L$, $\mathcal{L}_{\mathfrak{b}_\ell}f=0$, then one has $f(\mathbf{X}(t))=f(\mathbf{X}(s))$. As $G$ is a submanifold of $M_N(\mathbb{C})$ with tangent space isomorphic to $\mathfrak{g}$, it is locally defined by level sets of functions $f$ verifying locally this property. Hence, one deduces the expected result. See~\cite[Th.~VII.2.3]{FranchiLeJan} for more details.
\end{proof}
As a direct consequence, one finds that the law $\nu_t$ of $\mathbf{X}(t)$ is a probability measure on $G$ and~\eqref{e:sto-PDE} shows that $\nu_t$ solves the following partial differential equation:
\begin{equation}\label{e:heat-sto}
 \partial_t\nu_t=\frac{\mathscr{A}'(t)}{2}\sum_{\ell=1}^M\mathcal{L}_{\mathfrak{b}_\ell}^2\nu_t,\quad\nu_0=\delta_{\text{Id}}^G.
\end{equation}
Hence, one has
\begin{corollary}\label{c:heat-kernel} Let $\mathbf{X}(t)$ be the solution to~\eqref{e:reparametrized-brownian-general}. Then, the law of $\mathbf{X}_t$ is given by $p_{\mathscr{A}}(t)\mu_G$ where $\mu_G$ is the normalized Haar measure on $G$ and $p_\tau$ is the heat kernel associated to the elliptic operator
$$
\frac{1}{2}\Delta_G:=\frac{1}{2}\sum_{\ell=1}^M\mathcal{L}_{\mathfrak{b}_\ell}^2
$$
\end{corollary}
Finally, we record the following lemma stating that $\mathbf{X}(t)$ has the properties of a Brownian motion.
\begin{lemma}\label{l:properties-brownian} Let $\mathbf{X}(t)$ be the solution to~\eqref{e:reparametrized-brownian-general}. Then, the following holds.
 \begin{enumerate}
 \item for every $p\geqslant 1$ and for every $0=t_0< t_1< t_2<\ldots< t_p\leq 1$, $\mathbf{X}(t_1)$, $\mathbf{X}(t_1)^{-1}\mathbf{X}(t_2)$, $\ldots$, $\mathbf{X}(t_{p-1})^{-1}\mathbf{X}(t_p)$ are independent;
 \item for all $0\leqslant s\leqslant t\leqslant 1$, $\mathbf{X}(s)^{-1}\mathbf{X}(t)$ has the same distribution as $\mathbf{X}(t-s)$;
 \item for every $\mathrm{h}$ in $G$, $\mathbf{X}(t)$ and $\mathrm{h}\mathbf{X}(t)\mathrm{h}^{-1}$ have the same distribution.
 \end{enumerate}
\end{lemma}
\begin{proof}
 
 For the first property, we fix $s\in[0,1]$ and we observe that $(\mathbf{X}(s)^{-1}\mathbf{X}(t))_{t\geq s}$ solves~\eqref{e:reparametrized-brownian-general} with initial condition $\mathbf{X}(s)=\text{Id}$. As the solution to this equation is unique according to Theorem~\ref{t:reparametrized-brownian}, it is measurable with respect to the $\sigma$-algebra generated by $(\mathfrak{W}(t)-\mathfrak{W}(s))_{t\geq s}$. In particular, from the properties of the $\mathfrak{g}$-valued Brownian, it is independent of the one generated by $\mathfrak{W}(s)$. Hence, $\mathbf{X}(s)^{-1}\mathbf{X}(t)$ and $\mathbf{X}(s)$ are independent. The case $p>2$ follows the same argument. 
 
 For the second property, it follows from Corollary~\ref{c:heat-kernel} combined with the above discussion.

 Finally, the last property can be viewed as a consequence of the invariance of the heat kernel by conjugation -- see~\S\ref{ss:def_lie_group} for explicit expressions.
\end{proof}

\Addresses
\end{document}